\documentclass[11pt]{amsart} 
\usepackage{amsmath} \usepackage{amsxtra}
\usepackage{amscd} \usepackage{amsthm}  
\usepackage{amsfonts}\usepackage{amssymb} 
\usepackage{eucal}\usepackage{bm}
\usepackage{graphicx} 
\usepackage{hyperref}
\makeindex

\newtheorem{Cor}[subsubsection]{Corollary}
\newtheorem{Lm}[subsubsection]{Lemma}
\newtheorem{Pp}[subsubsection]{Proposition}

\newtheorem{Thm}[subsubsection]{Theorem}
\newtheorem{Def}[subsubsection]{Definition}
\newtheorem{Rem}[subsubsection]{Remark}

\theoremstyle{definition}

\theoremstyle{remark}

\newcommand{\nc}{\newcommand}
\nc{\renc}{\renewcommand}
\nc{\ssec}{\subsection}
\nc{\sssec}{\subsubsection}
\nc{\on}{\operatorname}

\nc\ol{\overline}
\nc\wt{\widetilde}
\nc\tboxtimes{\wt{\boxtimes}}

\newcommand{\cA}{{\mathcal A}}
\newcommand{\cB}{{\mathcal B}}
\newcommand{\cC}{{\mathcal C}}
\newcommand{\cD}{{\mathcal D}}
\newcommand{\cH}{{\mathcal H}}
\newcommand{\cE}{{\mathcal E}}
\newcommand{\cG}{{\mathcal G}}
\newcommand{\cI}{{\mathcal I}}
\newcommand{\cJ}{{\mathcal J}}
\newcommand{\cO}{{\mathcal O}}
\newcommand{\cL}{{\mathcal L}}
\newcommand{\cM}{{\mathcal M}}

\newcommand{\cF}{{\mathcal F}}
\newcommand{\cK}{{\mathcal K}}
\newcommand{\cP}{{\mathcal P}}

\newcommand{\cS}{{\mathcal S}}
\newcommand{\cT}{{\mathcal T}}

\newcommand{\cX}{{\mathcal X}}
\newcommand{\cY}{{\mathcal Y}}
\newcommand{\cZ}{{\mathcal Z}}

\newcommand{\NN}{{\mathbb N}}
\newcommand{\ZZ}{{\mathbb Z}}

\newcommand{\PP}{{\mathbb P}}

\nc{\gi}{\mathfrak{i}}

\newcommand{\gU}{\mathfrak{U}}

\nc{\gM}{\mathfrak{M}}
\nc{\gV}{\mathfrak{V}}
\nc{\gE}{\mathfrak{E}}
\nc{\gL}{\mathfrak{L}}

\nc{\bC}{{\mathbf C}}
\nc{\uZ}{\underline{\cZ}}

\newcommand{\Rep}{{\on{Rep}}}
\newcommand{\Sch}{{\on{Sch}}}

\newcommand{\Gm}{\mathbb{G}_m}

\newcommand{\toup}[1]{\stackrel{#1}{\to}}

\newcommand{\hook}[1]{\stackrel{#1}{\hookrightarrow}}

\newcommand{\getsup}[1]{\stackrel{#1}{\gets}}

\newcommand{\IC}{\on{IC}}

\newcommand{\Hom}{\on{Hom}}

\newcommand{\Aut}{\on{Aut}}

\newcommand{\RG}{\on{R\Gamma}}

\newcommand{\Spec}{\on{Spec}}

\newcommand{\supp}{\on{supp}}

\newcommand{\HOM}{{{\mathcal H}om}}

\newcommand{\Gr}{\on{Gr}}

\newcommand{\Sets}{\on{Sets}}

\newcommand{\Fun}{{\on{Fun}}}

\newcommand{\pr}{\on{pr}}
\newcommand{\id}{\on{id}}

\newcommand{\QED}{$\square$} 
  
\newcommand{\iso}{{\widetilde\to}}

\newcommand{\comp}{\circ}

\renewcommand{\H}{{\on{H}}}   

\newcommand{\DD}{\mathbb{D}}  
\newcommand{\D}{\on{D}}       
\newcommand{\ov}[1]{\overline{#1}}
\newcommand{\select}[1]{{\it{#1}}}

\newcommand{\und}[1]{\underline{#1}}

\newcommand{\ev}{\mathit{ev}}

\newcommand{\Conv}{\on{Conv}}
\newcommand{\Loc}{\on{Loc}}

\newcommand{\Sph}{\on{Sph}}

\newcommand{\ttimes}{\tilde\times}

\newcommand{\act}{\on{act}}

\newcommand{\Funct}{\on{Funct}}

\renewcommand{\Im}{\on{Im}}

\newcommand{\Vect}{\on{Vect}}

\newcommand{\Ind}{\on{Ind}}

\newcommand{\ra}{\rightarrow}

\nc{\Perv}{\on{Perv}}

\nc{\Gra}{\on{Gra}}
\nc{\PPerv}{\on{{\PP}erv}}

\nc{\oX}{\overset{\scriptscriptstyle\circ}{X}}
\nc{\ocL}{\overset{\scriptscriptstyle\circ}{\cL}}
\nc{\gRes}{\on{gRes}}
\nc{\Sign}{\on{Sign}}
\nc{\goodat}{\rm{good\, at}}
\nc{\Whit}{\on{Whit}}
\nc{\add}{\on{add}}
\nc{\FS}{\on{FS}}
\nc{\oo}[1]{\overset{\scriptscriptstyle\circ}{#1}}
\nc{\can}{\on{can}}
\nc{\summ}{\on{sum}}
\nc{\SiSu}{\on{SS}}
\nc{\Irr}{\on{Irr}}
\nc{\Hecke}{\on{Hecke}}
\nc{\oHecke}{\overset{\scriptstyle\bullet}{\Hecke}}
\nc{\og}[1]{\overset{\scriptscriptstyle\bullet}{#1}}
\nc{\of}{\overset{\scriptstyle\bullet}{f}}
\nc{\Exp}{\on{{\mathcal E}xp}}
\nc{\Chain}{\on{Chain}}
\nc{\Map}{\on{Map}}
\nc{\cSet}{\on{{\mathcal S}et}}
\nc{\Cat}{\on{\mathcal{C}at}}
\nc{\bfitDelta}{\bm{\mathit{\Delta}}}
\nc{\Grpd}{\on{Grpd}}
\nc{\Kan}{\on{{\mathcal K}an}}
\nc{\Spc}{\on{Spc}}
\nc{\Yon}{\on{Yon}}
\nc{\colim}{\on{colim}}
\nc{\Fin}{\on{{\mathcal F}in}}
\nc{\Alg}{\on{Alg}}
\nc{\Triv}{\on{Triv}}
\nc{\Grp}{\on{{\mathcal G}rp}}
\nc{\Mon}{\on{Mon}}
\nc{\Disc}{\on{Disc}}
\nc{\PreStk}{\on{PreStk}}
\nc{\Stk}{\on{Stk}}
\nc{\DGCat}{\on{DGCat}}
\nc{\AGCat}{\on{AGCat}}
\nc{\Ab}{\on{Ab}}
\nc{\ComGrp}{\on{{\mathcal C}omGrp}}
\nc{\Ptd}{\on{Ptd}}
\nc{\Shv}{\on{Shv}}
\nc{\DG}{\on{DG}}
\nc{\cLS}{\on{\mathcal L\mathcal S}}
\nc{\Ge}{\on{Ge}}
\nc{\Ran}{\on{Ran}}
\nc{\Surj}{\on{Surj}}
\nc{\FactGe}{\on{FactGe}}
\nc{\Quad}{\on{Quad}}
\nc{\detrel}{\on{detrel}}
\nc{\Fact}{\on{Fact}}
\nc{\Sptr}{\on{Sptr}}
\nc{\cPr}{\on{{\mathcal P}r}}
\nc{\Tw}{\on{Tw}}
\nc{\QCoh}{\on{QCoh}}
\nc{\uMap}{\und{\Map}}
\nc{\cTw}{{\mathcal T}w}
\nc{\FactPic}{\on{FactPic}}
\nc{\oblv}{\on{oblv}}
\nc{\com}{\on{com}}
\nc{\Glue}{\on{Glue}}
\nc{\Sat}{\on{Sat}}
\nc{\coind}{\on{coind}}
\nc{\counit}{\on{coun}}
\nc{\Conf}{\on{Conf}}
\nc{\Cart}{\on{Cart}}
\nc{\Cent}{\on{Cent}}

\makeatletter
\newcommand*{\doublerightarrow}[2]{\mathrel{
  \settowidth{\@tempdima}{$\scriptstyle#1$}
  \settowidth{\@tempdimb}{$\scriptstyle#2$}
  \ifdim\@tempdimb>\@tempdima \@tempdima=\@tempdimb\fi
  \mathop{\vcenter{
    \offinterlineskip\ialign{\hbox to\dimexpr\@tempdima+1em{##}\cr
    \rightarrowfill\cr\noalign{\kern.5ex}
    \rightarrowfill\cr}}}\limits^{\!#1}_{\!#2}}}
\newcommand*{\triplerightarrow}[1]{\mathrel{
  \settowidth{\@tempdima}{$\scriptstyle#1$}
  \mathop{\vcenter{
    \offinterlineskip\ialign{\hbox to\dimexpr\@tempdima+1em{##}\cr
    \rightarrowfill\cr\noalign{\kern.5ex}
    \rightarrowfill\cr\noalign{\kern.5ex}
    \rightarrowfill\cr}}}\limits^{\!#1}}}
\makeatother

\usepackage{lipsum}

\begin{document}

\title{Note on factorization categories}
\author{Sergey Lysenko}
\address{Institut Elie Cartan Lorraine, Universit\'e de Lorraine, 
B.P. 239, F-54506 Vandoeuvre-l\`es-Nancy Cedex, France}
\email{Sergey.Lysenko@univ-lorraine.fr}
\date{9 November, 2026}
\begin{abstract}
We systematically study the commutative factorization categories over the Ran space uniformly for $\cD$-modules and the constructible setting.
We summarize four different constructions of commutative factorization sheaves of categories on $\Ran$ and establish relations between them. We also generalize the Drinfeld-Pl\"ucker formalism from \cite{Gai19Ran, DL}. 

 We also study the commutative factorization categories $\Fact(C)$ with an additional grading of $C$. We apply our results to recast the construction of the factorizable Satake functor for a connected reductive group $G$ uniformly for $\cD$-modules and the constructible context. We also relate the versions of the Satake functors for the Ran space with that of the configuration space of colored divisors on a curve. 
\end{abstract}
\thanks{It is a pleasure to thank Sam Raskin for answering my questions and very useful email correspondence. I am grateful to Gurbir Dhillon and Misha Finkelberg for many fruitful conversations.}
\maketitle
\tableofcontents

\section{Introduction}

\sssec{} This book is intended for people working on or studying the geometric representation theory (and, more specifically, the geometric Langlands program).

Let $X$ be a smooth projective connected curve over an algebraically closed field $k$, $\Ran$ be the Ran space of $X$. The factorization categories over $X$ (in the setting of $\cD$-modules) played an important role in the recent proof of the geometric Langlands conjecture in \cite{GLC1, GLC2, GLC4, GLC4, GLC5}. 

  The reason to use $\cD$-modules (as opposed to other sheaf theories used in \cite{GLys}) is the presence of the K\"unneth formula for the category of sheaves on the product of two schemes $S_i$ of finite type. In the constructible context the natural functor $Shv(S_1)\otimes Shv(S_2)\to Shv(S_1\times S_2)$ is emphatically not an equivalence. For this reason it seems impossible to develop the theory of factorization categories in general in the constructible context.\footnote{For example, it is not clear if for reductive group $G$ the category of spherical sheaves on the affine grassmannian $\Sph_G$ is a factorization category in the constructible context.} 
  
  However, for commutative factorization categories it turns out possible to develop such a theory uniformly for $\cD$-modules and constructible context. This is the goal of this book. 
    
 We should say that the authors of \cite{GRV} have adapted another approach, which is more fundamental, but also substantially more complicated. Namely, they introduced certain symmetric monoidal $(\infty, 2)$-category $\AGCat$ together with a symmetric monoidal functor $\DGCat_{cont}\to\AGCat$, which allows to `turn' the K\"unneth formula into an equivalence in some universal way. 
 
 Our point is that for commutative factorization categories one may, with some additional work, remain in the more classical world of $\DGCat_{cont}$ instead of $\AGCat$. The application that we have in mind are direct construction of some automorphic sheaves in the constructible setting (such as residues of geometric Eisenstein series in \cite{DL3}), which further allow to directly take trace of Frobenius and get new autmorphic functions (without passing through the formalism of categorical traces and related recent technics of Gaitsgory, Raskin and their coauthors).
 
 This book is essentially self-contained. The exposition is aimed to be accessible by young researchers and PhD students in the area. 

\sssec{}  Work with a sheaf theory as in \cite{GLys}, which is either $\cD$-modules or a sheaf theory in the constructible context. We systematically study the commutative factorization categories over $\Ran$ attached to commutative algebras $C(X)$ in $Shv(X)$-module categories and related constructions. Some of these constructions have appeared in \cite{GLys}, where the proofs were left to a reader. One of our motivations is to provide complete proofs to facilitate reading of \cite{GLys}. It is also our purpose to provide a convenient reference for general results on the commutative factorization categories (as already used, for example, in \cite{DL2}).

 A detailed description of our results is found below in this introduction, the most important among them are: 
\begin{itemize} 
\item[i)] equivalence of different constructions of $\Fact(C)$; 
\item[ii)] development of some version of the Drinfeld-Pl\"ucker formalism more general than that of \cite{Gai19SI, Gai19Ran}; 
\item[iii)] structural results on $\Fact(C)$ under an additional assumption that $C(X)$ is equipped with a grading compatible with the symmetric monoidal structure on $C(X)$. 
\end{itemize} 
 
 Though the constructions of commutative factorization categories in Sections~\ref{Sect_Fact_cat_for_com_algebras} and \ref{Sect_4} go back to Gaitsgory-Lysenko \cite{GLys} and Raskin \cite{R2}, most of the results of this book are not available in the published litterature to the best of our knowledge (though certain are maybe known to some experts as folklore claims). 
 
\sssec{} We summarize several different constructions of commutative factorization categories and establish their basic properties and relations between them.
 
 Let $fSets$ be the category of finite non-empty sets, where the morphisms are surjections.
 
 If $C(X)$ is a non-unital commutative algebra in $Shv(X)-mod$ (cf. Section~\ref{Sect_Conventions} for notations) then following D. Gaitsgory one associates to it a factorization sheaf of categories $\Fact(C)$ on $\Ran$. This constructed is studied in Section~\ref{Sect_Fact_cat_for_com_algebras}, it uses closed subsets of $X^I$ for $I\in fSets$. 

 If $C(X)$ is a non-unital cocommutative coalgebra in $Shv(X)-mod$ then one associates to it a sheaf of categories $\Fact^{co}(C)$ on $\Ran$, which is a factorization sheaf of categories under some additional assumptions. This construction is studied in Section~\ref{Sect_Fact_cat_for_cocom_coalg}, it is obtained in a manner dual to that of Section~\ref{Sect_Fact_cat_for_com_algebras}. 
 
 If $C$ is a non-unital commutative algebra in $\DGCat_{cont}$ then following S. Raskin one associates to it a sheaf of categories $\ov{\Fact}(C)$ on $\Ran$, which is a factorization sheaf of categories under mild additional assumptions. This construction is studied in Section~\ref{Sect_4}. It uses a different combinatorics, namely, open subsets of $X^I$ for $I\in fSets$ instead of closed ones.
 
 If $C$ is a non-unital cocommutative coalgebra in $\DGCat_{cont}$ then one associated to it a factorization sheaf of categories $\ov{\Fact}^{co}(C)$ on $\Ran$ in a manner dual to that of Section~\ref{Sect_4}. This construction is studied in Section~\ref{Sect_5}. 

\sssec{} 
\label{Sect_1.0.3_now}
We study basic properties of the above sheaves of categories and construct natural functors between them. One of our main results is Theorem~\ref{Thm_5.1.11} which claims that under some mild assumptions all the four constructions produce the same category. 
For the convenience of the reader, the precise assumptions in Theorem~\ref{Thm_5.1.11} are as follows: 
\begin{itemize}
\item[] $C$ is a unital cocommutative coalgebra in $\DGCat_{cont}$ for which the comultiplication map $\com: C\to C^{\otimes 2}$ admits a left adjoint, the counit map $C\to \Vect$ admits a left adjoint, and $C$ is compactly generated.
\end{itemize}

 To obtain Theorem~\ref{Thm_5.1.11}, we extend to the constructible context a beautiful result of S. Raskin (\cite{Ras2}, Lemma~6.18.1), whose proof in \select{loc.cit.} works only for $\cD$-modules. The corresponding result is our Proposition~\ref{Pp_4.1.10_Raskin_dualizability}, its proof in Appendix~\ref{Sect_proof_Pp_4.1.10_Raskin_dualizability} is maybe of independent interest.  
 
\sssec{} Let un illustrate this by an example. Let $G$ be a connected reductive group over the base field $k$. Let $e$ be field of coefficients of our sheaf theory, $\check{G}$ be the Langlands dual group of $G$ over $e$. Take $C=\Rep(\check{G})$. 

 The product $m: C^{\otimes 2}\to C$ admits a right adjoint $m^R: C\to C^{\otimes 2}$ in $\DGCat_{cont}$. We may view $(C, m)$ as a commutative algebra in $\DGCat_{cont}$, $(C, m^R)$ as a cocommutative coalgebra in $\DGCat_{cont}$. We may similarly view $C(X)=C\otimes Shv(X)$ as a commutative algebra (resp., cocommutative coalgebra) in $Shv(X)-mod$.

 By Theorem~\ref{Thm_5.1.11}, we have canonical equivalences
$$
\ov{\Fact}^{co}(C)\,\iso\, \Fact^{co}(C)\,\iso\,\ov{\Fact}(C)\,\iso\,\Fact(C)
$$

\sssec{} Theorem~\ref{Thm_5.1.11} is not immediate: for example, let $C$ be a non-unital commutative algebra in $\DGCat_{cont}$ and $I\in fSets$. Consider the natural map $X^I\to \Ran$. By definition, the global sections $\Gamma(X^I, \ov{\Fact}(C))$ are given as a limit of some functor 
$$
\bar\cG_{I, C}: \Tw(I)^{op}\to Shv(X^I)-mod
$$ 
(cf. Section~\ref{Sect_4}). Assume the multiplication $m: C^{\otimes 2}\to C$ has a continuous right adjoint $m^R$. Then $(C, m^R)$ is a non-unital cocommutative coalgebra in $\DGCat_{cont}$, and by definition
$\Gamma(X^I, \ov{\Fact}^{co}(C))$ is the colimit of the functor 
$$
\bar\cG_{I, C}^R: \Tw(I)\to Shv(X^I)-mod
$$ 
obtained from $\bar\cG_{I, C}$ by passing to the right adjoints in $Shv(X^I)-mod$. In general, there is no reason for the latter colimit to coincide with the previous limit. 

\sssec{} Let $G$ be a connected reductive group over $k$. Our Proposition~\ref{Pp_4.1.10_Raskin_dualizability} is applied in Section~\ref{Sect_Application to Satake} to extend the construction of the factorizable Satake functor $\Sat_{G,\Ran}$ from the $\cD$-module case (obtained in \cite{Ras2}) to the constructible context. We also review the construction of the chiral Hecke algebra in a way uniform for our sheaf theories. 
 
 This fills in what we consider as a gap in the construction of the factorizable Satake functor $\Sat_{G, \Ran}: \Fact(\Rep(\check{G})\to \Sph_{G,\Ran}$ in the constructible context in \cite{Gai19Ran}, \cite{GLys} (cf. Section~\ref{Sect_Application to Satake} for the notation). In (\cite{Ras2}, Section~6) the functor $\Sat_{G, \Ran}$ is defined in the setting of $\cD$-modules, but some proofs in \select{loc.cit.} do not apply in the constructible context. In Section~\ref{Sect_Application to Satake} we explain how to modify the argument from \select{loc.cit.} so that it works uniformly for $\cD$-modules and in the constructible context to get the functor $\Sat_{G,\Ran}$.

\ssec{Detailed overview of results}
\label{Sect_Overview of results}

\sssec{} The content of Section~\ref{Sect_Fact_cat_for_com_algebras} is as follows. We show that $\Ran$ is 1-affine for our sheaf theories. This is a consequence of a more general result in Appendix~\ref{Sect_Safe pseudo-schemes}, where we introduce a notion of a \select{safe pseudo-scheme}. We show that any safe pseudo-scheme is 1-affine for our sheaf theories (this is a new result of independent interest in the constructible context).

Given $C(X)\in CAlg^{nu}(Shv(X)-mod)$, we define a factorization sheaf of categories $\Fact(C)$ over $\Ran$, equip it with the commutative chiral product. We equip $\Fact(C)$ with three distinct non-unital symmetric monoidal structures: $\star, \otimes^{ch}, \otimes^!$.
Here $\otimes^{ch}$ is called the chiral symmetric monoidal structure.

 For example, if $C(X)=Shv(X)$ then $\Fact(C)\,\iso\,Shv(\Ran)$ naturally, and the $\otimes^!$-symmetric monoidal structure on $\Fact(C)$ is the usual one on $Shv(\Ran)$, which justifies our notation $\otimes^!$. The existence of $\otimes^!$ comes from the fact that the functor $CAlg^{nu}(Shv(X)-mod)\to Shv(\Ran)-mod$, $C(X)\mapsto \Fact(C)$ itself is equipped with a non-unital symmetric monoidal structure. The non-unital symmetric monoidal structure $\otimes^!$ on $\Fact(C)$ is compatible with factorization. 
 
\sssec{} In Section~\ref{Sect_2.2.19_now} we explain that one may talk about 'sheaves of categories of $\Fact(C)$-modules over $\Ran$'. We show that $\Ran$ (equipped with the sheaf of commutative algebras $(\Fact(C), \otimes^!)$) is also 1-affine. That is, a sheaf of categories of $\Fact(C)$-modules over $\Ran$ is the same as its category of global sections.
 
 Further we discuss the commutative factorization algebras in $\Fact(C)$ and their properties. We explain that they can be viewed as certain non-unital commutative algebras in $(\Fact(C), \star)$ and also as non-unital commutative algebras in $(\Fact(C),\otimes^!)$. We add also unital version of some of the previous results. 
 
 We then study further properties of $\Fact(C)$ as dualizability, interaction with the duality, ULA property, interaction with the involution $(C(X), m)\mapsto (C(X)^{\vee}, (m^R)^{\vee})$, where $m: C(X)^{\otimes 2}\to C(X)$ is the multiplication (provided that $m$ has a $Shv(X)$-linear continuous right adjoint).  
 
\sssec{} We then study properties of some functors between the commutative factorization categories in Section~\ref{Sect_2.7_Properties of some functors}. In particular, given a monadic functor $D(X)\to C(X)$ of commutative algebras in $Shv(X)-mod$, we propose a sufficient condition for the functor obtained from it by `speading over $\Ran$' or over $X^I$ to remain monadic. We show that if $C(X)$ is relatively rigid over $X$ then $C_{X^I}$ (resp., $\Fact(C)$) is relatively rigid over $Shv(X^I)$ (resp., over $Shv(\Ran)$ (cf. Corollary~\ref{Cor_2.7.7}). The notion of relative rigidity generalizes a notion of a rigid monoidal category from (\cite{G}, ch. I.1, 9.1.2), we study this notion in details in Appendix~\ref{Sect_Appendix_G}. 
 
 Next, we explain that $(\Fact(C), \otimes^{ch})$ is pro-nilpotent in the sense of (\cite{FG}, 4.1.1). We apply this to extend some results of \cite{FG} to our setting. Namely, one may define the category of factorization cocommutative coalgebras in $(\Fact(C), \otimes^{ch})$ and establish a version of Koszul duality for this category (cf. Corollary~\ref{Cor_Koszul_duality}). More generally, for a reduced operad $\cP$ and its dual cooperad $Q=\cP^{\vee}$ (in the sense of \cite{G}, IV.2) one may consider the corresponding categories of $\cP$-algebras (resp., $Q$-coalgebras) in $(\Fact(C), \otimes^{ch})$ and $(\Fact(C), \star)$. We explain that the corresponding version of the basic diagram for them commutes as in \cite{FG}, that is, the Koszul duality functors intertwine induction and restriction (cf. Proposition~\ref{Pp_basic_diagram_2.8.10}). This uses the fact that $\id: (\Fact(C), \otimes^{ch})\to (\Fact(C), \star)$ is right-lax non-unital symmetric monoidal.
 
\sssec{}  In Section~\ref{Sect_Filtration_on_Fact(A)} we explain how a filtration on an algebra $A$ yields a filtration on $\Fact(A)$, this uses the ideas from \cite{G} and the Day convolution product. In Section~\ref{Sect_special_constr_context} we collected some properties of $\Fact(C)$ special to the constructible context (they do not hold for $\cD$-modules). 
  
   In Section~\ref{Sect_Fact_cat_for_cocom_coalg} we study the construction dual to that of Section~\ref{Sect_Fact_cat_for_com_algebras}. It attaches to $C(X)\in CoCAlg^{nu}(Shv(X)-mod)$ a sheaf of categories $\Fact^{co}(C)$. By construction, there are natural candidates for the factorization morphisms. In this sense the factorization structure on $\Fact^{co}(C)$ is not an additional structure, but a property. We give a sufficient condition for this property to hold. We then define the chiral comultiplication for $\Fact^{co}(C)$. 
\sssec{} In Section~\ref{Sect_4} we associate to $C\in CAlg^{nu}(\DGCat_{cont})$ a sheaf of categories $\ov{\Fact}(C)$ on $\Ran$. We equip it with the commutative chiral product, so that the factorization structure on $\ov{\Fact}(C)$ becomes a property (not an additional structure). We give a sufficient condition for this property to hold (cf. Remark~\ref{Rem_4.1.6}). 

We equip the functor $CAlg^{nu}(\DGCat_{cont})\to Shv(\Ran)-mod$, $C\mapsto \ov{\Fact}(C)$ with a non-unital right-lax symmetric monoidal structure (compatible with the chiral product). We construct a canonical arrow 
\begin{equation}
\label{map_intriduction_Fact(C)_to_ov(Fact)(C)}
\Fact(C)\to \ov{\Fact}(C)
\end{equation}
compatible with factorization.

 Given a t-structure on $C$ (satisfying some minor assumptions, cf. Section~\ref{Sect_4.1.12_now_t-str}) we equip $\ov{\Fact}(C)$ with an induced t-structure. We make observations about some cohomology groups of commutative factorization algebras with respect to this t-structure (Remark~\ref{Rem_4.1.13_now}). 
 
 Proposition~\ref{Pp_4.1.10_Raskin_dualizability} is one of our main technical results. It extends to the constructible context a beautiful result of Raskin (\cite{Ras2}, 6.18.1), whose proof in \select{loc.cit.} holds for $\cD$-modules only. It allows to conclude that $\ov{\Fact}(C)$ is a factorization category provided that $C$ is dualizable in $\DGCat_{cont}$ and the product $m; C^{\otimes 2}\to C$ has a continuous right adjoint. 
 
 We then analyze the interaction of $\ov{\Fact}(C)$ with duality. We give a sufficient condition for the arrow (\ref{map_intriduction_Fact(C)_to_ov(Fact)(C)}) to be an equivalence. 
 
\sssec{} In Section~\ref{Sect_5} we attach to $C\in CoCAlg^{nu}(\DGCat_{cont})$ a sheaf of categories $\ov{\Fact}^{co}(C)$ on $\Ran$ (cf. Section~\ref{Sect_Conventions} for notations). We equip it with a commutative chiral coproduct and show it is a factorization sheaf of categories (cf. Lemma~\ref{Lm_5.1.6_now}). We further define a canonical arrow $\ov{\Fact}^{co}(C)\to \Fact^{co}(C)$. We then establish our main Theorem~\ref{Thm_5.1.11} compairing all the four constructions. 

\sssec{Spreading right-lax functors} 
\label{Sect_Spreading right-lax functors_intro}
In Section~\ref{Sect_Spreading right-lax symmetric monoidal functors} we generalize a construction of Gaitsgory from (\cite{Gai19Ran}, Section~2.6). Informally speaking, it allows to spread along the $\Ran$ space certain right-lax non-unital symmetric monoidal functors given originally over $X$ in a factorizable way (so extending the construction of commutative factorization algebras in commutative factorization categories). 
 
 We apply it to generalize the result of D.~Gaitsgory (\cite{Gai19Ran}, Proposition~5.4.7), which allowed in \select{loc.cit.} to present the Ran space version of the semi-infinite $\IC$-sheaf $\IC^{\frac{\infty}{2}}_{\Ran}$ as certain colimit. This new version of the Drinfeld-Pl\"ucker formalism plays an essential role in \cite{DL2}.
 
  To explain the idea, let $\Lambda$ be a non-unital commutative algebra in $\Sets$. Let $(C(X),\otimes^!)$ be a non-unital commutative algebra in $Shv(X)-mod$, and 
$$
\Irr: \Lambda\to (C(X), \otimes^!), \; \lambda\mapsto V^{\lambda}
$$ 
a right-lax non-unital symmetric monoidal functor. The sheaf of categories $\Fact(C)$ on $\Ran$ ammounts to a datum of $C_{X^I}\in Shv(X^I)-mod$ for $I\in fSets$ together with isomorphisms 
$$
C_{X^I}\otimes_{Shv(X^I)} Shv(X^J)\,\iso\, C_{X^J}
$$ 
for maps $\phi: I\to J$ in $fSets$ together with coherent higher compatibilities.

 Similarly, we `spred' $\Lambda$ to a functor 
$$
fSets\to CAlg^{nu}(Sets), \; I\mapsto \Map(I,\Lambda),
$$ 
here $\Map(I, \Lambda)$ is equipped with the pointwise sum. The latter functor sends a map $\phi: I\to J$ in $fSets$ to the direct image map $\phi_*: \Map(I, \Lambda)\to\Map(J,\Lambda)$. The analog of global sections of this functor is $\bar\Lambda=\underset{I\in fSets}{\lim} \Map(I,\Lambda)$ taken in $CAlg^{nu}(\Sets)$. 
 
 We construct a natural transformation $\Map(I,\Lambda)\to C_{X^I}$, $\und{\lambda}\mapsto V^{\und{\lambda}}$ functorial in $I\in fSets$. Moreover, the  section $V^{\und{\lambda}}\in C_{X^I}$ factorizes naturally in the following sense.
 
 If $\phi: I\to J$ is a map in $fSets$ and $\und{\lambda}: I\to \Lambda$, let $I_j=\phi^{-1}(j)$ for $j\in J$, let $\und{\lambda}^j: I_j\to \Lambda$ be the restriction of $\und{\lambda}$. Set 
$$
X^I_{\phi, d}=\{(x_i)\in X^I\mid \mbox{if}\; \phi(i_1)\ne\phi(i_2)\; \mbox{then}\; x_{i_1}\ne x_{i_2}\},
$$ 
here $d$ stands for `disjoint'. One has an equivalence
$$
C_{X^I}\mid_{X^I_{\phi, d}}\,\iso\, (\underset{j\in J}{\boxtimes} C_{X^{I_j}})\mid_{X^I_{\phi, d}}
$$
of sheaves of categories over $X^I_{\phi, d}$, and a compatible factorization isomorphism for sections
$$
V^{\und{\lambda}}\mid_{X^I_{\phi, d}}\,\iso\, (\underset{j\in J}{\boxtimes} V^{\und{\lambda}^j})\mid_{X^I_{\phi, d}}
$$
in $C_{X^I}\mid_{X^I_{\phi, d}}$.

 In the special case $\Lambda=\{0\}$ the datum of $\Irr$ is a non-unital commutative algebra $V$ in $(C(X),\otimes^!)$, and $V^{\und{\lambda}}$ becomes the factorization algebra associated to $V$.
 
\sssec{Drinfeld-Pl\"ucker formalism} Assuming $\Lambda$ and $(C(X),\otimes^!)$ as in Section~\ref{Sect_Spreading right-lax functors_intro} unital in addition, assume $\Irr$ is in addition right-lax symmetric monoidal (that is, preserves units). Assume also given another right-lax symmetric monoidal functor $Irr^*: \Lambda\to (C(X),\otimes^!)$, $\lambda\mapsto (V^{\lambda})^*$ and a morphism
\begin{equation}
\label{map_unit_for_introduction}
1_{C(X)}\to V^{\lambda}\otimes^! (V^{\lambda})^*
\end{equation}
functorial in $\lambda\in\Lambda$ and, moreover, compatible with right-lax structures. We assume that for any $\lambda\in \Lambda$ the map (\ref{map_unit_for_introduction}) extends to a duality datum for the pair $(V^{\lambda}, (V^{\lambda})^*)$ in $(C(X),\otimes^!)$. 

 Let $\cA=\underset{\lambda\in\Lambda}{\oplus} V^{\lambda}$ taken in $C(X)$. Then $\cA\in CAlg(C(X))$. We assume in addition we are given a map $\cA\to A$ in $CAlg(C(X))$. Equip $\Lambda$ with the relation: $\lambda_1\le\lambda_2$ iff there is $\lambda\in\Lambda$ with $\lambda_2=\lambda+\lambda_1$, note that $(\Lambda,\le)$ is then a filtered category. Under these assumptions, we define a functor
$$
(\Lambda,\le)\to C(X), \; \lambda\mapsto A\otimes^! (V^{\lambda})^*
$$
and equip 
$$
B:=\underset{\lambda\in (\Lambda,\le)}{\colim} A\otimes^! (V^{\lambda})^*
$$
with a stricture of a commutative algebra in $C(X)$, so that one gets a morphism $A\to B$ in $CAlg(C(X))$. 

 The Drinfeld-Plucker formalism, as we propose to think about it, is the fact that the functor
$$
A-mod(C(X))\to B-mod(C(X)), M\mapsto B\otimes^!_A M
$$
admits a colimit presentation: 
\begin{equation}
\label{colimit_presentation_over_X_intro}
B\otimes_A M\,\iso\, \underset{\lambda\in (\Lambda,\le)}{\colim} M\otimes^! (V^{\lambda})^*
\end{equation}

 The key observation is that for $I\in fSets$ the $X^I$-version of the above functor
$$ 
A_{X^I}-mod(C_{X^I})\to B_{X^I}-mod(C_{X^I}), M\mapsto B_{X^I}\otimes^!_{A_{X^I}} M
$$
also admit a colimit presentation obtained from (\ref{colimit_presentation_over_X_intro}) by a similar `spreading' from $X$ to $X^I$. At the end of the day, this formalism is an application of (\cite{DL}, Lemma~C.4.9).

 This construction plays a crucial role in the definition of the $\Ran$-version of the parabolic semi-infinite $\IC$-sheaf in \cite{DL2}.
 
\sssec{} In Appendix~\ref{Sect_Sheaves of categories} we summarize the generalities about sheaves of categories for our sheaf theories.

\sssec{} In Appendix~\ref{appendix_some_generalities} we collect some properties of our sheaf theories, which are either not known or proved for other sheaf theories elsewhere. Especially, several results of Raskin from \cite{Ras2} are extended there from the case of $\cD$-modules to the constructible context. Some results seem new. For example, for $S\in\Sch_{ft}$ and $C\in Shv(S)-mod$ we show that dualizability of $C$ in $Shv(S)-mod$ is local in Zariski topology of $S$ (Lemma~\ref{Lm_dualizability_is_local}). Another example, we show that if $C$ is ULA over $Shv(S)$ then the exponent $\Fun_{Shv(S)}(C, Shv(S))$ is also ULA over $Shv(S)$ (Lemma~\ref{Lm_C^vee(X)_is_still_ULA}).  

\sssec{} In Appendix~\ref{Sect_Some categorical notions} we collect some general categorical observations. Given an object $C\in CAlg(\DGCat_{cont})$ and an automorphism $\epsilon$ of the identity functor on $C$ (as a symmetric monoidal functor), in (\cite{GLys}, 8.2.4) we constructed some twist $C^{\epsilon}\in CAlg(\DGCat_{cont})$. In Section~\ref{Sect_Twist} we give an alternative construction of the same twist. It is used in the construction of Satake functors in Section~\ref{Sect_Application to Satake}.
 
\sssec{} In Appendix~\ref{Sect_E} we study graded versions of commutative factorization categories. Let $\Lambda$ be a semigroup isomorphic to $\ZZ_+^n$ for some $n>0$, set $\Lambda^*=\Lambda-\{0\}$. Let $C(X)\in CAlg^{nu}(Shv(X)-mod$). In Appendix~\ref{Sect_E} we systematically study the sheaves of categories $\Fact(C)$ assuming given an additional $\Lambda$-grading (resp., $\Lambda^*$-grading) on $C(X)$  compatible with the symmetric monoidal structure. Then $\Fact(C)$ inherits a $\Lambda$-grading (resp., $\Lambda^*$-grading). Similarly, if $A\in CAlg^{nu}(C(X))$ is graded then 
then $\Fact(A)$ inherits the corresponding grading.

 We calculate $\Fact(C)$ in the simplest cases when $C(X)$ is the category of $\Lambda^*$-graded (resp., $\Lambda$-graded) sheaves on $X$ (cf. Proposition~\ref{Pp_E.1.5} and \ref{Pp_E.2.3}). Let $\Conf$ be the moduli scheme of $\Lambda^*$-valued divisors on $X$. As an important surprising byproduct we get a canonical $Shv(\Ran)$-module structure on $Shv(\Conf)$. 
 
 Let $(\Conf\times\Ran)^{\subset}$ be the subfunctor of $\Conf\times\Ran$ classifying $(D\in\Conf, \cI\in\Ran)$ such that $D$ is 'set-theoretically' supported on the set $\cI$ (cf. Section~\ref{Sect_E.0.5}). If $C(X)$ is $\Lambda$-graded (resp., $\Lambda^*$-graded) then we equip $\Fact(C)$ with an additional structure of a $Shv((\Conf\times\Ran)^{\subset})$-module category (resp., of a $Shv(\Conf)$-module category).  
  
 For $E(X)\in CAlg^{nu}(Shv(X)-mod)$ we get an equivalence
$$
\Fact(E)\otimes_{Shv(\Ran)} Shv(\Conf)\,\iso\,\Fact(C),
$$ 
where $C=\underset{\lambda\in\Lambda^*}{\oplus} E(X)$ is viewed as a $\Lambda^*$-graded non-unital commutative algebra in $Shv(X)-mod$ (cf. Proposition~\ref{Pp_E.1.22}). We also study the corresponding unital case. 
 
\sssec{}  For $C(X)=\underset{\lambda\in\Lambda}{\oplus} C(X)_{\lambda}$ as above let $C_{>0}=\underset{\lambda\in\Lambda}{\oplus} C(X)_{\lambda}$. In Section~\ref{Sect_Adding and removing units}
we discuss the procedures of adding and removing units relating $\Fact(C)$ with $\Fact(C_{>0})$ under the assumption that $C(X)_0\,\iso\, Shv(X)$, and similarly for commutative factorization algebras in $\Fact(C)$.
 
  This section culminates with Theorem~\ref{Con_E.3.8}, which is one of our main technical results. It affirms that $\Fact(C)$ is obtained from $\Fact(C_{>0})$ by the base change along the projection map $(\Conf\times\Ran)^{\subset}\to \Conf$. In Section~\ref{Sect_Adding units for graded factorization algebras} we establish a similar result for graded factorization algebras: $\Fact(A)$ is obtained from $\Fact(A_{>0})$ by base change. This result is very important for our joint paper \cite{DL2} with G. Dhillon. 
  
\sssec{} In Section~\ref{Sect_E.3} we apply the above considerations to the two versions of the affine Grassmannian $\Gr_{G,\Ran}$ and $\Gr_{G,\Conf}$ of a reductive group $G$. Though
there is no natural map $\can_{\Gr}: \Gr_{G,\Conf}\to \Gr_{G,\Ran}$ of prestacks, we explain that the functor $\can_{\Gr}^!$ does exist between the corresponding usual (or spherical-equivariant) $\DG$-categories of sheaves. For the corresponding $\DG$-categories of spherical sheaves $\Sph_{G,\Ran}$ and $\Sph_{G, \Conf}$ on these prestacks we promote the functor $\can_{\Gr}^!$ to a non-unital symmetric monoidal functor
$$
\can_{\Gr}^!: (\Sph_{G,\Ran}, \,\otimes^{ch})\to (\Sph_{G,\Conf}, \,\otimes^{ch}),
$$
where $\otimes^{ch}$ denotes the corresponding chiral symmetric monoidal structure (cf. Proposition~\ref{Pp_E.4.12}). The same functor is also non-unital monoidal
$$
\can_{\Gr}^!: (\Sph_{G,\Ran}, \,\star)\to (\Sph_{G,\Conf}, \,\star)
$$
with respect to the exterior convolutions (cf. Proposition~\ref{Pp_E.4.16}). 
 
 As an application, we obtain the version of the Satake functor for the configuration space $\Conf$ out of the usual Satake functor $\Sat_{G,\Ran}: \Fact(\Rep(\check{G}))\to \Sph_{G,\Ran}$ by the base change $\_\otimes_{Shv(\Ran)} Shv(\Conf)$.  
 
\sssec{} In Section~\ref{Sect_Factorization coalgebras} we study a construction attaching to $C\in CAlg^{nu}(\DGCat_{cont})$ and $A\in CoCAlg^{nu}(C)$ an object $\Fact^{coalg}(A)\in\Fact(C)$. We show that, under the assumption that $C$ is $\Lambda^*$-graded as in Section~\ref{Sect_E1}, $\Fact^{coalg}(A)$ is a factorization coalgebra in $(\Fact(C), \star)$ in the sense of \cite{FG}. 
  

\ssec{Conventions}
\label{Sect_Conventions} 

\sssec{} Our conventions are those of \cite{GLys}. We work with a sheaf theory, which is either in constructible context or the theory of $\cD$-modules. Our geometric objects are over an algebraically closed field $k$. Let $X$ be a smooth projective connected curve. Write $\Ran$ for the Ran space of $X$ (\cite{GLys}, Section~2.1). 

 We denote by $\Sch_{ft}$ (resp., $\Sch^{aff}_{ft}$) the category of classical schemes (resp., affine classical schemes) of finite type over $k$. 
 
\sssec{} Write $1-\Cat$ for the $\infty$-category of $\infty$-categories, $\Spc$ for the $\infty$-category of spaces. The field of coefficients of our sheaf theory is denoted $e$. Write $\DGCat_{cont}$ for the category of $e$-linear presentable cocomplete $\DG$-categories and continuous functors. For $C, D\in\DGCat_{cont}$, $C\otimes D$ denotes the Lurie tensor product in $\DGCat_{cont}$. For $A\in Alg(\DGCat_{cont})$ set $A-mod=A-mod(\DGCat_{cont})$, write $A^{rm}$ for $A$ equipped with the reversed mutliplication. For $M,L\in A-mod$ we denote by $\Fun_A(M, N)\in\DGCat_{cont}$ the category of continuous exact $A$-linear functors. 

 If $C,D\in \DGCat_{cont}$ are equipped with accessible t-structures then $C\otimes D$ is also equipped with an accessible t-structure defined in (\cite{Ly}, 9.3.17). 
 
\sssec{} If not mentioned otherwise, for $S\in\Sch_{ft}$ the category $Shv(S)$ is equipped with the $\otimes^!$-symmetric monoidal structure. Recall that $\PreStk_{lft}=\Fun((\Sch^{aff}_{ft})^{op}, \Spc)$.  

 The theory of sheaves of categories on prestacks (for a given theory of sheaves as above) is recalled in Appendix~\ref{Sect_Sheaves of categories}. For $S\in\Sch_{ft}$, $S$ is 1-affine, so $ShvCat(S)\,\iso\, Shv(S)-mod$ canonically. 
 
 For $S, S'\in \Sch_{ft}$, $C\in Shv(S)-mod, C'\in Shv(S')-mod$ we denote by $C\boxtimes C'$ the corresponding object in $Shv(S\times S')-mod$. In the constructible context we will at some point need also to consider the tensor product $C\otimes C'$ as $Shv(S)\otimes Shv(S')$-module, then $C\boxtimes C'$ is obtained from $C\otimes C'$ via extension of scalars $Shv(S)\otimes Shv(S')\to Shv(S\times S')$. (In the case of $\cD$-modules they coincide).

\sssec{} We denote by $Alg$ (resp., $Alg^{nu}$) the operad of unital (resp., non-unital) associative algebras. The corresponding operads for commutative algebras are denoted $CAlg, CAlg^{nu}$. For $A\in CAlg(1-\Cat)$ we write
$CoCAlg(A)$, (resp., $CoCAlg^{nu}(A)$) for the category of unital (resp., non-unital) cocommutative coalgebras in $A$. 

 By default, our factorization categories and algebras are non-unital. 
 
\sssec{} For $C, D\in CAlg(1-\Cat)$ a right-lax non-unital symmetric monoidal functor $f: C\to D$ gives rise, in particular, to a morphism $1_D\toup{\alpha} f(1_C)$. We call it right-lax symmetric monoidal if $\alpha$ is an isomorphism in $D$. 

\section{Factorization categories attached to commutative algebras}
\label{Sect_Fact_cat_for_com_algebras}

\ssec{} In this section we associate to $C(X)\in CAlg^{nu}(Shv(X)-mod)$ a factorization sheaf of categories $\Fact(C)$ on $\Ran$. This is the construction from (\cite{GLys}, Section~8.1).

\sssec{} Write $fSets$ for the category whose objects are finite nonempty sets, and morphisms are surjections. For a map $J\to K$ in $fSets$ write $\vartriangle^{(J/K)}: X^K\to X^J$ for the corresponding closed embedding. Recall that 
$$
\Ran\,\iso\, \underset{I\in fSets^{op}}{\colim} X^I
$$
in $\PreStk_{lft}$. For $I\in fSets$ write $\vartriangle^I: X^I\to \Ran$ for the structure map.
Recall also that for $I\in fSets$ the map $\vartriangle^I$ is pseudo-proper by (\cite{Ga}, 7.4.2). For the notion of a pseudo-proper map cf. (\cite{Ga}, 1.5). 

\index{Ran space}

Recall that for $I, J\in fSets$, 
$$
X^I\times_{\Ran} X^J\,\iso\, \underset{I\twoheadrightarrow K\twoheadleftarrow J}{\colim} X^K
$$ 
in $\PreStk_{lft}$ (cf. \cite{Ga}, 8.1.2), here the colimit is taken over $(fSets_{I/}\times_{fSets} fSets_{J/})^{op}$. For the notion of 1-affineness (for our sheaf theories!)
see Section~\ref{Sect_A.1.5}.

 The next observation is, from our point of view, of independent interest (especially in the constructible context). 
\begin{Lm} 
\label{Lm_Ran_is_1-affine}
The prestack $\Ran$ 1-affine.
\end{Lm}
\begin{proof} 
The natural map $X^I\times_{\Ran} X^J\to X^I\times X^J$ is a closed immersion. So, $\Ran$ is a safe pseudo-scheme in the sense of Definition~\ref{Def_safe_pseudo-scheme}. Our claim follows by Corollary~\ref{Cor_F.0.9} 
\end{proof}

 In what follows we do not distinguish a sheaf of categories $E$ on $\Ran$ from $\Gamma(\Ran, E)\in Shv(\Ran)-mod$. 

\begin{Rem} In the constructible context the category $Shv(\Ran)$ is equipped with the non-unital symmetric monoidal structure $\otimes$ defined in Section~\ref{Sect_F.2.6}.
\end{Rem}

\begin{Lm} 
\label{Lm_2.1.11} i) If  $I\in fSets$ then $Shv(X^I)-mod$ is dualizable in $Shv(\Ran)-mod$.\\
ii) If $L\in Shv(X^I)-mod$ is dualizable in $Shv(X^I)-mod$ then $L$ is dualizable in $Shv(\Ran)-mod$.
\end{Lm}
\begin{proof}
i) {\bf Step 1}. If $\alpha: Z\to S$ is a finite morphism in $\Sch_{ft}$ then $Shv(Z)$ is dualizable in $Shv(S)-mod$. Indeed, for $\cD$-modules this follows from the fact that $Shv(Z)$ is dualizable in $\DGCat_{cont}$.

 In the constructible context $\alpha_!$ has a left adjoint, so $\alpha_!$ is monadic and $Shv(Z)\,\iso\, (\alpha_!e)-mod(Shv(S))$. By (\cite{G}, I.1, 8.6.3), $(\alpha_!e)-mod(Shv(S))$ is dualizable in $Shv(S)-mod$.
 
 Combining the above with (\cite{HA}, 4.6.2.13), we see that $Shv(Z)$ is left dualizable as a $(Shv(S), Shv(Z))$-bimodule in $\DGCat_{cont}$.  
 
\medskip\noindent
{\bf Step 2}. By (\cite{Ga1}, 1.4.6), it suffices to show that for any $J\in fSets$, $Shv(X^I)\otimes_{Shv(\Ran)} Shv(X^J)$ is dualizable in $Shv(X^J)-mod$. By Theorem~\ref{Thm_Kunneth_formula}, $Shv(X^I)\otimes_{Shv(\Ran)} Shv(X^J)\,\iso\, Shv(X^I\times_{\Ran} X^J)$ canonically. Since the projection $X^I\times_{\Ran} X^J\to X^J$ is finite, our claim follows from Step 1.

\medskip\noindent
ii) Our claim follows from i) combined with Lemma~\ref{Lm_2.1.5_Lin_Chen} below.
\end{proof}

\begin{Lm}[\cite{Chen}, A.3.8] 
\label{Lm_2.1.5_Lin_Chen}
Let $A\to B$ be a map in $CAlg(\DGCat_{cont})$. Let $L\in B-mod$ be dualizable in $B-mod$. Let $B$ be left-dualizable as a $(A,B)$-bimodule in $\DGCat_{cont}$. Then $L$ is dualizable in $A-mod$. \QED
\end{Lm}

\begin{Cor} 
\label{Cor_2.2.9_dualizability}
An object $M\in Shv(\Ran)-mod\,\iso\, ShvCat(\Ran)$ is dualizable iff for each $I\in fSets$, $M\otimes_{Shv(\Ran)} Shv(X^I)$ is dualizable in $Shv(X^I)-mod$.
\end{Cor}
\begin{proof}
Apply (\cite{Ga1}, Lemma 1.4.6). 
\end{proof}

\sssec{} For $\cC\in 1-\Cat$ write $\cTw(\cC)$ for the twisted arrows category. By definition, $\cTw(\cC)\to \cC\times\cC^{op}$ is the cartesian fibration in spaces associated to the functor $\cC^{op}\times\cC\to \Spc$, $(c',c)\mapsto \Map_C(c',c)$. So, an object of $\cTw(\cC)$ is an arrow $\alpha: c'\to c$ in $\cC$. A map from $\alpha: c'\to c$ to $\beta: d'\to d$ in $\cTw(\cC)$ gives rise to a commutative diagram
$$
\begin{array}{ccc}
c' & \toup{\alpha} & c\\
\downarrow && \uparrow\\
d'& \toup{\beta} & d
\end{array}
$$
\index{Twisted arrows category}
\sssec{} For $I\in fSets$ set $Tw(I)=\cTw(fSets)\times_{fSets} fSets_{I/}$. An object of $Tw(I)$ is a diagram $\Sigma=(I\toup{p}J\to K)$ in $fSets$, a morphism from $\Sigma_1$ to $\Sigma_2$ is given by a diagram
\begin{equation}
\label{morphism_in_Tw_v2}
\begin{array}{ccccc}
I & \toup{p_1} & J_1 & \to &K_1\\
\Vert && \downarrow &&\uparrow\\
I &\toup{p_2} & J_2 & \to &K_2
\end{array}
\end{equation} 

\sssec{} 
\label{Sect_2.1.3_now}
Let $C(X)\in CAlg^{nu}(Shv(X)-mod)$. For $J\in fSets$ write $C^{\otimes J}(X)$ for the $J$-th tensor power of $C(X)$ in $Shv(X)-mod$. 

 Let $I\in fSets$. Write 
$$
\cF_{I,C}: Tw(I)\to Shv(X^I)-mod
$$ 
for the functor sending $\Sigma=(I\to J\toup{\phi} K)$ to $C^{\otimes \phi}:=\underset{k\in K}{\boxtimes} C^{\otimes J_k}(X)$. Here $J_k$ is the fibre over $k\in K$. The functor $\cF_{I,C}$ sends the map (\ref{morphism_in_Tw_v2}) to the composition 
$$
\underset{k\in K_1}{\boxtimes} C^{\otimes (J_1)_k}(X)\toup{m} \underset{k\in K_1}{\boxtimes} C^{\otimes (J_2)_k}(X)\to\underset{k\in K_2}{\boxtimes} C^{\otimes (J_2)_k}(X),
$$
where the first map is the product in $C(X)$ along $J_1\to J_2$, and the second is the natural map 
\begin{equation}
\label{map_transition_for_cF_I,C}
\vartriangle_!\vartriangle^!(\underset{k\in K_2}{\boxtimes} C^{\otimes (J_2)_k}(X))\to \underset{k\in K_2}{\boxtimes} C^{\otimes (J_2)_k}(X)
\end{equation}
for $\vartriangle: X^{K_1}\to X^{K_2}$. Set
$$
C_{X^I}=\underset{Tw(I)}{\colim} \; \cF_{I,C}
$$
The latter colimit can also be understood in $Shv(X^I)-mod$ or equivalently in 
$\DGCat_{cont}$.

\sssec{} 
\label{Sect_2.1.4_now}
Let us check that this produces a factorization sheaf of categories on $\Ran$. 
This amounts to the following. 

\index{Factorization sheaf of categories on $\Ran$}

Given a map $f: I\to I'$ in $fSets$, for the closed immersion $\vartriangle: X^{I'}\to X^I$ we need to construct an equivalence 
\begin{equation}
\label{equiv_Fact(C)_is_a_sheaf_Sect_2.1.4}
C_{X^I}\otimes_{Shv(X^I)} Shv(X^{I'})\,\iso\, C_{X^{I'}}
\end{equation}
in a way compatible with compositions, so giving rise to a sheaf of categories $\Fact(C)$ on $\Ran$ with $\Fact(C)\otimes_{Shv(\Ran)}\Shv(X^I)\,\iso\, C_{X^I}$.

For a map $I\toup{\phi} I'$ in $fSets$ set $X^I_{\phi, d}=\{(x_i)\in X^I\mid \mbox{if}\; \phi(i_1)\ne \phi(i_2)\;\mbox{then}\; x_{i_1}\ne x_{i_2}\}$, here the subscript $d$ stands for \select{disjoint}. The factorization structure on $\Fact(C)$ amounts to a collection of equivalences for any $I\toup{\phi} I'$ in $fSets$
\begin{equation}
\label{iso_after_Lm_3.2.4}
C_{X^I}\mid_{X^I_{\phi, d}}\,\iso\, (\mathop{\boxtimes}\limits_{i'\in I'}
C_{X^{I_{i'}}})\mid_{X^I_{\phi, d}}
\end{equation} 
in $Shv(X^I_{\phi, d})-mod$ together with compatibilities as in (\cite{BD_chiral}, 3.4.4). That is, the equivalences (\ref{iso_after_Lm_3.2.4}) have to be mutually compatible with respect to the compositions of $\phi$'s, and also compatible with (\ref{equiv_Fact(C)_is_a_sheaf_Sect_2.1.4}). 

 We construct this structure in Sections~\ref{Sect_2.1.5_now}-\ref{Sect_2.1.11_now} below.

\sssec{} 
\label{Sect_2.1.5_now}
Let $Q(I)$ be the set of equivalence relations on $I$. Recall that $Q(I)$ is partially ordered. As in \cite{BD_chiral}, we write $I'\in Q(I)$ for a map $I\to I'$ in $fSets$ viewed as an equivalence relation on $I$. We write $I''\le I'$ iff $I''\in Q(I')$. Then $Q(I)$ is a lattice (\cite{BD_chiral}, 1.3.1). For $I', I''\in Q(I)$ we have $\inf(I',I'')$ and $\sup(I', I'')$. Note that $\inf(I',I'')$ can be seen as a push-out of the diagram $I'\gets I\to I''$. 

Let $f: I\to I'$ be a map in $fSets$. Then $f$ induces a full embedding $\Tw(I')\subset \Tw(I)$ sending $I'\to J'\to K'$ to $I\toup{f'} J'\to K'$, where $f'$ is the composition $I\to I'\to J'$. 
Let 
$$
\xi: \Tw(I)\to \Tw(I')
$$ 
be the functor sending $I\to J\to K$ to the diagram $I'\to J'\to K'$ obtained by push-out along $f: I\to I'$. It sends the morphism (\ref{morphism_in_Tw_v2}) to the induced diagram
$$
\begin{array}{ccccc}
I' & \to & J'_1 & \to & K'_1\\
\Vert && \downarrow &&\uparrow\\
I' &\to & J'_2 & \to & K'_2
\end{array}
$$

 If $J_2\le J_1$ in $Q(I)$ then for $J'_i=\inf(J_i, I')$ we get $X^{J_2}\times_{X^{J_1}} X^{J'_1}\,\iso\, X^{J'_2}$.
 
\sssec{} 
\label{Sect_2.1.6_v4}
For $(I\to J\to K)\in \Tw(I)$ let $(I'\to J'\to K')\in\Tw(I')$ be its image under $\xi$. Consider the functor
\begin{equation}
\label{functor_for_Fact_v1}
\mathop{\boxtimes}\limits_{k\in K} C^{\otimes J_k}(X)\to \mathop{\boxtimes}\limits_{k'\in K'} C^{\otimes J'_{k'}}(X),
\end{equation}
given as the composition
$$
\mathop{\boxtimes}\limits_{k\in K} C^{\otimes J_k}(X)\to (\mathop{\boxtimes}\limits_{k\in K} C^{\otimes J_k})(X)\otimes_{\Shv(X^K)} Shv(X^{K'})\,\iso\, \mathop{\boxtimes}\limits_{k'\in K'} C^{\otimes J_{k'}}(X)\,\to\, \mathop{\boxtimes}\limits_{k'\in K'} C^{\otimes J'_{k'}}(X)
$$
where the second map is the product in $C(X)$ along $J_{k'}\to J'_{k'}$ for any $k'\in K'$. We used the closed immersion $X^{K'}\to X^K$. Now (\ref{functor_for_Fact_v1}) extends to a morphism  $\cF_{I,C}\to \cF_{I', C}\comp \xi$ in $\Funct(\Tw(I), \Shv(X^I)-mod)$. Namely, for any morphism (\ref{morphism_in_Tw_v2}) the diagram commutes
$$
\begin{array}{ccc}
\mathop{\boxtimes}\limits_{k\in K_1} C^{\otimes (J_1)_k}(X)& \to & 
\mathop{\boxtimes}\limits_{k\in K'_1} C^{\otimes (J'_1)_k}(X)
 \\
\downarrow && \downarrow\\
\mathop{\boxtimes}\limits_{k\in K_2} C^{\otimes (J_2)_k}(X) &
 \to & \mathop{\boxtimes}\limits_{k\in K'_2} C^{\otimes (J'_2)_k}(X)
\end{array}
$$  
It uses the fact that the square is cartesian
$$
\begin{array}{ccc}
X^{K_1} & \getsup{\vartriangle} & X^{K'_1}\\
\downarrow\lefteqn{\scriptstyle\vartriangle}   &&\downarrow\lefteqn{\scriptstyle\vartriangle}\\
X^{K_2} & \getsup{\vartriangle} & X^{K'_2}
\end{array}
$$
and the base change holds $\vartriangle^!\vartriangle_*\,\iso\, \vartriangle_*\vartriangle^!$. Here $K'_1=inf(K_1, K'_2)$. 
 
  This gives the desired functor 
\begin{equation}
\label{functor_restr_for_Fact(C)}
C_{X^I}\,\iso\, 
\underset{\Tw(I)}{\colim} \;\cF_{I,C} \to \underset{\Tw(I)}{\colim} \;(\cF_{I', C}\comp\xi)\to \underset{\Tw(I')}{\colim} \;\cF_{I', C}\,\iso\, C_{X^{I'}}
\end{equation}

\sssec{} 
\label{Sect_2.1.7_now}
Let us show that the functor
\begin{equation}
\label{functor_restr_for_Fact(C)_again}
C_{X^I}\otimes_{\Shv(X^I)} Shv(X^{I'})\to C_{X^{I'}}
\end{equation}
induced by (\ref{functor_restr_for_Fact(C)}) is an equivalence.

Denote by $\Tw(I)^f\subset \Tw(I)$ the full subcategory of $(I\to J\to K)$ such that $K\in Q(I')$. The embedding $Tw(I)^f\subset Tw(I)$ has a right adjoint $\beta: Tw(I)\to Tw(I)^f$ sending $(I\to J\to K)$ to $(I\to J\to K')$ with $K'=inf(I', K)$. We have 
$$
C_{X^I}\otimes_{\Shv(X^I)} Shv(X^{I'})\,\iso\,\mathop{\colim}\limits_{(I\to J\to K)\in Tw(I)} (\mathop{\boxtimes}\limits_{k'\in K'} C^{\otimes J_{k'}}(X)),
$$ 
here $(I\to J\to K')=\beta(I\to J\to K)$. The expression under the colimit is the composition
$$
Tw(I)\toup{\beta} Tw(I)^f\toup{\cF^f_{I,C}} Shv(X^{I'})-mod,
$$
where $\cF^f_{I,C}: Tw(I)^f\to Shv(X^{I'})-mod$ is the restriction of $\cF_{I,C}$ to this full subcategory. So, we first calculate the LKE under $\beta: Tw(I)\to Tw(I)^f$ of $\cF^f_I\comp \beta$. By (\cite{G}, ch. I.1, 2.2.3), $\beta$ is cofinal, so the above colimit identifies with
$$
\mathop{\colim}\limits_{(I\to J\to K)\in Tw(I)^f} (\mathop{\boxtimes}\limits_{k\in K} C^{\otimes J_k}(X)).
$$ 

 Consider the full embedding $Tw(I')\subset Tw(I)^f$. It has a left adjoint $\xi^f: Tw(I)^f\to Tw(I')$. Here $\xi^f$ is the restriction of $\xi$. 
So, the full embedding $Tw(I')\subset Tw(I)^f$ is cofinal. So, the latter colimit identifies with 
$$
\mathop{\colim}\limits_{(I'\to J'\to K')\in Tw(I')} (\mathop{\boxtimes}\limits_{k'\in K'} C^{\otimes J'_{k'}}(X))\,\iso\, C_{X^{I'}}
$$ 
So, (\ref{functor_restr_for_Fact(C)_again}) is an equivalence. These equivalences are compatible with compositions of maps in $fSets$, so we obatined a sheaf of categories $\Fact(C)$ on $\Ran$.  

\sssec{} 
\label{Sect_2.1.8_v4}
There is also a direct definition of $\Fact(C)$ as follows. Consider the functor
\begin{equation}
\label{functor_cF_Ran_C}
\cF_{\Ran, C}: \cTw(fSets)\to Shv(\Ran)-mod
\end{equation}
sending $(J\to K)$ to $\underset{k\in K}{\boxtimes} C^{\otimes J_k}(X)$. It sends a map
\begin{equation}
\label{map_in_cTw(fSets)}
\begin{array}{ccc}
J_1 &\to & K_1\\
\downarrow && \uparrow\\
J_2 & \to & K_2
\end{array}
\end{equation}
in $\cTw(fSets)$ to the composition 
$$
\underset{k\in K_1}{\boxtimes} C^{\otimes (J_1)_k}(X)\toup{m} \underset{k\in K_1}{\boxtimes} C^{\otimes (J_2)_k}(X)\to \underset{k\in K_2}{\boxtimes} C^{\otimes (J_2)_k}(X),
$$
here the first map is the product in $C(X)$ along $J_1\to J_2$, and the second is (\ref{map_transition_for_cF_I,C}). 

\begin{Lm} 
\label{Lm_2.1.2_about_Fact(C)}
One has canonically 
$$
\Fact(C)\,\iso\, \underset{\cTw(fSets)}{\colim}\; \cF_{\Ran, C}.
$$
\end{Lm}
\begin{proof}
Consider the functor $\Tw: fSets^{op}\to 1-\Cat$ sending $I$ to $\Tw(I)$. It sends a map $I\toup{f} I'$ in $fSets$ to the natural fully faithful embedding $\Tw(I')\hook{} \Tw(I)$. Let $\cY_{\Tw}\to fSets^{op}$ be the cocartesian fibration attached to $\Tw$ by strengthening. So, an object of $\cY_{\Tw}$ is a diagram $(I\to J\to K)$ in $fSets$. A morphism from $(I_1\to J_1\to K_1)$ to $(I_2\to J_2\to K_2)$ in $\cY_{\Tw}$ is a commutative diagram in $fSets$
$$
\begin{array}{ccccc}
I_2 & \to &J_2 & \to & K_2\\
\downarrow && \uparrow && \downarrow\\
I_1 & \to & J_1 & \to & K_1.
\end{array}
$$

Let $\zeta: \cY_{\Tw}\to\cTw(fSets)$ be the functor sending $(I\to J\to K)$ to $(J\to K)$. 
We have 
$$
\Gamma(\Ran, \Fact(C))\,\iso\, \underset{I\in fSets^{op}}{\colim}\;\underset{(I\to J\to K)\in \Tw(I)}{\colim} \; (\underset{k\in K}{\boxtimes} C^{\otimes J_k}(X))\,\iso\; \underset{\cY_{\Tw}}{\colim} \; \cF_{\Ran, C}\comp \zeta
$$
We used here (\cite{G}, I.1, 2.2.4) for the cocartesian fibration $\cY_{\Tw}\to fSets^{op}$.

 We claim that the LKE of $\cF_{\Ran, C}\comp \zeta$ along $\zeta$ identifies canonically with $\cF_{\Ran, C}$. Indeed, pick an object $\phi\in \cTw(fSets)$ given by a map $\phi: J\to K$ in $fSets$. Note that $\zeta$ is a cocartesian fibration. By (\cite{G}, I.1, 2.2.4), the canonical functor $(\cY_{\Tw})_{\phi}\to \cY_{\Tw}\times_{\cTw(fSets)} \cTw(fSets)_{/\phi}$ is cofinal. So, the value of the LKE of $\cF_{\Ran, C}\comp \zeta$ along $\zeta$ at $\phi$ is
$$
\underset{(I_1\to J_1\to K_1)\in (\cY_{\Tw})_{\phi}}{\colim} \cF_{\Ran, C}(J\to K)
$$
In turn, $(\cY_{\Tw})_{\phi}\,\iso\, fSets_{/J}$ is contractible, as it has a final object. Finally, for $I\in 1-\Cat$, $I\to \mid I\mid$ is cofinal. Our claim follows.
\end{proof}

\sssec{} To construct the factorization structure on $\Fact(C)$, recall the following.

\begin{Lm} 
\label{Lm_3.2.4_from_chiral_algebras}
Let $I'\getsup{\phi} I\to K$ be a diagram in $fSets$. Then $X^I_{\phi, d}\times_{X^I} X^K$ is empty unless $I'\in Q(K)$, that is, $\phi$ decomposes as $I\to K\toup{\phi'} I'$. In the latter case the square is cartesian
$$
\begin{array}{ccc}
X^I_{\phi, d} & \hook{} & X^I\\
\uparrow && \uparrow\lefteqn{\scriptstyle\vartriangle}\\
X^K_{\phi', d} & \hook{} & X^K,
\end{array}
$$
where $\vartriangle$ is the diagonal. \QED
\end{Lm} 

\sssec{} 
\label{Sect_2.1.11_now}
Let us construct the equivalence (\ref{iso_after_Lm_3.2.4}). Write $\Tw(I)_{\phi}$ for the full subcategory of $\Tw(I)$ spanned by objects $(I\to J\to K)$ such that $I'\in Q(K)$. We have the equivalence 
$$
\Tw(I)_{\phi}\,\iso\,\prod_{i'\in I'} \Tw(I_{i'})
$$ 
sending $(I\to J\to K)$ to the collection $(I_{i'}\to J_{i'}\to K_{i'})\in \Tw(I_{i'})$ for $i'\in I'$, the corresponding fibres over $i'$. It is easy to check that $\Tw(I)_{\phi}\hook{}\Tw(I)$ is zero-cofinal in the sense of Section~\ref{Def_zero-cofinal}. 

 The LHS of (\ref{iso_after_Lm_3.2.4}) identifies with
$$
\mathop{\colim}\limits_{(I\to J\to K)\in \Tw(I)}  ((\mathop{\boxtimes}\limits_{k\in K} C^{\otimes J_k}(X)\otimes_{\Shv(X^I)} Shv(X^I_{\phi, d}))
$$  
By Lemma~\ref{Lm_3.2.4_from_chiral_algebras} and Section~\ref{Sect_C.0.2}, the latter colimit rewrites as the colimit of the same diagram restricted to $Tw(I)_{\phi}$. For $(I\to J\to K)\in Tw(I)_{\phi}$ we get
$$
(\mathop{\boxtimes}\limits_{k\in K} C^{\otimes J_k}(X))\otimes_{\Shv(X^I)} Shv(X^I_{\phi, d})\,\iso\, (\mathop{\boxtimes}\limits_{i'\in I'}(\mathop{\boxtimes}\limits_{k\in K_{i'}} C^{\otimes J_k}(X))) \otimes_{\Shv(X^I)} Shv(X^I_{\phi, d})
$$ 
Since 
$$
\mathop{\colim}\limits_{(I_{i'}\to J_{i'}\to K_{i'})\in Tw(I_{i'})} \; (\mathop{\boxtimes}\limits_{k\in K_{i'}} C^{\otimes J_k}(X))\,\iso\, C_{X^{I_{i'}}},
$$ 
passing to the colimit over $Tw(I)_{\phi}$, we get the desired equivalence (\ref{iso_after_Lm_3.2.4}).  

 This defines on $\Fact(C)$ a structure of a factorization sheaf of categories on $\Ran$. 
 
\ssec{First properties of $\Fact(C)$}

\sssec{} Let $I\in fSets$. The category $Tw(I)$ has an object $(I\to I\to I)$. 
Denote by $\Loc: \mathop{\boxtimes}\limits_{i\in I} C(X)\to C_{X^I}$ the corresponding structure morphism in $\cF_{I, C}$.  
 
\begin{Lm} 
\label{Lm_3.5.2_generated_under_colim}
Assume $C(X)$ is unital, that is, $C(X)\in CAlg(Shv(X)-mod)$. Then
the functor $\Loc: \mathop{\boxtimes}\limits_{i\in I} C(X)\to C_{X^I}$ generates $C_{X^I}$ under colimits.
\end{Lm}
\begin{proof}
It suffices to show, by (\cite{G}, ch. I.1, 5.4.3), that the right adjoint $\Loc^R: C_{X^I}\to \mathop{\boxtimes}\limits_{i\in I} C(X)$ is conservative. Denote by $j: {^0\Tw(I)}\subset \Tw(I)$ the full subcategory spanned by objects of the form $(I\toup{p} J\to K)$, where $p$ is an isomorphism. We have an adjoint pair 
\begin{equation}
\label{adj_pair_j_and_j^R_Lm_3.5.2}
j: {^0Tw(I)}\leftrightarrows Tw(I): j^R,
\end{equation} 
where $j^R(I\toup{p} J\toup{q} K)=(I\toup{\id} I\toup{qp} K)$. 
 
 We have
$$
C_{X^I}\,\iso\,\mathop{\lim}\limits_{(I\to J\to K)\in Tw(I)^{op}} \; \mathop{\boxtimes}\limits_{k\in K} C^{\otimes J_k}(X),
$$ 
the limit of the functor $\cF_{I,C}^R$ obtained from $\cF_{I,C}$ by passing to right adjoints. Restricting $\cF_{I,C}^R$ to the full subcategory $^0Tw(I)^{op}$, we get a morphism 
\begin{equation}
\label{map_ffaith_for_3.3.8}
C_{X^I}\to \mathop{\lim}\limits_{(I\to J\to K)\in {^0Tw(I)^{op}}} \cF_{I,C}^R\comp j^{op}
\end{equation}
 
  Let $\cG$ denote the RKE of  $\cF_{I,C}^R\comp j^{op}$ under $^0Tw(I)^{op}\to Tw(I)^{op}$. By (\cite{Ly}, 2.2.39), 
$$
\cG=\cF_{I,C}^R\comp j^{op}\comp (j^R)^{op}.
$$ 
The map of functors $\cF^R_{I,C}\to \cG$ evaluated at an object $(I\toup{p} J\toup{q} K)=\Sigma\in Tw(I)^{op}$ becomes
$$
\mathop{\boxtimes}\limits_{k\in K} C^{\otimes J_k}(X)
\to \mathop{\boxtimes}\limits_{k\in K} C^{\otimes I_k}(X).
$$ 
It is conservative, as its left adjoint is surjective. So, passing to the limit over $Tw(I)^{op}$, we conclude by (\cite{Ly}, Cor. 2.5.3) that (\ref{map_ffaith_for_3.3.8}) is conservative.
  
  The category $^0Tw(I)^{op}$ has an initial object $(I\toup{\id} I\toup{id} I)$. So, 
$$
\mathop{\lim}\limits_{(I\to J\to K)\in {^0Tw(I)^{op}}} \cF_{I,C}^R\comp j^{op}\,\iso\, \mathop{\boxtimes}\limits_{i\in I} C(X)
$$ 
Thus, $\Loc^R$ is conservative.
\end{proof}  

 Under the assumptions of Lemma~\ref{Lm_3.5.2_generated_under_colim} there could be at most a unique structure of an object of $CAlg^{nu}(Shv(X^I)-mod)$ on $C_{X^I}$ for which $\Loc$ is non-unital symmetric monoidal. 
 
\sssec{Commutative chiral category structure} The structure of a commutative chiral category on $\Fact(C)$ is as follows. 

\index{Commutative chiral category structure}

 Given $I_1, I_2\in fSets$ let $I=I_1\sqcup I_2$. Consider the functor 
\begin{equation}
\label{map_alpha}
\alpha: Tw(I_1)\times Tw(I_2)\to Tw(I)
\end{equation}
sending a pair 
$(I_1\to J_1\to K_1), (I_2\to J_2\to K_2)$ to $(I\to J\to K)$ with 
$$
J=J_1\sqcup J_2, K=K_1\sqcup K_2.
$$  
 
Note that $\alpha$ is fully faithful. For an object of $Tw(I_1)\times Tw(I_2)$ whose image under $\alpha$ is $(I\to J\to K)$ we get an equivalence
\begin{equation}
\label{equiv_for_Fact(C)(Ran)_operation}
(\mathop{\boxtimes}\limits_{k\in K_1} C^{\otimes (J_1)_k}(X))\boxtimes (\mathop{\boxtimes}\limits_{k\in K_2} C^{\otimes (J_2)_k}(X))\,\iso\, \mathop{\boxtimes}\limits_{k\in K} C^{\otimes J_k}(X)
\end{equation} 
It extends naturally to an isomorphism of functors $\cF_{I_1, C}\boxtimes\cF_{I_2, C}\,\iso\,\cF_{I, C}\comp \alpha$
in 
$$
\Fun(Tw(I_1)\times Tw(I_2), Shv(X^I)-mod).
$$ 

 Passing to the colimit over $Tw(I_1)\times Tw(I_2)$ we get a morphism
\begin{equation}
\label{map_for_comm_chiral_cat_on_Fact(C)}
\eta_{I_1, I_2}: C_{X^{I_1}}\boxtimes C_{X^{I_2}}\,\iso\, \mathop{\colim}\limits_{Tw(I_1)\times Tw(I_2)} \cF_{I,C}\comp\alpha\to \mathop{\colim}\limits_{Tw(I)} \cF_{I,C}\,\iso\, C_{X^I}
\end{equation}
in $Shv(X^I)-mod$. 

 Set
$$
(X^{I_1}\times X^{I_2})_d=\{(x_i)\in X^I\mid \mbox{if}\; i_1\in I_1, i_2\in I_2\;\mbox{then}\; (x_{i_1}, x_{i_2})\in X^2-X\}
$$
 
 If $(I\to J\to K)\in Tw(I)$ then 
$$
X^K\times_{X^I} (X^{I_1}\times X^{I_2})_d
$$ 
is empty unless $(I\to J\to K)$ lies in the full subcategory $Tw(I_1)\times Tw(I_2)$. This shows that  (\ref{map_for_comm_chiral_cat_on_Fact(C)}) becomes an isomorphism after restriction to $(X^{I_1}\times X^{I_2})_d$. 
 
\sssec{} 
\label{Sect_2.2.4}
Let now $I_1\to I'_1, I_2\to I'_2$ be maps in $fSets$. Then (\ref{map_for_comm_chiral_cat_on_Fact(C)}) fits into a commutative diagram 
\begin{equation}
\label{diag_for_chiral_multiplication_in_Fact(C)}
\begin{array}{ccc} 
 C_{X^{I_1}}\boxtimes C_{X^{I_2}} & \toup{(\ref{map_for_comm_chiral_cat_on_Fact(C)})} & C_{X^{I_1\sqcup I_2}}\\
\downarrow && \downarrow \\
 C_{X^{I'_1}}\boxtimes C_{X^{I'_2}} & \toup{(\ref{map_for_comm_chiral_cat_on_Fact(C)})} & C_{X^{I'_1\sqcup I'_2}}, 
\end{array}
\end{equation} 
where the vertical maps are $!$-restrictions along the closed immersion $X^{I'_1\sqcup I'_2}\hook{} X^{I_1\sqcup I_2}$.  
Passing to the limit over $I_1, I_2\in fSet\times fSets$, the above diagram yields the functor
\begin{equation}
\label{map_beta_chiracl_mult_for_Fact(C)}
\eta: \Gamma(\Ran, \Fact(C))\boxtimes \Gamma(\Ran, \Fact(C))\to \Gamma(\Ran\times\Ran, u^!\Fact(C)).  
\end{equation}
for the sum map $u: \Ran\times\Ran\to \Ran$. Here the notation $u^!$ is that of Section~\ref{Sect_Sheaves of categories}.

\sssec{Non-unital symmetric monoidal structure} 
\label{Sect_Non-unital symmetric monoidal structure_Fact(C)(Ran)}
As in (\cite{R}, Section~7.17), we promote $\Fact(C)(\Ran)$ to an object of $CAlg^{nu}(\DGCat_{cont})$ in two ways as follows.

 Note that $\Ran$ is a pseudo-indscheme, and $u: \Ran\times\Ran\to \Ran$ is a pseudo-indproper map of pseudo-indschemes in the sense of (\cite{R}, 7.15.1). By Section~\ref{Sect_A.1.11}, for $\cE\in ShvCat(\Ran)$ there is a left adjoint 
$$
u_{!, \cE}: \Gamma(\Ran\times\Ran, u^!\cE)\to \Gamma(\Ran, \cE)
$$ 
to the restriction functor $u^!_{\cE}: \Gamma(\Ran, \cE)\to \Gamma(\Ran\times\Ran, u^!\cE)$. 

 The $\star$-product in $\Fact(C)(\Ran)$ is given by the composition
\begin{multline*}
\Gamma(\Ran, \Fact(C))\otimes \Gamma(\Ran, \Fact(C))\to \Gamma(\Ran\times\Ran, \Fact(C)\boxtimes \Fact(C))\toup{(\ref{map_beta_chiracl_mult_for_Fact(C)})}
\\ \Gamma(\Ran\times\Ran, u^!\Fact(C))\toup{u_{!, \Fact(C)}} \Gamma(\Ran, \Fact(C))
\end{multline*} 

\index{$\star$-product, $\otimes^{ch}$-product in $\Fact(C)(\Ran)$}

The $\otimes^{ch}$-product in $\Fact(C)(\Ran)$ is given by the composition
\begin{multline*}
\Gamma(\Ran, \Fact(C))\otimes \Gamma(\Ran, \Fact(C))\to \Gamma(\Ran\times\Ran, \Fact(C)\boxtimes \Fact(C))\toup{j_*j^*}\\
\Gamma(\Ran\times\Ran, \Fact(C)\boxtimes \Fact(C))
\toup{(\ref{map_beta_chiracl_mult_for_Fact(C)})}
\Gamma(\Ran\times\Ran, u^!\Fact(C))\toup{u_{!, \Fact(C)}} \Gamma(\Ran, \Fact(C)),
\end{multline*}
where $j:(\Ran\times\Ran)_d\to \Ran\times\Ran$ is the open immersion. Here the dual pair
$$
j^*: \Fact(C)\boxtimes\Fact(C)\leftrightarrows \Gamma(\Ran^2_d, \Fact(C)\boxtimes\Fact(C)): j_*
$$ 
in $Shv(\Ran\times\Ran)-mod$ is obtained from the adjoint pair $j^*: Shv(\Ran^2)\leftrightarrows Shv(\Ran^2_d): j_*$ by the base change $\_\otimes_{Shv(\Ran^2)} (\Fact(C)\boxtimes\Fact(C))$.  

\sssec{} 
\label{Sect_2.3.6_now}
A rigorous definition of the $\star$- and $\otimes^{ch}$-symmetric monoidal structure on the category $ \Gamma(\Ran, \Fact(C))$ is as follows. 

 Recall the $\infty$-operad $\Surj$ defined in (\cite{HA}, 5.4.4.1). The category $\cTw(fSets)$ is an object of $CAlg^{nu}(1-\Cat)$ via the operation sending $(J_1\to K_1), (J_2\to K_2)$ to 
$$
(J_1\sqcup J_2\to K_1\sqcup K_2).
$$ 
Consider the category $\Fun(\cTw(fSets), \DGCat_{cont})$ equipped with the Day convolution product (\cite{HA}, Construction 2.2.6.7) applied to the $\infty$-operad $\Surj$. This promotes $\Fun(\cTw(fSets), \DGCat_{cont})$ to an object of $CAlg^{nu}(1-\Cat)$. Then 
$$
\colim: \Fun(\cTw(fSets), \DGCat_{cont})\to \DGCat_{cont}
$$
is nonunital symmetric monoidal, hence yields a functor 
\begin{equation}
\label{map_between_CAlg_nu_for_Sect_2.2.6}
CAlg^{nu}(\Fun(\cTw(fSets), \DGCat_{cont}))\to CAlg^{nu}(\DGCat_{cont}).
\end{equation} 
By (\cite{HA}, 2.2.6.8), 
$$
CAlg^{nu}(\Fun(\cTw(fSets), \DGCat_{cont}))\,\iso\,\Fun^{rlax}(\cTw(fSets), \DGCat_{cont})),
$$
here the superscript $rlax$ denotes the category of right-lax non-unital symmetric monoidal functors. We equip the composition 
\begin{equation}
\label{functor_for_Sect_2.2.6}
\cTw(fSets)\toup{\cF_{\Ran, C}} Shv(\Ran)-mod\to \DGCat_{cont}
\end{equation} 
with two distinct right-lax non-unital symmetric monoidal structures as follows. 

 First, consider the right-lax non-unital symmetric monoidal structure on (\ref{functor_for_Sect_2.2.6}) given by (\ref{equiv_for_Fact(C)(Ran)_operation}). Applying (\ref{map_between_CAlg_nu_for_Sect_2.2.6}) to the corresponding composition, one obtains
$$
(\Fact(C)(\Ran), \star)\in CAlg^{nu}(\DGCat_{cont})
$$

 Second, let $(J_1\to K_1), (J_2\to K_2)\in \cTw(fSets)$ and $J=J_1\sqcup J_2, K=K_1\sqcup K_2$. Consider the right-lax non-unital symmetric monoidal structure on (\ref{functor_for_Sect_2.2.6}) given by the maps
\begin{multline*}
\cF_{\Ran, C}(J_1\to K_1)\otimes \cF_{\Ran, C}(J_2\to K_2)\to \cF_{\Ran, C}(J_1\to K_1)\boxtimes \cF_{\Ran, C}(J_2\to K_2)\toup{j_*j^*} \\ \cF_{\Ran, C}(J_1\to K_1)\boxtimes \cF_{\Ran, C}(J_2\to K_2)\, \toup{(\ref{equiv_for_Fact(C)(Ran)_operation})}\; \cF_{\Ran, C}(J\to K)
\end{multline*}
Applying (\ref{map_between_CAlg_nu_for_Sect_2.2.6}) to the corresponding composition, one obtains
$$
(\Fact(C)(\Ran), \otimes^{ch})\in CAlg^{nu}(\DGCat_{cont}).
$$

 As in \cite{FG}, the identity functor $\id: (\Fact(C), \otimes^{ch})\to (\Fact(C), \star)$ is non-unital right-lax symmetric monoidal.  

\sssec{Example} Take $C(X)=Shv(X)$. Then one has canonically in $Shv(\Ran)-mod$
$$
\Fact(C)\,\iso\, \Fact(C)(\Ran)\,\iso\,Shv(\Ran)
$$ 
 
\sssec{} 
\label{Sect_2.2.10_now}
For $C(X)\in CAlg^{nu}(Shv(X)-mod)$ one has
\begin{equation}
\label{Fact(C)_as_colimit_of_C_X^I}
\Gamma(\Ran,\Fact(C))\,\iso\, \mathop{\colim}\limits_{I\in fSets^{op}} C_{X^I}
\end{equation} 
in $\DGCat_{cont}$ and also in $\Shv(\Ran)-mod$. Write 
$$
\cF_{fSets, C}: fSets^{op}\to Shv(\Ran)-mod
$$ 
for the diagram giving the latter colimit. Let $\cF_{fSets, C}^R: fSets\to Shv(\Ran)-mod$ be the functor obtained from $\cF_{fSets, C}$ by passing to right adjoints in $Shv(\Ran)-mod$. One has $\Fact(C)\,\iso\, \underset{fSets}{\lim} \cF_{fSets, C}^R$. 

\sssec{} 
\label{Sect_2.2.11_v4}
Let $I\in fSets$. In Sections~\ref{Sect_2.2.11_v4}-\ref{Sect_2.2.15_v4}
we equip the functor 
\begin{equation}
\label{functor_C_goes_to_C_X^I_nonunital}
CAlg^{nu}(Shv(X)-mod)\to Shv(X^I)-mod, \; C(X)\mapsto C_{X^I}
\end{equation} 
with a non-unital symmetric monoidal structure. 

 Write $'Tw(I)$ for the category whose objects are collections: $I\to J_1, I\to J_2$ in $fSets$ and a surjection $J_1\sqcup J_2\to K$ with $K\in Q(I)$. That is, we require that $I\sqcup I\to K$ factors as $I\sqcup I\to I\to K$.  A map in $'Tw(I)$ is covariant in $J_1, J_2$ and contravariant in $K$, as for the category $Tw(I)$ itself. That is, a map from $(J_1, J_2, K)$ to $(J'_1, J'_2, K')$ is given by a diagram
$$
\begin{array}{ccccc}
I\sqcup I & \to & J_1\sqcup J_2 & \to & K\\
& \searrow & \downarrow\lefteqn{\scriptstyle a_1\sqcup a_2} && \uparrow\\
& & J'_1\sqcup J'_2 & \to & K',
\end{array}
$$ 
in $fSets$ with $a_i: J_i\to J'_i$. 

 For an object of $'Tw(I)$ as above and $k\in K$ we have the fibres $(J_1)_k, (J_2)_k$. 

We have the full embedding $'j: Tw(I)\to {'Tw(I)}$ given by the property that $J_1, J_2$ define the same element of the set $Q(I)$ of equivalence relations on $I$. 
Let 
$$
\upsilon: Tw(I)\times Tw(I)\to {'Tw(I)}
$$ 
be the map 
$$
(I\to J_1\to K_1), (I\to J_2\to K_2)\mapsto (J_1, J_2, K),
$$
where $K$ s the push-out of $I\gets I\sqcup I\to K_1\sqcup K_2$.

\sssec{} Let $C, D\in CAlg^{nu}(Shv(X)-mod)$. Let $\kappa: {'Tw(I)}\to Shv(X^I)-mod$ be the functor sending $(J_1, J_2, K)$ to
$$
\underset{k\in K}{\boxtimes} (C^{\otimes (J_1)_k}(X)\otimes_{Shv(X)} D^{\otimes (J_2)_k}(X))
$$
with the natural transition maps as in the definitions of $\cF_{I, C}, \cF_{I, D}$.
By definition, 
$$
C_{X^I}\otimes_{Shv(X^I)} D_{X^I}\;\iso\;\underset{\substack{(I\to J_1\to K_1)\in Tw(I)\\ 
(I\to J_2\to K_2)\in Tw(I)}}{\colim} (\underset{k\in K_1}{\boxtimes} (C^{\otimes (J_1)_k}(X))\otimes_{Shv(X^I)}(\underset{k\in K_2}{\boxtimes} D^{\otimes (J_2)_k}(X))
$$
Recall that
$$
X^{K_1\sqcup K_2}\times_{X^{I\sqcup I}} X^I\,\iso\, X^{(K_1\sqcup K_2)'},
$$ 
where $(K_1\sqcup K_2)'$ is the push-out of $I\gets I\sqcup I\to K_1\sqcup K_2$. 
So,
$$
C_{X^I}\otimes_{Shv(X^I)} D_{X^I}\;\iso\; \underset{Tw(I)\times Tw(I)}{
\colim} \kappa\comp \upsilon
$$
\begin{Lm}
\label{Lm_2.2.13_v4}
Both natural maps in the diagram 
\begin{equation}
\label{comp_maps_for_def_r_I}
\underset{Tw(I)\times Tw(I)}{\colim} \kappa\comp \upsilon\to \;\underset{'Tw(I)}{\colim} \; \kappa\gets \;\underset{Tw(I)}{\colim} \; \kappa\comp {'j}
\end{equation}
are equivalences.
\end{Lm}
\begin{proof} The map $\upsilon$ is cofinal, because it has a left adjoint $^L\upsilon: {'Tw(I)}\to Tw(I)\times Tw(I)$. Here $^L\upsilon$ sends an element $(J_1, J_2, K)$ to $(I\to J_1\to K_1), (I\to J_2\to K_2)$, where $K_i=\Im(J_i\to K)$. 

The map $'j: Tw(I)\to {'Tw(I)}$ is cofinal also, because it has a left adjoint $^L{'j}: {'Tw(I)}\to Tw(I)$ sending $(J_1, J_2, K)$ to $(I\to J\to K)$, where $J$ is the push-out of the diagram
$$
I\gets I\sqcup I\to J_1\sqcup J_2.
$$ 
Our claim follows.
\end{proof}

\sssec{} 
\label{Sect_2.2.15_v4}
Write $C\otimes D$ for $C(X)\otimes_{Shv(X)} D(X)$. Since $(C\otimes D)_{X^I}\,\iso\, \underset{Tw(I)}{\colim} \; \kappa\comp {'j}$ canonically, 
the diagram (\ref{comp_maps_for_def_r_I}) yields a canonical equivalence 
\begin{equation}
\label{functor_r_I_non-unital}
r_I: C_{X^{I}}\otimes_{Shv(X^I)} D_{X^{I}}\;\iso\; 
(C\otimes D)_{X^I}
\end{equation}
This equips (\ref{functor_C_goes_to_C_X^I_nonunital}) with the desired non-unital symmetric monoidal structure.

\sssec{} In particular, (\ref{functor_C_goes_to_C_X^I_nonunital}) upgrades to a functor 
$$
CAlg^{nu}(Shv(X)-mod)\to CAlg^{nu}(Shv(X^I)-mod), \; C(X)\mapsto C_{X^I}
$$
The non-unital symmetric monoidal structure on $C_{X^I}$ so obtained is denoted $\otimes^!$. In particular, this notation is used for the product in $C(X)\in CAlg^{nu}(Shv(X)-mod)$. 

\sssec{} Let now $C(X)\in CAlg^{nu}(Shv(X)-mod)$. For $(I\to J\to K)\in Tw(I)$ the structure functor 
\begin{equation}
\label{map_structure_C_X^I_non-unital}
\underset{k\in K}{\boxtimes} C^{\otimes J_k}(X)\to C_{X^I}
\end{equation}
is a map in $CAlg^{nu}(Shv(X^I)-mod)$. Indeed, the functor 
$$
CAlg^{nu}(Shv(X)-mod)\to Shv(X^I)-mod, \; C(X)\mapsto \underset{k\in K}{\boxtimes} C^{\otimes J_k}(X)
$$ 
is non-unital symmetric monoidal, and (\ref{map_structure_C_X^I_non-unital}) is a morphism of non-unital symmetric monoidal functors $CAlg^{nu}(Shv(X)-mod)\to Shv(X^I)-mod$. 

\sssec{} Let now $\phi: I\to I'$ be a map in $fSets$. Let $\vartriangle: X^{I'}\to X^I$ be the corresponding closed embedding. Then for $C(X)\in CAlg^{nu}(Shv(X)-mod)$
the functor $\vartriangle^!: C_{X^I}\to C_{X^{I'}}$ is non-unital symmetric monoidal, so $\underset{I\in fSets}{\lim} C_{X^I}$ can be understood in $CAlg^{nu}(Shv(\Ran)-mod)$. Denote the so obtained symmetric monoidal structure on $\Fact(C)$ by $\otimes^!$, it is distinct from $\star$.  

 More generally, given $\phi: I\to I'$ the diagram commutes
\begin{equation}
\label{diag_for_Sect_2.2.17_v4}
\begin{array}{ccc}
C_{X^I}\otimes_{Shv(X^I)} D_{X^I}&\toup{r_I} &(C\otimes D)_{X^I}\\
\downarrow &&\downarrow\\ 
C_{X^{I'}}\otimes_{Shv(X^{I'})} D_{X^{I'}}&\toup{r_{I'}} &(C\otimes D)_{X^{I'}}
\end{array}
\end{equation}
Here the vertical arrows are defined as in Section~\ref{Sect_2.1.6_v4} for
$\vartriangle: X^{I'}\to X^I$.

Passing to the limit over $I\in fSets$, we conclude that the functor
$$
CAlg^{nu}(Shv(X)-mod)\to Shv(\Ran)-mod, \; C(X)\mapsto \Fact(C)
$$
is equipped with a non-untal symmetric monoidal structure. Given $C(X), D(X)\in CAlg^{nu}(Shv(X)-mod)$ we get 
$$
r_{\Ran}: \Fact(C)\otimes_{Shv(\Ran)} \Fact(D)\,\iso\, \Fact(C\otimes D),
$$ 
where $C\otimes D$ stands for $C(X)\otimes_{Shv(X)} D(X)$. 

 In particular, our functor lifts to
$$
CAlg^{nu}(Shv(X)-mod)\to CAlg^{nu}(Shv(\Ran)-mod), \; C(X)\mapsto (\Fact(C),\otimes^!)
$$

\sssec{Example} Take $C=\Vect$, so that $\Fact(C)\,\iso\, Shv(\Ran)$. The symmetric monoidal structure denoted by $\otimes^!$ on $\Fact(C)$ in this case becomes the usual $\otimes^!$-symmetric monoidal structure on $Shv(\Ran)$. This explains our notatiion $\otimes^!$ in general.

\begin{Lm} 
\label{Lm_rlax_and_factorization_compatibility_non-unital}
The non-unital symmetric monoidal structure on (\ref{functor_C_goes_to_C_X^I_nonunital}) is compatible with the factorization. That is, for a map $\phi: I\to I'$ in $fSets$, $C,D\in CAlg^{nu}(Shv(X)-mod)$ the diagram commutes
$$
\begin{array}{ccc}
C_{X^I}\otimes_{Shv(X^I)} D_{X^I}\mid_{X^I_{\phi, d}} & \toup{r_I} & (C\otimes D)_{X^I}\mid_{X^I_{\phi, d}}\\
\downarrow && \downarrow\\
(\underset{i\in I'}{\boxtimes} C_{X^{I_i}})\otimes_{Shv(X^I)}(\underset{i\in I'}{\boxtimes} D_{X^{I_i}})\mid_{X^I_{\phi, d}} & \toup{\underset{i\in I'}{\boxtimes} r_{I_i}} & (\underset{i\in I'}{\boxtimes} (C\otimes D)_{X^{I_i}})\mid_{X^I_{\phi, d}},
\end{array}
$$
where the vertical arrows are the factorization equivalences. 
\end{Lm}
\begin{proof} Left to a reader.
\end{proof}

\sssec{} From Lemma~\ref{Lm_rlax_and_factorization_compatibility_non-unital} we conclude that we get a functor
$$
CAlg^{nu}(Shv(X)-mod)\to CAlg^{nu, Fact}(Shv(\Ran)-mod), \; C(X)\mapsto (\Fact(C), \otimes^!)
$$
where $CAlg^{nu, Fact}(Shv(\Ran)-mod)$ denotes the category of non-unital symmetric monoidal factorization categories on $\Ran$. 

\sssec{} 
\label{Sect_2.2.19_now}
Let $C(X)\in CAlg^{nu}(Shv(X)-mod)$. Set 
$$
ShvCat_C(\Ran)=\underset{I\in fSets}{\lim} C_{X^I}-mod
$$ 
calculated in $1-\Cat$. Let
$$
\Loc_C: (\Fact(C),\otimes^!)-\!\!\!\!\!\mod\to \underset{I\in fSets}{\lim} C_{X^I}-mod
$$ 
be the functor sending $M$ to the compatible family $\{M_I\}$ with $M_I=M\otimes_{(\Fact(C),\otimes^!)} C_{X^I}$ for $I\in fSets$. The functor $\Loc_C$ has a right adjoint 
$$
\{M_{X^I}\}\mapsto \Gamma(\Ran, M)=\underset{I\in fSets}{lim} M_{X^I}
$$  
calculated in $(\Fact(C),\otimes^!)-\!\!\!\mod$ or equivalently in $Shv(\Ran)-mod$ or in $\DGCat_{cont}$. 
\begin{Lm} 
\label{Lm_Loc_C_is_an_equivalence}
The functor $\Loc_C$ is an equivalence.
\end{Lm}
\begin{proof} 
{\bf Step 1} Given a map $I\to J$ in $fSets$, one checks that for $\vartriangle: X^J\to X^I$ the dual pair
$$
\vartriangle_!: C_{X^J}\leftrightarrows C_{X^I}: \vartriangle^!
$$
takes place in $C_{X^I}-mod$. In other words, for $K\in C_{X^J}, L\in C_{X^I}$ one has the projection formula: $(\vartriangle_! K)\otimes^! L\,\iso\, \vartriangle_!(K\otimes^! \vartriangle^! L)$ canonically. 

 So, for $L\in C_{X^I}-mod$ tensoring by $L$ one gets an adjoint pair $
\vartriangle_!: C_{X^J}\otimes_{C_{X^I}} L\leftrightarrows L: \vartriangle^!$
in $C_{X^I}-mod$. 

 Thus, given $\{M_I\}\in ShvCat_C(\Ran)$, we may pass to the left adjoint in the diagram $\underset{I\in fSets}{\lim} M_I$ in $\DGCat_{cont}$ and get
$$
\underset{I\in fSets}{\lim} M_I\,\iso\,\underset{I\in fSets^{op}}{\colim} M_I,
$$ 
where the colimit is calculated in $\DGCat_{cont}$ or equivalently, in $Shv(\Ran)-mod$ or equivalently in $(\Fact(C),\otimes^!)-mod$. 

\medskip\noindent
{\bf Step 2} The rest of the argument is as in Lemma~\ref{Lm_Ran_is_1-affine}. Namely, let $\cM:=\{M_I\}\in ShvCat_C(\Ran)$. We must show that the natural map
$$
\Gamma(\Ran, \cM)\otimes_{\Fact(C)} C_{X^I}\to M_I
$$
is an equivalence. We have 
$$
\Gamma(\Ran, \cM)\otimes_{\Fact(C)} C_{X^I}\,\iso\, \underset{J\in fSets^{op}}{\colim} M_J\otimes_{\Fact(C)} C_{X^I}
$$
From the 1-affineness of $\Ran$ we get $C_{X^J}\otimes_{\Fact(C)} C_{X^I}\,\iso\, \Gamma(X^I\times_{\Ran} X^J, C)$, so  
$$
M_J\otimes_{\Fact(C)} C_{X^I}\,\iso\, M_J\otimes_{C_{X^J}} C_{X^J}\otimes_{\Fact(C)} C_{X^I}\,\iso\, \Gamma(X^I\times_{\Ran} X^J, \cM).
$$
Since $\underset{J\in fSets^{op}}{\colim} \Gamma(X^I\times_{\Ran} X^J, \cM)\,\iso\, \Gamma(X^I, \cM)$, our claim follows.
\end{proof}

\ssec{Decomposition of colimits into pieces} 

\sssec{} Let $I\in fSets$ with $\mid I\mid \ge 2$. Let $\Tw(I)^{>1}\subset \Tw(I)$ be the full subcategory of those $(I\to J\to K)$ for which $\mid J\mid >1$. Let $f: I\to *$ in $fSets$.
Recall the full subcategory $\Tw(I)^f\subset \Tw(I)$, it is spaned by objects of the form $(I\to J\to *)$ in $\Tw(I)$. 
\begin{Lm} 
\label{Lm_2.2.12}
The diagram is cocartesian in the category of small categories
$$
\begin{array}{ccc}
\Tw(I)^{>1} & \to & \Tw(I)\\
\uparrow && \uparrow\\
\Tw(I)^{>1}\cap \Tw(I)^f & \hook{} & \Tw(I)^f
\end{array}
$$
\end{Lm}  
\begin{proof}
Any object of $\Tw(I)$ lies either in $\Tw(I)^{>1}$ or in $\Tw(I)^f$. Let 
$$
\Sigma_1=(I\to J_1\to K_1), \Sigma_2=(I\to J_2\to K_2)\in \Tw(I)
$$ 
Assume given a map $\Sigma_1\to \Sigma_2$ in $\Tw(I)$.

1) If $\Sigma_2\in \Tw(I)^f$ then $\Sigma_1\in \Tw(I)^f$ also.

2) If $\Sigma_2\in \Tw(I)^{>1}$ then $\Sigma_1\in Tw(I)^{>1}$ also. 

\noindent
The claim follows from Lemma~\ref{Lm_C.0.3}. 
\end{proof}

 Lemma~\ref{Lm_2.2.12} allows to decompose $\underset{\Tw(I)}\colim$ into pieces as follows.

\begin{Cor} 
\label{Cor_2.3.3}
Let $I\in fSets$. Let $h: \Tw(I)\to E$ be a map in $1-\Cat$, where $E$ is cocomplete. Then the square is cocartesian in $E$
$$
\begin{array}{ccc}
\underset{\Tw(I)^{>1}}{\colim}\; h & \to & \underset{\Tw(I)}{\colim}\;  h\\
\uparrow && \uparrow\\
\underset{\Tw(I)^{>1}\cap \Tw(I)^f}{\colim}\;  h & \to & \underset{\Tw(I)^f}{\colim} \; h,
\end{array}
$$
where all the colimits are calculated in $E$. Besides, $\underset{\Tw(I)^f}{\colim} \; h\,\iso\, h(I\to *\to *)$ 
\end{Cor}
\begin{proof}
Apply \cite{PT} to Lemma~\ref{Lm_2.2.12}. 
\end{proof}

\begin{Cor} 
\label{Cor_2.2.14}
The square is cocartesian in $Shv(X^I)-mod$
$$
\begin{array}{ccc}
\underset{\Tw(I)^{>1}}{\colim}\; \cF_{I,C} & \to & C_{X^I}\\
\uparrow && \uparrow\\
\underset{\Tw(I)^{>1}\cap \Tw(I)^f}{\colim}\; \cF_{I,C} & \toup{m} & C(X),
\end{array}
$$
where the colimits of the pieces are calculated in $Shv(X^I)-mod$. Here the right vertical map is the structure map for $(I\to *\to *)$ in $\cF_{I,C}$. \QED
\end{Cor}

\sssec{} Write $\cTw(fSets)^{>1}\subset \cTw(fSets)$ for the full subcategory of $(J\to K)\in \cTw(fSets)$ with $\mid J\mid >1$. View $fSets$ as a full subcategory of $\cTw(fSets)$ via the inclusion sending $J$ to $J\to *$. As in Lemma~\ref{Lm_2.2.12}, one shows that the square
$$
\begin{array}{ccc}
\cTw(fSets)^{>1} & \to & \cTw(fSets)\\
\uparrow && \uparrow\\
\cTw(fSets)^{>1}\cap fSets & \to & fSets
\end{array}
$$
is cocartesian in the category of small categories. 

\ssec{Factorization algebras in $\Fact(C)$} 
\label{Sect_Factorization algebras in Fact(C)}

\sssec{} For the convenience of the reader, recall the construction of the commutative factorization algebras in $\Fact(C)$ essentially given in \cite{Gai19Ran}. Let $A\in CAlg^{nu}(C(X))$. 

\index{Commutative factorization algebras in $\Fact(C)$}

\sssec{} 
\label{Sect_2.4.2_v4}
 Given a map $\phi: J\to J'$ in $fSets$, for the product map $m_{\phi}: C^{\otimes J}(X)\to C^{\otimes J'}(X)$ we get the product map $m_{\phi}(A^{\otimes J})\to A^{\otimes J'}$ in $C^{\otimes J'}(X)$ for the algebra $A$. Here $A^{\otimes J}$ denotes the $J$-th tensor power of $A$ in $C^{\otimes J}(X)$. 

 Let $I\in fSets$. Define the functor $\cF_{I, A}: Tw(I)\to C_{X^I}$ as follows. We will write $\cF_{I, A}^C=\cF_{I, A}$ if we need to express its dependence on $C$. The functor $\cF_{I, A}$ sends $(I\to J\toup{\phi} K)$ to the image under $\cF_{I,C}(I\to J\to K)\to C_{X^I}$ of the object
$$
A^{\otimes\phi}:=\mathop{\boxtimes}\limits_{k\in K} A^{\otimes J_k}\in  \mathop{\boxtimes}\limits_{k\in K} C^{\otimes J_k}(X)=\cF_{I,C}(I\to J\to K).
$$

 For a map (\ref{morphism_in_Tw_v2}) in $Tw(I)$ we get a morphism in $\mathop{\boxtimes}\limits_{k\in K_2} C^{\otimes (J_2)_k}(X)$ and hence in $C_{X^I}$
$$
\cF_{I,A}(I\to J_1\to K_1)\to  \cF_{I,A}(I\to J_2\to K_2)
$$
as follows. First, for the diagram (defining the transition functor for $\cF_{I,C}$)
\begin{equation}
\label{transition_map_for_cF_IC}
\mathop{\boxtimes}\limits_{k\in K_1} C^{\otimes (J_1)_k}(X)\toup{m} \mathop{\boxtimes}\limits_{k\in K_1} C^{\otimes (J_2)_k}(X)\to \mathop{\boxtimes}\limits_{k\in K_2} C^{\otimes (J_2)_k}(X)
\end{equation}
we get natural product map $m(\mathop{\boxtimes}\limits_{k\in K_1} A^{\otimes (J_1)_k})\to \mathop{\boxtimes}\limits_{k\in K_1} A^{\otimes (J_2)_k}$ in $\mathop{\boxtimes}\limits_{k\in K_1} C^{\otimes (J_2)_k}(X)$ for the algebra $A$. Further, for $\vartriangle: X^{K_1}\hook{} X^{K_2}$ we have 
$$
\vartriangle^!(\mathop{\boxtimes}\limits_{k\in K_2} A^{\otimes (J_2)_k})\,\iso\, \mathop{\boxtimes}\limits_{k\in K_1} A^{\otimes (J_2)_k}
$$ 
in $\mathop{\boxtimes}\limits_{k\in K_1} C^{\otimes (J_2)_k}(X)$. So, we compose the previous product map with 
$$
\vartriangle_!(\mathop{\boxtimes}\limits_{k\in K_1} A^{\otimes (J_2)_k})\,\iso\,
\vartriangle_!\vartriangle^!(\mathop{\boxtimes}\limits_{k\in K_2} A^{\otimes (J_2)_k})\to \mathop{\boxtimes}\limits_{k\in K_2} A^{\otimes (J_2)_k}
$$

 Finally, $A_{X^I}\in C_{X^I}$ is defined as $\mathop{\colim}\limits_{(I\to J\to K)\in Tw(I)} \cF_{I, A}$ in $C_{X^I}$. That is, 
$$
A_{X^I}\,\iso\, \mathop{\colim}\limits_{(I\to J\to K)\in Tw(I)} \; (\mathop{\boxtimes}\limits_{k\in K} A^{\otimes J_k})
$$
taken in $C_{X^I}$. 
 
\sssec{} Let us check that this defines indeed an object of the category of global sections $\Fact(C)(\Ran)$. For a surjection $I\to I'$ in $fSets$, we identify canonically the image of $A_{X^I}$ under (\ref{functor_restr_for_Fact(C)}) with $A_{X^{I'}}$. 

 This is done as in Section~\ref{Sect_2.1.7_now}. First, the image of $A_{X^I}$ under (\ref{functor_restr_for_Fact(C)}) writes as
$$
\mathop{\colim}\limits_{(I\to J\to K)\in Tw(I)} (\mathop{\boxtimes}\limits_{k\in K'} A^{\otimes J_k})
$$
taken in $C_{X^{I'}}$, where $K'=\inf(K, I')$. So, this is the colimit of the composition
$$
Tw(I)\toup{\beta} Tw(I)^f\toup{\cF_{I, A}^f} C_{X^{I'}},
$$
where $\cF_{I,A}^f$ is the restriction of $\cF_{I, A}$ to the full subcategory $Tw(I)^f\subset Tw(I)$ composed with the natural map $C_{X^I}\to C_{X^{I'}}$. Since $\beta$ is cofinal, the above colimit rewrites as
$$
\mathop{\colim}\limits_{(I\to J\to K)\in Tw(I)^f} \,(\mathop{\boxtimes}\limits_{k\in K} A^{\otimes J_k})
$$
taken in $C_{X^{I'}}$. Since $Tw(I')\to Tw(I)^f$ is cofinal, the above colimit rewrites as
$$
\mathop{\colim}\limits_{(I'\to J'\to K')\in Tw(I')} \, (\mathop{\boxtimes}\limits_{k\in K'} A^{\otimes J'_k})
$$
taken in $C_{X^{I'}}$, hence identifies with $A_{X^{I'}}$. This defines $\Fact(A)$ in $\Fact(C)(\Ran)$.  

\sssec{} 
\label{Sect_2.4.4_now}
As in the case of $\Fact(C)$ itself, there is also the following direct definition of $\Fact(A)$. Consider the functor
$$
\cF_{\Ran, A}: \cTw(fSets)\to \Fact(C)
$$
sending $(J\to K)$ to the image of $\mathop{\boxtimes}\limits_{k\in K} A^{\otimes J_k}$ under $\cF_{\Ran, C}(J\to K)\to \Fact(C)$. It sends a map (\ref{map_in_cTw(fSets)}) to the composition of 
$$
m(\mathop{\boxtimes}\limits_{k\in K_1} A^{\otimes (J_1)_k}) \to \mathop{\boxtimes}\limits_{k\in K_1} A^{\otimes (J_2)_k} 
$$
with 
$$
\mathop{\boxtimes}\limits_{k\in K_1} A^{\otimes (J_2)_k} \,\iso\, \vartriangle_!\vartriangle^!(\mathop{\boxtimes}\limits_{k\in K_2} A^{\otimes (J_2)_k}\to \mathop{\boxtimes}\limits_{k\in K_2} A^{\otimes (J_2)_k} 
$$
Here $m$ is that of (\ref{transition_map_for_cF_IC}). Then
$$
\Fact(A)\,\iso\, \underset{\cTw(fSets)}{\colim}\; \cF_{\Ran, A}
$$

 As in Lemma~\ref{Lm_2.1.2_about_Fact(C)}, one checks that the two definitions are equivalent.
 
\sssec{} The factorization structure on $\Fact(A)$ amounts to a collection (indexed by maps $\phi: I\to I'$ in $fSets$) of isomorphisms
\begin{equation}
\label{fact_isom_for_Fact(A)}
A_{X^I}\mid_{X^I_{\phi, d}}\,\iso\, (\mathop{\boxtimes}\limits_{i'\in I'} A_{X^{I_{i'}}})\mid_{X^I_{\phi, d}}
\end{equation} 
in $C_{X^I}\otimes_{Shv(X^I)} Shv(X^I_{\phi, d})$, where we use the equivalence (\ref{iso_after_Lm_3.2.4}) to see both sides in the same category. They have to satisfy the compatibilities similar to those for the collection of $C_{X^I}$ theirself. 

 The isomorphisms (\ref{fact_isom_for_Fact(A)}) are constructed as in Section~\ref{Sect_2.1.11_now}. This concludes the construction of $\Fact(A)$. 

\sssec{} In Section~\ref{Sect_Non-unital symmetric monoidal structure_Fact(C)(Ran)} we defined the object $(\Fact(C)(\Ran), \star)\in CAlg^{nu}(\DGCat_{cont})$. Let us promote $\Fact(A)\in \Fact(C)(\Ran)$ to a non-unital commutative algebra in $(\Fact(C)(\Ran), \star)$. 

 Let $I_1, I_2\in fSets$ with $I=I_1\sqcup I_2$. Let 
$$
(I_1\to J_1\to K_1)\in Tw(I_1), (I_2\to J_2\to K_2)\in Tw(I_2)
$$ 
Recall the functor $\alpha: Tw(I_1)\times Tw(I_2)\to Tw(I)$, let $(I\to J\to K)$ be the image of this pair under $\alpha$. Under the equivalence (\ref{equiv_for_Fact(C)(Ran)_operation}) one gets an isomorphism 
$$
(\mathop{\boxtimes}\limits_{k\in K_1}A^{\otimes (J_1)_k})\boxtimes (\mathop{\boxtimes}\limits_{k\in K_2}A^{\otimes (J_2)_k})\,\iso\, \mathop{\boxtimes}\limits_{k\in K} A^{\otimes J_k}
$$
in $\mathop{\boxtimes}\limits_{k\in K} C^{\otimes J_k}(X)$, hence also in $C_{X^{I_1}}\boxtimes C_{X^{I_2}}$. Passing to the colimit over $Tw(I_1)\times Tw(I_2)$ in $C_{X^{I_1}}\boxtimes C_{X^{I_2}}$, we get 
$$
A_{X^{I_1}}\boxtimes A_{X^{I_2}}\,\iso\, \underset{\begin{array}{cc}
{\scriptstyle (I_1\to J_1\to K_1)\in Tw(I_1)}\\
{\scriptstyle (I_2\to J_2\to K_2)\in Tw(I_2)}
\end{array}}
{\colim} \mathop{\boxtimes}\limits_{k\in K} A^{\otimes J_k}.
$$
 Applying further the functor (\ref{map_for_comm_chiral_cat_on_Fact(C)}), we get a natural map in $C_{X^I}$
\begin{equation}
\label{map_product_for_Fact(A)_for_I_1_I_2}
\beta_{I_1, I_2}: \eta_{I_1, I_2}(A_{X^{I_1}}\boxtimes A_{X^{I_2}}) \to A_{X^I}
\end{equation}
 
  Now if $I_1\to I'_1, I_2\to I'_2$ are maps in $fSets$, let $I'=I'_1\sqcup I'_2$.  
Using the commutative diagram (\ref{diag_for_chiral_multiplication_in_Fact(C)}), we identify the $!$-restriction of (\ref{map_product_for_Fact(A)_for_I_1_I_2}) under $X^{I'}\hook{} X^I$ with
$$
\beta_{I'_1, I'_2}: \eta_{I'_1, I'_2}(A_{X^{I'_1}}\boxtimes A_{X^{I'_2}}) \to A_{X^{I'}}
$$
Passing to the limit over $I, I'\in fSet\times fSets$, this gives a map 
$$
\eta(\Fact(A)\boxtimes\Fact(A))\to u^!_{\Fact(C)}\Fact(A)
$$ 
in $\Gamma(\Ran\times\Ran, u^!\Fact(C))$, here $\eta$ is the morphism (\ref{map_beta_chiracl_mult_for_Fact(C)}). 
Recall that here  $u^!_{\Fact(C)}: \Gamma(\Ran, \Fact(C))\to \Gamma(\Ran\times\Ran, u^!\Fact(C))$ the pullback for sections, which has a left adjoint
adjoint 
$$
u_{!,\Fact(C)}: \Gamma(\Ran\times\Ran, u^!\Fact(C))\to \Gamma(\Ran, \Fact(C)),
$$ 
cf. Section~\ref{Sect_Non-unital symmetric monoidal structure_Fact(C)(Ran)}. By definition of the monoindal category $(\Gamma(\Ran,\Fact(C)), \star)$, this gives a map
$$
\Fact(A)\star \Fact(A)\to \Fact(A)
$$
 
  This is the product on $\Fact(A)$ in $(\Gamma(\Ran, \Fact(C)),\star)$. 
 
\sssec{} The rigorous way to promote $\Fact(A)$ to an object of $CAlg^{nu}(\Gamma(\Ran, \Fact(C)),\star)$ is as in Section~\ref{Sect_2.3.6_now}. Namely, one shows that 
$$
\cF_{\Ran, A}: \cTw(fSets)\to (\Fact(C), \star)
$$ 
is right-lax non-unital symmetric monoidal. The category $\Fun(\cTw(fSets), \Fact(C))$ has a Day convolution structure, so is an object of $CAlg^{nu}(\DGCat_{cont})$ as in (\cite{Ly}, 9.2.41). The functor 
$$
\colim: \Fun(\cTw(fSets), \Fact(C))\to (\Fact(C), \star)
$$ 
is nonunital symmetric monoidal. So, it gives
$$
\colim: CAlg^{nu}(\Fun(\cTw(fSets), \Fact(C))\to  CAlg^{nu}(Fact(C), \star)
$$ 
By (\cite{HA}, 2.2.6.8), 
$$
CAlg^{nu}(\Fun(\cTw(fSets), \Fact(C))\,\iso\, \Fun^{rlax}(\cTw(fSets), \Fact(C)).
$$ 
So, $\underset{\cTw(fSets)}{\colim} \cF_{\Ran, A}$ gets a structure of a non-unital commutative algebra in $(\Fact(C), \star)$. 

\begin{Lm} 
\label{Lm_2.4.8_v4}
Let $C,D\in CAlg^{nu}(Shv(X)-mod)$ and $A\in CAlg^{nu}(C(X), \otimes^!)$, $B\in CAlg^{nu}(D(X), \otimes^!)$. \\
i) Let $I\in fSets$. Under the equivalence (\ref{functor_r_I_non-unital}) one has a canonical isomorphism
$$
r_I(A_{X^I}\otimes_{Shv(X^I)} B_{X^I})\,\iso\, (A\otimes B)_{X^I}
$$
in $(C\otimes D)_{X^I}$. Here $A\otimes B$ denotes their exterior tensor product in $C(X)\otimes_{Shv(X)} D(X)$, so $A\otimes B\in CAlg^{nu}(C\otimes D)$. 

\smallskip\noindent
ii) Let $\phi: I\to I'$ be a map in $fSets$. Then the image of the above isomorphism under the $!$-restriction for $X^{I'}\to X^I$ identifies via (\ref{diag_for_Sect_2.2.17_v4}) with the corresponding isomorphism
$$
r_{I'}(A_{X^{I'}}\otimes_{Shv(X^{I'})} B_{X^{I'}})\,\iso\, (A\otimes B)_{X^{I'}}
$$ 
So, we also have canonically 
$$
r_{\Ran}(\Fact(A)\otimes_{Shv(\Ran)} \Fact(B))\,\iso\, \Fact(A\otimes B)
$$
in $\Fact(C\otimes D)$. 
\end{Lm}
\begin{proof}
i) Let $(I\to J_1\to K_1), (I\to J_2\to K_2)\in Tw(I)\times Tw(I)$. Let $(J_1, J_2, K)$ be its image under $\upsilon$. The image of 
$$
(\underset{k\in K_1}{\boxtimes} A^{\otimes (J_1)_k})\otimes_{Shv(X^I)}(\underset{k\in K_2}{\boxtimes} B^{\otimes (J_2)_k})
$$ 
under
$$
(\underset{k\in K_1}{\boxtimes} C^{\otimes (J_1)_k}(X))\otimes_{Shv(X^I)}(\underset{k\in K_2}{\boxtimes} D^{\otimes (J_2)_k}(X))\,\iso\, \underset{k\in K}{\boxtimes} (C^{\otimes (J_1)_k}(X)\otimes_{Shv(X)} D^{\otimes (J_2)_k}(X))
$$
is
\begin{equation}
\label{functor_bar_kappa}
\underset{k\in K}{\boxtimes} (A^{\otimes (J_1)_k}\otimes B^{\otimes (J_2)_k})
\end{equation}
So, $A_{X^I}\otimes_{Shv(X^I)} B_{X^I}$ is the colimit of the composition 
$$
Tw(I)\times Tw(I)\toup{\upsilon} {'Tw(I)}\toup{\bar\kappa} C_{X^I}\otimes_{Shv(X^I)} D_{X^I},$$
where $\bar\kappa$ sends $(J_1, J_2, K)$ to (\ref{functor_bar_kappa}). Arguing now as in Lemma~\ref{Lm_2.2.13_v4}, one concludes the proof.

\smallskip\noindent
ii) Left to a reader. 
\end{proof}
\begin{Rem} The isomorphisms of Lemma~\ref{Lm_2.4.8_v4} are compatible with factorization as in Lemma~\ref{Lm_rlax_and_factorization_compatibility_non-unital}.
\end{Rem}

\sssec{} 
\label{Sect_2.4.10_v4}
Let now $C(X)\in CAlg^{nu}(Shv(X)-mod)$. Let us equip the functor 
\begin{equation}
\label{functor_A_goes_to_A_X^I_non-unital}
CAlg^{nu}(C(X))\to (C_{X^I}, \otimes^!),\; A\mapsto A_{X^I}
\end{equation}
with a non-unital symmetric monoidal structure.
The product in $(C_{X^I}, \otimes^!)$ is given by the composition
$$
C_{X^I}\otimes_{Shv(X^I)} C_{X^I}\toup{r_I} (C\otimes C)_{X^I}\toup{\otimes^!_{X^I}} C_{X^I},
$$
where the map $\otimes^!_{X^I}$ comes by functoriality from the morphism $\otimes^!: C\otimes C\to C$ in $CAlg^{nu}(Shv(X)-mod)$. 

 Given $A, B\in CAlg^{nu}(C(X))$, write $A\otimes B$ for their exterior tensor product in $(C\otimes C)(X)$. Apply Remark~\ref{Rem_image_of_A_X^I} to the map $\otimes^!: C\otimes C\to C$. It shows that the image of $(A\otimes B)_{X^I}$ under $\otimes^!_{X^I}: (C\otimes C)_{X^I}\to C_{X^I}$ identifies canonically with $(A\otimes^!B)_{X^I}$. This gives a canonical isomorphism
$$
A_{X^I}\otimes^! B_{X^I}\,\iso\, (A\otimes^! B)_{X^I}
$$
in $C_{X^I}$. This is the desired non-unital symmetric monoidal structure on (\ref{functor_A_goes_to_A_X^I_non-unital}). 

\begin{Rem} 
\label{Rem_image_of_A_X^I}
If $f: C(X)\to C'(X)$ is a map in $CAlg^{nu}(Shv(X))$ and $A\in CAlg^{nu}(C(X))$ then for any $I\in fSets$, the image under $f_{X^I}: C_{X^I}\to C'_{X^I}$ of $A_{X_I}$ identifies canonically with $(f(A))_{X^I}$. 
\end{Rem}

\sssec{} In particular, (\ref{functor_A_goes_to_A_X^I_non-unital}) lifts naturally to a functor
$$
CAlg^{nu}(C(X))\to CAlg^{nu}(C_{X^I}, \otimes^!),\; A\mapsto A_{X^I}
$$

\sssec{} Let now $A\in CAlg^{nu}(C(X),\otimes^!)$. For each $I\in fSets$ one gets $A_{X^I}-mod(C_{X^I})\in C_{X^I}-mod$. Note that for $I\to J$ in $fSets$ and $\vartriangle: X^J\to X^I$ one has canonically
$$
A_{X^I}-mod(C_{X^I})\otimes_{C_{X^I}} C_{X^J}\,\iso\, A_{X^J}-mod(C_{X^J})
$$
by (\cite{G}, ch. I.1, 8.5.7), here the functor $C_{X^I}\to C_{X^J}$ is $\vartriangle^!$. So, the compatible collection $\{A_{X^I}-mod(C_{X^I})\}$ is an object of $ShvCat_C(\Ran)$. In view of Lemma~\ref{Lm_Loc_C_is_an_equivalence} it is the image under $\Loc_C$ of
$$
\Fact(A)-mod(\Fact(C))\in (\Fact(C),\otimes^!)-mod
$$ 
As in Lemma~\ref{Lm_Loc_C_is_an_equivalence}, 
$$
\Fact(A)-mod(\Fact(C))\,\iso\, \underset{I\in fSets^{op}}{
\colim} A_{X^I}-mod(C_{X^I}), 
$$
where the colimit is understood in $\DGCat_{cont}$ or equivalently in $(\Fact(C), \otimes^!)-mod$.

 Note that $\Fact(A)-mod(\Fact(C))$ gets a structure of a factorization category over $\Ran$. Namely, for a map $\phi: I\to I'$ in $fSets$ recall the isomorphism (\ref{fact_isom_for_Fact(A)}) via the equivalence (\ref{iso_after_Lm_3.2.4}). By (\cite{G}, ch. I.1, 8.5.4), we get canonically
$$
A_{X^I}-mod(C_{X^I})\mid_{X^I_{\phi, d}}\,\iso\, (\underset{i'\in I'}{\boxtimes} A_{X^{I_{i'}}}-mod(C_{X^{I_{i'}}}))\mid_{X^I_{\phi, d}}
$$  
in $Shv(X^I_{\phi, d})-mod$ and also in $C_{X^I}\mid_{X^I_{\phi, d}}-mod$.

 Note that if $M\in (\Fact(C),\otimes^!)-mod$ one has
\begin{multline*}
\Fact(A)-mod(M)\,\iso\, \Fact(A)-mod(\Fact(C))\otimes_{\Fact(C)} M\,\iso\,
\underset{I\in fSets^{op}}{\colim} A_{X^I}-mod(M_I),
\end{multline*}
where $M_I=C_{X^I}\otimes_{\Fact(C)} M$.  

\sssec{} Consider the natural map $\vartriangle: X\to \Ran$. It is easy to see that the functor $\vartriangle^!: (\Fact(C), \star)\to (C(X),\otimes^!)$ is non-unital symmetric monoidal. So, it gives rise to the functor $\vartriangle^!: CAlg^{nu}(\Fact(C), \star)\to CAlg^{nu}(C(X))$. The latter functor preserves limits, and both categories are presentable by (\cite{HA}, 3.2.3.5), so it admits a left adjoint $\cL$. It is easy to see that $\cL$ is fully faithful.

\begin{Pp} 
\label{Pp_2.4.15}
The functor 
$$
\Fact: CAlg^{nu}(C(X))\to CAlg^{nu}(\Fact(C), \star), \; A\mapsto \Fact(A)
$$ 
is left adjoint to $\vartriangle^!: CAlg^{nu}(\Fact(C), \star)\to CAlg^{nu}(C(X),\otimes^!)$. Besides, $\Fact$ is fully faithful.
\end{Pp} 
\begin{proof}
More generally, let $\cC,\cD\in CAlg^{nu}(\DGCat_{cont})$ and $F: \cC\leftrightarrows \cD: G$
be an adjoint pair in $\DGCat_{cont}$, where $G$ is non-unital symmetric monoidal. So, $F$ is left-lax nonunital symmetric monoidal. Then $G$ gives rise to the functor
$G_{enh}: CAlg^{nu}(\cD)\to CAlg^{nu}(\cC)$ satisfying $\oblv G_{enh}\,\iso\, G\oblv$ for the forgetful functors $\oblv: CAlg^{nu}(\cD)\to \cD$, $\oblv: CAlg^{nu}(\cC)\to\cC$.
In this situation $G_{enh}$ admits a left adjoint $F_{enh}$ described in (\cite{Ly8}, 1.2.14). 
Moreover, if $F$ is fully faithful then $F_{enh}$ is also fully faithful by (\cite{Ly8}, 1.1.3).

The functor $F_{enh}$ is given by an explicit construction from (\cite{Ly8}, 1.2.14), which translates in our case into the definition of $\Fact(A)$ from Section~\ref{Sect_2.4.4_now}. Our claim follows.
\end{proof}

\begin{Cor} The diagram commutes naturally
$$
\begin{array}{ccc}
CAlg^{nu}(C(X)) & \toup{\Fact} & CAlg^{nu}(\Fact(C), \star)\\
\uparrow\lefteqn{\scriptstyle free}&& \uparrow\lefteqn{\scriptstyle free}\\
C(X) &\toup{\vartriangle_!} & \Fact(C),
\end{array}
$$
where $free$ denotes the left adjoint to the corresponding oblivion functor. 
\end{Cor}
\begin{proof} The diagram obtained by passing to right adjoints commutes.
\end{proof}

\ssec{Unital versions}
\label{Subsect_unital_case}

\sssec{} Many of the previous results have unital versions. We indicate some of them. 

As in Section~\ref{Sect_2.2.11_v4}, the functor
$$
CAlg(Shv(X)-mod)\to Shv(X^I)-mod, \; C(X)\mapsto C_{X^I}
$$
has a symmetric monoidal structure (it preserves the units). So, it upgrades to a functor
$$
CAlg(Shv(X)-mod)\to CAlg(Shv(X^I)-mod), \; C(X)\mapsto (C_{X^I}, \otimes^!)
$$
The above two functors are compatible with the factorization. The unit of $C_{X^I}$ is the image of $\omega_{X^I}$ under the unit map $Shv(X^I)\to C_{X^I}$. 

\sssec{} If $C\in CAlg(Shv(X)-mod)$. For $(I\to J\to K)\in Tw(I)$ the structure functor (\ref{map_structure_C_X^I_non-unital}) is a map only in $CAlg^{nu}(Shv(X)-mod)$. 
 
\sssec{} The functor 
$$
CAlg(Shv(X)-mod)\to Shv(\Ran)-mod,\; C(X)\mapsto \Fact(C)
$$
has a symmetric monoidal structure, so lifts to a functor
$$
CAlg(Shv(X)-mod)\to CAlg(Shv(\Ran)-mod),\; C(X)\mapsto \Fact(C)
$$ 
The correspoonding symmetric monoidal structure on $\Fact(C)$ is $\otimes^!$. These two functors are compatible with the factorization. The unit of $(\Fact(C),\otimes^!)$ is the image of $\omega_{\Ran}$ under the unit map $Shv(\Ran)\to \Fact(C)$.  

 This gives a functor 
$$
CAlg(Shv(X)-mod)\to CAlg^{Fact}(Shv(\Ran)-mod), C(X)\mapsto (\Fact(C), \otimes^!),
$$
where $CAlg^{Fact}(Shv(\Ran)-mod)$ denotes the category of symmetric monoidal factorization categories on $\Ran$.

\sssec{} Let $C(X)\in CAlg(Shv(X)-mod)$. As in Section~\ref{Sect_2.4.10_v4}, the functor 
\begin{equation}
\label{functor_A_goes_to_A_X^I}
CAlg(C(X))\to (C_{X^I}, \otimes^!),\; A\mapsto A_{X^I}
\end{equation}
is symmetric monoidal. The fact that it preserves the units follows from Remark~\ref{Rem_image_of_A_X^I}. In particular, (\ref{functor_A_goes_to_A_X^I}) lifts to a functor
$$
CAlg(C(X))\to CAlg(C_{X^I}, \otimes^!),\; A\mapsto A_{X^I}. 
$$

\ssec{Further properties}
\label{Sect_Further properties}

\sssec{} In Section~\ref{Sect_Further properties} we assume that $C(X)\in CAlg^{nu}(Shv(X)-mod)$, $C(X)$ is dualizable as a $Shv(X)$-module, the functor $C^{\otimes 2}(X)\toup{m} C(X)$ admits a continuous right adjoint, which is $Shv(X)$-linear.

\begin{Lm} 
\label{Lm_3.2.6_for_sheaves_of_cat}
Let $I\in fSets$. \\
i) For each $(I\to J\toup{\phi} K)\in Tw(I)$, $C^{\otimes\phi}$ is dualizable in $Shv(X^I)-mod$. 

\smallskip\noindent
ii) Each transition functor in the diagram $\cF_{I, C}: Tw(I)\to Shv(X^I)-mod$ admits a continuous $Shv(X^I)$-linear right adjoint. The category $\cC_{X^I}$ is dualizable in $Shv(X^I)-mod$. Besides, for any $(I\to J\to K)\in Tw(I)$ the natural functor 
\begin{equation}
\label{functor_for_Lm_3.5.4_Fact}
\mathop{\boxtimes}\limits_{k\in K} \cC^{\otimes J_k}(X)\to \cC_{X^I}
\end{equation}
admits a continuous right adjoint, which is a strict morphism of $Shv(X^I)$-modules. 
\end{Lm}  
\begin{proof}
\noindent
i) First, $\underset{j\in J}{\boxtimes} \cC(X)$ is dualizable as a $Shv(X^J)$-module, as the functor 
$$
\prod_J Shv(X)-mod\to Shv(X^J)-mod
$$ 
of exterior product is symmetric monoidal. Now the extensions of scalars functor $Shv(X^J)-mod\to Shv(X^K)-mod$ with respect to $\vartriangle^!: Shv(X^J)\to Shv(X^K)$ is symmetric monodal. So, $\mathop{\boxtimes}\limits_{k\in K} C^{\otimes J_k}(X)$ is dualizable in $Shv(X^K)-mod$. Applying (\cite{Ly}, 9.2.32) for the colocalization $\vartriangle^{(I/K)}_!: 
Shv(X^K)\leftrightarrows Shv(X^I): \vartriangle^{(I/K) !}$, we conclude that $\mathop{\boxtimes}\limits_{k\in K} C^{\otimes J_k}(X)$ is dualizable in $Shv(X^I)-mod$. 

\medskip\noindent
ii) Consider a morphism in $Tw(I)$ given by (\ref{morphism_in_Tw_v2}). We claim that in the diagram
$$
\mathop{\boxtimes}\limits_{k\in K_1} \cC^{\otimes (J_1)_k}(X)\to \mathop{\boxtimes}\limits_{k\in K_1} \cC^{\otimes (J_2)_k}(X)\to \mathop{\boxtimes}\limits_{k\in K_2} \cC^{\otimes (J_2)_k}(X)
$$
both maps admit continuous right adjoints, which are $Shv(X^I)$-linear. For the first map we first check that it is $Shv(X^{K_1})$-linear using (\cite{Ly}, 4.1.6), and apply the functor of direct image $Shv(X^{K_1})-mod\to Shv(X^I)-mod$. For the second map we use the fact that for any $M\in Shv(X^{K_2})-mod$, we have an adjoint pair 
$$
\vartriangle_!: M\otimes_{Shv(X^{K_2})} Shv(X^{K_1})\leftrightarrows M: \vartriangle^!
$$
in $Shv(K^2)-mod$, which is also an adjoint pair in $Shv(X^I)-mod$. 

 So, in the functor $\cF_{I,C}: Tw(I)\to Shv(X^I)-mod$ we may pass to right adjoints and get a functor $\cF_{I,C}^R: Tw(I)^{op}\to Shv(X^I)-mod$. Recall that $\oblv: Shv(X^I)-mod\to \DGCat_{cont}$ preserves limits and colimits, so we may understand $\lim \cF^R_I$ either in $\DGCat_{cont}$ or in $Shv(X^I)-mod$. Recall that $\colim \cF_I\,\iso\, \lim \cF^R_I$, where the limit is understood in $\DGCat_{cont}$, the claim about the right adjoint to (\ref{functor_for_Lm_3.5.4_Fact}) follows. To get the dualizability of $C_{X^I}$ in $Shv(X^I)-mod$ apply (\cite{Ly}, 3.1.10). 
\end{proof} 

\sssec{} The proof of Lemma~\ref{Lm_3.2.6_for_sheaves_of_cat} shows that each transition map in the diagram (\ref{functor_cF_Ran_C}) admits a $Shv(\Ran)$-linear continuous right adjoint. Let $\cF^R_{\Ran, C}: \cTw(fSets)^{op}\to Shv(\Ran)-mod$ be obtained from $\cF_{\Ran, C}$ by passing to right adjoints. Note that
$$
\Fact(C)\,\iso\, \underset{\cTw(fSets)^{op}}{\lim}\, \cF^R_{\Ran, C}
$$
in $Shv(\Ran)-mod$ and in $\DGCat_{cont}$. 

\begin{Cor} The category $\Fact(C)$ is dualizable in $Shv(\Ran)-mod$. 
\end{Cor}
\begin{proof}
For each $I\in fSets$, $\Fact(C)\otimes_{Shv(\Ran)} Shv(X^I)\,\iso\, C_{X^I}$ is dualizable in $Shv(X^I)-mod$ by Lemma~\ref{Lm_3.2.6_for_sheaves_of_cat}. Our claim follows now from Corollary~\ref{Cor_2.2.9_dualizability}. 
\end{proof}

 Let $\cF_{fSets, C}^{\vee}: fSets\to Shv(\Ran)-mod$ be obtained from the functor $\cF_{fSets, C}$ by passing to the duals in $Shv(\Ran)-mod$.

\sssec{} 
\label{Sect_2.5.5}
The right adjoint $m^R: C(X)\to C^{\otimes 2}(X)$ of $m$ defines on $C(X)$ the structure of a non-unital cocommutative coalgebra in $Shv(X)-mod$, that is, an object of $CoCAlg^{nu}(Shv(X)-mod)$. Write $C^{\vee}(X)$ for the dual of $C(X)$ in $Shv(X)-mod$. It
becomes an object of $CAlg^{nu}(Shv(X)-mod)$ with the product 
$$
(m^R)^{\vee}: (C^{\vee})^{\otimes 2}(X)\to C^{\vee}(X)
$$  

 Under our assumptions, the map $C(X)\mapsto C^{\vee}(X)$ is an involution. It interacts nicely with the construction of $\Fact(C)$, we discuss this in the next subsection.
 
\sssec{} Write $\cF_{I, C}^{\vee}: Tw(I)^{op}\to Shv(X^I)-mod$ for the functor obtained from $\cF_{I,C}$ by passing to the duals.

\begin{Lm} The functor $\cF_{I, C}^{\vee}: Tw(I)^{op}\to Shv(X^I)-mod$ identifies canonically with $\cF_{I, C^{\vee}}^R$. For $I\in fSets$ one has canonically in $Shv(X^I)-mod$
\begin{equation}
\label{iso_for_Sect_2.4.6}
(C_{X^I})^{\vee}\,\iso\, (C^{\vee})_{X^I}
\end{equation}
\end{Lm}
\begin{proof}
For $\Sigma=(I\to J\to K)\in Tw(I)$, the dual of $\mathop{\boxtimes}\limits_{k\in K} C^{\otimes J_k}(X)$ in $Shv(X^I)-mod$ is $\mathop{\boxtimes}\limits_{k\in K} (C^{\vee})^{\otimes J_k}(X)$. From Lemma~\ref{Lm_3.2.6_for_sheaves_of_cat} we conclude that the dual of $C_{X^I}$ in $Shv(X^I)-mod$ writes as
\begin{equation}
(C_{X^I})^{\vee}\,\iso\,
\label{diag_for_(C_X^I)_dual}
\mathop{\lim}\limits_{(I\to J\to K\in Tw(I)^{op}} \; \mathop{\boxtimes}\limits_{k\in K}\; (C^{\vee})^{\otimes J_k}(X)
\end{equation}
the limit taken in $Shv(X^I)-mod$. For a map (\ref{morphism_in_Tw_v2}) in $Tw(I)$ the transition map in the latter limit is
$$
\mathop{\boxtimes}\limits_{k\in K_2} (C^{\vee})^{\otimes (J_2)_k}(X)\toup{\vartriangle^!} \mathop{\boxtimes}\limits_{k\in K_1} (C^{\vee})^{\otimes (J_2)_k}(X)\toup{m^{\vee}}
 \mathop{\boxtimes}\limits_{k\in K_1} (C^{\vee})^{\otimes (J_1)_k}(X)
$$
for $\vartriangle: X^{K_1}\to X^{K_2}$.

 We may pass to the left adjoints in $Shv(X^I)-mod$ in the diagram (\ref{diag_for_(C_X^I)_dual}), and get
$$
(C_{X^I})^{\vee}\,\iso\, \mathop{\colim}\limits_{(I\to J\to K\in Tw(I)} \; \mathop{\boxtimes}\limits_{k\in K} (C^{\vee})^{\otimes J_k}(X)
$$
The corresponding diagram is nothing but the functor $\cF_{I, C^{\vee}}$. This yields the equivalence 
(\ref{iso_for_Sect_2.4.6}).
\end{proof}

 Note that for $D\in Shv(X^I)-mod$ one has 
\begin{equation}
\label{isom_after_Lm_2.5.7}
\Fun_{Shv(X^I)}(C_{X^I}, D)\,\iso\, (C_{X^I})^{\vee}\otimes_{Shv(X^I)} D\,\iso\, (C^{\vee})_{X^I}\otimes_{Shv(X^I)} D
\end{equation}  

\begin{Pp} 
\label{Pp_2.4.7}
i) The functor 
$$
\cF^{\vee}_{fSets, C}: fSets\to Shv(\Ran)-mod
$$ 
identifies canonically with $\cF^R_{fSets, C^{\vee}}$.\\
ii) The canonical pairing between $\underset{fSets^{op}}{
\colim} \cF_{fSets, C}$ and $\underset{fSets}{\lim} \cF^{\vee}_{fSets, C}$ provides a counit of a duality datum 
$$
\Fact(C)\otimes_{Shv(\Ran)} \Fact(C^{\vee})\to Shv(\Ran)
$$ 
in $Shv(\Ran)-mod$. So, 
$$
\Fact(C)^{\vee}\,\iso\, \Fact(C^{\vee})
$$ 
canonically in $Shv(\Ran)-mod$. 
\end{Pp}
\begin{proof}
i) Let $J\to K$ be a map in $fSets$. From (\ref{iso_for_Sect_2.4.6}) we conclude that the corresponding transition functor in the diagram $\cF^{\vee}_{fSets, C}$ identifies with the map $\vartriangle^!: (C^{\vee})_{X^J}\to ((C^{\vee})_{X^J})\otimes_{Shv(X^J)} Shv(X^K)\,\iso\, (C^{\vee})_{X^K}$, where $\vartriangle: X^K\to X^J$. 

\medskip\noindent
ii) This follows from i) as in (\cite{G}, I.1, 6.3.4).
\end{proof}

\begin{Rem} The $\star$-product $\Fact(C)\otimes \Fact(C)\to\Fact(C)$ admits a continuous right adjoint.
\end{Rem}
\begin{proof}
Given $(J^1\to K^1), (J^2\to K^2)\in\cTw(fSets)$ let $J=J_1\cup J_2, K=K_1\cup K_2$, so $(J\to K)\in\cTw(fSets)$. 
The functor under consideration is the natural map
\begin{multline*}
\underset{(J_1\to K_1, J_2\to K_2)\in\cTw(fSets)^2}{\colim} (\underset{k\in K_1}{\boxtimes} C^{\otimes J^1_k}(X))\otimes (\underset{k\in K_2}{\boxtimes} C^{\otimes J^2_k}(X))\to \\ \underset{(J_1\to K_1, J_2\to K_2)\in\cTw(fSets)^2}{\colim} (\underset{k\in K}{\boxtimes} C^{\otimes J_k}(X))\to \underset{(J\to K)\in\cTw(fSets)}{\colim} (\underset{k\in K}{\boxtimes} C^{\otimes J_k}(X)).
\end{multline*}
For $(J^1\to K^1), (J^2\to K^2)\in\cTw(fSets)$ the functor
$$
(\underset{k\in K_1}{\boxtimes} C^{\otimes J^1_k}(X))\otimes (\underset{k\in K_2}{\boxtimes} C^{\otimes J^2_k}(X))\to \underset{k\in K}{\boxtimes} C^{\otimes J_k}(X)
$$ 
admits a continuous right adjoint. Indeed, $\boxtimes: Shv(X^{K_1})\otimes Shv(X^{K_2})\to Shv(X^K)$ has a continuous $Shv(X^{K_1})\otimes Shv(X^{K_2})$-linear right adjoint by (\cite{Ly4}, 0.0.7). Now, since the transition map for the diagram (\ref{functor_cF_Ran_C}) admit continuous right adjoints, our claim follows from (\cite{Ly}, 9.2.6).
\end{proof}
\sssec{} 
\label{Sect_2.5.9}
For the rest of Section~\ref{Sect_Further properties} assume in addition that $C(X)\in CAlg(Shv(X)-mod)$, $C(X)$ is compactly generated and ULA over $Shv(X)$. Once again, $C(X)$ is now assumed unital. The ULA property is discussed in Section~\ref{appendix_some_generalities}. 

\index{category ULA over some base}

\begin{Rem}
By Lemma~\ref{Lm_C^vee(X)_is_still_ULA}, under our assumptions $C^{\vee}(X)$ is also ULA over $Shv(X)$. In particular, $C^{\vee}(X)$ is compactly generated. 
\end{Rem}

\begin{Lm} 
\label{Lm_ULA_general_lemma}
For any $I\in fSets$, $C_{X^I}$ is ULA over $Shv(X^I)$ in the sense of Definition~\ref{Def_ULA_module-category}. In particular, $C_{X^I}$ is compactly generated. 
\end{Lm} 
\begin{proof} 
{\bf Step 1} Recall that our notation $\mathop{\boxtimes}\limits_{i\in I} C(X)$ means $(C(X)^{\otimes I})\otimes_{(Shv(X)^{\otimes I})} Shv(X^I)$. Let us show that  the latter category is compactly generated by objects of the form 
\begin{equation}
\label{object_for_Lm3.5.6}
(\mathop{\boxtimes}\limits_{i\in I} c_i)\otimes_{(Shv(X)^{\otimes I})} z
\end{equation} 
with $c_i\in C(X)$ ULA over $Shv(X)$, and $z\in Shv(X^I)^c$. In the case of $\cD$-modules, $Shv(X)^{\boxtimes I}\to Shv(X^I)$ is an equivalence, and there is nothing to prove. Assume now we are in the constructible context. 

 In this case for any $S\in\Sch_{ft}$, $\otimes^!: Shv(S)\otimes Shv(S)\to Shv(S)$ has a continuous right adjoint. Note that if $c_i\in \cC(X)$ is ULA over $Shv(X)$ then $\otimes_{i\in I} \; c_i\in C(X)^{\otimes I}$ is ULA over $Shv(X)^{\otimes I}$. So, for any $z\in Shv(X^I)^c$, 
(\ref{object_for_Lm3.5.6}) is compact in $(C(X)^{\otimes I})\otimes_{(Shv(X)^{\otimes I})} Shv(X^I)$ by Remark~\ref{Rem_3.6.3_now}.

 Let $\cD\subset (C(X)^{\otimes I})\otimes_{(Shv(X)^{\otimes I})} Shv(X^I)$ be a full embedding in $\DGCat_{cont}$ such that $\cD$ contains all the objects of the form (\ref{object_for_Lm3.5.6}). Then it contains all the objects $c'\otimes_{(Shv(X)^{\otimes I})} z$ for $c'\in C(X)^{\boxtimes I}, z\in Shv(X^I)$ by (\cite{G}, I.1, 7.4.2). Applying in addition (\cite{G}, I.1, 8.2.6), we see that $\cD=(C(X)^{\otimes I})\otimes_{(Shv(X)^{\otimes I})} Shv(X^I)$.
 
\medskip\noindent
{\bf Step 2} If $c_i\in \cC(X)$ are ULA over $Shv(X)$ then $\underset{i\in I}{\boxtimes} c_i\,\iso\,(\mathop{\otimes}\limits_{i\in I} c_i)\otimes_{(Shv(X)^{\otimes I})} \omega_{X^I}$ is ULA over $Shv(X^I)$. 

 Indeed, consider the adjoint pair $Shv(X)^{\otimes I}\leftrightarrows \cC(X)^{\otimes I}$ in $Shv(X)^{\otimes I}-mod$, where the left adjoint is the multimlication by $\otimes_{i\in I}\,  c_i$. Tensoring with $Shv(X^I)$ over $Shv(X)^{\otimes I}$, we get the desired adjoint pair in $Shv(X^I)$. 
  
\medskip\noindent
{\bf Step 3} By Lemma~\ref{Lm_3.5.2_generated_under_colim}, the essential image of $\Loc: \mathop{\boxtimes}\limits_{i\in I} C(X)\to C_{X^I}$ generates $C_{X^I}$ under colimits. Now if $c_i\in \cC(X)$ are ULA over $Shv(X)$, $\underset{i\in I}{\boxtimes} c_i\in \mathop{\boxtimes}\limits_{i\in I} C(X)$ is ULA over $Shv(X^I)$. By Proposition~\ref{Pp_3.6.6} and Lemma~\ref{Lm_3.2.6_for_sheaves_of_cat}, $\Loc(\underset{i\in I}{\boxtimes} c_i)$ is ULA over $Shv(X^I)$. 

 Let $c_i\in C(X)$ be ULA over $Shv(X)$, $z\in Shv(X^I)^c$. By Remark~\ref{Rem_3.6.3_now}, 
$$
(\mathop{\boxtimes}\limits_{i\in I} c_i)\otimes_{(Shv(X)^{\boxtimes I})} z \in (\mathop{\boxtimes}\limits_{i\in I} C(X))^c.
$$ 
By Lemma~\ref{Lm_3.2.6_for_sheaves_of_cat}, 
$$
\Loc((\mathop{\boxtimes}\limits_{i\in I} c_i)\otimes_{(Shv(X)^{\boxtimes I})} z)
$$
is compact in $C_{X^I}$, and these objects generate $C_{X^I}$ by Lemma~\ref{Lm_3.5.2_generated_under_colim}.
\end{proof}

\begin{Cor} $\Fact(C)$ is compactly generated.
\end{Cor}
\begin{proof}
By Lemma~\ref{Lm_ULA_general_lemma}, if $I\in fSets$ then $C_{X^I}$ is compactly generated. Now Apply (\cite{G}, I.1, 7.2.7) to the formula (\ref{Fact(C)_as_colimit_of_C_X^I}).
\end{proof}

\ssec{Properties of some functors}
\label{Sect_2.7_Properties of some functors}

\sssec{} In this subsection we establish some properties of functors between commutative factorization categories.

\sssec{} 
\label{Sect_2.7.2}
Let $C(X), D(X)\in CAlg^{nu}(Shv(X)-mod)$ and $L: C(X)\leftrightarrows D(X): R$ is an adjoint pair in $Shv(X)-mod$, where $L$ is a map in $CAlg^{nu}(Shv(X)-mod)$, so $R$ is right-lax non-unital symmetric monoidal. Assume that both $m_C: C^{\otimes 2}(X)\to C(X)$ and $m_D: D^{\otimes 2}(X)\to D(X)$ admit continuous $Shv(X)$-linear right adjoints. 

 Let $I\in fSets$. Write $L_I: (C_{X^I}, \otimes^!)\to (D_{X^I},\otimes^!)$ for the non-unital symmetric monoidal functor obtained from $L$ by functoriality, so $L_I$ is a map in $CAlg^{nu}(Shv(X^I)-mod)$. Let $R_I: D_{X^I}\to C_{X^I}$ denote the right adjoint to $L_I$. 
 
\begin{Lm} i) $R_I$ is a map in $Shv(X^I)-mod$. \\
ii) For $I\to J$ in $fSets$ the !-restriction of $R_I$ under $X^J\to X^I$ is the right adjoint to $L_J: C_{X^J}\to D_{X^J}$. \\
iii) Both $L_I$ and $R_I$ factorize. Namely, given a map $\phi: I\to I'$ in $fSets$, one has canonically
$$
L_I\mid_{X^I_{\phi, d}}\,\iso\,\underset{i\in I'}{\boxtimes} L_{I_i}\mid_{X^I_{\phi, d}}\;\;\;\;\;\;\mbox{and}\;\;\;\;\;\;\;\;  R_I\mid_{X^I_{\phi, d}}\,\iso\,\underset{i\in I'}{\boxtimes} R_{I_i}\mid_{X^I_{\phi, d}}.
$$   
\end{Lm}
\begin{proof} i) We have the functor $\cF_{I, C}: \Tw(I)\to Shv(X^I)-mod$ with $C_{X^I}\,\iso\, \underset{\Tw(I)}{\colim} \; \cF_{I, C}$. We may pass to right adjoints in the latter functor and get $\cF^R_{I, C}: \Tw(I)^{op}\to Shv(X^I)-mod$ as in Lemma~\ref{Lm_3.2.6_for_sheaves_of_cat}. So, $C_{X^I}\,\iso\, \underset{\Tw(I)^{op}}{\lim} \cF^R_{I, C}$. 

 Consider the functor $\cF: \Tw(I)\times [1]\to Shv(X^I)-mod$ sending $(I\to J\toup{\phi} K)$ to the map $\alpha_{\phi}: C^{\otimes\phi}\to D^{\otimes\phi}$. Clearly, $\alpha_{\phi}$ has a right adjoint $\alpha_{\phi}^R$ in $Shv(X^I)-mod$. So, we may pass to right adjonts in $\cF$ and get $\cF^R: (\Tw(I)\times [1])^{op}\to Shv(X^I)-mod$ sending $(I\to J\toup{\phi} K)$ to
\begin{equation} 
\label{functor_alpha_phi^R_for_Lm2.7.3} 
 \alpha_{\phi}^R: D^{\otimes\phi}\to C^{\otimes\phi}.
\end{equation} 
Now (\cite{Ly}, 9.2.39) shows that $R_I$ is $\underset{\Tw(I)^{op}}{\lim}$ of the functors $\alpha_{\phi}^R: D^{\otimes\phi}\to C^{\otimes\phi}$. 
 
\smallskip\noindent
ii and iii) follow from i).  
\end{proof} 

\index{Factorizable functors}

\begin{Lm} 
\label{Lm_A.1.3_for_appendix} 
In the situation of Section~\ref{Sect_2.7.2} assume in addition that the right-lax $C(X)$-module structure on the
map $R: D(X)\to C(X)$ is strict. Then for any $I\in fSets$ the dual pair $L_I: C_{X^I}\leftrightarrows D_{X^I}: R_I$ takes place in $(C_{X^I}, \otimes^!)-mod$.
\end{Lm}
\begin{proof}
Take $c\in C_{X^I}, d\in D_{X^I}$. We must show that the natural map $c\otimes^! R_I(d)\to R_I(d\otimes^! L_I(c))$ is an isomorphism in $C_{X^I}$. Since the functors $R_I, L_I$ are compatible with $!$-restrictiions under $X^J\to X^I$ for a map $I\to J$ in $fSets$ and with factorization, it suffices to establish the above isomorphism over $\oo{X}{}^I$.

 The dual pair $L^{\boxtimes I}: \underset{i\in I}{\boxtimes} C(X)\leftrightarrows \underset{i\in I}{\boxtimes} D(X): R^{\boxtimes I}$ takes place in $\underset{i\in I}{\boxtimes} C(X)-mod$. Our claim follows. 
\end{proof}

\begin{Lm} In the situation of Section~\ref{Sect_2.7.2} assume $C(X), D(X)\in CAlg(Shv(X)-mod)$, $L$ is symmetric monoidal, and the dual pair $L: C(X)\leftrightarrows D(X): R$ takes place in $C(X)-mod$.  Assume also that $R$ is monadic. 
Set 
$$
A=R(e_D)\in CAlg(C(X)),
$$ 
where $e_D\in D(X)$ is the unit. Let $I\in fSets$. Assume that for any $(I\to J\to K)\in \Tw(I)$ the functor (\ref{functor_alpha_phi^R_for_Lm2.7.3}) is conservative. Then $R_I: D_{X^I}\to C_{X^I}$ is monadic, and the corresponding monad identifies canonically with $R_IL_I(e_{X^I})\,\iso\, A_{X^I}\in CAlg(C_{X^I}, \otimes^!)$. Here $e_{X^I}\in (C_{X^I},\otimes^!)$ is the unit. So, 
$$
D_{X^I}\,\iso\, A_{X^I}-mod(C_{X^I}).
$$ 
\end{Lm}
\begin{proof}
Let us show that $R_I$ is monadic. Since $R_I$ is continuous, for this it suffices to show that $R_I$ is conservative. This follows from (\cite{Ly}, 2.5.3) by passing to the limit over $(I\to J\to K)\in \Tw(I)$, because each (\ref{functor_alpha_phi^R_for_Lm2.7.3}) is assumed conservative.
 
 The dual pair $(L_I, R_I)$ takes place in $(C_{X^I}, \otimes^!)-mod$ by Lemma~\ref{Lm_A.1.3_for_appendix}. It remains to calculate the monad $R_IL_I(e_{X^I})\in CAlg(C_{X^I}, \otimes^!)$. 
 
 We get an adjoint pair $L_I: CAlg(C_{X^I}, \otimes^!)\leftrightarrows CAlg(D_{X^I}, \otimes^!): R_I$. So, the adjointness map $L(A)=LR(e_D)\to e_D\,\iso\, L(e_C)$ is a map in $CAlg(D(X))$, where $e_C\in C(X)$ is the unit. By functoriality, it gives the morphism $(L(A))_{X^I}\to L_I(e_{X^I})$ in $D_{X^I}$. By Remark~\ref{Rem_image_of_A_X^I}, $(L(A))_{X^I}\,\iso\, L_I(A_{X^I})$ canonically. By adjointness, this gives a morphism $A_{X^I}\to R_IL_I(e_{X^I})$ in $D_{X^I}$. We claim this is an isomorphism. Indeed, it is compatible with $!$-restrictions with respect to $X^J\to X^I$ for a map $I\to J$ in $fSets$, as well as with the factorization. It is clearly an isomorphism over $\oo{X}{}^I$. Our claim follows. 
\end{proof}

\begin{Cor} 
\label{Cor_2.7.6}
Let $C(X)\in CAlg^{nu}(Shv(X)-mod)$. Assume that $m: C^{\otimes 2}(X)\to C(X)$ admits a continuous $C^{\otimes 2}(X)$-linear right adjoint. Then \\
i) for $I\in fSets$ the multiplication $\otimes^!: (C_{X^I})\otimes_{Shv(X^I)} (C_{X^I})\to C_{X^I}$ admits a continuous $(C_{X^I})\otimes_{Shv(X^I)} (C_{X^I})$-linear right adjoint.

\noindent
ii) the multiplication $\otimes^!: \Fact(C)\otimes_{Shv(\Ran)}\Fact(C)\to\Fact(C)$ admits a continuous $\Fact(C)\otimes_{Shv(\Ran)}\Fact(C)$-linear right adjoint. 
\end{Cor}
\begin{proof} i) This follows from Lemma~\ref{Lm_A.1.3_for_appendix}.\\
ii) Combine i) and (\cite{Ly}, 9.2.39). 
\end{proof}

\begin{Cor} 
\label{Cor_2.7.7}
Let $C(X)\in CAlg(Shv(X)-mod)$. Assume that $m: C^{\otimes 2}(X)\to C(X)$ admits a continuous $C^{\otimes 2}(X)$-linear right adjoint, and $1_C: Shv(X)\to C(X)$ has a continuous $Shv(X)$-linear right adjoint. That is, $C(X)$ is relatively rigid over $(Shv(X), \otimes^!)$ in the sense of Definition~\ref{Def_relatively_rigid_G}. Then\\
i) for $I\in fSets$ the product $\otimes^!: (C_{X^I})\otimes_{Shv(X^I)} (C_{X^I})\to C_{X^I}$ admits a continuous $(C_{X^I})\otimes_{Shv(X^I)} (C_{X^I})$-linear right adjoint, and the unit $Shv(X^I)\to C_{X^I}$ admits a continuous $Shv(X^I)$-linear right adjoint. So, $(C_{X^I}, \otimes^!)$ is retaively rigid over $(Shv(X^I), \otimes^!)$.
\\
i) the multiplication $\otimes^!: \Fact(C)\otimes_{Shv(\Ran)}\Fact(C)\to\Fact(C)$ admits a continuous $\Fact(C)\otimes_{Shv(\Ran)}\Fact(C)$-linear right adjoint, and the unit $Shv(\Ran)\to \Fact(C)$ admits a continuous $Shv(\Ran)$-linear right adjoint. So, $(\Fact(C),\otimes^!)$ is relatively rigid over $(Shv(\Ran), \otimes^!)$. 
\end{Cor}
\begin{proof} As in Corollary~\ref{Cor_2.7.6}. 
\end{proof}

\index{Relatively rigid monoidal category}

\sssec{} In the situation of Corollary~\ref{Cor_2.7.7} it is not true that $(C(X), \otimes^!)$ or $(C_{X^I}, \otimes^!)$ or $(\Fact(C),\otimes^!)$ is rigid.

\ssec{Pro-nilpotence and the basic diagram}
\label{Sect_Pro-nilpotence}

\sssec{} 
\label{Sect_2.8.1}
Let $C(X)\in CAlg^{nu}(Shv(X)-mod)$. The main results of this subsection are Proposition~\ref{Pp_2.8.7} and \ref{Pp_basic_diagram_2.8.10} below.

\sssec{} 
\label{Sect_2.8.2}
Write $fSets_{\le n}\subset fSets$ for the full subcategory of those $I$ with $\mid I\mid\le n$. Set 
$$
\Ran_{\le n}=\underset{I\in (fSets_{\le n})^{op}}{\colim} X^I
$$ 
in $\PreStk$. We have natural maps $\Ran_{\le n}\to \Ran_{\le n+1}\to \Ran$ for $n\ge 1$. Moreover, $\Ran\,\iso\,\underset{n\in\NN}{\colim}\Ran_{\le n}$. Indeed, $\underset{n\in\NN}{\colim} fSets_{\le n}\to fSets$ is fully faithful by (\cite{Ly}, 2.7.16), and essentially surjective by (\cite{Ly}, 13.1.14), so an equivalence. Our claim follows now from \cite{PT}. 

 As in \cite{FG}, let $X^{I,\le n}$ be the union of the images of $\vartriangle: X^J\to X^I$ for all maps $I\to J$ in $fSets$ with $\mid J\mid\le n$. It is closed in $X^I$. Let $X^{I, >n}$ be the complement open subset. Then $fSets^{op}\to \PreStk, I\mapsto X^{I,>n}$ is a functor, and we let $\Ran_{>n}=\underset{I\in fSets^{op}}{\colim} X^{I, >n}$ in $\PreStk$. We have the natural maps $i^n: \Ran_{\le n}\to \Ran$, $j^n: \Ran_{>n}\to \Ran$.  
 
  For a map $I\to J$ in $fSets$ both squares are cartesian
$$
\begin{array}{ccccc}
X^{I,\le n} & \toup{i^{I, n}} &
X^I &\getsup{j^{I, n}} &X^{I, >n}\\
\uparrow && \uparrow && \uparrow\\
X^{J,\le n} & \toup{i^{J, n}} & X^J & \getsup{j^{J, n}} & X^{J, >n},
\end{array}
$$
here we denoted by $i^{I,n }, j^{I, n}$ the corresponding natural maps.

\sssec{}  Define $(j^n)^*: Shv(\Ran)\to Shv(\Ran_{>n})$ as the limit of the functors $(j^{I, n})^*: Shv(X^I)\to Shv(X^{I, >n})$ and apply (\cite{G}, ch. I.1, 2.6.3) to describe its right adjoint $(j^n)_*: Shv(\Ran_{>n})\to Shv(\Ran)$. So, $(j^n)_*$ is obtained by passing to the limit over $I\in fSets$ in the fully faithful functors $(j^{I, n})_*: Shv(X^{I, >n})\to Shv(X^I)$. For $I\in fSets$ the diagram commutes
$$
\begin{array}{ccc}
Shv(\Ran_{>n}) & \toup{(j^n)_*} & Shv(\Ran)\\
\downarrow\lefteqn{\scriptstyle (\vartriangle^I)^!} && \downarrow\lefteqn{\scriptstyle (\vartriangle^I)^!}\\
Shv(X^{I, >n}) & \toup{(j^{I, n})_*} & Shv(X^I)
\end{array}
$$
In particular, $(j^n)_*$ is continuous and fully faithful. It is easy to see that $(j^n)_*$ is $Shv(\Ran)$-linear naturally.

 We have $Shv(\Ran_{\le n})\,\iso\, \underset{I\in fSets^{op}}{\colim} Shv(X^{I,\le n})$, where the transition functors are direct images. The functor $(i^n)_!: Shv(\Ran_{\le n})\to Shv(\Ran)$ is obtained via (\cite{Ly}, 9.2.21). The functor $(i^n)_!$ is obtained by passing to the limit over $I\in fSets$ in the fully faihtful functors $(i^{I, n})_!: Shv(X^{I,\le n})\to Shv(X^I)$, so $(i^n)_!$ is fully faithful by (\cite{Ly}, 2.2.17). 
 
 As in (\cite{FG}, 5.1.3), we get a short exact sequence in $Shv(\Ran)-mod$
$$
Shv(\Ran_{\le n})\leftrightarrows Shv(\Ran)\leftrightarrows Shv(\Ran_{>n})
$$ 
Namely, $(j^n)^*(i^n)_*=0$, and $(j^n)_*Shv(\Ran_{>n})$ identifies with the right orthogonal of $Shv(\Ran_{\le n})$ inside $Shv(\Ran)$. For $K\in Shv(\Ran)$ one has a fibre sequence
\begin{equation}
\label{f_seq_for_Sect_2.8.3}
(i^n)_!(i^n)^!K\to K\to (j^n)_*(j^n)^*K
\end{equation} 
in $Shv(\Ran)$. 

\sssec{} Let $D\in Shv(\Ran)-mod$, which we also view as a sheaf of categories on $\Ran$.
Set $D_{>n}=D\otimes_{Shv(\Ran)} Shv(\Ran_{>n})$.
 Applying $\otimes_{Shv(\Ran)} D$, one gets adjoint pairs $(j^n)^*: D \leftrightarrows D_{>n}: (j^n)_*$ and $(i^n)_!: D\otimes_{Shv(\Ran)} Shv(\Ran_{\le n})\leftrightarrows D: (i^n)^!$ in $Shv(\Ran)-mod$. For $K\in D$ we still have a canonical fibre sequence (\ref{f_seq_for_Sect_2.8.3}) in $D$.  

 For $I\in fSets$ set $D_I=D\otimes_{Shv(\Ran)} Shv(X^I)$, so $D\,\iso\, \underset{I\in fSets^{op}}{\colim} D_I\,\iso\, \underset{I\in fSets}{\lim} D_I$ in $\DGCat_{cont}$, where the transition functors are !-direct image and !-pullbacks respectively. Let $D_{\le n}=D\otimes_{Shv(\Ran)} Shv(\Ran_{\le n})$ then similarly
$$
D_{\le n}=\underset{I\in fSets_{\le n}^{op}}{\colim} D_I.
$$ 
Using \cite{PT}, as in Section~\ref{Sect_2.8.2} we get 
$D\,\iso\,\underset{n\in\NN}{\colim} D_{\le n}\,\iso\, \underset{n\in\NN^{op}}{\lim} D_{\le n}$ taken in $\DGCat_{cont}$. 

\sssec{} Assume we are in the setting of Secton~\ref{Sect_2.8.1}. 

\begin{Lm} 
\label{Lm_2.8.6_now}
If $n\ge 1$ then the chiral symmetric monoidal structure on $\Fact(C)$ is compatible with the localization functor $(i^n)^!: \Fact(C)\to \Fact(C)_{\le n}$ in the sense of (\cite{HA}, 2.2.1.6). 
\end{Lm}
\begin{proof}
Let $(\Ran_{>n}\times\Ran)_d$ be the preimage of $\Ran^2_d$ under $j^n\times\id: \Ran_{>n}\times\Ran\to \Ran^2$. The limit of the diagram $(\Ran_{>n}\times\Ran)_d\toup{j^n\times\id} \Ran^2_d\toup{u\comp j} \Ran\getsup{i^n} \Ran_{\le n}$ is empty.
\end{proof}

 By Lemma~\ref{Lm_2.8.6_now} combined with (\cite{HA}, 2.2.1.9), $\Fact(C)_{\le n}$ is equipped with a symmetric monoidal structure still denoted $\otimes^{ch}$ such that the functor $(i^n)^!: (\Fact(C),\otimes^{ch})\to (\Fact(C)_{\le n}, \otimes^{ch})$ is non-unital symmetric monoidal. If $n<m\in\NN$ then the similar !-restriction functor $(\Fact(C)_{\le m}, \otimes^{ch})\to (\Fact(C)_{\le n}, \otimes^{ch})$ is non-unital symmetric monoidal. The limit $\underset{n\in\NN^{op}}{\lim} \Fact(C)_{\le n}$ can equivalently be understood in $CAlg^{nu}(Shv(X)-mod)$ or in $Shv(X)-mod$, so the equivalence 
$$
(\Fact(C), \otimes^{ch})\,\iso\, \underset{n\in\NN^{op}}{\lim} (\Fact(C)_{\le n}, \otimes^{ch})
$$
holds in $CAlg^{nu}(Shv(X)-mod)$. We have proved the following.

\begin{Pp} 
\label{Pp_2.8.7}
The $\DG$-category $(\Fact(C), \otimes^{ch})$ is pro-nilpotent in the sense of (\cite{FG}, 4.1.1). \QED 
\end{Pp}

\index{Pro-nilpotent non-unital symmetric monoidal $\DG$-category} 

So, (\cite{FG}, Proposition~4.3.3) applies to $(\Fact(C),\otimes^{ch})$ and yields the following (in the notations of \cite{FG}).
\begin{Cor} 
\label{Cor_Koszul_duality}
The Koszul duality functors
$$
C^{ch}: Lie-alg(\Fact(C), \otimes^{ch})\leftrightarrows Com-coalg(\Fact(C), \otimes^{ch}): Prim^{ch}[-1]
$$
are mutually inverse equivalences in $\DGCat_{cont}$. These equivalences identify the full subcategory $\{M\in Lie-alg(\Fact(C), \otimes^{ch})\mid \oblv(M)\in \Fact(C)$ is the extension by zero under $X\to\Ran\}$ with the full subcategory 
$$
Com-coalg^{fact}(\Fact(C), \otimes^{ch})
$$ 
of factorization cocommutative coalgebras in $(\Fact(C), \otimes^{ch})$ in the sense of \cite{FG}. 
\QED
\end{Cor}

\index{Koszul duality functors}
\index{Factorization cocommutative coalgebras in $(\Fact(C), \otimes^{ch})$}
\index{Reduced operad}
\index{Koszul dual cooperad}

\sssec{} For the rest of Section~\ref{Sect_Pro-nilpotence} we will use some notions and notations from (\cite{G}, chapter IV.2). Namely, let $\Sigma$ be the groupoid of finite non-empty sets, let $\cP\in\Vect^{\Sigma}$ be a reduced operad (in the sense of \cite{G}, IV.2, 1.1.2), and $Q=\cP^{\vee}$ be the Koszul dual cooperad (in the sense of \cite{G}, IV.2, 2.3.1).  

 Recall that 
$\id: (\Fact(C), \otimes^{ch})\to (\Fact(C),\star)$ is right-lax non-unital symmetric monoidal. So, it yields the functors 
\begin{equation}
\label{functor_cP-alg(Fact(C))}
\cP-alg^{ch}(\Fact(C))\to \cP-alg^{\star}(\Fact(C))
\end{equation}
and
\begin{equation}
\label{functor_cQ-coalg(Fact(C))}
Q-coalg^{\star, in}(\Fact(C))\to Q-coalg^{ch}(\Fact(C))
\end{equation} 
commuting with the corresponding oblivion functors to $\Fact(C)$. Here we have denoted by $Q-coalg^{\star, in}(\Fact(C))$ the category of $Q$-coalgebras in $(\Fact(C), \star)$ with respect to the action of $\Vect^{\Sigma}$ on $\Fact(C)$ such that $V\in \Vect^{\Sigma}$ sends $c\in\Fact(C)$ to 
$$
\underset{I\in\Sigma}{\colim} (V(I)\otimes x^{\otimes^{\star} I})
$$ 
taken in $\Fact(C)$. The superscript $in$ stands for \select{ind-nilpotent} according to the conventions of (\cite{G}, IV.2, 2.2.1). This superscript is omitted for $(\Fact(C),\otimes^{ch})$, as the latter category is pro-nilpotent. 

 The categories $\cP-alg^{ch}(\Fact(C))$ and $\cP-alg^{\star}(\Fact(C))$ are presentable by (\cite{Ly}, 9.4.12). The functor (\ref{functor_cP-alg(Fact(C))}) preserves limits, so admits a left adjoint denoted $\Ind_{\cP}^{\star\to ch}$. 
 
\begin{Pp} 
\label{Pp_basic_diagram_2.8.10}
The diagram of functors canonically commutes
$$
\begin{array}{ccc}
\cP-alg^{ch}(\Fact(C)) & \toup{coPrym_{\cP}^{enh}} & Q-coalg^{ch}(\Fact(C)\\
\uparrow\lefteqn{\scriptstyle \Ind_{\cP}^{\star\to ch}} && \uparrow\lefteqn{\scriptstyle (\ref{functor_cQ-coalg(Fact(C))})}\\
\cP-alg^{\star}(\Fact(C)) & \toup{coPrym_{\cP}^{enh, in}} & Q-coalg^{\star, in}(\Fact(C)),
\end{array}
$$
where $coPrym_{\cP}^{enh}$ and $coPrym_{\cP}^{enh, in}$ denote the Koszul duality functors from (\cite{G}, IV.2, 2.4.3).
\end{Pp}
\begin{proof}
The proof of (\cite{FG}, 6.1.2) applies in our situation without changes, that is, we apply (\cite{FG}, Lemma~6.2.6) to the adjoint pair $\id: \Fact(C)\leftrightarrows \Fact(C): \id$.
\end{proof}

\ssec{Filtration on a factorization algebra $\Fact(A)$}
\label{Sect_Filtration_on_Fact(A)}

\sssec{} Let $C\in CAlg^{nu}(\DGCat_{cont})$ and $A\in CAlg^{nu}(C)$. The purpose of this section is, given a filtration on $A$, to get an induced filtration on $\Fact(A)\in\Fact(C)$, where $C(X)=C\otimes Shv(X)$, and calculate its associate graded. 

\index{Filtration on a factorization algebra $\Fact(A)$}

\sssec{} Let $C^{Fil, \ge 0}=\Fun(\ZZ_+, C)$, here $\ZZ_+=\{0,1,\ldots\}$ is viewed as an ordered set, hence a category. Let $C^{Fil}=\Fun(\ZZ, C)$, where $\ZZ$ is viewed as an ordered set, hence a category.
We have the functor $C^{Fil,\ge 0}\subset C^{Fil}$ given by extending by zero from $\ZZ_+$ to $\ZZ$. Write $C^{gr}=\Fun(\ZZ^{\Spc}, C)$, here $\ZZ^{\Spc}$ is the space underlying $\ZZ$. The functor $ass-gr: C^{Fil}\to C^{gr}$ sends an object $(V_n)_{n\in\ZZ}$ to the object $n\mapsto coFib(V_{n-1}\to V_n)$. The map $ass-gr$ is a morphism in $\DGCat_{cont}$. 

\sssec{} As in (\cite{G}, vol.2, ch. 5, Section 1.3), $C^{Fil}, C^{gr}\in CAlg^{nu}(\DGCat_{cont})$ via the Day convolution. Namely, for $V,W\in C^{Fil}$ we get
$$
(V\times W)_n=\underset{n_1+n_2\le n}{\colim} V_{n_1}\otimes W_{n_2}
$$
For $V,W\in C^{gr}$ we get
$$
(V\otimes W)_n=\underset{n_1,n_2\in\ZZ, n_1+n_2=n}{\oplus} V_{n_1}\otimes W_{n_2}
$$
By (\cite{G}, vol.2, ch. 5, 1.3.5), the functor $ass-gr$ is non-unital symmetric monoidal. 

 In particular, for $A\in C^{Fil}, I\in fSets$ we have 
$$
ass-gr(A^{\otimes I})\,\iso\, (ass-gr(A))^{\otimes I}
$$ 
in $C^{gr}$ canonically. Besides, $ass-gr$ induces a functor $ass-gr: CAlg^{nu}(C^{Fil})\to CAlg^{nu}(C^{gr})$.

\sssec{} The functor $\oblv_{Fil}: C^{Fil}\to C$ sends $f$ to $\underset{n\in \ZZ}{\colim} f(n)$. This functor is non-unital symmetric monoidal. In particular, it induces a functor $CAlg^{nu}(C^{Fil})\to CAlg^{nu}(C)$. 

\sssec{} Let $A\in CAlg^{nu}(C^{Fil})$, set $B=\oblv_{Fil}(A)\in CAlg^{nu}(C)$. We will get a filtration on $\Fact(B)\in\Fact(C)$ and calculate its associated graded.

 Note that $ass-gr(A)\in CAlg^{nu}(C^{gr})$. By the above, we have the $Shv(\Ran)$-module categories $\Fact(C^{Fil})$ and $\Fact(C^{gr})$. Since $ass-gr: C^{Fil}\to C^{gr}$ is non-unital symmetric monoidal, it induces by functoriality the functor
$$
ass-gr: \Fact(C^{Fil})\to \Fact(C^{gr})
$$ 
As in Section~\ref{Sect_Factorization algebras in Fact(C)},
$$
\Fact(A)\,\iso\, \underset{(J\to K)\in \cTw(fSets)}{\colim} A^{\otimes J}\otimes \omega_{X^K}
$$ 
taken in $\Fact(C^{Fil})$. Since $ass-gr: C^{Fil}\to C^{gr}$ is non-unital symmetric monoidal, we get
$$
ass-gr(\Fact(A))\,\iso\, \Fact(ass-gr(A))
$$
in $\Fact(C^{gr})$ by functoriality.

\sssec{} Assume now there is $n>0$ such that $A_n\iso A_{n+1}\iso A_{n+2}\iso\ldots$ and so on are isomorphisms, so $B\,\iso\, A_n$.  

 As above $B^{\otimes J}\in C^{Fil}$ naturally, so $B^{\otimes J}\otimes\omega_{X^J}\in (C^{\otimes J}\otimes Shv(X^K))^{Fil}$ naturally. For a map (\ref{map_in_cTw(fSets)}) in $\cTw(fSets)$ the multiplication morphism $B^{\otimes J_1}\to B^{\otimes J_2}$ is a map in $C^{Fil}$, so the functor
$$
\cF_{\Ran, B}: \cTw(fSets)\to \Fact(C), (J\to K)\mapsto  B^{\otimes J}\otimes\omega_{X^J}
$$
is lifted to a functor
$$
\cF_{\Ran, B}^{Fil}: \cTw(fSets)\to \Fact(C)^{Fil}
$$
by the above.\footnote{The idea of this construction is due to Joakim Faergeman, we are grateful to him for his help.}

 The composition $\cTw(fSets)\toup{\cF_{\Ran, B}^{Fil}}\Fact(C)^{Fil}\toup{ass-gr} \Fact(C)^{gr}$ sends $(J\to K)$ to 
$$
(ass-gr(B))^{\otimes I}\otimes\omega_{X^K} 
$$ 
So, it coinsides with 
$$
\cF_{\Ran, ass-gr(B)}: \cTw(fSets)\to \Fact(C)^{gr}.
$$

Passing to the colimits over $\cTw(fSets)$ we get
\begin{multline*}
ass-gr(\Fact(B))\,\iso\,
\underset{(J\to K)\in \cTw(fSets)}{\colim} ass-gr\comp \cF_{\Ran, A}^{Fil}\,\iso\\\underset{(J\to K)\in \cTw(fSets)}{\colim} \cF_{\Ran, ass-gr(B)}=\Fact(ass-gr(B))
\end{multline*}
in $\Fact(C)^{gr}$.  In particular, 
$$
ass-gr(\Fact(B))\,\iso\, \Fact(ass-gr(B))
$$
in $\Fact(C)$. 

\ssec{Special properties of $\Fact(C)$ in the constructible context}
\label{Sect_special_constr_context}

\sssec{}  For Section~\ref{Sect_special_constr_context} work in the constructible context. Let $C\in CAlg^{nu}(\DGCat_{cont})$ and $C(X)=C\otimes Shv(X)$. Then $Shv(\Ran)$ is equipped with the non-unital $\otimes$-symmetric monoidal structure defined in Section~\ref{Sect_F.2.6}. For $I\in fSets$ and the map $\vartriangle^I: X^I\to\Ran$ the functor $(\vartriangle^I)^*: Shv(\Ran)\to Shv(X^I)$ is defined in Section~\ref{Sect_F.2.2}, and $(\vartriangle^I)^*: (Shv(\Ran),\otimes)\to (Shv(X^I),\otimes)$ is non-unital symmetric monoidal by Lemma~\ref{Lm_F.2.7}

\index{$\otimes$-symmetric monoidal structure on $Shv(\Ran)$}

Our $\Fact(C)$ is naturally an object of $(Shv(\Ran), \otimes)-mod$. Indeed, 
$$
\Fact(C)\,\iso\, \underset{(J\to K)\in\cTw(fSets)}{\colim} C^{\otimes J}\otimes Shv(X^K).
$$ 
For each $(J\to K)\in \cTw(fSets)$ view $C^{\otimes J}\otimes Shv(X^K)\in (Shv(\Ran), \otimes)-mod$ via the monoidal functor 
$$
(\vartriangle^K)^*: (Shv(\Ran),\otimes)\to (Shv(X^K),\otimes).
$$ 

 For a map (\ref{map_in_cTw(fSets)}) in $\cTw(fSets)$ the transition functor 
$$
C^{\otimes J_1}\otimes Shv(X^{K_1})\to C^{\otimes J_2}\otimes Shv(X^{K_2})
$$ 
is a map in $(Shv(\Ran),\otimes)-mod$. Indeed, $(\vartriangle)_!: Shv(X^{K_1})\to Shv(X^{K_2})$ is a map in $(Shv(X^{K_2}),\otimes)-mod$. So, the colimit in the definition of $\Fact(C)$ can be calculated in $(Shv(\Ran), \otimes)-mod$.

\begin{Lm} i) For $I\in fSets$ one has canonically
$$
\Fact(C)\otimes_{(Shv(\Ran),\otimes)} Shv(X^I)\,\iso\, \Fact(C)\otimes_{Shv(\Ran)} Shv(X^I)=C_{X^I}.
$$
ii) The functor $\vartriangle_*=\vartriangle_!: C_{X^I}\to \Fact(C)$ has a left adjoint $\vartriangle^*: \Fact(C)\to C_{X^I}$.
\end{Lm}
\begin{proof}
i) Using Lemma~\ref{Lm_functor_cF'_is_equivalence}, we get
\begin{multline*} 
\Fact(C)\otimes_{(Shv(\Ran),\otimes)} Shv(X^I)\,\iso\, \underset{(J\to K)\in\cTw(fSets)}{\colim} C^{\otimes J}\otimes Shv(X^K)\otimes_{(Shv(\Ran),\otimes)} Shv(X^I)\\
\iso\, \underset{(J\to K)\in\cTw(fSets)}{\colim} C^{\otimes J}\otimes Shv(X^K\times_{\Ran} X^I)\,\iso\\ \underset{(J\to K)\in\cTw(fSets)}{\colim} C^{\otimes J}\otimes Shv(X^K)\otimes_{Shv(\Ran)} Shv(X^I),
\end{multline*}
where we used Theorem~\ref{Thm_Kunneth_formula} to get the last isomorphism. The latter expression identifies with $\Fact(C)\otimes_{Shv(\Ran)} Shv(X^I)$.

\medskip\noindent
ii) The adjoint pair $(\vartriangle^I)^*: Shv(\Ran)\leftrightarrows Shv(X^I): (\vartriangle^I)_*$ takes place in $(Shv(\Ran),\otimes)-mod$. Tensoring by $\Fact(C)$, we get an adjoint pair
$$
(\vartriangle^I)^*: \Fact(C)\leftrightarrows \Fact(C)\otimes_{(Shv(\Ran),\otimes)} Shv(X^I): (\vartriangle^I)_*,
$$
where the right adjoint identifies with $(\vartriangle^I)_*: C_{X^I}\to \Fact(C)$ by i).
\end{proof}

\section{Factorization categories attached to cocommutative coalgebras}
\label{Sect_Fact_cat_for_cocom_coalg}

\ssec{} In this section we associate to $C(X)\in CoCAlg^{nu}(Shv(X)-mod)$ (under strong additional assumptions) a factorization sheaf of categories $\Fact^{co}(C)$ on $\Ran$. This is a dual version of the construction from Section~\ref{Sect_Fact_cat_for_com_algebras}. We expect that it produces a factorization sheaf of categories under some weak assumptions, we formulate the corresponding questions.

\sssec{} For a map $(J\toup{\phi} K)$ in $fSets$ recall our notation $C^{\otimes\phi}=\underset{k\in K}{\boxtimes} C^{\otimes J_k}(X)$. Denote by $\com: C^{\otimes K}(X)\to C^{\otimes J}(X)$ the comultiplication map for $C(X)$. 
Denote by
$$
\cG_{\Ran, C}: \cTw(fSets)^{op}\to Shv(\Ran)-mod
$$
the functor sending $(J\toup{\phi} K)$ to $C^{\otimes\phi}$, and sending the map (\ref{map_in_cTw(fSets)}) to the composition
$$
\underset{k\in K_2}{\boxtimes} C^{\otimes (J_2)_k}(X)\toup{\com} \underset{k\in K_2}{\boxtimes} C^{\otimes (J_1)_k}(X)\toup{\vartriangle^!} \underset{k\in K_1}{\boxtimes} C^{\otimes (J_1)_k}(X),
$$
where the first map is the product of comultiplications along $(J_1)_k\to (J_2)_k$ for $k\in K_2$, and the second map is $\vartriangle^!$ for $\vartriangle: X^{K_1}\to X^{K_2}$.  

 Set 
$$
\Fact^{co}(C)=\mathop{\lim}\limits_{\cTw(fSets)^{op}} \cG_{\Ran, C}
$$

\sssec{} Let $I\in fSets$. Consider the composition 
$$
Tw(I)^{op}\to \cTw(fSets)^{op}\toup{\cG_{\Ran,C}} Shv(\Ran)-mod,
$$
where the first functor sends $(I\to J\to K)$ to $(J\to K)$. We may view this composition as taking values in $Shv(X^I)-mod$ naturally. Denote by
$$
\cG_{I,C}: Tw(I)^{op}\to Shv(X^I)-mod
$$
the functor so obtained. Set
$$
C_{X^I, co}=\underset{Tw(I)^{op}}{\lim} \, \cG_{I,C}
$$
in $Shv(X^I)-mod$.

\sssec{Example} For $I=\{1,2\}$ the category $C_{X^I, co}$ is defined by the cartesian square in $Shv(X^I)-mod$
$$
\begin{array}{ccc}
C^{\otimes 2}(X) & \getsup{\vartriangle^!} & \underset{i\in I}{\boxtimes} C(X)\\
\uparrow\lefteqn{\scriptstyle \com}&& \uparrow\\
C(X) & \gets & C_{X^I, co}
\end{array}
$$

\begin{Lm} For $I\in fSets$ one has canonically in $Shv(X^I)-mod$
$$
\Fact^{co}(C)\otimes_{Shv(\Ran)} Shv(X^I)\,\iso\, C_{X^I, co}.
$$
\end{Lm}
\begin{proof}
By Lemma~\ref{Lm_2.1.11},  $Shv(X^I)$ is dualizable in $Shv(\Ran)-mod$. So,
$$
\Fact^{co}(C)\otimes_{Shv(\Ran)} Shv(X^I)\,\iso\, \underset{(J\toup{\phi} K)\in\cTw(fSets)^{op}}{\lim} C^{\otimes\phi}\otimes_{Shv(\Ran)} Shv(X^I)
$$
in $Shv(X^I)-mod$. Now argue as in Lemma~\ref{Lm_2.1.2_about_Fact(C)}. 
 Recall the category $\cE$ defined in the proof of Lemma~\ref{Lm_2.1.2_about_Fact(C)}. The latter limit identifies with
\begin{equation}
\label{limit_inside_Lm_3.1.3}
\underset{(I\to \bar K\gets K\getsup{\phi} K)\in \cE^{op}}{\lim} \vartriangle^{(I/\bar K)}_!\vartriangle^{(K/\bar K)!} C^{\otimes\phi}
\end{equation}
Write $\bar\phi$ for the composition $J\toup{\phi} K\to \bar K$ for an object of $\cE$. Recall that $\vartriangle^{(K/\bar K)!} C^{\otimes\phi}\,\iso\, C^{\bar\phi}$. 

Recall the adjoint pair 
$$
r_0: \cE_0\leftrightarrows \cE: l_0
$$ 
from Lemma~\ref{Lm_2.1.2_about_Fact(C)}. The diagram defining (\ref{limit_inside_Lm_3.1.3}) is the composition of $l_0^{op}: \cE^{op}\to \cE_0^{op}$ with another functor, and $l_0$ is cofinal. Thus, (\ref{limit_inside_Lm_3.1.3}) identifies with
$$
\underset{(I\to K\getsup{\phi} J)\in \cE_0^{op}}{\lim} \vartriangle^{(I/ K)}_! C^{\otimes\phi}
$$
in $Shv(X^I)-mod$. 

 Recall the functor $q: Tw(I)\to \cE_0$ from Lemma~\ref{Lm_2.1.2_about_Fact(C)}. Since $q$ is cofinal, the latter limit identifies with
$$
\underset{(I\to J\to K)\in Tw(I)^{op}}{\lim} \vartriangle^{(I/ K)}_! C^{\otimes\phi},
$$
and the corresponding diagram identifies with the functor $\cG_{I, C}$ as required.
\end{proof}

\sssec{} Think of $\Fact^{co}(C)$ as a sheaf of categories on $\Ran$. Denote by 
$$
\cG_{fSets, C}: fSets\to Shv(\Ran)-mod
$$ 
the functor sending $I$ to $\Gamma(X^I, \Fact^{co}(C))\,\iso\, C_{X^I, co}$, and $(J\to K)$ to the restriction 
$$
\vartriangle^{(J/K)!}: \Gamma(X^J, \Fact^{co}(C))\to \Gamma(X^K, \Fact^{co}(C))
$$
for $\vartriangle^{(J/K)}: X^K\to X^J$. Let $\cG^L_{fSets, C}: fSets^{op}\to Shv(\Ran)-mod$ be obtained from $\cG_{fSets, C}$ by passing to the left adjoints in $Shv(\Ran)-mod$.
So, 
$$
\Fact^{co}(C)\,\iso\; \underset{fSets}{\lim} \cG_{fSets, C}\;\iso\; \underset{fSets^{op}}{
\colim} \cG^L_{fSets, C}
$$

\sssec{} Now we discuss the factorization structure on $\Fact^{co}(C)$ under additional assumptions. Let $\phi: I\to I'$ be a map in $fSets$. The factorization structure on $\Fact^{co}(C)$ amounts to the equivalences
\begin{equation}
\label{fact_iso_for_C_co}
C_{X^I, co}\mid_{X^I_{\phi, d}}\,\iso\,(\underset{i'\in I'}{\boxtimes} C_{X^{I_{i'}}, co})\mid_{X^I_{\phi, d}}
\end{equation}
together with the compatibilites as in Section~\ref{Sect_Fact_cat_for_com_algebras}.    

 By Section~\ref{Sect_B.1.1}, $Shv(X^I_{\phi, d})$ is canonically self-dual in $Shv(X^I)-mod$. 
So, the LHS of (\ref{fact_iso_for_C_co}) identifies with
$$
\underset{(I\to J\toup{\varphi} K)\in Tw(I)^{op}}{\lim} C^{\otimes\varphi}\otimes_{Shv(X^I)} Shv(X^I_{\phi, d})
$$ 
Recall that the full embedding $Tw(I)_{\phi}\subset Tw(I)$ from Section~\ref{Sect_2.1.11_now} is zero-cofinal. Using Lemma~\ref{Lm_3.2.4_from_chiral_algebras} and Section~\ref{Sect_C.0.2}, one identifies the latter limit with the limit of the same diagram restricted to $Tw(I)_{\phi}^{op}$
$$
\underset{(I\to J\toup{\varphi} K)\in Tw(I)^{op}_{\phi}}{\lim} C^{\otimes\varphi}\otimes_{Shv(X^I)} Shv(X^I_{\phi, d})
$$ 

 For $(I\to J\toup{\varphi} K)\in Tw(I)_{\phi}$, $i'\in I'$ let $(I_{i'}\to J_{i'}\toup{\varphi_{i'}} K_{i'})\in Tw(I_{i'})$ be the corresponding fibre. Then
$$
C^{\otimes\varphi}\,\iso\, \underset{i\in I'}{\boxtimes} C^{\otimes \varphi_{i'}},
$$
in $Shv(X^I)-mod$. So, we get a canonical map in $Shv(X^I_{\phi, d})-mod$
\begin{equation}
\label{map_for_Sect_3.1.5_minterious_iso?}
(\underset{i\in I'}{\boxtimes} C_{X^{I_{i'}}, co})\mid_{X^I_{\phi, d}}\to \underset{(I\to J\toup{\varphi} K)\in Tw(I)^{op}_{\phi}}{\lim} C^{\otimes\varphi}\otimes_{Shv(X^I)} Shv(X^I_{\phi, d})\,\iso\,C_{X^I, co}\mid_{X^I_{\phi, d}}
\end{equation}

\sssec{} 
\label{Sect_3.1.7_stratifications}
For $I\in fSets$ write $\oo{X}{}^I\subset X^I$ for the complement to all the diagonals. The scheme $X^I$ is stratified by locally closed subschemes indexed by $(I\to J)\in Q(I)$. The corresponding subscheme is $\vartriangle^{(I/J)}(\oo{X}{}^J)$ for $\vartriangle: X^J\to X^I$. 

 For our $\phi: I\to I'$ as above the subscheme $X^I_{\phi, d}\subset X^I$ is a union of some of these strata. In this sense $X^I_{\phi, d}$ is also stratified.

 It is easy to see that for each stratum $Y=\vartriangle^{(I/J)}(\oo{X}{}^J)$ of $X^I_{\phi, d}$ the functor (\ref{map_for_Sect_3.1.5_minterious_iso?}) becomes an equivalence after applying 
$$
\cdot\otimes_{Shv(X^I_{\phi, d})} Shv(Y)
$$
 
 We do not know if (\ref{map_for_Sect_3.1.5_minterious_iso?}) is an equivalence in general. The nature of this question is to understand if the corresponding stratum by stratum equivalences  glue to a $Shv(X^I_{\phi,d})$-module in the same way on both sides. 
In this sense $\Fact^{co}(C)$ \select{looks like a factorization category}. 

\begin{Rem} 
\label{Rem_3.1.8}
Suppose that for each $(J\toup{\phi} K)\in \cTw(fSets)$, $C^{\otimes \phi}$ is dualizable in $Shv(X^J)-mod$, and for each $I\in fSets$ the category $C_{X^I, co}$ is dualizable in $Shv(X^I)-mod$. Then the functor (\ref{map_for_Sect_3.1.5_minterious_iso?}) is an equivalence. In this case $\Fact^{co}(C)$ becomes a factorization sheaf of categories on $\Ran$. Besides, by Corollary~\ref{Cor_2.2.9_dualizability}, $\Fact^{co}(C)$ is dualizable in $Shv(\Ran)-mod$. 
\end{Rem}
\begin{Rem} Assume that the LHS of (\ref{map_for_Sect_3.1.5_minterious_iso?}) is ULA over $Shv(X^I_{\phi, d})$, and the functor (\ref{map_for_Sect_3.1.5_minterious_iso?}) preserves the ULA objects. Then (\ref{map_for_Sect_3.1.5_minterious_iso?}) is an equivalence by Proposition~\ref{Pp_3.7.8_devissage_using_ULA_preservation}.
\end{Rem}
 
\sssec{Question} 
\label{Question_3.1.6}
i) Is (\ref{map_for_Sect_3.1.5_minterious_iso?}) an equivalence 
for any $C(X)\in CoCAlg^{nu}(Shv(X)-mod)$?

\smallskip\noindent
ii) Assuming $C(X)$ dualizable in $Shv(X)-mod$, is (\ref{map_for_Sect_3.1.5_minterious_iso?}) an equivalence? For this it suffices to prove that for any $I\in fSets$, $C_{X^I, co}$ is dualizable in $Shv(X^I)-mod$. 

\begin{Rem} 
\label{Rem_3.1.11}
If $\com: C(X)\to C^{\otimes 2}(X)$ admits a left adjoint $\com^L: C^{\otimes 2}(X)\to C(X)$, which is $Shv(X)$-linear, then $(C(X), \com^L)\in CAlg^{nu}(Shv(X)-mod)$. In this case we may pass to left adjoints in the functor $\cG_{\Ran, C}$ and get a functor 
$$
\cG_{\Ran, C}^L: \cTw(fSets)\to Shv(\Ran)-mod,
$$
which identifies with $\cF_{\Ran, C}$ for the algebra $(C(X), \com^L)$. In this case
$$
\Fact^{co}(C)\,\iso\, \underset{\cTw(fSets)}{\colim} \cF_{\Ran, C}\,\iso\, \Fact(C)
$$
is a sheaf of factorization categories on $\Ran$ described in Section~\ref{Sect_Fact_cat_for_com_algebras}. 
\end{Rem}

\sssec{} Assume for the rest of Section~\ref{Sect_Fact_cat_for_cocom_coalg} that $C(X)$ is dualizable in $Shv(X)-mod$. 

 Write $C^{\vee}(X)$ for its dual in $Shv(X)-mod$. Then $C^{\vee}(X)$ with the product $\com^{\vee}: (C^{\vee})^{\otimes 2}(X)\to C^{\vee}(X)$ becomes an object of $CAlg^{nu}(Shv(X)-mod)$.

 We have seen in Lemma~\ref{Lm_3.2.6_for_sheaves_of_cat} i) that for any $(J\toup{\phi} K)\in \cTw(fSets)$, $C^{\otimes\phi}$ is dualizable in $Shv(X^I)-mod$, hence also in $Shv(\Ran)-mod$ by Lemma~\ref{Lm_2.1.11}. 
 So, we may pass to the duals in $Shv(\Ran)-mod$ in the functor $\cG_{\Ran, C}$ and get the functor denoted 
$$
\cG^{\vee}_{\Ran, C}: \cTw(fSets)\to Shv(\Ran)-mod
$$ 
In this case we get an isomorphism of functors
$$
\cF_{\Ran, C^{\vee}}\;\iso\; \cG^{\vee}_{\Ran, C}
$$
from $\cTw(fSets)$ to $Shv(\Ran)-mod$.  

 Similarly, for $I\in fSets$ we get a canonical isomorphism of functors 
$$
\cF_{I, C^{\vee}}\;\iso\; \cG^{\vee}_{I, C}
$$ 
from $Tw(I)$ to $Shv(X^I)-mod$. The canonical pairing in $Shv(\Ran)-mod$
between $\underset{\cTw(fSets)^{op}}{\lim}\cG_{\Ran, C}$ and $\underset{\cTw(fSets)}{\colim}\cG^{\vee}_{\Ran, C}$ is a functor
\begin{equation}
\label{pairing_canonical_Sect_3.1.12}
\Fact(C^{\vee})\otimes_{Shv(\Ran)}\Fact^{co}(C)\to Shv(\Ran)
\end{equation}
Similarly, we have a canonical pairing
$$
(C^{\vee})_{X^I}\otimes_{Shv(X^I)}C_{X^I, co}\to Shv(X^I)
$$

It is not clear if the above two pairings extend to a duality datum. 

\begin{Rem} If in addition $\com: C(X)\to C^{\otimes 2}(X)$ admits a left adjoint $\com^L$, which is $Shv(X)$-linear, then the pairing (\ref{pairing_canonical_Sect_3.1.12}) is a counit of a duality datum. This follows from Proposition~\ref{Pp_2.4.7}.
\end{Rem}

\sssec{} For $I\in fSets$ denote by $\Loc^{co}: C_{X^I, co}\to \underset{i\in I}{\boxtimes} C(X)$ the structure functor for $(I\toup{\id} I\toup{\id} I)\in Tw(I)^{op}$. 

\begin{Lm} If in addition $C(X)\in CoCAlg(Shv(X)-mod)$, that is, $C$ is unital then, for each $I\in fSets$, $\Loc^{co}$ is conservative.
\end{Lm}
\begin{proof}
Argue as in Lemma~\ref{Lm_3.5.2_generated_under_colim}. We have a natural map of functors 
$$
\cG_{I, C}\to \cG_{I, C}\comp j^{op}\comp (j^R)^{op}
$$ 
from $Tw(I)^{op}$ to $Shv(X^I)-mod$, where we used the adjoint pair (\ref{adj_pair_j_and_j^R_Lm_3.5.2}). For each $\Sigma=(I\to J\to K)\in Tw(I)^{op}$ this map evaluated on $\Sigma$ is
$$
\underset{k\in K}C^{\otimes J_k}(X)\to \underset{k\in K}C^{\otimes I_k}(X).
$$
It is given by the exterior product of comultpilications along maps $I_k\to J_k$ for $k\in K$. 
The latter functor is conservative. To see this pick a section $J\to I$ of the surjection $I\to J$. The conservativity then follows from the commutative diagram
$$
\begin{array}{ccc}
\underset{k\in K}C^{\otimes J_k}(X) & \to & \underset{k\in K}C^{\otimes I_k}(X)\\
& \searrow\lefteqn{\scriptstyle\id} & \downarrow\lefteqn{\scriptstyle counit}\\
&& \underset{k\in K}C^{\otimes J_k}(X)
\end{array}
$$

 Applying now (\cite{Ly}, 9.2.67) one gets the desired result.
\end{proof}

\sssec{} For the rest of Section~\ref{Sect_Fact_cat_for_cocom_coalg} we make the assumptions of Remark~\ref{Rem_3.1.8}. So, $\Fact^{co}(C)$ is a factorization sheaf of categories on $\Ran$. Though $\Fact^{co}(C)$ is dualizable in $Shv(\Ran)-mod$ in this case, we can not conclude that (\ref{pairing_canonical_Sect_3.1.12}) extends to a duality datum for $\Fact^{co}(C)$. 

 Then we define the chiral comultiplication as follows. 
 
Let $I_1, I_2\in fSets$ and $I=I_1\sqcup I_2$. Recall the full embedding (\ref{map_alpha}). 
We have the natural functor
$$
C_{X^I, co}\to \underset{Tw(I_1)^{op}\times Tw(I_2)^{op}}{\lim} \;\cG_{I,C}\comp\alpha^{op}
$$
Under our assumptions the target of this map identifies with $C_{X^{I_1}, co}\boxtimes C_{X^{I_2}, co}$, so we get the morphism
\begin{equation}
\label{map_eta_for_Sect_3.1.15}
\eta_{I_1, I_2}: C_{X^I, co}\to C_{X^{I_1}, co}\boxtimes C_{X^{I_2}, co}
\end{equation}
in $Shv(X^I)-mod$. Moreover, the restriction of $\eta_{I_1, I_2}$ to $(X^{I_1}\times X^{I_2})_d$ is a special case of the factorization isomorphism (\ref{fact_iso_for_C_co}).  

\sssec{} Let now $I_1\to I'_1, I_2\to I'_2$ be maps in $fSets$, let $I'=I'_1\sqcup I'_2$. As in Section~\ref{Sect_2.2.4}, $\eta_{I_1, I_2}$. fits into a commutative diagram  
$$
\begin{array}{ccc}
C_{X^I, co} & \toup{(\ref{map_eta_for_Sect_3.1.15})} & C_{X^{I_1}, co}\boxtimes C_{X^{I_2}, co}\\
\downarrow && \downarrow\\
C_{X^{I'}, co} & \toup{(\ref{map_eta_for_Sect_3.1.15})}  & C_{X^{I'_1}, co}\boxtimes C_{X^{I'_2}, co},
\end{array}
$$
where the vertical arrows are the !-restrictions along $X^{I'}\hook{} X^I$. Passing to the limit over $(I_1, I_2)\in fSets\times fSets$, the above diagram yields the functor 
$$
\eta: \Gamma(\Ran\times\Ran, u^!\Fact^{co}(C))\to \Fact^{co}(C)\boxtimes\Fact^{co}(C)
$$
in $Shv(\Ran\times\Ran)-mod$. Here $u: \Ran\times\Ran\to\Ran$ is the sum map. 

\sssec{} We may define the analog of the $\star$ operation in this case as the composition
$$
\Fact^{co}(C)\toup{u^!_{\Fact^{co}(C)}} \Gamma(\Ran\times\Ran, u^!\Fact^{co}(C))\toup{\eta} \Fact^{co}(C)\boxtimes\Fact^{co}(C)
$$

Note that if our sheaf theory is $\cD$-modules, we may view the latter as a comultiplication on $\Fact^{co}(C)$.  

\index{Chiral comultiplication on $\Fact^{co}(C)$} 

\section{Factorization categories attached to constant commutative algebras}
\label{Sect_4}

\ssec{} In this section we recall the construction of factorization categories attached to $C\in CAlg^{nu}(\DGCat_{cont})$ by S. Raskin (\cite{Ras2}, Section~6). It uses a different combinatorics (namely, opens of $X^I$ instead of closed parts that we used in Section~\ref{Sect_Fact_cat_for_com_algebras}), and the construction makes sense only for constant commutative algebras along the curve $X$. This construction has appeared in some preliminary form in \cite{Gai_de_Jong}. 

 Our mnemonics is as follows: the functors $\cF, \bar\cF$ are designed to calculate colimits, and $\cG, \bar\cG$ are designed to calculate limits. One more way to remember notations: $\bar\cG,\bar\cF$ use open subsets, while $\cF,\cG$ use closed ones. 

\sssec{} 
\label{Sect_4.1.1_now}
Let $C\in CAlg^{nu}(\DGCat_{cont})$. Let $I\in fSets$. Consider the functor 
$$
\bar\cG_{I, C}: Tw(I)^{op}\to Shv(X^I)-mod
$$
sending $\Sigma=(I\toup{p}J\to K)$ to $C_{\Sigma}:=C^{\otimes K}\otimes Shv(X^I_{p, d})$, and sending the map (\ref{morphism_in_Tw_v2}) to the composition
$$
C^{\otimes K_2}\otimes Shv(X^I_{p_2, d})\to C^{\otimes K_1}\otimes Shv(X^I_{p_2, d})\to C^{\otimes K_1}\otimes Shv(X^I_{p_1, d}),
$$
where the first map comes from the product $C^{\otimes K_2}\to C^{\otimes K_1}$, and the second is the restriction along the open embedding $X^I_{p_1, d}\subset X^I_{p_2, d}$. 
 
 Set 
$$
\bar C_{X^I}=\underset{Tw(I)^{op}}{\lim} \bar\cG_{I,C}
$$
in $Shv(X^I)-mod$. Note that we may view $\bar\cG_{I, C}$ as a functor $Tw(I)^{op}\to CAlg^{nu}(Shv(X^I)-mod)$. We may understand the latter limit in $CAlg^{nu}(Shv(X^I)-mod)$, as 
$$
\oblv: CAlg^{nu}(Shv(X^I)-mod)\to Shv(X^I)-mod
$$ 
preserves limits. In particular, $\bar C_{X^I}\in CAlg^{nu}(Shv(X^I)-mod)$.  
 
\begin{Lm} For a map $f: I\to I'$ in $fSets$ one has a canonical equivalence 
\label{Lm_4.1.2_restrictions}
$$
\bar C_{X^I}\otimes_{Shv(X^I)} Shv(X^{I'})\,\iso\, \bar C_{X^{I'}}
$$
in $CAlg^{nu}(Shv(X^{I'})-mod)$ in a way compatible with compositions, so giving rise to a sheaf of categories $\ov{\Fact}(C)$ on $\Ran$ with $\ov{\Fact}(C)\otimes_{Shv(\Ran)} Shv(X^I)\,\iso\, \bar C_{X^I}$ for all $I\in fSets$.
\end{Lm}
\begin{proof} Recall that $Shv(X^{I'})$ is dualizable in $Shv(X^I)-mod$, so 
$$
\bar C_{X^I}\otimes_{Shv(X^I)} Shv(X^{I'})\,\iso\,\underset{(I\toup{p} J\to K)\in \Tw(I)^{op}}{\lim} C^{\otimes K}\otimes Shv(X^I_{p, d})\otimes_{Shv(X^I)} Shv(X^{I'})
$$
By Lemma~\ref{Lm_3.2.4_from_chiral_algebras}, $X^I_{p, d}\times_{X^I} X^{I'}$ is empty unless $J\in Q(I')$, and identifies with $X^{I'}_{p', d}$ in the latter case, where $p$ is written as the compostion $I\toup{f} I'\toup{p'} J$. 

 Recall the full embedding $\Tw(I')\subset \Tw(I)$ from Section~\ref{Sect_2.1.5_now}. It is zero-cofinal. So, the above limit identifies with
$$
\underset{(I'\toup{p'} J'\to K')\in \Tw(I')^{op}}{\lim} C^{\otimes K'}\otimes Shv(X^{I'}_{p', d})\,\iso\, \bar C_{X^{I'}}.
$$
Other properties follow from the construction.
\end{proof} 

\sssec{} 
\label{Sect_4.1.3_now}
Write $\bar\cG_{fSets, C}: fSets\to Shv(\Ran)-mod$ for the functor sending $I$ to $\bar C_{X^I}$ and $(J\to K)$ to the pullback $\vartriangle^{(J/K)!}: \bar C_{X^J}\to \bar C_{X^K}$, here $\vartriangle^{(J/K)}: X^K\to X^J$. Let 
$$
\bar\cG_{fSets, C}^L: fSets^{op}\to Shv(\Ran)-mod
$$ 
be obtained from $\bar\cG_{fSets, C}$ by passing to left adjoints. So,
$$
\ov{\Fact}(C)\,\iso\, \underset{fSets}{\lim} \bar\cG_{fSets, C}\,\iso\, \underset{fSets^{op}}{\colim} \bar\cG_{fSets, C}^L
$$

 Since $\Ran$ is 1-affine for our sheaf theories, for any $E\in Shv(\Ran)-mod$ the natural map
$$
E\otimes_{Shv(\Ran)} \underset{fSets}{\lim} \bar\cG_{fSets, C}\to \underset{I\in fSets}{\lim} (\bar C_{X^I}\otimes_{Shv(\Ran)} E)
$$
is an equivalence.

\begin{Cor} 
\label{Cor_4.1.3}
Let $I\in fSets$ and $f: I\to *$. \\
i) The square is cartesian in $Shv(X^I)-mod$
$$
\begin{array}{ccc}
\underset{(\Tw(I)^{>1})^{op}}{\lim}\; \bar\cG_{I,C} & \gets & \bar C_{X^I}\\
\downarrow && \downarrow\lefteqn{\scriptstyle h}\\
\underset{(\Tw(I)^{>1}\cap \Tw(I)^f)^{op}}{\lim}\; \bar\cG_{I,C} & \gets & 
C\otimes Shv(X^I),
\end{array}
$$
where $h$ is the projection corresponding to $(I\to *\to *)\in \Tw(I)^{op}$, and 
all the limits are calculated in $Shv(X^I)-mod$. \\
ii) Assume $C$ dualizable in $\DGCat_{cont}$. Let $U=X^I-X$. Then the square is cartesian
in $Shv(X^I)-mod$
$$
\begin{array}{ccc}
\bar C_{X^I}\mid_U & \gets & \bar C_{X^I}\\
\downarrow\lefteqn{\scriptstyle h\mid_U} && \downarrow\lefteqn{\scriptstyle h}\\
C\otimes Shv(U) & \gets & C\otimes Shv(X^I),
\end{array}
$$
where all the limits are calculated in $Shv(X^I)-mod$.
\end{Cor}
\begin{proof}
i) is a version of Corollary~\ref{Cor_2.2.14} in our setting. 

\medskip\noindent
ii) follows from i) using the following. First, note that $\underset{(I\toup{p} J\to *)\in \Tw(I)^{>1}\cap \Tw(I)^f}{\cup} \; X^I_{p, d}=U$. Let $L: \PreStk_{lft}\to\PreStk_{lft}$ be the functor of sheafification in the etale topology. By (\cite{Ly}, 10.2.2), the natural map
$$
L(\underset{(I\toup{p} J\to *)\in \Tw(I)^{>1}\cap \Tw(I)^f}{\colim} X^I_{p, d})\to U
$$
is an isomorphism in $\PreStk_{lft}$, where the colimit is understood in $\PreStk_{lft}$. 
By \'etale descent for $Shv$, 
$$
\underset{(I\toup{p} J\to *)\in
(\Tw(I)^{>1}\cap \Tw(I)^f)^{op}}{\lim} Shv(X^I_{p, d})\,\iso\, Shv(U).
$$
Since $C$ is dualizable, tensoring by $C$ commutes with the limit in the latter formula.  
\end{proof} 

\sssec{Commutative chiral product} 
\label{Sect_4.1.5_now}
First, define the commutative product as follows. Pick $\phi: I\to I'$ in $fSets$. Let us define a canonical map
\begin{equation}
\label{functor_product_for_barC_X^I}
\underset{i\in I'}{\boxtimes} \bar C_{X^{I_i}}\to \bar C_{X^I}
\end{equation}
in $CAlg^{nu}(Shv(X^I)-mod)$. 

 Recall the full subcategory $\Tw(I)_{\phi}\,\iso\,\prod_{i\in I'} \Tw(I_i)$ of $\Tw(I)$ from 
Section~\ref{Sect_2.1.11_now}.  
Denote by $\nu: Tw(I)\to \prod_{i\in I'} \Tw(I_{i'})$ the functor
sending $(I\toup{p} J\to K)$ to the collection 
\begin{equation}
\label{collection_for_Sect_4.1.3}
\{(I_i\toup{p_i} J_i\to K_i)\}_{i\in I'},
\end{equation} 
where $J_i$ is the image of $I_i$ under $p$, and $K_i$ is the image of $J_i$ under $J\to K$.
We warn the reader that $\nu$ is not an adjoint (on either side) of the inclusion $\Tw(I)_{\phi}\subset Tw(I)$. 

 For $(I\to J\to K)\in Tw(I)$ let (\ref{collection_for_Sect_4.1.3}) be its image under $\nu$. Then 
\begin{equation}
\label{open_imm_for_Sect_4.1.3}
X^I_{p, d}\subset \prod_{i\in I'} X^{I_i}_{p_i, d}\subset X^I, 
\end{equation}
and we get canonical maps
\begin{equation}
\label{comm_chiral_mult_for_barC_maps}
\underset{i\in I'}{\boxtimes} \bar C_{X^{I_i}} \to (\underset{i\in I'}{\otimes} C^{\otimes K_i})\otimes Shv(\prod_{i\in I'} X^{I_i}_{p_i, d})\to C^{\otimes K}\otimes Shv(X^I_{p, d})
\end{equation}
The first map here comes from the projections (in the definition of the limit). The second map comes from the product in $C$ along the surjection $\sqcup_{i\in I'} K_i\to K$ together with restriction under the open immersion (\ref{open_imm_for_Sect_4.1.3}). The compositions in (\ref{comm_chiral_mult_for_barC_maps}) are compatible with the transition functors in the diagram $\bar\cG_{I,C}$, so by definition of $\lim\bar\cG_{I,C}$ yield the desired functor (\ref{functor_product_for_barC_X^I}). 

 In the case when $\phi: I\to I'$ is $\id: I\to I$ the functor (\ref{functor_product_for_barC_X^I}) is denoted by $\Loc: C^{\otimes I}\otimes Shv(X^I)\to \bar C_{X^I}$ by abuse of notations. 

\sssec{} Now we discuss the factorization structure on $\ov{\Fact}(C)$ under additional assumptions.

A factorization structure $\ov{\Fact}(C)$ now is a property. Namely, $\ov{\Fact}(C)$ is a factorization category iff for each $\phi: I\to I'$ in $fSets$ the restriction of  (\ref{functor_product_for_barC_X^I}) to the open part $X^I_{\phi, d}\subset X^I$ is an equivalence. 

 Recall the stratification of $X^I_{\phi, d}$ introduced in Section~\ref{Sect_3.1.7_stratifications}. It is easy to see that for a map $I\to J$ in $fSets$ giving the corresponding stratum $Y=\vartriangle^{(I/J)}(\oo{X}{}^J)$ of $X^I_{\phi, d}$ the functor (\ref{functor_product_for_barC_X^I}) becomes an equivalence after applying
$$
\cdot\otimes_{Shv(X^I)} Shv(Y)
$$
So, the situation here is similar to that of Section~\ref{Sect_3.1.7_stratifications} (for cocommutative coalgebras). In general, $\ov{\Fact}(C)$ \select{looks like a factorization category} in the same sense as in Section~\ref{Sect_3.1.7_stratifications}. 
An analog of Remark~\ref{Rem_3.1.8} also holds:

\begin{Rem}
\label{Rem_4.1.6}
Suppose that $C$ is dualizable in $\DGCat_{cont}$, and for any $I\in fSets$, $\bar C_{X^I}$ is dualizable in $Shv(X^I)-mod$. Then for each $\phi: I\to I'$ in $fSets$ the restriction of  (\ref{functor_product_for_barC_X^I}) to the open part $X^I_{\phi, d}\subset X^I$ is an equivalence. (This follows from Section~\ref{Sect_4.1.8} below). So, $\ov{\Fact}(C)$ is a factorization category. Besides, by Corollary~\ref{Cor_2.2.9_dualizability}, $\ov{\Fact}(C)$ is dualizable in $Shv(\Ran)-mod$.
\end{Rem}

\sssec{}  Denote by $j: {^{\phi}\Tw(I)}\subset\Tw(I)$ the full subcategory of those $(I\toup{p} J\to K)$ for which $I'\in Q(J)$. If $\phi: I\to I'$ is $\id: I\to I$ then the inclusion ${^{\phi}\Tw(I)}\subset\Tw(I)$ coincides with the left adjoint in (\ref{adj_pair_j_and_j^R_Lm_3.5.2}), this explains our notation. 
 
 For a diagram $J\getsup{p} I\toup{\phi} I'$ in $fSets$ consider the category $Diag(p, \phi)$. Its objects are diagrams
$$
\begin{array}{ccccc}
J & \getsup{p} & I &\toup{\phi} & I'\\
& \nwarrow& \downarrow & \nearrow\\
&& \bar I
\end{array}
$$
in $fSets$ and morphisms are morphisms of diagrams in $fSets$ covariant in $\bar I$, so that we have the projection $Diag(p, \phi)\to fSets_{I/}$. Denote by $\sup(J, I')$ a final object of $Diag(p, \phi)$, it always exists. The functor $j$ has a right adjoint $j^R: \Tw(I)\to {^{\phi}\Tw(I)}$ sending $(I\toup{p} J\to K)$ to $(I\toup{\tau} \sup(J, I')\to K)$. 

\sssec{} 
\label{Sect_4.1.8}
The category $\bar C_{X^I}\mid_{X^I_{\phi, d}}$ identifies with
\begin{equation}
\label{limit_for_Sect_4.1.7}
\underset{(I\toup{p} J\to K)\in \Tw(I)^{op}}{\lim} C^{\otimes K}\otimes Shv(X^I_{p, d})
\otimes_{Shv(X^I)} Shv(X^I_{\phi, d})
\end{equation}  
Note that 
$$
Shv(X^I_{p, d})\otimes_{Shv(X^I)} Shv(X^I_{\phi, d})\,\iso\, Shv(X^I_{p, d}\times_{X^I} X^I_{\phi, d})
$$
For $I\toup{\tau} \sup(J, I')$ one has 
$$
X^I_{p, d}\times_{X^I} X^I_{\phi, d}=X^I_{\tau, d}
$$  
So, (\ref{limit_for_Sect_4.1.7}) is the limit of the composition 
$$
\Tw(I)^{op}\toup{(j^R)^{op}}  {^{\phi}\Tw(I)^{op}}\toup{\mu} Shv(X^I_{\phi,d})-mod,
$$ 
where $\mu$ is the restriction of $\bar\cG_{I,C}$ to the full subcategory $^{\phi}\Tw(I)^{op}$. 
Now $\mu\comp (j^R)^{op}$ is the right Kan extension of $\mu$ along $j^{op}$, because $j^{op}$ is right adjoint to $(j^R)^{op}$. So, (\ref{limit_for_Sect_4.1.7}) identifies with
$$
\underset{^{\phi}\Tw(I)^{op}}{\lim}\mu
$$

 Recall the full subcategory $\Tw(I)_{\phi}\subset {^{\phi}\Tw(I)}$ from Section~\ref{Sect_2.1.11_now}. This inclusion has a left adjoint ${^{\phi}\Tw(I)}\to \Tw(I)_{\phi}$ sending $(I\to J\to K)$ to $(I\to J\to \sup(K, I'))$. Thus, the inclusion $\Tw(I)_{\phi}\subset {^{\phi}\Tw(I)}$ is cofinal. So, 
$$
\underset{^{\phi}\Tw(I)^{op}}{\lim}\mu\;\,\iso\;  \underset{(\Tw(I)_{\phi})^{op}}{\lim}\; \mu
$$

\sssec{Canonical arrow} 
\label{Sect_4.1.10}
Take $C(X)=C\otimes Shv(X)$ viewed now as an object of $CAlg^{nu}(Shv(X)-mod)$. In this section we construct a canonical map
\begin{equation}
\label{map_from_Fact_to_ov_Fact}
\Fact(C)\to \ov{\Fact}(C)
\end{equation}
in $Shv(\Ran)-mod$.

\index{Canonical arrow $\Fact(C)\to \ov{\Fact}(C)$}

 It amounts to a collection of functors in $Shv(X^I)-mod$
\begin{equation}
\label{functor_from_C_X^I_to_barC_X^I} 
\zeta_I: C_{X^I}\to \bar C_{X^I}
\end{equation} 
compatible with $!$-restrictions under $\vartriangle^{(I/J)}: X^J\to X^I$ for any map $I\to J$ in $fSets$. 
 
 Pick $(I\toup{p} J\to K), (I\to J_1\to K_1)\in Tw(I)$. Recall that $X^I_{p, d}\times_{X^I} X^{K_1}$ is empty unless $J\in Q(K_1)$. First, we define a functor
\begin{equation}
\label{map_from_Dennis_to_Sam} 
 \Shv(X^{K_1})\otimes C^{\otimes J_1}\to Shv(X^I_{p, d})\otimes C^{\otimes K}
\end{equation} 
as follows. It vanishes unless $J\in Q(K_1)$. In the latter case we get a diagram $I\to J_1\to K_1\to J\to K$, hence a map $C^{\otimes J_1}\to C^{\otimes K}$ given by the product along $J_1\to K$. Then (\ref{map_from_Dennis_to_Sam}) is the composition 
$$
\Shv(X^{K_1})\otimes C^{\otimes J_1}\to \Shv(X^{K_1})\otimes C^{\otimes K}\to 
Shv(X^I_{p, d})\otimes C^{\otimes K},
$$
where the second map is the direct image under $X^{K_1}\to X^I$ followed by restriction to $X^I_{p, d}$. These maps are compatible with transition functors in the definitions of $C_{X^I}, \bar C_{X^I}$, hence yield the desired functor (\ref{functor_from_C_X^I_to_barC_X^I}). For each $(I\to J_1\to K_1)\in Tw(I)$ the so obtained map $\Shv(X^{K_1})\otimes C^{\otimes J_1}\to \bar C_{X^I}$ is a morphism in $CAlg^{nu}(Shv(X^I)-mod)$. 

  One checks that $\zeta_I$ are compatible with $!$-restrictions along $\vartriangle^{(I/J)}: X^J\to X^I$ for maps $I\to J$ in $fSets$, so yield the desired functor (\ref{map_from_Fact_to_ov_Fact}).  
  
\sssec{} Let $C, D\in CAlg^{nu}(\DGCat_{cont})$. For $I\in fSets$ let us construct a $Shv(X^I)$-linear functor
\begin{equation}
\label{functor_for_barC_barD_first}
\gamma_I: \bar C_{X^I}\otimes_{Shv(X^I)} \bar D_{X^I}\to \ov{C\otimes D}_{X^I}
\end{equation}

 By definition, we have a natural functor
$$
\bar C_{X^I}\otimes_{Shv(X^I)} \bar D_{X^I}\to \underset{
\substack{(I\toup{p_1} J_1\to K_1)\in \Tw(I)^{op}\\ (I\toup{p_2} J_2\to K_2)\in\Tw(I)^{op}}}
{\lim}  C^{\otimes K_1}\otimes D^{\otimes K_2}\otimes Shv(X^I_{p_1, d})\otimes_{Shv(X^I)} Shv(X^I_{p_2, d})
$$
The latter maps naturally (via restriction to the diagonal) to
$$
 \underset{(I\toup{p} J_1\to K_1)\in \Tw(I)^{op}}{\lim} (C\otimes D)^{\otimes K}\otimes Shv(X^I_{p,d})\,\iso\, \ov{C\otimes D}_{X^I}
$$
Composing, we get the desired functor (\ref{functor_for_barC_barD_first}). 

 As in Lemma~\ref{Lm_4.1.2_restrictions}, the functors (\ref{functor_for_barC_barD_first}) are compatible with $(\vartriangle^{(I/J)})^!$ for $\vartriangle^{(I/J)}: X^J\to X^I$. So, they yield $Shv(\Ran)$-linear morphism
\begin{equation}
\label{functor_sym_mon_str_on_ovFact}
\ov{\Fact}(C)\otimes_{Shv(\Ran)} \ov{\Fact}(D)\to \ov{\Fact}(C\otimes D)
\end{equation}

This equips the functor $CAlg^{nu}(\DGCat_{cont})\to Shv(\Ran)-mod$, $C\mapsto \ov{\Fact}(C)$ with a non-unital right-lax symmetric monoidal structure. One checks also that the functors (\ref{functor_sym_mon_str_on_ovFact}) are compatible with the commutative chiral products. 

\sssec{} 
\label{Sect_4.1.12_now_t-str}
Assume for Section~\ref{Sect_4.1.12_now_t-str} only that 
\begin{itemize}
\item $C$ is compactly generated and equipped with a compactly generated t-structure, that is, $C^{\le 0}$ is generated under filtered colimits by $C^{\le 0}\cap C^c$;
\item the product $m: C^{\otimes 2}\to C$ is t-exact.
\end{itemize}
In this case $\ov{\Fact}(C)$ is equipped with a natural t-structure that we are going to define.

\index{t-structure on $\ov{\Fact}(C)$}

 For $K\in fSets$ we equip $C^{\otimes K}$ with the t-structure via the construction from (\cite{Ly}, 9.3.17). Then using (\cite{Ly}, 9.3.10) for any map $K_2\to K_1$ in $fSets$ the product $C^{\otimes K_2}\to C^{\otimes K_1}$ is t-exact.

 For each $I\in fSets$ equip $\bar C_{X^I}$ with a t-structure as follows. For $(I\toup{p} J\to K)\in\Tw(I)$ we equip $C^{\otimes K}\otimes Shv(X^I_{p, d})$ with the t-structure via the coinstruction from (\cite{Ly}, 9.3.17), where $Shv(X^I_{p, d})$ is equipped with the perverse t-structure. Then the transition maps in the diagram $\bar\cG_{I, C}: \Tw(I)^{op}\to Shv(X^I)-mod$ are t-exact. By (\cite{G}, I.3, 1.5.8), there is a unique t-structure on $\bar C_{X^I}\,\iso\, \underset{\Tw(I)^{op}}{\lim} \bar\cG_{I,C}$ such that each evaluation functor
$$
\bar C_{X^I}\to C^{\otimes K}\otimes Shv(X^I_{p, d})
$$
is t-exact. Since the t-structure on $Shv(X^I_{p, d})$ is compactly generated, the obtained t-structure on $\bar C_{X^I}$ is compatible with filtered colimits by (we used \cite{Ly}, 9.3.5).  

 If $I\to J$ is a map in $fSets$ we have the adjoint pair $\vartriangle_!: \bar C_{X^J}\leftrightarrows \bar C_{X^I}: \vartriangle^!$ in $Shv(\Ran)-mod$ for $\vartriangle: X^J\to X^I$. By construction, the functor $\vartriangle_!: \bar C_{X^J}\to \bar C_{X^I}$ is t-exact. 
  
  Let $\ov{\Fact}(C)^{\le 0}\subset \ov{\Fact}(C)$ be the smallest full subcategory closed under extensions, colimits, and containing for each $I\in fSets$ the image of $\bar C_{X^I}^{\le 0}$. By (\cite{HA}, 1.4.4.11), this is an accessible t-structure on $\ov{\Fact}(C)$.   
  
  An object of $\ov{\Fact}(C)$ is coconnective iff for any $I\in fSets$ its !-restriction under $X^I\to \Ran$ is coconnective in $\bar C_{X^I}$. This shows that the t-structure on $\ov{\Fact}(C)$ is compatible with filtered colimits.  
  
\begin{Rem} 
\label{Rem_4.1.13_now}
In the situation of Section~\ref{Sect_4.1.12_now_t-str} assume that the canonical arrow $\Fact(C)\to\ov{\Fact}(C)$ is an equivalence. For any $(I\to J\to K)\in\Tw(I)$ the category $C^{\otimes J}\otimes Shv(X^K)$ is equipped with a t-structure naturally, and the transition functors in the diagram $\cF_{I,C}: \Tw(I)\to Shv(X^I)-mod$ are t-exact. It is easy to see that for any $(I\to J\to K)\in\Tw(I)$ the composition $C^{\otimes J}\otimes Shv(X^K)\to C_{X^I}\iso \bar C_{X^I}$ is t-exact. 

Let $A\in CAlg^{nu}(C)$ lying in $C^{\heartsuit}$. Then for any $I\in fSets$, the image of $A_{X^I}$ under $C_{X^I}\,\iso\, \bar C_{X^I}$ is contained in $\bar C_{X^I}^{<0}$. Indeed, 
$$
A_{X^I}\,\iso\, \underset{(I\to J\to K)\in\Tw(I)}{\colim} A^{\otimes J}\otimes\omega_{X^K}
$$ 
taken in $C_{X^I}$. For each $(I\to J\to K)\in\Tw(I)$ the image of $A^{\otimes J}\otimes\omega_{X^K}$ under 
$$
C^{\otimes J}\otimes Shv(X^K)\to \bar C_{X^I}
$$ 
is placed in degrees $<0$, and $\bar C_{X^I}^{<0}$ is closed under colimits.
\end{Rem}  

\sssec{} From now on for the rest of Section~\ref{Sect_4} we assume $C$ dualizable in $\DGCat_{cont}$, and $m: C^{\otimes 2}\to C$ has a continuous right adjoint in $\DGCat_{cont}$. 

\begin{Pp} 
\label{Pp_4.1.10_Raskin_dualizability}
For any $I\in fSets$, $\bar C_{X^I}$ is dualizable in $Shv(X^I)-mod$. Moreover, for any $D\in Shv(X^I)-mod$ the canonical map
\begin{equation}
\label{map_for_Pp_4.1.10}
\bar C_{X^I}\otimes_{Shv(X^I)} D\to \underset{(I\toup{p} J\to K)\in \Tw(I)^{op}}
{\lim} C^{\otimes K}\otimes Shv(X^I_{p, d})\otimes_{Shv(X^I)} D
\end{equation}
is an equivalence.
\end{Pp}

 For $\cD$-modules this is a beautiful result of Raskin (\cite{Ras2}, 6.18.1). The proof given in \select{loc.cit.} does not apply in the constructible context as in (\cite{Ras2}, A.3) it uses the forgetful functor $Shv(S)\to \QCoh(S)$ for $S\in \Sch_{ft}$. Our contribution is to adapt his argument, so that it holds also in the constructible context. The proof of Proposition~\ref{Pp_4.1.10_Raskin_dualizability} is postponed to Appendix~\ref{Sect_proof_Pp_4.1.10_Raskin_dualizability}.  

 By Remark~\ref{Rem_4.1.6} and Proposition~\ref{Pp_4.1.10_Raskin_dualizability}, $\ov{\Fact}(C)$ is a factorization category over $\Ran$, and $\ov{\Fact}(C)$ is dualizable in $Shv(\Ran)-mod$. Moreover, (\ref{map_from_Fact_to_ov_Fact}) is compatible with factorization. 
 
\sssec{} For $I\in fSets$ let $\bar\cG_{I,C}^{\vee}: \Tw(I)\to Shv(X^I)-mod$ be the functor obtained from $\bar\cG_{I,C}$ by passing to the duals in $Shv(X^I)-mod$. Denote by 
$$
\bar\cG_{fSets, C}^{\vee}: fSets^{op}\to Shv(\Ran)-mod
$$ 
the functor obtained from $\bar\cG_{fSets, C}$ by passing to the duals in $Shv(\Ran)-mod$. 

\begin{Cor} 
\label{Cor_4.1.14}
i)  For $I\in fSets$ the canonical pairing between $\underset{\Tw(I)}{\colim}\bar\cG_{I,C}^{\vee}$ and $\underset{\Tw(I)^{op}}{\lim}\bar\cG_{I, C}$ is a duality datum realizing $\underset{\Tw(I)}{\colim}\bar\cG_{I,C}^{\vee}$ as the dual of $\bar C_{X^I}$ in $Shv(X^I)-mod$. 

\smallskip\noindent
ii) The canonical pairing between $\underset{fSets^{op}}{\colim} \bar\cG_{fSets, C}^{\vee}$ and $\underset{fSets}{\lim} \bar\cG_{fSets, C}\,\iso\; \ov{\Fact}(C)$ is a duality datum in $Shv(\Ran)-mod$. 
\end{Cor}
\begin{proof}
i) is immediate from Proposition~\ref{Pp_4.1.10_Raskin_dualizability}.

\smallskip\noindent
ii) follows from Section~\ref{Sect_4.1.3_now}. 
\end{proof}

\begin{Rem} 
\label{Rem_4.1.15_now}
View $C^{\vee}$ as an object of $CoCAlg^{nu}(\DGCat_{cont})$ with the comultiplication $m^{\vee}: C^{\vee}\to (C^{\vee})^{\otimes 2}$. Then the functor $\bar\cG_{I, C}^{\vee}: \Tw(I)\to Shv(X^I)-mod$ identifies canonically with the functor $\bar\cF_{I, C^{\vee}}$ defined in Section~\ref{Sect_5}. Corollary~\ref{Cor_4.1.14} affirms that $\bar C_{X^I}$ and $\ov{C^{\vee}}_{X^I, co}$ are canonically dual to each other.

 Besides, we have canonically $\bar\cG^{\vee}_{fSets, C}\,\iso\, \bar\cF_{fSets, C^{\vee}}$. Thus, $\ov{\Fact}(C)$ and $\ov{\Fact}^{co}(C^{\vee})$ are canonically dual to each other.
\end{Rem}

\sssec{} Under our assumptions, the functor $m^R: C\to C^{\otimes 2}$ defines on $C$ a structure of an object of $CoCAlg^{nu}(\DGCat_{cont})$. Further passing to the duals, $C^{\vee}$ is promoted to an object of $CAlg^{nu}(\DGCat_{cont})$ with the product $(m^R)^{\vee}: (C^{\vee})^{\otimes 2}\to C^{\vee}$, compare with Section~\ref{Sect_2.5.5}. 

 Under our assumptions, we may pass to right adjoints in the functor $\bar\cG_{I, C}$ and get a functor $\bar\cG_{I,C}^R: \Tw(I)\to Shv(X^I)-mod$. 
 In Section~\ref{Sect_5} we will define the functor 
$$
\bar\cF_{I, C}: \Tw(I)\to Shv(X^I)-mod
$$ 
attached to $(C, m^R)\in CoCAlg^{nu}(\DGCat_{cont})$. We have canonically $\bar\cF_{I,C}\,\iso\, \bar\cG_{I,C}^R$. 
  
\sssec{} For the rest of Section~\ref{Sect_4} we assume in addition that $C$ is compactly generated and $C\in CAlg(\DGCat_{cont})$, that is, $C$ is unital.

\begin{Lm} $C(X)=C\otimes Shv(X)$ is ULA over $Shv(X)$. 
\end{Lm}
\begin{proof}
If $z\in C^c$ then $z\otimes\omega_X\in C(X)$ is ULA over $Shv(X)$. The objects of the form $z\otimes K\in C(X)$ with $K\in Shv(X)^c, z\in C^c$ generate $C(X)$.
\end{proof}

 So, the assumptions of Section~\ref{Sect_2.5.9} hold for $C(X)\in CAlg(Shv(X)-mod)$.  
 
\begin{Pp} 
\label{Pp_4.1.19}
i) For $I\in fSets$ the functor $\zeta_I: C_{X^I}\to \bar C_{X^I}$ is an equivalence. So, (\ref{map_from_Fact_to_ov_Fact}) is also an equivalence.

\smallskip\noindent
ii) for each $(I\to J\to K)\in Tw(I)$ the projection $\bar C_{X^I}\to Shv(X^I_{p, d})\otimes C^{\otimes K}$ admits a continuous $Shv(X^I)$-linear right adjoint.
\end{Pp} 
\begin{proof} 
For $S\in\Sch_{ft}$, $E\in Shv(S)-mod, x, x'\in E$ write $\und{\HOM}_E(x, x')\in Shv(S)$ for the relative inner hom for the $Shv(S)$-action. 

 By Lemma~\ref{Lm_ULA_general_lemma}, if $I\in fSets$ then $C_{X^I}$ is ULA over $Shv(X^I)$. 
  
\medskip
\noindent 
i) We argue by induction on $\mid I\mid$, the case $\mid I\mid=1$ is evident.

\smallskip\noindent
{\bf Step 1}. We claim that $\zeta_I: \cC_{X^I}\to \bar\cC_{X^I}$ admits a $Shv(X^I)$-linear continuous right adjoint. Recall that 
$$
\Loc: C^{\otimes I}\otimes Shv(X^I)\to C_{X^I}
$$ 
generates $C_{X^I}$ under colimits by Lemma~\ref{Lm_3.5.2_generated_under_colim}. 
Using Proposition~\ref{Pp_3.6.6}, it suffices to show that if $c\in (C^{\otimes I})^c$ then 
$$
\zeta_I(\Loc(c\otimes \omega))\in \bar C_{X^I}
$$ 
is ULA over $Shv(X^I)$. Indeed, the objects of the form $c\otimes \cK$ with $\cK\in Shv(X^I)^c$, $c\in (C^{\otimes I})^c$ generate $C^{\otimes I}\otimes Shv(X^I)$. Let $c\in (C^{\otimes I})^c$. 

 By (\cite{Ly}, 2.4.7), if $L\in \bar C_{X^I}$ is such that for any $(I\toup{p} J\to K)\in Tw(I)$, the image of $L$ in $Shv(X^I_{p, d})\otimes \cC^{\otimes K}$ is compact then $L$ is compact in $\bar C_{X^I}$, because $Tw(I)$ is finite. For $\cK\in Shv(X^I)^c$ the image of $\zeta_I(\Loc(c\otimes \cK))$ in each $Shv(X^I_{p, d})\otimes \cC^{\times K}$ is compact, so $$
\zeta_I(\Loc(c\otimes \cK))\in (\bar C_{X^I})^c.
$$ 
This shows that $\zeta$ admits a continuous right adjoint $\zeta^R$. 
 
  Let $L\in Shv(X^I), M\in\bar C_{X^I}$.  We must show that the natural map
\begin{equation}
\label{map_for_Lm_3.4.2_to_be_isom}
 L\otimes^! \und{\HOM}_{\bar C_{X^I}}(\zeta(\Loc(c\otimes\omega)), M)\to \und{\HOM}_{\bar C_{X^I}}(\zeta(\Loc(c\otimes\omega), L\otimes M)
\end{equation}
 is an isomorphism in $Shv(X^I)$. For $\Sigma=(I\toup{p} J\to K)\in Tw(I)$ write $M_{\Sigma}$ for the projection of $M$ to $Shv(X^I_{p, d})\otimes C^{\otimes K}$, write
$f_{\Sigma}$ for the composition 
$$
Shv(X^I)\otimes C^{\otimes I}\toup{\Loc} C_{X^I}\toup{\zeta_I}\bar C_{X^I}\to Shv(X^I_{p, d})\otimes C^{\otimes K}.
$$ 

 One has 
$$
\und{\HOM}_{\bar C_{X^I}}(\zeta(\Loc(c\otimes\omega)), M)\,\iso\,\underset{(I\toup{p} J\to K)\in Tw(I)^{op}}{\lim} \und{\HOM}_{Shv(X^I_{p, d})\otimes C^{\otimes K}}(f_{\Sigma}(c\otimes\omega), M_{\Sigma})
$$
in $Shv(X^I)$. Clearly, $f_{\Sigma}$ has a $Shv(X^I)$-linear continuous right adjoint $f_{\Sigma}^R$, and
$$
\und{\HOM}_{Shv(X^I_{p, d})\otimes C^{\otimes K}}(f_{\Sigma}(c\otimes\omega), M_{\Sigma})\,\iso\,\und{\HOM}_{Shv(X^I)\otimes C^{\otimes I}}(c\otimes \omega, f_{\Sigma}^R(M_{\Sigma}))
$$

The key point is that the functor  $Shv(X^I)\to Shv(X^I), \; \cdot\mapsto L\otimes^! \cdot$ commutes with finite limits, as this functor is exact. So, the LHS of (\ref{map_for_Lm_3.4.2_to_be_isom}) identifies with
$$
\underset{(I\toup{p} J\to K)\in Tw(I)^{op}}{\lim} L\otimes^! \und{\HOM}_{Shv(X^I)\otimes C^{\otimes I}}(c\otimes \omega, f_{\Sigma}^R(M_{\Sigma}))
$$
Since $c\otimes\omega\in Shv(X^I)\otimes C^{\otimes I}$ is ULA over $Shv(X^I)$, the latter limit becomes
\begin{multline*}
\underset{(I\toup{p} J\to K)\in Tw(I)^{op}}{\lim} \und{\HOM}_{Shv(X^I)\otimes C^{\otimes I}}(c\otimes \omega, L\otimes f_{\Sigma}^R(M_{\Sigma}))\,\iso\\
\underset{(I\toup{p} J\to K)\in Tw(I)^{op}}{\lim} \und{\HOM}_{Shv(X^I)\otimes C^{\otimes I}}(c\otimes \omega, f_{\Sigma}^R(L\otimes M_{\Sigma}))\,\iso\\
\underset{(I\toup{p} J\to K)\in Tw(I)^{op}}{\lim}\und{\HOM}_{Shv(X^I_{p, d})\otimes C^{\otimes K}}(f_{\Sigma}(c\otimes \omega), L\otimes M_{\Sigma})\,\iso\,
\und{\HOM}_{\bar C_{X^I}}(\zeta(\Loc(c\otimes \omega)), L\otimes M)
\end{multline*}

\medskip\noindent
{\bf Step 2} Let $U\subset X^I$ be the complement to the main diagonal $X\hook{} X^I$.
By Proposition~\ref{Pp_3.7.8_devissage_using_ULA_preservation}, it suffices to show now that $\zeta$ becomes an isomorphism after applying $\cdot\otimes_{Shv(X^I)} Shv(X)$ and $\cdot\otimes_{Shv(X^I)} Shv(U)$. But both properties follow from factorization and induction hyppothesis. For the open part, we used that the union of $X^I_{p, d}$ for $p: I\to J$ in $fSets$ with $\mid J\mid>1$ is $U$. We also used Zariski descent: if $\nu: B\to B'$ is a map in $Shv(U)-mod$, which becomes an equvalence after Zariski localization then $\nu$ is an equivalence. 
So, i) is proved.

\medskip\noindent
ii) For any $(I\to J_1\to K_1)\in Tw(I)$ the functor (\ref{map_from_Dennis_to_Sam}) admits a continuous $Shv(X^I)$-linear right adjoint. Recall that each transition functor in the diagram $\cF_{I, C}$ admits a $Shv(X^I)$-linear continuous right adjoint by Lemma~\ref{Lm_3.2.6_for_sheaves_of_cat}. Passing to the right adjoints in $Shv(X^I)-mod$, we get a canonical map 
$$
Shv(X^I_{p, d})\otimes \cC^{\otimes K}\to \underset{Tw(I)^{op}}{\lim} \cF^R_{I, C}\,\iso\, C_{X^I}
$$ 
in $Shv(X^I)-mod$. By (\cite{Ly}, 9.2.6), this is the desired $Shv(X)$-linear continuous right adjoint to the projection $\bar C_{X^I}\to Shv(X^I_{p, d})\otimes \cC^{\otimes K}$.
\end{proof}

\sssec{} By Proposition~\ref{Pp_4.1.19}, we may pass to right adjoints in the functor $\bar\cG_{I,C}: Tw(I)^{op}\to Shv(X^I)-mod$ and get a functor denoted $\bar\cG_{I,C}^R: Tw(I)\to Shv(X^I)-mod$. Moreover, by the above lemma we may pass to right adjoints in the limit diagram $^{\triangleleft}(Tw(I)^{op})\to Shv(X^I)-mod$ of the functor $\bar\cG_{I,C}$, this produces a functor denoted 
$$
(\bar\cG_{I,C}^R)^{\triangleright}: Tw(I)^{\triangleright}\to Shv(X^I)-mod, 
$$
whose value on the final object is $\bar C_{X^I}$. In other words, we constructed a map in $Shv(X^I)-mod$
\begin{equation}
\label{map_another_one_to_barC_X,I}
\underset{Tw(I)}{\colim}  \; \bar\cG_{I,C}^R\to \bar C_{X^I}.
\end{equation}

Viewing $(C, m^R)\in CoCAlg^{nu}(\DGCat_{cont})$, recall that $\bar\cG_{I,C}^R\,\iso\, \bar\cF_{I, C}$ canonically. In the notations of Section~\ref{Sect_5}, (\ref{map_another_one_to_barC_X,I}) is a functor $\bar C_{X^I, co}\to \bar C_{X^I}$. 

\begin{Pp} Assume that the unit $u: \Vect\to C$ has a continuous right adjoint $u^R: C\to \Vect$. Then the functor (\ref{map_another_one_to_barC_X,I}) an equivalence.
\end{Pp}
\begin{proof}  
View $C$ as an object of $CoCAlg(\DGCat_{cont})$ with the comultiplication $m^R: C\to C^{\otimes2}$ and counit $u^R: C\to \Vect$. In the notations of Section~\ref{Sect_5} the functor (\ref{map_another_one_to_barC_X,I}) is the composition
$$
\bar C_{X^I, co}\toup{\zeta_{X^I, co}} C_{X^I, co}\,\iso\, C_{X^I}\toup{\zeta_I} \bar C_{X^I},
$$
where the middle equivalence is given by Remark~\ref{Rem_5.1.9}. Our claim follows now from Theorem~\ref{Thm_5.1.11} in Section~\ref{Sect_5}. 
\end{proof}

\section{Factorization categories attached to constant cocommutative coalgebras}
\label{Sect_5}

\ssec{} In this section we assume $C\in CoCAlg^{nu}(\DGCat_{cont})$. We give a construction of a factorization category associated to $C$ in a manner dual to that of S. Raskin (\cite{Ras2}, Section~6), that is, dual to the construction of Section~\ref{Sect_4}. 

\sssec{} Let $I\in fSets$. Consider the functor
$$
\bar\cF_{I, C}: \Tw(I)\to Shv(X^I)-mod
$$
sending $(I\toup{p} J\to K)$ to $C_{\Sigma}=C^{\otimes K}\otimes Shv(X^I_{p, d})$. It sends a map (\ref{morphism_in_Tw_v2}) to the composition
$$
C^{\otimes K_1}\otimes Shv(X^I_{p_1, d})\to C^{\otimes K_1}\otimes Shv(X^I_{p_2, d})\to
C^{\otimes K_2}\otimes Shv(X^I_{p_2, d}), 
$$
where for $j: X^I_{p_1, d}\hook{} X^I_{p_2, d}$ the first map comes from $j_*: Shv(X^I_{p_1, d})\to Shv(X^I_{p_2, d})$, the second one from the comultiplication $\com: C^{\otimes K_1}\to C^{\otimes K_2}$. Set
$$
\bar C_{X^I, co}=\underset{\Tw(I)}{\colim} \; \bar\cF_{I, C}
$$
The subscript $co$ everywhere means that we apply some construction to a coalgebra.

\sssec{Example} For $I=\{1,2\}$, $U=X^I-X$, and the open embedding $j: U\to X^I$ one gets a cocartesian square in $Shv(X^I)-mod$
$$
\begin{array}{ccc}
C\otimes Shv(U) & \toup{j_*} & C\otimes Shv(X^I)\\
\downarrow\lefteqn{\scriptstyle \com} &&\downarrow\\
C^{\otimes 2}\otimes Shv(U) & \to & \bar C_{X^I, co}.
\end{array}
$$

\begin{Lm} For a map $f: I\to I'$ in $fSets$ one has canonically
$$
\bar C_{X^I, co}\otimes_{Shv(X^I)} Shv(X^{I'})\,\iso\, \bar C_{X^{I'}, co}
$$
in $Shv(X^{I'})-mod$ in a way compatible with compositions, so giving rise to a sheaf of categories $\ov{\Fact}^{co}(C)$ on $\Ran$ with $\ov{\Fact}^{co}(C)\otimes_{Shv(\Ran)} Shv(X^I)\,\iso\, \bar C_{X^I, co}$ for all $I\in fSets$.
\end{Lm}
\index{Sheaf of categories $\ov{\Fact}^{co}(C)$ on $\Ran$}
\begin{proof} One has canonically
$$
\bar C_{X^I, co}\otimes_{Shv(X^I)} Shv(X^{I'})\,\iso\,\underset{(I\to J\to K)\in Tw(I)}{\colim} C^{\otimes K}\otimes Shv(X^I_{p, d}\times_{X^I} X^{I'})
$$
Argue as in Lemma~\ref{Lm_4.1.2_restrictions}. Namely, $X^I_{p, d}\times_{X^I} X^{I'}$ is empty unless $J\in Q(I')$, and in the latter case it identifies with $X^{I'}_{p', d}$, where $p$ is the composition $I\toup{f} I'\toup{p'} J$. Since the embedding $\Tw(I')\subset \Tw(I)$ is zero-cofinal,
the above colimit identifies with
$$
\underset{(I'\toup{p'} J'\to K')\in \Tw(I')}{\colim} C^{\otimes K'}\otimes Shv(X^{I'}_{p', d})\,\iso\, \bar C_{X^{I'}, co}
$$
as desired. Other properties follow from the construction.
\end{proof}

\sssec{} As for any sheaf of categories over $\Ran$, we get
$$
\ov{\Fact}^{co}(C)\,\iso\, \underset{I\in fSets^{op}}{\colim} \bar C_{X^I, co}
$$
in $Shv(\Ran)-mod$. Write 
$$
\bar\cF_{fSets, C}: fSets^{op}\to Shv(\Ran)-mod
$$
for the diagram defining the latter colimit. Let $\bar\cF^R_{fSets, C}: fSets \to Shv(\Ran)-mod$ be obtained from $\bar\cF_{fSets, C}$ by passing to the right adjoints. 

\sssec{Commutative chiral coproduct} First, we define the commutative coproduct as follows. Pick a map $\phi: I\to I'$ in $fSets$. Let us define a canonical map
\begin{equation}
\label{comm_chiral_comproduct_Sect5.1.5}
\bar C_{X^I, co}\to \underset{i\in I'}{\boxtimes} \bar C_{X^{I_i}, co}
\end{equation} 

 Recall the full subcategory $\Tw(I)_{\phi}\,\iso\,\prod_{i\in I'} \Tw(I_i)$ of $\Tw(I)$ from 
Section~\ref{Sect_2.1.11_now} and the functor $\nu$ from Section~\ref{Sect_4.1.5_now}. 
The diagram (\ref{open_imm_for_Sect_4.1.3}) is functorial in $(I\to J\to K)\in \Tw(I)$. 

Consider the functor $\lambda: \Tw(I)_{\phi}\to Shv(X^I)-mod$ defined as the exterior product of the functors $\bar \cF_{I_i, C}$ for $i\in I'$. That is, it sends (\ref{collection_for_Sect_4.1.3})
to
$$
C^{\otimes (\sqcup_{i\in I'} K_i)}\otimes Shv(\prod_{i\in I'} X^{I_i}_{p_i, d})
$$
and the transition maps are given by exterior product of the transition functors for each $\bar\cF_{I_i, C}$. Then there is a natural transformation of functions from $\Tw(I)$ to $Shv(X^I)-mod$
\begin{equation}
\label{natural_transformation_for_Sect_Commutative chiral coproduct}
\bar\cF_{I, C}\to \lambda\comp\nu
\end{equation}
At the level of objects, let $(I\to J\to K)\in \Tw(I)$ and let (\ref{collection_for_Sect_4.1.3}) be its image under $\nu$. The natural transformation (\ref{natural_transformation_for_Sect_Commutative chiral coproduct}) evaluated at $(I\to J\to K)$ is the functor
$$
C^{\otimes K}\otimes Shv(X^I_{p, d})\to C^{\otimes (\sqcup_{i\in I'} K_i)}\otimes Shv(\prod_{i\in I'} X^{I_i}_{p_i, d}).
$$
The latter is the tensor product of two maps: the coproduct map $C^{\otimes K}\to C^{\otimes (\sqcup_{i\in I'} K_i)}$ along $\sqcup_{i\in I'} K_i\to K$, and the functor $j_*$ for the open embedding
$$
j: X^I_{p, d}\hook{} \prod_{i\in I'} X^{I_i}_{p_i, d}
$$
It is defined naturally on morphisms. Passing to the colimit over $\Tw(I)$,  (\ref{natural_transformation_for_Sect_Commutative chiral coproduct}) yields morphisms
$$
\bar C_{X^I, co}\,\iso\,\underset{\Tw(I)}{\colim}\; \bar\cF_{I, C}\to \underset{\Tw(I)}{\colim} \lambda\comp\nu\to \underset{\Tw(I)_{\phi}}{\colim} \lambda\,\iso\, \underset{i\in I'}{\boxtimes} \bar C_{X^{I_i}, co}
$$
in $Shv(X^I)-mod$. The composition is the desired morphism (\ref{comm_chiral_comproduct_Sect5.1.5}). 

\index{Commutative chiral coproduct on $\ov{\Fact}^{co}(C)$}

 In the case when $\phi: I\to I'$ is $\id: I\to I$ the map (\ref{comm_chiral_comproduct_Sect5.1.5}) is denoted 
$$
\Loc^{co}: \bar C_{X^I, co}\to C^{\otimes I}\otimes Shv(X^I)
$$ 
 
\begin{Lm} 
\label{Lm_5.1.6_now}
Restricting (\ref{comm_chiral_comproduct_Sect5.1.5}) to the open subscheme $X^I_{\phi, d}\subset X^I$, one gets an isomorphism
$$
\bar C_{X^I, co}\mid_{X^I_{\phi, d}}\,\iso\, (\underset{i\in I'}{\boxtimes} \bar C_{X^{I_i}, co})\mid_{X^I_{\phi, d}}
$$
So, $\ov{\Fact}^{co}(C)$ is a factorization sheaf of categories on $\Ran$. 
\end{Lm}
\begin{proof} The argument is analogous to that of Section~\ref{Sect_4.1.8}, whose notations we use. Namely, we get
$$
\bar C_{X^I, co}\mid_{X^I_{\phi, d}}\,\iso\,\underset{(I\to J\to K)\in Tw(I)}{\colim} C^{\otimes K}\otimes Shv(X^I_{\tau, d}),
$$
where $(I\toup{\tau} \sup(J, I')\to K)=j^R(I\to J\to K)$. Let $\lambda: {^{\phi}\Tw(I)}\to Shv(X^I_{\phi, d})-mod$ be the restriction of $\bar\cF_{I, C}$ to the full subcategory $^{\phi}\Tw(I)\subset \Tw(I)$. We get
$$
\bar C_{X^I, co}\mid_{X^I_{\phi, d}}\,\iso\, \underset{\Tw(I)}{\colim} \, \lambda\comp j^R
$$
Now $\lambda\comp j^R$ is the left Kan extension of $\lambda$ along $j: {^{\phi}\Tw(I)}\to \Tw(I)$. So, the latter colimit identifies with
$$
\underset{^{\phi}\Tw(I)}{\colim} \, \lambda
$$
Now as in Section~\ref{Sect_4.1.8} we get
$$
\underset{^{\phi}\Tw(I)}{\colim} \, \lambda\,\iso\, \underset{\Tw(I)_{\phi}}{\colim} \lambda\,\iso\, (\underset{i\in I'}{\boxtimes} \bar C_{X^{I_i}, co})\mid_{X^I_{\phi, d}}
$$
as desired.
\end{proof}

\sssec{Canonical arrow} Take $C(X)=C\otimes Shv(X)$. We view it as an object of $CoCAlg^{nu}(Shv(X)-mod)$. In a way dual to that of Section~\ref{Sect_4.1.10} we construct a canonical map
\begin{equation}
\label{arrow_from_ovFact^co_to_Fact^co}
\ov{\Fact}^{co}(C)\to \Fact^{co}(C)
\end{equation}
in $Shv(\Ran)-mod$. It amounts to a collection of functors
\begin{equation}
\label{functor_zeta_I^co}  
\zeta_{I, co}: \bar C_{X^I, co}\to C_{X^I, co} 
\end{equation} 
compatible with !-restrictions under $\vartriangle^{(I/J)}: X^J\to X^I$ for any map $I\to J$ in $fSets$.

\index{Canonical arrow $\ov{\Fact}^{co}(C)\to \Fact^{co}(C)$}

 Pick $(I\to J\to K), (I\toup{p_1} J_1\to K_1)\in \Tw(I)$. First, we define a functor
\begin{equation}
\label{functor_for_Sect_5.1.7}
C^{\otimes K_1}\otimes Shv(X^I_{p_1, d})\to C^{\otimes J}\otimes Shv(X^K)
\end{equation} 
as follows. Recall that $X^I_{p_1, d}\times_{X^I} X^K$ is empty unless $J_1\in Q(K)$.  In the latter case we get a diagram in $fSets$
$$
I\to J\to K\to J_1\to K_1,
$$ 
hence the comultipliation $\com: C^{\otimes K_1}\to C^{\otimes J}$ along $J\to K_1$. Then (\ref{functor_for_Sect_5.1.7}) is the composition
$$
C^{\otimes K_1}\otimes Shv(X^I_{p_1, d})\toup{\com} C^{\otimes J}\otimes Shv(X^I_{p_1, d})\to
C^{\otimes J}\otimes Shv(X^K),
$$
here the second map is the $*$-direct image along $X^I_{p_1, d}\to X^I$ followed by $!$-restriction to $X^K$. The maps (\ref{functor_for_Sect_5.1.7}) are compatible with the transition functors in the definitions of $\bar C_{X^I, co}$, $C_{X^I, co}$, so yield the desired functor (\ref{functor_zeta_I^co}).  

 One checks that $\zeta_{I, co}$ are canonically compatible with !-restrictions along $\vartriangle^{(I/J)}: X^J\to X^I$ for maps $I\to J$ in $fSets$, so yield the desired functor (\ref{arrow_from_ovFact^co_to_Fact^co}). 
 
\begin{Lm} Assume $C$ dualizable in $\DGCat_{cont}$, $I\in fSets$, $U=X^I-X$. Let $j: U\hook{} X$ be the inclusion. Then the square is cocartesian in $Shv(X^I)$
$$
\begin{array}{ccc}
\bar C_{X^I, co}\mid_U & \to & \bar C_{X^I, co}\\
\uparrow && \uparrow\\
C\otimes Shv(U) & \toup{j_*} & C\otimes Shv(X^I)
\end{array}
$$
\end{Lm}
\begin{proof}
This is an application of Corollary~\ref{Cor_2.3.3}. 

\noindent
{\bf Step 1}.
Let us show that the natural functor
$$
\underset{(I\toup{p} J\to *)\in \Tw(I)^{>1}\cap \Tw(I)^f}{\colim} Shv(X^I_{p, d})\to Shv(U)
$$
is an equivalence, where the transition functors in the latter diagram are given by the $*$-direct images. This follows from Zarizki descent for sheaves of categories over $U$. Namely, for any $(I\toup{p} J\to *)\in \Tw(I)^{>1}\cap \Tw(I)^f$ over the open subset $X^I_{p, d}$ the above arrow becomes an equivalence, hence it is an equivalence.

\medskip
\noindent
{\bf Step 2}. It remains to show that $\underset{\Tw(I)^{>1}}\colim \bar\cF_{I, C}\,\iso\, \bar C_{X^I, co}\mid_U$. We claim that the functor $\bar\cF_{I,C}\mid_U: \Tw(I)\to Shv(U)-mod$ is the left Kan extension of its restriction to $Tw(I)^{>1}$. Indeed, the only object of $Tw(I)$ not lying in $Tw(I)^{>1}$ is $\Sigma=(I\to *\to *)$, and 
$$
\bar\cF_{I, C}(\Sigma)\mid_U\,\iso\, Shv(U)\otimes C
$$
On the other hand, 
$$
\underset{\Tw(I)^{>1}\times_{\Tw(I)} \Tw(I)_{/\Sigma}}{\colim} \bar\cF_{I, C}\mid_U\,\iso\, C\otimes Shv(U)
$$
by Step 1.
\end{proof} 
 
\begin{Rem} 
\label{Rem_5.1.9}
If $\com: C\to C^{\otimes 2}$ admits a left adjoint $\com^L$ then $(C, \com^L)\in CAlg^{nu}(\DGCat_{cont})$ naturally. In this case we may pass to left adjoints in the functor $\bar\cF_{I, C}$ and get a functor $\bar\cF_{I, C}^L: \Tw(I)^{op}\to Shv(X^I)-mod$. Then we have a canonically 
$$
\bar\cF_{I, C}^L\,\iso\,\bar \cG_{I, C}
$$
Note that in this case by Remark~\ref{Rem_3.1.11} we have $\Fact(C)\,\iso\,\Fact^{co}(C)$. 
\end{Rem}

\sssec{} 
\label{Sect_5.1.10}
For the rest of Section~\ref{Sect_5} assume that $\com: C\to C^{\otimes 2}$ admits a left adjoint $\com^L$, and $C$ dualizable in $\DGCat_{cont}$.

 View $C^{\vee}$ as an object of $CAlg^{nu}(\DGCat_{cont})$ with the product $\com^{\vee}: (C^{\vee})^{\otimes 2}\to C^{\vee}$. By Remark~\ref{Rem_4.1.15_now}, $\bar C_{X^I, co}$ and $\ov{C^{\vee}}_{X^I}$ are canonically in duality in $Shv(X^I)-mod$.
Besides, $\ov{\Fact}(C^{\vee})$ and $\ov{\Fact}^{co}(C)$ are canonically in duality in $Shv(\Ran)-mod$. Moreover, by Proposition~\ref{Pp_2.4.7}, $\Fact(C)$ is dualizable in $Shv(\Ran)-mod$, and the functor (\ref{arrow_from_ovFact^co_to_Fact^co}) is dual to the functor
$$
\Fact(C^{\vee})\to \ov{\Fact}(C^{\vee})
$$
given by (\ref{map_from_Fact_to_ov_Fact}) for the category $(C^{\vee}, com^{\vee})\in CAlg^{nu}(\DGCat_{cont})$.

\begin{Thm} 
\label{Thm_5.1.11}
Assume $C\in CoCAlg(\DGCat_{cont})$ such that $\com: C\to C^{\otimes 2}$ has a left adjoint $\com^L: C^{\otimes 2}\to C$, and the counit $C\toup{cou} \Vect$ has a left adjoint $cou^L: \Vect\to C$. Assume $C$ compactly generated. Then one has the diagram of equivalences
$$
\ov{\Fact}^{co}(C)\toup{(\ref{arrow_from_ovFact^co_to_Fact^co})} \Fact^{co}(C)\,\iso\, \Fact(C)\toup{(\ref{map_from_Fact_to_ov_Fact})} \ov{\Fact}(C)
$$
Here the middle equivalence is that of Remark~\ref{Rem_5.1.9}. Here we also view $C$ as an object of $CAlg(\DGCat_{cont})$ with the multiplication $\com^L$ and unit $cou^L$.
\end{Thm}
\begin{proof}
The third functor is an equivalence by Proposition~\ref{Pp_4.1.19}. It is essential here that $C\in CAlg(\DGCat_{cont})$ is unital. 

Note that $C^{\vee}\in CAlg(\DGCat_{cont})$, that is, $C^{\vee}$ is unital. Applying Proposition~\ref{Pp_4.1.19} to $C^{\vee}$ and combining with Section~\ref{Sect_5.1.10}, we see that the first functor is an equivalence. 
\end{proof}

\section{Spreading right-lax symmetric monoidal functors}
\label{Sect_Spreading right-lax symmetric monoidal functors}

\ssec{} The purpose of this section is to generalize a construction of Gaitsgory from \cite{Gai19Ran}, Section~2.6). The construction comes in a packet of basic properties of the sheaves of categories $\Fact(C)$ for $C\in CAlg^{nu}(Shv(X)-mod)$. We expose it in a separate section as it is logically independent of our main Theorem~\ref{Thm_5.1.11}.

 Informally speaking, it allows to spread along the $\Ran$ space certain right-lax symmetric monoidal functors given originally over $X$. It originated from a particular case when $C$ is the category of representations, which explains our related notations and terminology.
 
 We apply it to generalize a result of D.~Gaitsgory (\cite{Gai19Ran}, Proposition~5.4.7), which allowed in \select{loc.cit.} to present the Ran space version of the semi-infinite $\IC$-sheaf $\IC^{\frac{\infty}{2}}_{\Ran}$ as certain colimit. This generalization is used in our paper \cite{DL2}.
 
\sssec{} 
\label{Sect_1.12.1_spreading_begins}
Let $\Lambda\in CAlg^{nu}(Sets)$. We think of it as the "monoid of highest weights". For $\lambda_i\in\Lambda$ the operation in $\Lambda$ is denoted $\lambda_1+\lambda_2$. If in addition $\Lambda$ is unital, the neutral object is denoted $0\in\Lambda$. 

Let $(C(X),\otimes^!)\in CAlg^{nu}(Shv(X)-mod)$. Assume given a right-lax non-unital symmetric monoidal functor 
$$
\Irr: \Lambda\to (C(X),\otimes^!), \; \lambda\mapsto V^{\lambda}
$$ 
We write $\Irr_C$, $V^{\lambda}_C$ if we need to express the dependence on $C(X)$. We think of $V^{\lambda}$ as an analog of the irreducible representation of highest weight $\lambda$. 

For $\lambda_i\in\Lambda$ we have by definition a canonical map in $C(X)$
$$
u: V^{\lambda_1}\otimes^! V^{\lambda_2}\to V^{\lambda_1+\lambda_2}.
$$ 

\sssec{} Let $I\in fSets$ and $\und{\lambda}: I\to \Lambda$ be a map. Define the object $V^{\und{\lambda}}\in C_{X^I}$ as follows. We think of it as a "spread representation of highest weight $\und{\lambda}$".

\index{Spread representation of highest weight $\und{\lambda}$}

 Recall that $C^{\otimes J}(X)$ is the tensor power of $C(X)$ in $Shv(X)-mod$. 
Given $V_j\in C(X)$, our notation $\underset{j\in J}{\otimes} V_j$ refers to the corresponding object of $C^{\otimes J}(X)$, not to be confused with 
$
\underset{j\in J}{\otimes^!} V_j\in C(X)
$.   

 First, define a functor 
$$
\cF_{\und{\lambda}, \Irr}: Tw(I)\to C_{X^I}
$$ 
as follows. It sends $(I\to J\toup{\phi} K)$ to the image of 
$$
V^{\otimes \phi}:=\underset{k\in K}{\boxtimes} (\underset{j\in J_k}{\otimes} V^{\lambda_j})
$$
under $\underset{k\in K}{\boxtimes} (C^{\otimes J_k}(X))\to C_{X^I}$, where $\lambda_j=\underset{i\in I_j}{\sum} \und{\lambda}(i)$ for $j\in J$. Given a map (\ref{morphism_in_Tw_v2}) in $Tw(I)$, we get the corresponding transition morphism in $\underset{k\in K_2}{\boxtimes} C^{\otimes (J_2)_k}(X)$ and hence in $C_{X^I}$ 
$$
\cF_{\und{\lambda}, \Irr}(I\to J_1\to K_1)\to \cF_{\und{\lambda}, \Irr}(I\to J_2\to K_2)
$$
as follows. First, for the diagram (defining the transition functor for $\cF_{I, C}$)
$$
\underset{k\in K_1}{\boxtimes} C^{\otimes (J_1)_k}(X)\toup{m} \underset{k\in K_1}{\boxtimes} C^{\otimes (J_2)_k}(X)\to \underset{k\in K_2}{\boxtimes} C^{\otimes (J_2)_k}(X)
$$
we get a natural map
\begin{equation}
\label{diag_for_Sect_1.11.23_second}
m(\underset{k\in K_1}{\boxtimes} (\underset{j\in (J_1)_k}{\otimes} V^{\lambda_j}))\,\iso\,
\underset{k\in K_1}{\boxtimes} (\underset{j\in (J_2)_k}{\otimes} (\underset{j'\in (J_1)_j}
{\otimes^!} V^{\lambda_{j'}}))\to \underset{k\in K_1}{\boxtimes} (\underset{j\in (J_2)_k}{\otimes} V^{\lambda_j})
\end{equation}
in $\underset{k\in K_1}{\boxtimes} C^{\otimes (J_2)_k}(X)$. Here the second arrow in (\ref{diag_for_Sect_1.11.23_second}) comes from the right-lax structure on $\Irr$ giving the maps 
$$
\underset{j'\in (J_1)_j}{\otimes^!} V^{\lambda_{j'}}\to V^{\lambda_j}
$$
for $j\in J_2$ with $\lambda_j=\underset{j'\in (J_1)_j}{\sum} \lambda_{j'}
$.
 Further, for $\vartriangle: X^{K_1}\to X^{K_2}$ we have 
$$
\vartriangle^!(\underset{k\in K_2}{\boxtimes} (\underset{j\in (J_2)_k}{\otimes} V^{\lambda_j}))\,\iso\, \underset{k\in K_1}{\boxtimes} (\underset{j\in (J_2)_k}{\otimes} V^{\lambda_j})
$$ 
So, we compose the previous map with
$$
\vartriangle_! \underset{k\in K_1}{\boxtimes} (\underset{j\in (J_2)_k}{\otimes} V^{\lambda_j})\,\iso\, \vartriangle_! \vartriangle^!(\underset{k\in K_2}{\boxtimes} (\underset{j\in (J_2)_k}{\otimes} V^{\lambda_j}))\to \underset{k\in K_2}{\boxtimes} (\underset{j\in (J_2)_k}{\otimes} V^{\lambda_j})
$$
 This concludes the definition of $\cF_{\und{\lambda}, \Irr}$. We write $\cF_{\und{\lambda}, \Irr_C}$ if we need to express the dependence on $C(X)$.  

Set  
$$
V^{\und{\lambda}}=\underset{Tw(I)}{\colim}\; \cF_{\und{\lambda}, \Irr}
$$
in $C_{X^I}$. 

\begin{Rem} In (\cite{Gai19Ran}, Section~2.6) the above objects $V^{\und{\lambda}}$ appeared in the following special case: $\Lambda$ is the set of dominant weights of a reductive group $\check{G}$ defined over $e$, $C(X)=\Rep(\check{G})\otimes Shv(X)$, and $\Irr: \Lambda\to C(X)$ sends $\lambda$ to an irreducible representation of $\check{G}$ of highest weight $\lambda$ tensored by $\omega_X$. 
\end{Rem}

\begin{Lm} Let $I\in fSets$. The object $V^{\und{\lambda}}\in C_{X^I}$ factorizes naturally. Namely, if $\phi: I\to I'$ is a map in $fSets$ then one has canonically
$$
V^{\und{\lambda}}\mid_{X^I_{\phi, d}}\,\iso\, (\underset{i\in I'}{\boxtimes} V^{\und{\lambda}^i})\mid_{X^I_{\phi, d}}
$$
in $(C_{X^I})\mid_{X^I_{\phi, d}}$.
Here for $i\in I'$, $\und{\lambda}^i: I_i\to \Lambda$ is the restriction of $\und{\lambda}$. 
\end{Lm}
\begin{proof}
Recall the full subcategory $\Tw(I)_{\phi}\subset \Tw(I)$ from Section~\ref{Sect_2.1.11_now}. The latter inclusion is zero-cofinal. For $(I\to J\to K)\in \Tw(I)$, $X^K\times_{X^I} X^I_{\phi, d}$ is empty unless $I'\in Q(K)$, so 
$V^{\und{\lambda}}\mid_{X^I_{\phi, d}}$ identifies with
$$
\underset{(I\to J\to K)\in \Tw(I)_{\phi}}{\colim} \underset{i'\in I'}{\boxtimes}(\underset{k\in K_{i'}}{\boxtimes} \underset{j\in (J_{i'})_k}{\otimes} V^{\lambda_j})\mid_{X^I_{\phi, d}}, 
$$
which identifies with the RHS. 
\end{proof}

\sssec{} Write $\Fun^{rlax}(\Lambda, C(X))$ for the category of right-lax non-unital symmetric monoidal functors. Let $I\in fSets$, $\und{\lambda}: I\to \Lambda$ as above. Then the assignment $\Irr\mapsto \cF_{\und{\lambda}, \Irr}$ defines a functor
$$
\Fun^{rlax}(\Lambda, C(X))\to \Fun(\Tw(I), C_{X^I})
$$
In particular, $V^{\und{\lambda}}$ is functorial in $\Irr\in \Fun^{rlax}(\Lambda, C(X))$. 

 If $\Lambda$ and $C(X)$ are unital we denote by $\Fun^{rlaxu}(\Lambda, C(X))\subset \Fun^{rlax}(\Lambda, C(X))$ the full subcategory of right-lax symmetric monoidal functors. These are functors $f\in \Fun^{rlax}(\Lambda, C(X))$ such that the natural map $1_{C(X)}\to f(0)$ is an isomorphism. 

\sssec{} For $I\in fSets$ view $\Map(I, \Lambda)$ as an object of $CAlg^{nu}(Sets)$ with the pointwise operation. Consider the functor 
$$
fSets\to CAlg^{nu}(Sets), \; I\mapsto \Map(I, \Lambda).
$$ 
It sends $f: I\to I'$ in $fSets$ to the direct image map $f_*: \Map(I, \Lambda)\to\Map(I',\Lambda)$. That is, $(f_*\und{\lambda})(i')=\underset{i\in I_{i'}}{\sum} \und{\lambda}(i)$.
Let $\ov{\Lambda}=\underset{I\in fSets}{\lim} \Map(I,\Lambda)$ be the limit of this functor in 
$CAlg^{nu}(Sets)$ or, equivalently, in $Sets$ or in $\Spc$. 

\sssec{} As in Section~\ref{Sect_2.1.7_now}, one shows the following. If $\phi: I\to I'$ is a map in $fSets$ then for $\vartriangle: X^{I'}\to X^I$ under the equivalence $C_{X^I}\otimes_{Shv(X^I)} Shv(X^{I'})\,\iso\, C_{X^{I'}}$ one has canonically
$$
\vartriangle^! V^{\und{\lambda}}\,\iso\ V^{\und{\lambda'}}
$$
in $C_{X^{I'}}$, where $\und{\lambda'}: I'\to\Lambda$ is given by $\und{\lambda'}=\phi_*\und{\lambda}$. 

 Moreover, for $\ov{\lambda}=\{\und{\lambda}_I\}_{I\in fSets}\in \ov{\Lambda}$ with $\und{\lambda}_I\in \Map(I, \Lambda)$, the collection $\{V^{\und{\lambda}_I}\}_{I\in fSets}$ is equipped with a coherent system of higher compatibilites for the above restrictions, so defines an object of $\underset{I\in fSets}{\lim} C_{X^I}\,\iso\,\Fact(C)$. We denote this section by 
$$
V^{\ov{\lambda}}\in \Fact(C)
$$ 

\sssec{} There is also a direct definition of $V^{\ov{\lambda}}$ as in Section~\ref{Sect_2.1.8_v4}. Namely, one defines a functor 
$$
\cF_{\ov{\lambda}, \Irr}: \cTw(fSets)\to \Fact(C)
$$
sending $(J\to K)\in Tw(fSets)$ to the image of 
$$
\underset{k\in K}{\boxtimes} (\underset{j\in J_k}{\otimes} V^{\und{\lambda}_J(j)}).
$$ 
under $\underset{k\in K}{\boxtimes} (C^{\otimes J_k}(X))\to \Fact(C)$. The transition morphisms are defined in the same way as for $\cF_{\und{\lambda}, \Irr}$. One has canonically
$$
V^{\ov{\lambda}}\,\iso\, \underset{Tw(fSets)}{\colim} \cF_{\ov{\lambda}, \Irr}
$$
in $\Fact(C)$. 

\sssec{Example} Assume for a moment that $\Lambda$ is unital, $C(X)=Shv(X)$, 
and $\Irr$ is right-lax symmetric monoidal (so, preserves units). Let $I\in fSets$ and $\und{\lambda}: I\to \Lambda$ be the constant zero map denoted $\und{0}$. Then $V^{\und{0}}\;\iso\, \omega_{X^I}$. 

\sssec{} 
\label{Sect_1.12.3_images}
Let $C(X)\to D(X)$ be a map in $CAlg^{nu}(Shv(X)-mod)$ and $I\in fSets, \und{\lambda}\in\Map(I,\Lambda)$. Let $\Irr_D$ be the composition $\Lambda\toup{\Irr_C} C(X)\to D(X)$. We get the functor $\cF_{\und{\lambda}, \Irr_D}: Tw(I)\to D_{X^I}$ as above and the corresponding objects 
$$
V^{\und{\lambda}}_D=\underset{Tw(I)}{\colim} \; \cF_{\und{\lambda}, \Irr_D} \in D_{X^I}.
$$ 
Then the image of $V^{\und{\lambda}}_C$ under $C_{X^I}\to D_{X^I}$ identifies canonically with $V^{\und{\lambda}}_D$. 

\sssec{Example} If $\Lambda$ is unital, $C(X)$ is unital, and $\Irr$ is right-lax symmetric monoidal then $V^{\und{0}}\in C_{X^I}$ is the unit of $(C_{X^I}, \otimes^!)$. 

\sssec{} Let us be in the situation of Section~\ref{Sect_1.12.1_spreading_begins}.
Since $\Map(I,\Lambda)$ is a set, we defined a functor 
\begin{equation}
\label{functor_first_for_Sect_6.1.11}
\Map(I,\Lambda)\to (C_{X^I}, \otimes^!), \; \und{\lambda}\mapsto V^{\und{\lambda}}
\end{equation} 
Let us equip this functor with a right-lax non-unital symmetric monoidal structure.
 
 Let $\und{\lambda},\und{\nu}\in \Map(I,\Lambda)$. We define a map 
\begin{equation}
\label{map_rlax_structure_for_Sect_6.1.11}
V^{\und{\lambda}}\otimes^! V^{\und{\nu}}\to V^{\und{\lambda}+\und{\nu}}
\end{equation}
in $C_{X^I}$ as follows. Write $\cF_{\und{\lambda}, \Irr}\otimes \cF_{\und{\nu}, \Irr}: \Tw(I)\to (C\otimes C)_{X^I}$ for the functor
$$
\Tw(I)\toup{\cF_{\und{\lambda}, \Irr}\times \cF_{\und{\nu}, \Irr}} \; C_{X^I}\times C_{X^I}\toup{\otimes} C_{X^I}\otimes_{Shv(X^I)} C_{X^I}
$$
Arguing as in Sections~\ref{Sect_2.2.11_v4}-\ref{Sect_2.2.15_v4}, 
the object 
$$
V^{\und{\lambda}}\otimes V^{\und{\nu}}\in
C_{X^I}\otimes_{Shv(X^I)} C_{X^I}\,\iso\, (C\otimes C)_{X^I}
$$
is identified with the colimit
$$
\underset{(I\to J\to K)\in\Tw(I)}{\colim} \;\underset{k\in K}{\boxtimes} (\underset{j\in J_k}{\otimes} (V^{\lambda_j}\otimes V^{\nu_j}))
$$
calculated in $(C\otimes C)_{X^I}$ of the functor $\cF_{\und{\lambda}, \Irr}\otimes \cF_{\und{\nu}, \Irr}$. The image of this object under 
\begin{equation}
\label{map_for_Sect_6.1.11_v4}
(C\otimes C)_{X^I}\toup{\otimes^!_{X^I}} C_{X^I}
\end{equation}
identifies with
$$
\underset{(I\to J\to K)\in\Tw(I)}{\colim} \;\underset{k\in K}{\boxtimes} (\underset{j\in J}{\otimes} (V^{\lambda_j}\otimes^! V^{\nu_j}))
$$
This is the colimit of the functor $\cF_{\und{\lambda}, \Irr}\otimes^! \cF_{\und{\nu}, \Irr}$ obtained from $\cF_{\und{\lambda}, \Irr}\otimes \cF_{\und{\nu}, \Irr}$ by composing with (\ref{map_for_Sect_6.1.11_v4}).  

Moreover, the maps $u: V^{\lambda_j}\otimes^! V^{\nu_j}\to V^{\lambda_j+\nu_j}$ organize into a morphism of functors 
$$
\cF_{\und{\lambda}, \Irr}\otimes^! \cF_{\und{\nu}, \Irr}\to \cF_{\und{\lambda}+\und{\nu}, \Irr}
$$ 
in $\Fun(Tw(I), C_{X^I})$. Passing to the colimit over $\Tw(I)$ this gives the desired map (\ref{map_rlax_structure_for_Sect_6.1.11}). 

\begin{Rem} If $\Lambda$ is unital, $(C(X),\otimes^!)\in CAlg(Shv(X)-mod)$ and $\Irr$ is right-lax symmetric monoidal then (\ref{functor_first_for_Sect_6.1.11}) is unital.
\end{Rem}

\begin{Rem} The right-lax monoidal structure on (\ref{functor_first_for_Sect_6.1.11}) is compatible with restrictions. That is, for $f: I\to I'$ in $fSets$ the $!$-restriction of (\ref{map_rlax_structure_for_Sect_6.1.11}) under $X^{I'}\to X^I$ is canonically identified with the corresponding map
$$
V^{\und{\lambda}'}\otimes^! V^{\und{\nu}'}\to V^{\und{\lambda}'+\und{\nu}'}
$$
with $\und{\lambda}'=f_*\und{\lambda}$ and $\und{\nu}'=f_*\und{\nu}$. 
\end{Rem}

\sssec{} 
\label{Sect_6.1.5_v19}
Equip $\Fun^{rlax}(\Lambda, C(X))$ with a non-unital symmetric monoidal structure as follows. For 
$$
\Irr^V, \Irr^W\in \Fun^{rlax}(\Lambda, C(X)), \lambda\mapsto V^{\lambda}, \lambda\mapsto W^{\lambda}
$$
we define the functor $\Irr^{V}\otimes^! \Irr^W\in \Fun^{rlax}(\Lambda, C(X))$ as the composition
$$
\Lambda\toup{\Irr^V\times \Irr^W} C(X)\times C(X)\toup{\otimes} C^{\otimes 2}(X)\toup{\otimes^!} C(X)
$$
It is naturally equipped with a right-lax non-unital symmetric monoidal structure. 

Let now $\und{\lambda}, \und{\nu}\in \Map(I, \Lambda)$. Denote by $\cF_{\und{\lambda}, \Irr^V}\otimes \cF_{\und{\nu}, \Irr^W}$ the composition
$$
\Tw(I)\toup{\cF_{\und{\lambda}, \Irr^V}\times \cF_{\und{\nu}, \Irr^ W}} C_{X^I}\times C_{X^I}\toup{\otimes} C_{X^I}\otimes_{Shv(X^I)} C_{X^I}
$$
Arguing as in Sections~\ref{Sect_2.2.11_v4}-\ref{Sect_2.2.15_v4}, the object
$$
V^{\und{\lambda}}\otimes W^{\und{\nu}}\in C_{X^I}\otimes_{Shv(X^I)} C_{X^I}\,\iso\, (C\otimes C)_{X^I}
$$
is identified with the colimit
$$
\underset{(I\to J\to K)\in \Tw(I)}{\colim} \underset{k\in K}{\boxtimes} (\underset{j\in J_k}{
\otimes} (V^{\lambda_j}\otimes W^{\nu_j}))
$$
in $(C\otimes C)_{X^I}$ of the functor $\cF_{\und{\lambda}, \Irr^V}\otimes \cF_{\und{\nu}, \Irr^ W}$. 

 Write $\cF_{\und{\lambda}, \Irr^V}\otimes^! \cF_{\und{\nu}, \Irr^W}$ for the composition of $\cF_{\und{\lambda}, \Irr^V}\otimes \cF_{\und{\nu}, \Irr^W}$ with 
$$
\otimes^!_{X^I}: (C\otimes C)_{X^I}\to C_{X^I}.
$$ 

 Then $V^{\und{\lambda}}\otimes^! W^{\und{\nu}}\in C_{X^I}$ identifies canonically with the colimit
$$
\underset{(I\to J\to K)\in \Tw(I)}{\colim} \underset{k\in K}{\boxtimes} (\underset{j\in J_k}{
\otimes} (V^{\lambda_j}\otimes^! W^{\nu_j}))
$$
in $C_{X^I}$ of the functor $\cF_{\und{\lambda}, \Irr^V}\otimes^! \cF_{\und{\nu}, \Irr^W}$. 
Note that 
$$
\cF_{\und{\lambda}, \Irr^V}\otimes^! \cF_{\und{\lambda}, \Irr^W}\,\iso\, \cF_{\und{\lambda}, \Irr^V\otimes^! \Irr^W}.
$$ 
Thus, we get canonically
$$
V^{\und{\lambda}}\otimes^! W^{\und{\lambda}}\in C_{X^I}\,\iso\, 
\underset{Tw(I)}{\colim} \, \cF_{\und{\lambda}, \Irr^V\otimes^! \Irr^W}
$$
We have established the following.

\begin{Lm} 
\label{Lm_6.1.16_v4}
For any $\und{\lambda}: I\to \Lambda$ the functor 
\begin{equation}
\label{functor_for_Lm_6.1.16}
\Fun^{rlax}(\Lambda, C(X))\to (C_{X^I},\otimes^!), \; \Irr\mapsto \underset{Tw(I)}{\colim} \, \cF_{\und{\lambda}, \Irr}
\end{equation}
is non-unital symmetric monoidal. \QED
\end{Lm}

\sssec{} 
\label{Sect_6.1.17_v4}
Let us be in the situation of Section~\ref{Sect_1.12.1_spreading_begins}. Assume in addition that $A\in CAlg^{nu}(C(X))$. Denote by $\cF_{I, A}\otimes \cF_{\und{\lambda}, \Irr}: \Tw(I)\to (C\otimes C)_{X^I}$ the composition
$$
\Tw(I)\toup{\cF_{I, A}\times \cF_{\und{\lambda}, \Irr}} C_{X^I}\times C_{X^I}\toup{\otimes} C_{X^I}\otimes_{Shv(X^I)} C_{X^I}
$$
Arguing as in Sections~\ref{Sect_2.2.11_v4}-\ref{Sect_2.2.15_v4}, the object
$A_{X^I}\otimes V^{\und{\lambda}}\in (C\otimes C)_{X^I}$ is identified with
$$
\underset{\Tw(I)}{\colim} \; \cF_{I, A}\otimes \cF_{\und{\lambda}, \Irr}
$$ 
Further applying $\otimes^!_{X^I}: (C\otimes C)_{X^I}\to C_{X^I}$, we get canonically
$$
A_{X^I}\otimes^! V^{\und{\lambda}}\;\iso\, \underset{(I\to J\to K)\in \Tw(I)}{\colim}\, \underset{k\in K}{\boxtimes} (A^{\otimes J_k}\otimes^! (\underset{j\in J_k}{\otimes} V^{\lambda_j})),
$$
the colimit being calculated in $C_{X^I}$. In the RHS of the latter formula the operation $\otimes^!$ stands for the corresponding product in $C^{\otimes J_k}(X)$. 

\sssec{}  Let $C(X), D(X)\in CAlg^{nu}(Shv(X)-mod)$.  Let $\Irr_C: \Lambda\to (C(X), \otimes^!)$, $\lambda\mapsto V^{\lambda}$ and $\Irr_D: \Lambda\to (D(X),\otimes^!)$, $\lambda\mapsto U^{\lambda}$ be right-lax non-unital symmetric monoidal functors. Let $I\in fSets$ and $\und{\lambda}\in\Map(I, \Lambda)$. So, we get 
$$
V^{\und{\lambda}}=\underset{\Tw(I)}{\colim} \; \cF_{\und{\lambda},\Irr_C} \in C_{X^I},\;\;\; U^{\und{\lambda}}=\underset{\Tw(I)}{\colim} \; \cF_{\und{\lambda},\Irr_D}\in D_{X^I}.
$$ 
 
  Denote by $\cF_{\und{\lambda},\Irr_C}\otimes \cF_{\und{\lambda},\Irr_D}$ the composition
$$
\Tw(I)\;\toup{\cF_{\und{\lambda},\Irr_C}\times \cF_{\und{\lambda},\Irr_D}} \; C_{X^I}\times D_{X^I}\toup{\otimes} C_{X^I}\otimes_{Shv(X^I)} D_{X^I}\,\iso\, (C\otimes D)_{X^I}
$$
Then as above we get $\underset{\Tw(I)}{\colim} \; \cF_{\und{\lambda},\Irr_C}\otimes \cF_{\und{\lambda},\Irr_D}\,\iso\, V^{\und{\lambda}}\otimes U^{\und{\lambda}}$.  

\sssec{} If both $\Lambda$ and $C(X)$ are unital then the full subcategory $\Fun^{rlaxu}(\Lambda, C(X))\subset \Fun^{rlax}(\Lambda, C(X))$ inherits a symmetric monoidal structure. 

\ssec{Spreading colimit formulas}
\label{Sect_6.2_v4}

\sssec{} 
In Section~\ref{Sect_6.2_v4} we propose a setting to put in a perspective a result of D.~Gaitsgory (\cite{Gai19Ran}, Proposition~5.4.7). It was used to present the Ran space version of the semi-infinite $\IC$-sheaf $\IC^{\frac{\infty}{2}}_{\Ran}$ in \select{loc.cit.} as certain colimit.

\sssec{}
\label{Sect_6.1.18_v4}
Keep the assumptions of Section~\ref{Sect_6.1.17_v4}. Assume also $C(X)\in CAlg(Shv(X)-mod)$, that is, $C(X)$ is unital, $\Lambda$ is unital, $\Irr: \Lambda\to (C(X), \otimes^!)$ is right-lax symmetric monoidal (so, preserves units), and $A\in CAlg(C(X))$. 
Note that the functor (\ref{functor_for_Lm_6.1.16}) preserves units, so is symmetric monoidal.
 
 Assume given also a symmetric monoidal functor
\begin{equation}
\label{functor_V^lambda_dual} 
\Irr^*: \lambda\mapsto (V^{\lambda})^*
\end{equation}
and a morphism
\begin{equation}
\label{map_unit_for_V_lambda}
1_{C(X)}\to V^{\lambda}\otimes^! (V^{\lambda})^*
\end{equation}
functorial in $\lambda\in \Lambda$, which gives a 1-morphism 
in $\Fun^{rlaxu}(\Lambda, (C(X),\otimes^!))$. We assume that for each $\lambda\in\Lambda$ the map (\ref{map_unit_for_V_lambda}) extends to a duality datum on the pair $(V^{\lambda}, (V^{\lambda})^*)$ in $(C(X),\otimes^!)$. So, we get the counit map $V^{\lambda}\otimes^!(V^{\lambda})^*\to 1_{C(X)}$, it is a morphism in $\Fun(\Lambda, C(X))$. However, we do not require the latter functor to be a morphism in $\Fun^{rlaxu}(\Lambda, C(X))$. 

 So, under our assumptions, $\Irr$ is not necessarily dualizable in $\Fun^{rlaxu}(\Lambda, C(X))$. 

For $I\in fSets$, $\und{\lambda}\in \Map(I, \Lambda)$ we then set 
$$
(V^{\und{\lambda}})^*=\underset{\Tw(I)}{\colim}\, \cF_{\und{\lambda}, \Irr^*}.
$$ 
Using Section~\ref{Sect_6.1.5_v19}, (\ref{map_unit_for_V_lambda}) yields a canonical map in $C_{X^I}$
$$
1_{C_{X^I}}\to V^{\und{\lambda}}\otimes^! (V^{\und{\lambda}})^*.
$$

 Equip $\Lambda$ with a relation: $\lambda_1\le\lambda_2$ iff there is $\lambda\in\Lambda$ with $\lambda_2=\lambda_1+\lambda$. This is not an order relation in general, but it makes $\Lambda$ into a filtered category.
 
  View the constant map $\Lambda\to C(X)$, $\lambda\mapsto A$ as an object of $\Fun^{rlax}(\Lambda, C(X))$, where the right lax non-unital symmetric monoidal structure comes from the algebra structure on $A$. Assume given morphisms 
$$
\tau^{\lambda}: V^{\lambda}\to A,
$$
for $\lambda\in\Lambda$, which are promoted to a map in $\Fun^{rlax}(\Lambda, C(X))$. In other words, for $\lambda,\nu\in\Lambda$ we are given a commutativity datum for the diagrams
$$
\begin{array}{ccc}
V^{\lambda}\otimes^! V^{\nu} &\toup{\tau^{\lambda}\otimes^!\tau^{\nu}} & A\otimes^!A\\
\downarrow && \downarrow\lefteqn{\scriptstyle m_A}\\
V^{\lambda+\mu} & \toup{\tau^{\lambda+\nu}} & A
\end{array}  
$$
together with a coherent system of higher compatibilities. Assume also $\tau^0: V^0\to A$ is the unit of $A$.\footnote{Examples of such functors $\tau$ appear in \cite{Gai19Ran, DL}}. 

Set $\cA=\underset{\lambda\in\Lambda}{\oplus} V^{\lambda}=\underset{\Lambda}{\colim} \Irr$ taken in $C(X)$. By Remark~\ref{Rem_6.2.4_v4} below, $\cA\in CAlg(C(X))$. The product in $\cA$ is given by the maps $V^{\lambda}\otimes^! V^{\mu}\to V^{\lambda+\mu}$ coming from the right lax structure on the functor $\Irr$. Then our assumption on $\tau^{\lambda}$ means that the map $\cA\to A$ (equal to $\tau^{\lambda}$ on $V^{\lambda}\to \cA$) is a map in $CAlg(C(X))$. 

\sssec{} For $E\in C(X)-mod$ we may consider $\cA-mod(E)\in\DGCat_{cont}$. This is the category of objects $e\in E$ together with morphisms $b^{\lambda}:V^{\lambda}\ast e\to e$ for $\lambda\in\Lambda$ such that:
\begin{itemize} 
\item $b^0: V^0\ast e\to e$ is identified with the identity map;
\item For $\lambda,\mu\in \Lambda$ we are given a commutativity datum for the diagram
$$
\begin{array}{ccc}
V^{\lambda}\ast (V^{\mu}\ast e) & \toup{b^{\mu}} & V^{\lambda}\ast e\\
\downarrow && \downarrow\lefteqn{\scriptstyle b^{\lambda}}\\
V^{\lambda+\mu}\ast e & \toup{b^{\lambda+\mu}} & e.
\end{array}
$$
Here the left vertical arrow is given by the lax structure on $\Irr$. 

\item a system of higher compatibilities is given.  
\end{itemize}

 The theory of lax central objects is developed in (\cite{DL}, Appendix~C), its informal version first appeared (\cite{Gai19SI}, 2.7.1). So, $\cA-mod(E)$ is the category of lax central objects $\Cent^{lax,\ra}_{\Lambda}(E)$ in the sense of (\cite{DL}, Appendix~C). Here the left lax action of $\Lambda$ on $E$ is given by $z\mapsto V^{\lambda}\ast z$, the lax structure being given by the right lax structure on the functor $\Irr$, and the right $\Lambda$-action on $E$ is trivial. 
 
\index{Lax central objects}
 
\sssec{} Since $(C(X),\otimes^!)$ is symmetric monoidal, we may equivalently write its action on $e\in E$ on the right. Then $\cA-mod(E)$ can equivalently be described as the category of objects $e\in E$ together with morphisms $a^{\lambda}: e\to e\ast (V^{\lambda})^*$ for $\lambda\in\Lambda$ such that:
\begin{itemize}
\item $a^0$ is the identity; recall that $V^0$ is the unit of $C(X)$.
\item For $\lambda,\mu\in \Lambda$ the diagram commutes
\begin{equation}
\label{diag_two}
\begin{array}{ccc}
e&\toup{a^{\lambda}} & e\ast (V^{\lambda})^*\\
\downarrow\lefteqn{\scriptstyle a^{\lambda+\mu}} && \downarrow\lefteqn{\scriptstyle a^{\mu}}\\
e\ast (V^{\lambda+\mu})^* & \gets & e\ast (V^{\mu})^*\otimes (V^{\lambda})^*.
\end{array}
\end{equation}
The low horizontal arrow is given by the right lax structure on the functor $\Irr^*$;
\item a system of higher compatibilities is given.  
\end{itemize}

 The fact that the first description implies the second one follows from our assumption on (\ref{map_unit_for_V_lambda}), namely the commutativity for $\lambda,\mu\in\Lambda$ of the diagram
$$
\begin{array}{ccc}
1_C & \toup{u\otimes u} & (V^{\mu})^*\otimes V^{\mu}\otimes (V^{\lambda})^*\otimes V^{\lambda}\\
& \searrow\lefteqn{\scriptstyle u} &\downarrow\\
&& (V^{\lambda+\mu})^*\otimes V^{\lambda+\mu},
\end{array}
$$
where the right vertical arrow comes from the right lax structure on both $\Irr$ and $\Irr^*$, and $u$ everytime denotes the unit of the corresponding duality.

\sssec{} Consider the lax action of $\Lambda$ on $E$ such that $\lambda\in\Lambda$ sends $z\in E$ to $z\ast (V^{\lambda})^*$, the lax structure on the action is the morphisms
$$
(V^{\lambda})^*\otimes (V^{\mu})^*\to (V^{\lambda+\mu})^*
$$
coming from the right lax structure on $\Irr^*$. Consider also the trivial left action of $\Lambda$ on $E$. For these two actions the category of lax central objects 
$\Cent^{lax, \ra}_{\Lambda}(E)$ identifies canonically with $\cA-mod(E)$ by the above.   

\sssec{} By (\cite{DL}, Lemma~C.4.10), for $z\in \cA-mod(E)$ this gives a well-defined functor $(\Lambda,\le)\to E$, $\lambda\mapsto z\ast (V^{\lambda})^*$. Given $\lambda, \lambda_i\in\Lambda$ with $\lambda_2=\lambda+\lambda_1$ the transition map 
$$
z\ast (V^{\lambda_1})^*\to z\ast (V^{\lambda_2})^*
$$
is the composition
$$
z\ast (V^{\lambda_1})^*\toup{a^{\lambda}} z\ast (V^{\lambda})^*\otimes (V^{\lambda_1})^*\to z\ast (V^{\lambda_2})^*,
$$
where the second arrow comes from the right-lax structure on the functor $\Irr^*$. The coherent system of higher compatibiliies for compositions comes automatically from the above data.

\sssec{} Assume $E=C(X)\in C(X)-mod$. Consider now the following two lax actions of $\Lambda$ on $C(X)$. The left action is trivial. The right action of $\lambda\in\Lambda$ sends $c$ to $c\otimes^! (V^{\lambda})^*$. The lax structure on the right action is then given for $\lambda_i\in\Lambda$ by the maps
$$
c\otimes^! (V^{\lambda_1})^*\otimes^! (V^{\lambda_2})^*\to c\otimes^! (V^{\lambda_1+\lambda_2})^*
$$
coming from the right-lax structure on the functor (\ref{functor_V^lambda_dual}). For each $\lambda\in\Lambda$ the composition $V^{\lambda}\otimes^! A\toup{\tau^{\lambda}} A\otimes^! A\toup{m_A} A$ gives by duality a map $\phi(\lambda): A\to A\otimes^! (V^{\lambda})^*$ in $C(X)$. 

 Since $A\in \cA-mod(C(X))$, $A$ together with the collection of maps $\{\phi(\lambda)\}$ becomes a lax central object $A\in \Cent^{lax,\ra}_{\Lambda}(C(X))$ for these actions. Then we get a well-defined functor
\begin{equation}
\label{functor_for_DrPl_formalism}
(\Lambda,\le)\to C(X), \; \lambda\mapsto A\otimes^! (V^{\lambda})^*.
\end{equation}
If $\lambda_i,\lambda\in\Lambda$ and $\lambda_2=\lambda+\lambda_1$
then the transition map $A\otimes^! (V^{\lambda_1})^*\to A\otimes^! (V^{\lambda_2})^*$
is the composition
$$
A\otimes^! (V^{\lambda_1})^*\to A\otimes^! V^{\lambda}\otimes^! (V^{\lambda})^*\otimes^! 
(V^{\lambda_1})^*\toup{m_A\comp \tau^{\lambda}}
A\otimes^! (V^{\lambda})^*\otimes^!(V^{\lambda_1})^*\to A\otimes^! (V^{\lambda_2})^*,
$$
here the last arrow comes from the right-lax symmetric monoidal structure on the functor (\ref{functor_V^lambda_dual}). 

\sssec{} Set
$$
B=\underset{\lambda\in (\Lambda, \le)}{\colim} A\otimes^! (V^{\lambda})^*
$$
calculated in $C(X)$. Let us equip $B$ with a structure of an object in $CAlg(A-mod(C(X),\otimes^!))$.
  
   To do so, view $(\Lambda,\le)$ as a symmetric monoidal category with operation $(\lambda_1,\lambda_2)\mapsto \lambda_1+\lambda_2$. View (\ref{functor_for_DrPl_formalism}) as a functor $(\Lambda,\le)\to A-mod(C(X))$. Then it    
is naturally right-lax symmetric monoidal. Namely, for $\lambda,\mu\in\Lambda$ the corresponding morphism
$$
(A\otimes^! (V^{\lambda})^*)\otimes^!_A (A\otimes^! (V^{\mu})^*)=A\otimes^! (V^{\mu})^*\otimes^! (V^{\lambda})^*\to A\otimes^! (V^{\lambda+\mu})^*
$$
is given by the right-lax structure on the functor (\ref{functor_V^lambda_dual}). The desired structure on $B$ is now obtained using Remark~\ref{Rem_6.2.4_v4} below. 

\begin{Rem} 
\label{Rem_6.2.4_v4}
Let $K$ be a small symmetric monoidal category, $\cA$ another symmetric monoidal $\infty$-category, which admits colimits. Let $f: K\to \cA$ be a right-lax symmetric monoidal functor. Then $\colim f$ naturally has a structure of an object of $CAlg(\cA)$ provided that the tensor product in $\cA$ commutes with $K$-indexed colimits.
\end{Rem}
 
\sssec{} 
\label{Sect_6.2.5_v4_before_factorization}
The unit map $A\to B$ for $B$ is a morphism in $CAlg(C(X))$. 

Let $M\in (C(X),\otimes^!)-mod$. Consider the functor $A-mod(M)\to B-mod(M)$, $c\mapsto B\otimes_A c$. In the situation of Section~\ref{Sect_6.1.18_v4} we have the following. The functor 
$$
A-mod(M)\to B-mod(M), \; c\mapsto B\otimes_A c
$$ 
identifies canonically with the functor 
$$
c\mapsto \underset{\lambda\in (\Lambda,\le)}{\colim} c\ast (V^{\lambda})^*
$$

\sssec{} Pick $I\in fSets$. Our purpose now is to spread the above isomorphism over $X^I$. 
 
 Equip $\Map(I, \Lambda)$ with the relation $\und{\lambda}_1\le \und{\lambda}_2$ iff there is $\und{\lambda}\in \Map(I,\Lambda)$ with $\und{\lambda}_2=\und{\lambda}+\und{\lambda}_1$. Then $(\Map(I,\Lambda), \le)$ is a filtered category.

\begin{Pp} 
\label{Pp_6.2.7_v4}
Let $M\in (C_{X^I},\otimes^!)-mod$. The functor 
$$
A_{X^I}-mod(M)\to B_{X^I}-mod(M), \; c\mapsto B_{X^I}\otimes_{A_{X^I}} c
$$
identifies canonically with the functor
$$
c\mapsto \underset{\und{\lambda}\in (\Map(I,\Lambda), \le)}{\colim} \;\, c\ast (V^{\und{\lambda}})^*
$$
\end{Pp}
\begin{proof} Recall that $A_{X_I}-mod(M)\,\iso\, A_{X^I}-mod(C_{X^I})\otimes_{C_{X^I}} M$ canonically. So, it suffices to do the universal case, and we assume $M=C_{X^I}$ viewed naturally as a module over itself. Moreover, it suffices to establish an isomorphism
\begin{equation}
\label{iso_key_for_Pp_6.2.7_v4}
\underset{\und{\lambda}\in (\Map(I,\Lambda), \le)}{\colim} \; A_{X^I}\otimes^!(V^{\und{\lambda}})^*\,\iso\, B_{X^I}
\end{equation}
in $A_{X^I}-mod(C_{X^I})$. Indeed, it will give for any $c\in A_{X^I}-mod(C_{X^I})$
\begin{multline*}
B_{X^I}\otimes_{A_{X^I}} c\,\iso\, \underset{\und{\lambda}\in (\Map(I,\Lambda), \le)}{\colim} \; (A_{X^I}\otimes^!(V^{\und{\lambda}})^*)\otimes^!_{A_{X^I}} c\,\iso\\
\underset{\und{\lambda}\in (\Map(I,\Lambda), \le)}{\colim} \; c\otimes^! (V^{\und{\lambda}})^*
\end{multline*}

 As in Section~\ref{Sect_6.1.17_v4}, one has canonically
$$
A_{X^I}\otimes^! (V^{\und{\lambda}})^*\,\iso\, \underset{(I\to J\to K)\in \Tw(I)}{\colim}\, \underset{k\in K}{\boxtimes} (A^{\otimes J_k}\otimes^! (\underset{j\in J_k}{\otimes} (V^{\lambda_j})^*))
$$
Let $\und{\lambda}\in \Map(I, \Lambda)$ and $(I\to J\to K)\in \Tw(I)$. For each $k\in K$ we get a morphism 
$$
A^{\otimes J_k}\otimes^! (\underset{j\in J_k}{\otimes} (V^{\lambda_j})^*)\to B^{\otimes J_k}
$$ 
in $C^{\otimes J_k}(X)$ as the $\otimes$-product of the structure morphisms
$$
A\otimes^! (V^{\lambda_j})^*\to B
$$
in $C(X)$. Taking further the exterior product $\underset{k\in K}{\boxtimes}$, we get a morphism
$$
\underset{k\in K}{\boxtimes} (A^{\otimes J_k}\otimes^! (\underset{j\in J_k}{\otimes} (V^{\lambda_j})^*))\to \underset{k\in K}{\boxtimes} B^{\otimes J_k}
$$
in $C_{X^I}$ functorial in $(I\to J\to K)\in \Tw(I)$. Passing to the colimit over $\Tw(I)$, this gives the desired morphism (\ref{iso_key_for_Pp_6.2.7_v4}). The fact that it is an 
isomorphism easily follows from the factorizaton and Section~\ref{Sect_6.2.5_v4_before_factorization}. 
\end{proof}

\section{Application to the Satake functor}
\label{Sect_Application to Satake}

\ssec{} In (\cite{Ras2}, Section~6) a factorizable Satake functor is defined in the setting of $\cD$-modules, and some proofs in \select{loc.cit.} do not apply in the constructible context. In Section~\ref{Sect_Application to Satake} we explain how to modify the argument from \select{loc.cit.} so that it works uniformly for $\cD$-modules and in the constructible context to get the factorizable Satake functor $\Sat_{G,\Ran}$. Our purpose is to fill in what we think was a gap in the definition of the factorizable Satake functor in the constructible setting in (\cite{GLys}, \cite{Gai19Ran}), it is described in Section~\ref{Sect_gap_described}. 
 
\sssec{} Use the following notations from \cite{GLys}. Let $G$ be a connected reductive group over $k$, $\check{G}$ its Langlands dual group over $e$. 

 For $I\in fSets$ let $\Gr_{G, X^I}$ (resp., $\Gr_{G, \Ran}$) be the version of the affine grassmannian of $G$ over $X^I$ (resp., $\Ran$). For $I\in fSets$ one has the group schemes 
$$
\gL^+(G)_{\Ran}\to\Ran, \; \gL^+(G)_I:=X^I\times_{\Ran} \gL^+(G)_{\Ran}\to X^I
$$ 
defined in (\cite{Gai19Ran}, 1.1.3). 
 
\sssec{}  One defines the sheaf of categories 
$\Sph_{G,\Ran}$ over $\Ran$ as a collection of sheaves of categories over $X^I$
$$
\Sph_{G, I}=Shv(\Gr_{G, X^I})^{\gL^+(G)_I}
$$ together with natural equivalences $\Sph_{G, I}\otimes_{Shv(X^I)} Shv(X^J)\,\iso\, Sph_{G, J}$ in $ShvCat(X^J)$ for maps $I\to J$ in $fSets$. 

\index{Sperical sheaves on the affine Grassmannian $\Gr_{G,\Ran}$}

Our purpose is to precise the definition of the functors
\begin{equation}
\label{functor_Sat_G_I}
\Sat_{G, I}: \Rep(\check{G}))_{X^I}\to \Sph_{G, I}
\end{equation}
and glue them into a functor 
\begin{equation}
\label{functor_Sat_G_Ran}
\Sat_{G,\Ran}: \Fact(\Rep(\check{G}))\to \Sph_{G,\Ran}
\end{equation}
in $ShvCat(\Ran)\,\iso\, Shv(\Ran)-mod$. 

\ssec{Case of $C=\Rep(\Gamma)$} Let $\Gamma$ be an affine algebraic group of finite type. Take $C=\Rep(\Gamma)$. It is compactly generated by irreducible finite-dimensional representations in $C^{\heartsuit}$, it is rigid (by \cite{G}, I.3, 3.7.4). Set $C(X)=C\otimes Shv(X)$. By Proposition~\ref{Pp_4.1.19}, $\zeta_I: C_{X^I}\,\iso\, \bar C_{X^I}$ is an equivalence. We extend some results from \cite{Ras2} to the category $\Fact(C)$ in the constructible setting. 

\sssec{} Recall that for any $D\in \DGCat_{cont}$, $\oblv: \Rep(\Gamma)\otimes D\to D$ is comonadic by (\cite{Ga1}, 5.5.5). The corresponding comonad is the action of the coalgebra $\cO_{\Gamma}\in coAlg(\Vect)$, here $\cO_{\Gamma}$ is the space of functions on $\Gamma$. 

 The map $\oblv: C(X)\to Shv(X)$ is a morphism in $CAlg(Shv(X)-mod)$, so by functoriality yields a morphism $\Fact(C)\to Shv(\Ran)$ in $Shv(\Ran)-mod$. For $I\in fSets$ it manifests itself by the oblivion functor $\oblv_{X^I}: C_{X^I}\to Shv(X^I)$.

\begin{Lm} The functor $\oblv_{X^I}: C_{X^I}\to Shv(X^I)$ has a continuous $Shv(X^I)$-linear right adjoint $\oblv_{X^I}^R$.
\end{Lm}
\begin{proof}
By Lemma~\ref{Lm_ULA_general_lemma}, $C_{X^I}$ is ULA over $Shv(X^I)$. By Lemma~\ref{Lm_3.5.2_generated_under_colim}, $\Loc: C^{\otimes I}\otimes Shv(X^I)\to C_{X^I}$ generates $C_{X^I}$ under colimits. Let $c\in C^{\otimes I}$ be compact. Then $\Loc(c\otimes\omega_{X^I})$ is ULA over $Shv(X^I)$, cf. Lemma~\ref{Lm_ULA_general_lemma}. By Proposition~\ref{Pp_3.6.6}, it suffices to show that $\oblv_{X^I}\Loc(c\otimes\omega_{X^I})$ is ULA over $Shv(X^I)$. We have by functoriality $\oblv_{X^I}\Loc(c\otimes\omega_{X^I})\,\iso\,\oblv(c)\otimes \omega_{X^I}$, where $\oblv: C\to\Vect$ is the forgetful functor. We are done.
\end{proof}
\begin{Lm} For any $D\in Shv(X^I)-mod$, the functor 
\begin{equation}
\label{functor_in_Lm_6.2.3}
\oblv_{X^I}: C_{X^I}\otimes_{Shv(X^I)} D\to D
\end{equation}
is comonadic. 
\end{Lm}
\begin{proof}
For $\cD$-modules this is (\cite{Ras2}, 6.23.2). Use the isomorphism $C_{X^I}\,\iso\, \bar C_{X^I}$. By Proposition~\ref{Pp_4.1.10_Raskin_dualizability},
$$
C_{X^I}\otimes_{Shv(X^I)} D\,\iso\, \underset{(I\toup{p} J\to K)\in \Tw(I)^{op}}{\lim} C^{\otimes K}\otimes Shv(X^I_{p, d})\otimes_{Shv(X^I)} D
$$
By (\cite{Ga1}, Lemma~5.5.4), for each $\Sigma=(I\to J\to K)$ the functor
$$
\oblv_{\Sigma}: 
C^{\otimes K}\otimes Shv(X^I_{p, d})\otimes_{Shv(X^I)} D\to  Shv(X^I_{p, d})\otimes_{Shv(X^I)} D
$$
is comonadic, so is conservative and commutes with $\oblv_{\Sigma}$-split totalizations. We used that $C^{\otimes K}\,\iso\,\Rep(\Gamma^K)$. By (\cite{Ly}, 2.5.3), $\oblv_{X^I}$ is conservative. The claim follows now from Lemma~\ref{Lm_7.2.4_now} below.
\end{proof}

\begin{Lm} 
\label{Lm_7.2.4_now}
Let $I\in 1-\Cat$ be small, $I\times[1]\to 1-\Cat$ be a diagram sending $i$ to $(f_i: C_i\to D_i)$. Let $f: C\to D$ be obtained from $f_i$ by passing to the limit over $i\in I$. Assume $f$ has a right adjoint. If each $f_i$ is comonadic then $f$ is also comonadic. \QED
\end{Lm}

\begin{Lm} 
\label{Lm_7.2.5_t-structure}
For $I\in fSets$ there is a natural t-structure on $C_{X^I}$ such that $\oblv_{X^I}: C_{X^I}\to Shv(X^I)$ is t-exact. It is accessible, compatible with filtered colimits, left and right complete. In fact, $C_{X^I}^{\le 0}=\oblv_{X^I}^{-1}(Shv(X^I)^{\le 0})$ and $C_{X^I}^{\ge 0}=\oblv_{X^I}^{-1}(Shv(X^I)^{\ge 0})$.
\end{Lm}
\begin{proof}
For each $\Sigma=(I\toup{p} J\to K)$ the forgetful functor
$$
\oblv_{\Sigma}: C^{\otimes K}\otimes Shv(X^I_{p, d})\to Shv(X^I_{p, d})
$$
is comonadic with the comonad given by the coalgebra $\cO_{\Gamma^K}\in coAlg(\Vect)$. Since $\cO_{\Gamma^K}$ is placed in degree zero, the functor $\cdot\otimes\cO_{\Gamma^K}: \Vect\to \Vect$ is t-exact. We equip $C^{\otimes K}\otimes Shv(X^I_{p, d})$ with a t-structure defined in (\cite{Ly}, 9.3.23). So, both functors in the adjoint pair 
$$
\oblv_{\Sigma}: C^{\otimes K}\otimes Shv(X^I_{p, d})\leftrightarrows Shv(X^I_{p, d}): \coind_{\Sigma}
$$
are t-exact. By definition, $(C^{\otimes K}\otimes Shv(X^I_{p, d}))^{\le 0}$ is the preimage of
$Shv(X^I_{p, d})^{\le 0}$ under $\oblv$, so the t-structure on $C^{\otimes K}\otimes Shv(X^I_{p, d})$ is accessible. Besides, $(C^{\otimes K}\otimes Shv(X^I_{p, d}))^{\ge 0}$ is the preimage of $Shv(X^I_{p, d})^{\ge 0}$ under $\oblv$. So, the t-structure on $C^{\otimes K}\otimes Shv(X^I_{p, d})$ is compatible with filtered colimits.

By (\cite{G}, I.3, 1.5.8), there is a unique t-structure on 
\begin{equation}
\label{limit_defining_barC_X^I}
\underset{(I\toup{p} J\to K)\in \Tw(I)^{op}}{\lim} C^{\otimes K}\otimes Shv(X^I_{p, d})
\end{equation}
such that each evaluation functor to $C^{\otimes K}\otimes Shv(X^I_{p, d})$ is t-exact, because the transition functors in our limit diagram are t-exact. Moreover, the t-structure on 
$C_{X^I}$ is compatible with filtered colimits and accessible, as 
$$
\underset{(I\toup{p} J\to K)\in \Tw(I)^{op}}{\lim} C^{\otimes K}\otimes Shv(X^I_{p, d})^{\le 0}
$$
is presentable. 

 It remains to show that the t-structure on $C_{X^I}$ is left and right complete. For $\cD$-modules this is (\cite{Ras2}, 6.24.1). Assume we are in the constructible context. By (\cite{G}, I.3, 1.5.8), it suffices to show that for each $(I\toup{p} J\to K)\in\Tw(I)$ the t-structure on $C^{\otimes K}\otimes Shv(X^I_{p, d})$ is both left and right complete.
 
 The t-structure on $Shv(X^I_{p, d})$ is right complete by (\cite{Ly4}, 0.0.10). So, the t-structure on $C^{\otimes K}\otimes Shv(X^I_{p, d})$ is right complete by (\cite{Ly}, 9.3.23).
The t-structure on $Shv(X^I_{p, d})$ is left complete by (\cite{AGKRRV}, Theorem~1.1.6). 

 To see that the t-structure on $C^{\otimes K}\otimes Shv(X^I_{p, d})$ is left complete, apply (\cite{AGKRRV}, E.9.6). Namely, $Shv(X^I_{p, d})\,\iso\,\Ind(\D^b(\Perv(X^I_{p, d})))$ by (\cite{AGKRRV}, E.1.2). Now the t-structure on $\QCoh(B(\Gamma^K))$ is left complete by (\cite{G}, I.3, 1.5.7), as $B(\Gamma^K)$ is an Artin stack. Now by (\cite{AGKRRV}, E.9.6), the t-structure on $C^{\otimes K}\otimes Shv(X^I_{p, d})$ is left complete. 

In turn, by (\cite{G}, I.3, 1.5.8) applied to the diagram 
$$
\underset{(I\toup{p} J\to K)\in \Tw(I)^{op}}{\lim} C^{\otimes K}\otimes Shv(X^I_{p, d})
$$
we see that the t-structure on $\bar C_{X^I}$ is left complete.
\end{proof}

\begin{Rem} 
\label{Rem_6.2.5_perverse_degree}
If $A_0\in CAlg(\Vect^{\heartsuit})$, consider $A=A_0\otimes\omega_X\in CAlg(Shv(X))$ and the corresponding factorization algebra $\Fact(A)\in Shv(\Ran)$. For each $I\in fSets$, $A_{X^I}\in Shv(X^I)$ is placed in perverse degree $-\mid I\mid$. Indeed, for $\cD$-modules this is (\cite{Ras2}, 6.24.3), and its proof holds also in the constructible context.
\end{Rem}

\sssec{} Consider $A:=\cO_{\Gamma}\otimes \omega_X\in CAlg(C(X))$, here we view $\cO_{\Gamma}\in \Rep(\Gamma)$ via $\Gamma$-action on itself by left translations. Applying the construction of factorization algebras from Section~\ref{Sect_Factorization algebras in Fact(C)}, for $I\in fSets$ we get $A_{X^I}\in C_{X^I}$ and $\Fact(A)\in \Fact(C)$. 

 The counit map $A\to \omega_X$ in $CAlg(Shv(X))$ gives by functoriality of $\Fact$ a morphism 
\begin{equation}
\label{map_for_Sect_6.2.5}
\oblv_{X^I}(A_{X^I})\to (\omega_X)_{X^I}\,\iso\, \omega_{X^I}
\end{equation} 
in $Shv(X^I)$ compatible with factorizations. As $I\in fSets$ varies, this gives a morphism 
$$
\oblv_{\Ran}(\Fact(A))\to \Fact(\omega_X)\,\iso\,\omega_{\Ran}
$$ 
in $Shv(\Ran)$. Here $\oblv_{\Ran}: \Fact(C)\to Shv(\Ran)$ is the forgetful functor. We may view $\oblv_X(A)$ as an object of $CAlg(Shv(X))$, which in turn gives a factorization algebra $\Fact(\oblv_X(A))\in Shv(\Ran)$. We have 
$$
\Fact(\oblv_X(A))\,\iso\,\oblv_{\Ran}(\Fact(A))
$$ 
in $Shv(\Ran)$. In particular, $(\oblv_X(A))_{X^I}\,\iso\,\oblv_{X^I}(A_{X^I})$ for each $I\in fSets$. 

By adjointness, (\ref{map_for_Sect_6.2.5}) gives a map $A_{X^I}\to \oblv_{X^I}^R(\omega_{X^I})$ in $C_{X^I}$. As in (\cite{Ras2}, 6.22.1), the latter map is an isomorphism. So, the comonad for (\ref{functor_in_Lm_6.2.3}) is 
$$
\oblv_{X^I}(A_{X^I})\in coAlg(Shv(X^I))
$$ 

 Note that $A$ is naturally promoted to an object of $CAlg(coAlg(Shv(X)))$. To unburden the notation, write 
$$
\cF_{I, A}^{Shv}: \Tw(I)\to Shv(X^I)
$$ 
for the functor $\cF^{Shv}_{I, \, \oblv_X(A)}$ from Section~\ref{Sect_Factorization algebras in Fact(C)}, it sends $(I\to J\to K)$ to $\underset{k\in K}{\boxtimes} A^{\otimes J_k}$. It is naturally promoted to a functor
$$
\cF_{I, A}^{Shv, coAlg}: \Tw(I)\to coAlg(Shv(X^I))
$$
The forgetful functor $coAlg(Shv(X^I))\to Shv(X^I)$ preserves colimits, so
$$
(\oblv_X(A))_{X^I}\,\iso\, \underset{(I\to J\to K)\in \Tw(I)}{\colim} \cF_{I, A}^{Shv}
$$
can be understood in $Shv(X^I)$ or equivalently in $coAlg(Shv(X^I))$. This is how $\oblv_{X^I}(A_{X^I})$ becomes a coalgebra in $Shv(X^I)$. 

\begin{Lm} The functor $\oblv_{X^I}^R: Shv(X^I)\to C_{X^I}$ is t-exact.
\end{Lm}
\begin{proof}
Since $\oblv_{X^I}$ is t-exact, its right adjoint is left t-exact. To see that $\oblv_{X^I}^R$ is right t-exact, it suffices to show that $\oblv_{X^I}\oblv_{X^I}^R$ is right t-exact. This follows from the fact that the diagonal $X^I\to X^I\times X^I$ is a locally complete intersection of codimension $\mid I\mid$ combined with (\cite{BBD}, Corollare~4.1.10). Indeed, $A_{X^I}$ is placed in perverse degree $-\mid I\mid$ by Remark~\ref{Rem_6.2.5_perverse_degree}.  
\end{proof}

\begin{Cor} i) The category $(C_{X^I})^{\heartsuit}$ is a Grothendieck abelian category in the sense of (\cite{HA}, 1.3.5.1), in particular it has enough injective objects. 
The canonical functor
$$
\D((C_{X^I})^{\heartsuit})\to C_{X^I}
$$ 
is an equivalence, where $\D((C_{X^I})^{\heartsuit})$ is the left completion of $\D((C_{X^I})^{\heartsuit})^+$.

\smallskip\noindent
ii) The category $C_{X^I}$ is of finite cohomological dimension.  
\end{Cor}
\index{Grothendieck abelian category}
\begin{proof}
i) The first claim follows from (\cite{HA}, 1.3.5.23), since the t-structure on $C_{X^I}$ is accessible and compatible with filtered colimits. Then $(C_{X^I})^{\heartsuit}$ has enough injective objects by (\cite{HA}, Corollary 1.3.5.7), so one has the canonical functor 
$$
\D((C_{X^I})^{\heartsuit})^+\to C_{X^I}^+
$$
We invoke (\cite{G}, I.3, Lemma~2.4.5) to conclude that the latter functor is an equivalence. 
The rest follows, because the t-structure on $C_{X^I}$ is left complete. 

\smallskip\noindent
ii) If $K\in fSets$ then $C^{\otimes K}$ is of finite cohomological dimension by \cite{HR}. Now $C_{X^I}$ is written as the limit (\ref{limit_defining_barC_X^I}) under t-exact functors between $\DG$-categories of finite cohomological dimension.
\end{proof}

\begin{Rem} 
\label{Rem_Loc_is_t-exact}
The functor $\Loc: C^{\otimes I}\otimes Shv(X^I)\to C_{X^I}$ is t-exact.
\end{Rem}

\ssec{Chiral Hecke algebra} 
\label{Sect_Chiral Hecke algebra}
In Section~\ref{Sect_Chiral Hecke algebra} we recall the construction of the chiral Hecke algebra from (\cite{Gai_de_Jong}, 7.5).

\index{Chiral Hecke algebra}

\sssec{} Set $\cO=k[[t]]\subset F=k((t))$. Denote by 
$\Perv(\Gr_{G, X^I})^{\gL^+(G)_I}$ the abelian category of $\gL^+(G)_I$-equivariant perverse sheaves on $\Gr_{G, X^I}$. Recall the convolution diagram from (\cite{MV}, Section 5, diagram (5.2))
\begin{equation}
\label{diag_convolution_for_X^2}
\Gr_{G,X}\times \Gr_{G, X}\getsup{p} \wt{\Gr_{G,X}\times\Gr_{G,X}} \toup{q} \Gr_{G,X}\ttimes \Gr_{G,X}\toup{m} \Gr_{G, X^2}
\end{equation} 

For $I$ consisting of one element, write $\gL^+(G)_X=\gL^+(G)_I$. Given $\cB_i\in \Perv(\Gr_{G, X})^{\gL^+(G)_X}$ one defines their convolution 
$$
\cB_1\ast_X \cB_2\in Shv(\Gr_{G, X^2})^{\gL^+(G)_I}
$$ 
for $I=\{1,2\}$ by the formula (5.6) from \cite{MV}. Namely, 
$$
\cB_1\ast_X \cB_2=m_*(\cB_1\tboxtimes \cB_2)
$$
where $q^*(\cB_1\tboxtimes \cB_2)\,\iso\, p^*(\cB_1\boxtimes\cB_2)$. Let 
\begin{equation}
\label{functor_tau^0_first}
\tau^0: \Perv(\Gr_G)^{G(\cO)}\to \Perv(\Gr_{G,X})^{\gL^+(G)_X}
\end{equation} 
be the functor defined as in (\cite{MV}, Remark 5.1). Let $j: U\hook{} X^2$ be the complement to the diagonal. The base change of $m$ by $j: U\hook{} X^2$ becomes canonically the identity map 
$$
\id: (\Gr_{G,X}\times \Gr_{G,X})\mid_U\to (\Gr_{G,X}\times \Gr_{G,X})\mid_U
$$
in view of the factorization structure of $\Gr_{G, X^I}$. By abuse of notations, we also write 
$$
(\Gr_{G,X}\times \Gr_{G,X})\mid_U\hook{j} \Gr_{G,X}\ttimes \Gr_{G,X}\getsup{i} (\Gr_{G,X}\ttimes \Gr_{G,X})\times_{X^2} X
$$
for the corresponding closed immersion and its complement.

\sssec{} Write 
$$
\Sat: \Rep(\check{G})\to Shv(\Gr_G)^{G(\cO)}
$$
for the Satake functor at one point of the curve $X$. Recall the convolution diagram
$$
\Gr_G\times\Gr_G\getsup{p} G(F)\times \Gr_G \toup{q} G(F)\times^{G(\cO)}\Gr_G \toup{m} \Gr_G
$$
at one point of the curve. For $\cS_1,\cS_2\in \Perv(\Gr_G)^{G(\cO)}$ let similarly $\cS_1\tboxtimes\cS_2$ be the perverse sheaf on $G(F)\times^{G(\cO)}\Gr_G$ equipped with 
$$
q^*(\cS_1\tboxtimes \cS_2)\,\iso\, p^*(\cS_1\boxtimes \cS_2).
$$ 
So, the usual convolution on $\Gr_G$ is defined by 
$$
\cS_1\ast\cS_2=m_*(\cS_1\tboxtimes\cS_2).
$$ 

For $V\in\Rep(\check{G})^{\heartsuit}$ set $\cT_V=\tau^0\Sat(V)$. Write also by abuse of notations
$$
\tau^0: \Perv(G(F)\times^{G(\cO)}\Gr_G)^{G(\cO)}\to \Perv((\Gr_{G,X}\ttimes \Gr_{G,X})\times_{X^2} X)^{\gL^+(G)_X}
$$
for the functor analogous to (\ref{functor_tau^0_first}). 

\sssec{} 
\label{Sect_7.3.3_now}
Denote by 
$$
(\Gr_{G,X}\times \Gr_{G,X})\mid_U\hook{\bar j} \Gr_{G,X^2} \getsup{\bar i} \Gr_{G, X}
$$
the corresponding closed immersion and its complement. Recall that $\cT_V\ast_X \cT_W$ is perverse, the intermediate extension under $\bar j$ by \cite{MV}. 

 For $V,W\in \Rep(\check{G})^{\heartsuit}$ one has canonically 
$$
i^!(\cT_V\tboxtimes \cT_W)\,\iso\, \tau^0(\Sat(V)\tboxtimes \Sat(W))[-1].
$$ 
The fibre sequence
$$
i_*i^!(\cT_V\tboxtimes \cT_W)\to \cT_V\tboxtimes \cT_W\to j_*j^*(\cT_V\tboxtimes \cT_W)
$$
on $\Gr_{G,X}\ttimes \Gr_{G,X}$ becomes an exact sequence of perverse sheaves 
$$
0\to \cT_V\tboxtimes \cT_W\to j_*j^*(\cT_V\tboxtimes \cT_W)\to \tau^0(\Sat(V)\tboxtimes \Sat(W))\to 0.
$$

Applying $m_*$ it yields an exact sequence of perverse sheaves on $\Gr_{G, X^2}$
\begin{equation}
\label{eq_seq_perv_sheaves_for_Sect_1.11.4}
0\to\cT_V\ast_X \cT_W\to  \bar j_*(\cT_V\boxtimes \cT_W)\to \bar i_*\cT_{V\otimes W}\to 0.
\end{equation}

\sssec{} As in \cite{BD_chiral}, for $I\in fSets$ set $\lambda_I=(\underset{i\in I}{\boxtimes} e[1])[-\mid I\mid]$, so $\lambda_I\,\iso\, e$, but the group $\Aut_I$ acts on it by the sign character. This line is introduced to avoid signs problems in the construction of the chiral Hecke algebra.

 Let $A\in CAlg^{nu}(\Rep(\check{G}))$ lying in $\Rep(\check{G})^{\heartsuit}$. The product on $A$ gives a map $\cT_{A\otimes A}\to \cT_A$ on $\Gr_{G, X}$. Composing with the second map from (\ref{eq_seq_perv_sheaves_for_Sect_1.11.4}) one gets the chiral multiplication map
\begin{equation}
\label{map_ch_mult_for_cR_X}
\bar j_*(\cT_A\boxtimes \cT_A)\to \bar i_*\cT_A
\end{equation}
on $\Gr_{G, X^2}$. In fact, the definition of (\ref{map_ch_mult_for_cR_X}) depends on the order of elements in $\{1,2\}$, as the prestack $\Gr_{G,X}\ttimes \Gr_{G, X}$ is not symmetric. What we get canonically is rather the map
$$
\bar j_*(\cT_A^{\boxtimes I})\otimes \lambda_I \to \bar i_*\cT_A 
$$
for a set $I$ of two elements. 

\sssec{Jacobi identity} It is well known that (\ref{map_ch_mult_for_cR_X}) satisfies the Jacobi identity. Let us sketch a proof for the convenience of the reader.

 For $I=\{1,\ldots, n\}$ write $(\Gr_{G,X})^{\ttimes I}$ for the corresponding version of the convolution diagram. The linear order on $I$ is used for the definition of the latter prestack.
We still denote by $m: (\Gr_{G,X})^{\ttimes I}\to \Gr_{G, X^I}$ the convolution map.  
 For $V_i\in\Rep(\check{G})^{\heartsuit}$ we similarly get the perverse sheaf $\cT_{V_1}\tboxtimes\ldots\tboxtimes\cT_{V_n}$ on $(\Gr_{G,X})^{\ttimes I}$.
Consider the diagram
$$
(\Gr_{G,X})^{\ttimes I}\times_{X^I} X\hook{i} (\Gr_{G,X})^{\ttimes I}\getsup{\, j} \Gr_{G, X}^I\mid_{(X^I-X)},
$$ 
where $j$ is the complement open to the closed immersion $i$. We get
$$
i^!(\cT_{V_1}\tboxtimes\ldots\tboxtimes\cT_{V_n})\,\iso\, \tau^0(\Sat(V_1)\tboxtimes\ldots\tboxtimes\Sat(V_n))[1-n]
$$

 Applying $m_*$ to the fibre sequence
$$
i_*i^!(\cT_{V_1}\tboxtimes\ldots\tboxtimes\cT_{V_n})\to (\cT_{V_1}\tboxtimes\ldots\tboxtimes\cT_{V_n})\to j_*(\underset{i\in I}{\boxtimes}
\cT_{V_i})
$$
we get the fibre sequence 
$$
\bar i_*\cT_{V_1\otimes\ldots\otimes V_n}[1-n]\to \cT_{V_1}\ast_X\ldots\ast_X \cT_{V_n}\to\bar j_*(\underset{i\in I}{\boxtimes}\cT_{V_i})
$$
for the similar diagram
$$
(\Gr_{G,X}^I)\mid_{X^I-X}\hook{\bar j} \Gr_{G, X^I} \getsup{\bar i} \Gr_{G,X}
$$

 This is sufficient to get the exactness of the Cousin complex on $\Gr_{G, X^I}$ for the stratification coming from the diagonal stratification of $X^I$. 
 
 In details, the diagonal stratification of $X^I$ is as follows. For $d\ge 0$ let
$$
\bar Y_d=\underset{T\in Q(I), \, \mid T\mid=n-d}{\cup}\, \vartriangle^{(I/T)}(X^T)
$$  
Here $\vartriangle^{(I/T)}: X^T\to X^I$. Let $Y_d=\bar Y_d-\bar Y_{d+1}$. So, $Y_d$ is smooth of dimension $n-d$. The inclusion $Y_d\hook{} X^I$ is affine. 
  
 Let $Z_d=Y_d\times_{X^I} \Gr_{G, X^I}$. Let $j_d: Z_d\hook{} \Gr_{G, X^I}$ be the inclusion. Apply Section~\ref{Sect_B.1.12} to the perverse sheaf $\cT_{V_1}\ast_X\ldots\ast_X \cT_{V_n}$ on $\Gr_{G, X^I}$ and the stratification $\{Z_d\}$ of $\Gr_{G, X^I}$. Since $j_d^!(\cT_{V_1}\ast_X\ldots\ast_X \cT_{V_n})$ is placed in perverse degree $d$ for all $d\ge 0$, we get an exact sequence of perverse sheaves on $\Gr_{G, X^I}$
\begin{equation}
\label{ex_seq_for_Sect_1.11.5}
\cT_{V_1}\ast_X\ldots\ast_X \cT_{V_n}\to \cF_0\to \cF_1\to \cF_2\ldots
\end{equation}
with 
$$
\cF_d=(j_d)_*j_d^!(\cT_{V_1}\ast_X\ldots\ast_X \cT_{V_n})[d]
$$   
The Jacobi identity for $\cT_A$ follows from the fact that the square of the differential in (\ref{ex_seq_for_Sect_1.11.5}) vanishes for $n=3$ with $V_1=V_2=V_3=A$ and the functoriality of the construction.

 Thus, $\cT_A$ is a chiral algebra on $\Gr_{G, X}$. 
 
\index{Chiral algebra on $\Gr_{G, X}$} 
 
\sssec{} According to the construction from (\cite{BD_chiral}, 3.4.11), one forms the Chevalley-Cousin complex $\cC(\cT_A)$ of $\cT_A$, it is a collection $\cC(\cT_A)_{X^I}
\in \Sph_{G, I}$ for each $I\in fSets$ together with compatible system of isomorphisms 
\begin{equation}
\label{compatibility_isom_restrictions_for_cC(cT_A)}
\vartriangle^{(I/J)!}\cC(\cT_A)_{X^I}\,\iso\, \cC(\cT_A)_{X^J}
\end{equation} 
for $I\to J$ in $fSets$. Here $\vartriangle^{(I/J)}: X^J\to X^I$ is the corresponding diagonal.  

\index{Chevalley-Cousin complex} 

  Moreover, if the chiral algebra $\cT_A$ is unital then $\cC(\cT_A)_{X^I}$ is placed in the perverse degree $-\mid I\mid$ only, and the corresponding perverse sheaf lies in $\Perv(\Gr_{G, X^I})^{\gL^+(G)_I}$. Let $\oo{X}{}^I\hook{} X^I$ be the complement to all the diagonals. Recall the factorization isomorphism 
$$
\Gr_{G,X^I}\times_{X^I} \oo{X}{}^I\,\iso\, (\Gr_{G,X})^I\times_{X^I} \oo{X}{}^I.
$$
In this case there is a canonical injective $\cL^+(G)_I$-equivariant map of perverse sheaves
$$
\H^{-\mid I\mid}(\cC(\cT_A)_{X^I})\hook{}\bar j^{(I)}_*(\underset{i\in I}{\boxtimes} \cT_A)\otimes \lambda_I,
$$
where  
$$
\bar j^{(I)}: (\Gr_{G,X})^I\times_{X^I} \oo{X}{}^I\hook{} \Gr_{G, X^I}
$$
is the open immersion. We used that $\bar j^{(I)}$ is an affine open embedding. For example, if $I=\{*\}$ then $\cC(\cT_A)_X\,\iso\, \cT_A[1]$. In general,
$$
\cC(\cT_A)_{X^I}=\underset{T\in Q(I)}{\oplus} \vartriangle_*^{(I/T)}
\bar j^{(T)}_*(\bar j^{(T)})^*((\cT_A[1])^{\boxtimes T})
$$
with a differential discribed in (\cite{BD_chiral}, 3.4.11). Here $\vartriangle^{(I/T)}: \Gr_{G, X^T}\to \Gr_{G, X^I}$.

   Besides, $\cC(\cT_A)$ is equipped with the factorization isomorphisms: for any $\phi: I\to I'$ in $fSets$ we have
\begin{equation}
\label{iso_fact_for_C(T_A)}
(\underset{i\in I'}{\boxtimes} \cC(\cT_A)_{X^{I_i}})\mid_{X^I_{\phi, d}} \,\iso\, (\cC(\cT_A)_{X^I})\mid_{X^I_{\phi, d}}
\end{equation}
as in (\cite{BD_chiral}, 3.4.11). 

 Write $\cC(\cT_A)\in \Sph_{G,\Ran}$ for the global section of this sheaf of categories on $\Ran$ defined by the above collection $\cC(\cT_A)_{X^I}, I\in fSets$. 

\sssec{} 
\label{Sect_7.3.7_now}
Set $C=\Rep(\check{G})$. We view it as an object of $CAlg(\DGCat_{cont})$.
The construction of the chiral Hecke algebra from (\cite{Gai_de_Jong}, Section 7) uses the categories $\ov{\Fact}(C)$ and $\bar C_{X^I}$ with $I\in fSets$ of Section~\ref{Sect_4}. By Proposition~\ref{Pp_4.1.19}, 
$$
\zeta_I: C_{X^I}\to \bar C_{X^I}
$$ 
and $\Fact(C)\to \ov{\Fact}(C)$ are equivalences. 

   Take $A=\cO_{\check{G}}$, the algebra of functions on $\check{G}$. View it first as an object of $CAlg(C)$ via the $\check{G}$-action on itself by left translations. By the above, we get $\cC(\cT_A)\in \Sph_{G,\Ran}$. To define (\ref{functor_Sat_G_I}), one needs for $I\in fSets$ to promote $\cC(\cT_A)_{X^I}$ to an object of
$$
\bar C_{X^I}\otimes_{Shv(X^I)} \Sph_{G, I}
$$   
For $\cD$-modules this is done in (\cite{Ras2}, Section~6.33). However, the argument from \select{loc.cit.} does not work in the constructible context. We explain this below, and also explain how to argue instead in a way uniform for all our sheaf theories.
   
\begin{Lm} If $I\in fSets$ then $\cC(\cT_A)_{X^I}$ is naturally an object of 
$$   
\underset{(I\toup{p} J\to K)\in \Tw(I)^{op}}{\lim} C^{\otimes K}\otimes Shv(X^I_{p, d})\otimes_{Shv(X^I)} \Sph_{G, I},
$$
where the transition maps in the diagram are the same as for the functor $\bar\cG_{I,C}$, cf. Section~\ref{Sect_4.1.1_now}.  
\end{Lm}
\begin{proof}
Note that $A\in CAlg(\Rep(\check{G}\times\check{G}))$. Namely, for the diagonal map $q: B(\check{G})\to B(\check{G}\times \check{G})$ the functor $q_*: \QCoh(B(\check{G}))\to \QCoh(B(\check{G}\times \check{G}))$ is right-lax symmetric monoidal, so sends the commutative algebra $e$ to the commutative algebra $A$. 
 
  For this reason for $I\in fSet$ the shifted perverse sheaf $\cC(\cT_A)_{X^I}$ is equipped with an action of $\check{G}$, so it is an object of $C\otimes \Sph_{G, I}$. Now for $\phi: I\to I'$ the restriction $(\cC(\cT_A)_{X^I})\mid_{X^I_{\phi, d}}$ is equipped with an action of $\check{G}^{\times I'}$, the natural one on the LHS of (\ref{iso_fact_for_C(T_A)}). Moreover, the restriction of this action to the diagonal $\check{G}\to \check{G}^{\times I'}$ makes the isomorphism (\ref{iso_fact_for_C(T_A)}) $\check{G}$-equivariant. This observation generalizes to any morphism (\ref{morphism_in_Tw_v2}) in $\Tw(I)^{op}$. 
\end{proof}  

\sssec{} 
\label{Sect_gap_described}
To proceed, one needs to know that the natural functor
\begin{equation}
\label{map_for_Sect_6.2.9}
\bar C_{X^I}\otimes_{Shv(X^I)} \Sph_{G, I}\to  \underset{(I\toup{p} J\to K)\in \Tw(I)^{op}}{\lim} C^{\otimes K}\otimes Shv(X^I_{p, d})\otimes_{Shv(X^I)} \Sph_{G, I}
\end{equation}
is an equivalence. In the case of $\cD$-modules this follows from (\cite{Ras2}, Lemma~6.18.1) or (\cite{Ras2}, Remark~6.18.3). However, both these argument do not work literally in the constructible context.

 Now our Proposition~\ref{Pp_4.1.10_Raskin_dualizability} garantees that (\ref{map_for_Sect_6.2.9}) is an equivalence for all our sheaf theories, so 
$$
\cC(\cT_A)_{X^I}\in \bar C_{X^I}\otimes_{Shv(X^I)} \Sph_{G, I}
$$ 

\sssec{} Recall that $C$ is rigid, so canonically self-dual. By formula (\ref{isom_after_Lm_2.5.7}), we get canonically
$$
\Fun_{Shv(X^I)}(C_{X^I}, \Sph_{G, I})\,\iso\, C_{X^I}\otimes_{Shv(X^I)} \Sph_{G,I}
$$
Thus, 
$$
\Sat_{G,I}:=\cC(\cT_A)_{X^I}\in \Fun_{Shv(X^I)}(C_{X^I}, \Sph_{G, I})
$$ 
is the desired functor (\ref{functor_Sat_G_I}). The isomorphisms (\ref{compatibility_isom_restrictions_for_cC(cT_A)}) imply that $\Sat_{G,I}$ glue to give a functor (\ref{functor_Sat_G_Ran}). 

\sssec{} Recall that $\Sph_{G, \Ran}$ is a factorization sheaf of monoidal categories over $\Ran$ (\cite{GLys}, 9.1.1) in the case of $\cD$-modules.\footnote{In the constructible context it has a weak factorization structure in the sense of Section~\ref{Sect_B2_factorization}. We do not know if this structure is actually strict in this case.} In particular, each $\Sph_{G, I}\in Alg(Shv(X^I)-mod)$. One has to pay attention to shifts in the definition of the monoidal structure on $\Sph_{G, I}$. It can be formulated as follows.

\index{Weak factorization structure}

 For $I, J\in fSets$ similarly to the diagram (\ref{diag_convolution_for_X^2}) one defines a diagram
$$
\Gr_{G, X^I}\times \Gr_{G, X^J}\getsup{p}  \wt{\Gr_{G,X^I}\times\Gr_{G,X^J}} \toup{q} \Gr_{G,X^I}\ttimes \Gr_{G,X^J}\toup{m} \Gr_{G, X^{I\sqcup J}}
$$
One similarly has the exterior convolution functor
$$
\Sph_{G, I}\otimes \Sph_{G, J}\to \Sph_{G, I\sqcup J}
$$
It sends a pair $\cB\in \Sph_{G, I}, \cB'\in\Sph_{G, J}$ to 
$$
m_*(\cB\tboxtimes \cB'),
$$
where $\cB\tboxtimes \cB'$ is given informally by $q^*(\cB\tboxtimes \cB')\,\iso\, p^*(\cB\boxtimes \cB')$. A formal definition is found in Section~\ref{Sect_ext_convolution_precisely}. The exterior convolution is t-exact. 

\index{Exterior convolution on $\Sph_{G,\Ran}$}

The monoidal product on $\Sph_{G,I}$ is defined by
$$
\cB\ast_I \cB'=\bar i^!m_*(\cB\tboxtimes \cB'),
$$
where $\bar i: \Gr_{X^I}\hook{} \Gr_{G, I\sqcup I}$ is the closed immersion corresponding to 
the diagonal $X^I\to X^I\times X^I$. We underline that there is no shift in the definition of $\ast_I$.
 
 For example, if $G=\Spec k$ then $\Gr_{G, X^I}=X^I$ and $\Sph_{G, I}=Shv(X^I)$. The $\ast_I$-monoidal structure that we get on $\Sph_{G,I}$ is just the $\otimes^!$-monoidal structure. 

\sssec{}
Recall that $\bar C_{X^I}\in CAlg(Shv(X^I)-mod)$, so 
$$
C_{X^I}\otimes_{Shv(X^I)} \Sph_{G,I}\in Alg(Shv(X^I)-mod),
$$ 
cf. Section~\ref{Sect_C.0.3} for the corresponding structure on $\Fun_{Shv(X^I)}(C_{X^I}, \Sph_{G, I})$. 

 The structure on $A$ of an algebra in $C$ gives for $\cC(\cT_A)_{X^I}$ a structure of an algebra in $C_{X^I}\otimes_{Shv(X^I)} \Sph_{G,I}$. 

 For example, for $I$ consisting of one element we get $\cC(\cT_A)_X=\cT_A[1]$. The product on $\cC(\cT_A)_X$ is the composition
$$
\cC(\cT_A)_X\ast_I \cC(\cT_A)_X\,\iso\; \bar i^!(\cT_A\ast_X\cT_A)[2]\,\iso\, \cT_{A\otimes A}[1]\toup{mult} \cT_A[1]=\cC(\cT_A)_X,
$$   
where $mult$ comes by functoriality from the product map $A\otimes A\to A$.  

 According to Section~\ref{Sect_C.0.3}, this equips the functor (\ref{functor_Sat_G_I}) with a right-lax monoidal structure. By (\cite{Ras2}, 6.33.3) this right-lax structure is strict. 
 
\begin{Lm} 
\label{Lm_6.3.13}
View $\cC(\cT_A)_{X^I}$ as a mere object of $\Sph_{G, I}$. The functor $\Sph_{G, I}\to \Sph_{G, I}$, $\cB\mapsto \cB\ast_I \cC(\cT_A)_{X^I}$ is t-exact.
\end{Lm}
\begin{proof}
As in (\cite{Ras2}, Lemma~6.34.2). One should also invoke (\cite{Ga_central}, Proposition~6) when proving that this functor is left t-exact.
\end{proof}

For example, if $u: X^I\hook{}\Gr_{G, X^I}$ is the unit section then 
$$
(u_*\IC_{X^I})\ast_I \cC(\cT_A)_{X^I}\,\iso\, \cC(\cT_A)_{X^I}[-\mid I\mid]
$$ 

\sssec{} Write 
\begin{equation}
\label{functor_coun_Sect_6.3.14} 
 \counit: C_{X^I}\otimes_{Shv(X^I)} C_{X^I}\to Shv(X^I)
\end{equation} 
for the counit coming from the self-diality on $C$, hence also on $C_{X^I}$. For example, over $\oo{X}{}^I$ this is the functor 
\begin{equation}
\label{functor_coun_over_open}
C^{\otimes I}\otimes Shv(\oo{X}{}^I)\otimes_{Shv(X^I)} C^{\otimes I}\otimes Shv(\oo{X}{}^I)\to Shv(\oo{X}{}^I)
\end{equation} 
given by taking the invariants under the diagonal action of $\check{G}^I$ and tensoring by  $\id: Shv(\oo{X}{}^I)\to Shv(\oo{X}{}^I)$. The following is obtained as in (\cite{Ras2}, 6.34.1). 
\begin{Lm} 
\label{Lm_6.3.14}
The functor $\Sat_{G, I}: C_{X^I}\to \Sph_{G, I}$ is t-exact. It sends $\cF\in C_{X^I}$ to $\counit(\cF\otimes_{Shv(X^I)} \cC(\cT_A)_{X^I})$. 
\end{Lm}

For example, for $I=\{*\}$ and $V\in\Rep(\check{G})$ one gets
$\Sat_{G, I}(V\otimes\IC_X)\,\iso\, (V\otimes \cT_A)^{\check{G}}$. 

\begin{proof}[Proof of Lemma~\ref{Lm_6.3.14}]
The functor $\Sat_{G, I}$ writes as the composition
$$
C_{X^I} \to C_{X^I}\otimes_{Shv(X^I)} \Sph_{G, I}\,\toup{\cdot\ast_I \cC(\cT_A)_{X^I}}\; C_{X^I}\otimes_{Shv(X^I)} C_{X^I}\otimes_{Shv(X^I)}\Sph_{G, I}\toup{\counit} \Sph_{G, I},
$$
where the first map is $\cF\mapsto \cF\otimes_{Shv(X^I)} u_*\omega_{X^I}$. Here $u: X^I\to \Gr_{G, X^I}$ is the unit section, and $u_*\omega_{X^I}$ is the unit of the monoidal category $\Sph_{G, I}$. 

 The functor (\ref{functor_coun_Sect_6.3.14}) is left t-exact. Indeed, we may stratify $X^I$ by $\vartriangle^{(I/J)}(\oo{X}{}^J)$ for $J\in Q(I)$ and argue by devissage. By factorization, the claim follows from the fact that (\ref{functor_coun_over_open}) is left t-exact. Indeed, the counit $C\otimes C\to \Vect$ is left t-exact, and $Shv(\oo{X}{}^I)\otimes_{Shv(X^I)} Shv(\oo{X}{}^I)\to Shv(X^I)$ is the functor $j_*: Shv(\oo{X}{}^I)\to Shv(X^I)$ for the open immersion $j: \oo{X}{}^I\to X^I$, which is also left t-exact. Using Lemma~\ref{Lm_6.3.13}, this shows that $\Sat_{G, I}$ is left t-exact. 
 
  By Remark~\ref{Rem_Loc_is_t-exact} and Lemma~\ref{Lm_3.5.2_generated_under_colim}, $C_{X^I}^{\le 0}\subset C_{X^I}$ is the smallest full full subcategory containing $\Loc(V\otimes K)$ for $V\in (C^{\otimes I})^c\cap (C^{\otimes I})^{\le 0}$, $K\in Shv(X^I)^c\cap Shv(X^I)^{\le 0}$, closed under extensions and colimits. So, it suffices to show that for such $V, K$ one has $\Sat_{G, I}\Loc(V\otimes K)\in (\Sph_{G, I})^{\le 0}$. By devissage, we may assume $V\in (C^{\otimes I})^{\heartsuit}$. The proof of (\cite{Ras2}, 6.34.1) shows that $\Sat_{G, I}\Loc(V\otimes\omega_{X^I})$ is concentrated in cohomological degree $-\mid I\mid$. The claim follows, as the tensor product $K\otimes_{Shv(X^I)} \Sat_{G, I}\Loc(V\otimes\omega_{X^I})$ is then placed in perverse degrees $\le 0$. Indeed, the graph $\Gr_{G, X^I}\to \Gr_{G, X^I}\times X^I$ of the projection $\Gr_{G, X^I}\to X^I$ is a locally complete intersection of codimension $\mid I\mid$. 
\end{proof} 

\begin{Lm} The commutative chiral product on $\Fact(C)$ is compatible with the exterior convolution on $\Sph_{G,\Ran}$. Namely, for a map $I\to J$ in $fSets$, once a linear order on $J$ is chosen, one has a commutativity datum for the diagram
$$
\begin{array}{ccc}
\underset{j\in J}{\boxtimes} C_{X^{I_j}} & \to & C_{X^I}\\
\downarrow\lefteqn{\scriptstyle \underset{j\in J}{\boxtimes}\Sat_{G, I_j}} && \downarrow\lefteqn{\scriptstyle \Sat_{G, I}}\\
\underset{j\in J}{\boxtimes}  \Sph_{G, I_j} & \toup{c_{ext}} & \Sph_{G, I},
\end{array}
$$
where $c_{ext}$ is the corresponding exterior convolution. 
\end{Lm}
\begin{proof}
Since the functor $\Loc: \underset{i\in I}{\boxtimes} C(X)\to C_{X^I}$ generates $C_{X^I}$ under colimits by Lemma~\ref{Lm_3.5.2_generated_under_colim}, it suffices to prove our claim in the particular case of $\id: I\to I=J$. Recall that $C^{\otimes I}$ is generated by $\underset{i\in I}{\boxtimes} V_i$, where $V_i\in C^{\heartsuit}$ are irreducible finite dimensional representations of $\check{G}$. Let $V_i$ for $i\in I$ be such representations. It suffices to check that both compositions in the diagram
$$
\begin{array}{ccc}
\underset{i\in I}{\boxtimes} C_X & \toup{\Loc} & C_{X^I}\\
\downarrow\lefteqn{\scriptstyle \underset{i\in I}{\boxtimes}\Sat_{G, X}} && \downarrow\lefteqn{\scriptstyle \Sat_{G, I}}\\
\underset{i\in I}{\boxtimes}  \Sph_{G, X} & \toup{c_{ext}} & \Sph_{G, I}
\end{array}
$$
are isomorphic when evaluated on $\underset{i\in I}{\boxtimes} (V_i\otimes\IC_X)$. We have $\Sat_{G, X}(V_i\otimes \IC_X)\,\iso\, \cT_{V_i}$. Write $\bar j: \Gr_{G, X^I}\times_{X^I} \oo{X}{}^I\hook{} \Gr_{G, X^I}$ for the open immersion. As in Section~\ref{Sect_7.3.3_now}, $c_{ext}(\underset{i\in I}{\boxtimes} \cT_{V_i})$ is the intermediate extension under $\bar j$. Both morphisms of sheaves of categories $\Loc$, $c_{ext}$ becomes equvalences after the base change by $\oo{X}{}^I\hook{} X^I$. 

 The object $\Loc(\underset{i\in I}{\boxtimes} (V_i\otimes\IC_X))\in C_{X^I}$ lies in $C_{X^I}$ and is the intermediate extension of its restriction under $\bar j$. This follows from the description of the t-structure in Lemma~\ref{Lm_7.2.5_t-structure}. Indeed, $\oblv_{X^I}(\Loc(\underset{i\in I}{\boxtimes} (V_i\otimes\IC_X))$ identifies with $\IC_{X^I}$ tensored by some non-zero vector space. 
 
 Thus, it suffices to establish the desired isomorphism over $\Gr_{G, X^I}\times_{X^I} \oo{X}{}^I$. Over that locus it follows from the compatibility of $\Sat_{G, I}$ with factorization.
\end{proof}

\sssec{} Let $\epsilon\in Z(\check{G})$ be the image of $-1$ under $2\check{\rho}:\Gm\to \check{T}$. Recall that the center $Z(\check{G})$ of $\check{G}$ acts on $C=\Rep(\check{G})$ by automorphisms of the identity functor (as a symmetric monoidal functor). We associate to it the category $C^{\epsilon}\in CAlg(\DGCat_{cont})$ as in Section~\ref{Sect_Twist}. 

 Note that $\Fact(C)\,\iso\,\Fact(C^{\epsilon})$ in $Shv(\Ran)-mod$. The latter equivalence does not respect the factorization structures on these categories. 

The description of $\Sat_{G, I}$ in Lemma~\ref{Lm_6.3.14} gives the following (see also \cite{GLys}, 9.2). The functor $Sat_{G, \Ran}$ preserves the weak factorization structures. Namely, for a map $I\toup{\phi} J$ in $fSets$ the diagram commutes
$$
\begin{array}{ccc}
C^{\epsilon}_{X^I}\mid_{X^I_{\phi, d}} & \overset{\sim}{\gets} & (\underset{j\in J}{\boxtimes} C^{\epsilon}_{X^{I_j}})\mid_{X^I_{\phi, d}}\\
\downarrow\lefteqn{\scriptstyle  Sat_{G, I}} && \downarrow\lefteqn{\scriptstyle \underset{j\in I}{\boxtimes} Sat_{G, I_j}}\\
\Sph_{G, I}\mid_{X^I_{\phi, d}} & \gets & (\underset{j\in J}{\boxtimes} \Sph_{G, I_j})\mid_{X^I_{\phi, d}}  
\end{array}
$$
In the case of $\cD$-modules the low horizontal arrow in this square is an equivalence (and
in the constructible context this is maybe not true). 

\sssec{} 
\label{Sect_ext_convolution_precisely}
For the convenience of the reader, here is a precise definition of the exterior convolution on $\Sph_{G,\Ran}$. 

 For $S\in\Sch^{aff}$ and $\cJ\in \Ran(S)$ write $\Gamma_{\cJ}$ for the union of the graphs $\Gamma_i\subset S\times X$ for $i\in \cJ$. Write $\hat D_{\cJ}$ for the formal completion of $S\times X$ along $\Gamma_{\cJ}$ viewed as a formal scheme. Let $D_{\cJ}$ be the affine scheme obtained from $\hat D_{\cJ}$, the image of $\hat D_{\cJ}$ under $\colim: \Ind(\Sch^{aff})\to \Sch^{aff}$. Write also $\oo{D}_{\cJ}=D_{\cJ}-\Gamma_{\cJ}$.
 
 Recall that $\gL^+(G)_{\Ran}\backslash \Gr_{G,\Ran}$ is the prestack sending $S\in\Sch^{aff}$ to the groupoid of collections: $\cJ\in\Ran(S)$, $G$-torsors $\cF_G, \cF'_G$ on $D_{\cJ}$ and an isomorphism $\cF_G\,\iso\, \cF'_G\mid_{\oo{D}_{\cJ}}$. 
 
 Write $\Conv_{G,\Ran^2}$ for the prestack sending $S\in\Sch^{aff}$ to the collection $\cJ,\cJ'\in\Ran(S)$, $\cG$-torsors $\cF_G, \cF'_G, \cF''_G$ on $D_{\cJ\cup\cJ'}$ with isomorphisms $\beta: \cF_G\,\iso\,\cF'_G\mid_{D_{\cJ\cup\cJ'}-\Gamma_{\cJ}}$, $\beta': \cF'_G\,\iso\,\cF''_G\mid_{D_{\cJ\cup\cJ'}-\Gamma_{\cJ'}}$. We have the diagram
$$
(\gL^+(G)_{\Ran}\backslash \Gr_{G,\Ran})\times (\gL^+(G)_{\Ran}\backslash \Gr_{G,\Ran})\getsup{p}\Conv_{G,\Ran^2} \toup{m} \gL^+(G)_{\Ran}\backslash \Gr_{G,\Ran},
$$
where the map $p$ sends the above point to the collection: 
$$
(\cJ, \cF_G\mid_{D_{\cJ}}, \cF_G\mid_{D_{\cJ}}, \bar\beta: \cF_G\,\iso\, \cF'_G\mid_{\oo{D}_{\cJ}})\in \gL^+(G)_{\Ran}\backslash \Gr_{G,\Ran},
$$
$$
(\cJ', \cF'_G\mid_{D_{\cJ'}}, \cF''_G\mid_{D_{\cJ'}}, \bar\beta': \cF_G\,\iso\, \cF'_G\mid_{\oo{D}_{\cJ'}})\in \gL^+(G)_{\Ran}\backslash \Gr_{G,\Ran}
$$
Here $\bar\beta$ is the restriction of $\beta$, and $\bar\beta'$ is the restriction of $\beta'$. 

 The map $m$ sends the above point to 
$$
(\cJ\cup\cJ', \cF_G, \cF''_G, \; \beta'\beta: \cF_G\,\iso\, \cF'_G\mid_{\oo{D}_{\cJ\cup \cJ'}}).
$$ 
The map $m$ is schematic and proper, so $m_*=m_!$. The exterior convolution sends $K, K'\in \Sph_{G,\Ran}$ to  
$$
K\star K'=m_*p^!(K\boxtimes K').
$$
This defines a non-unital monoidal category $(\Sph_{G,\Ran},\star)$.

\sssec{} Write $(Shv(\Ran),\star)$ for the convolution monoidal structure on $Shv(\Ran)$. Consider the diagram
$$
\Ran\;\getsup{q}\; \gL^+(G)_{\Ran}/\Ran\;\toup{i} \;\gL^+(G)_{\Ran}\backslash\Gr_{G,\Ran},
$$
where $i$ is the closed immersion of the unit section. It is easy to see that the functor $i_*q^!: (Shv(\Ran),\star)\to (\Sph_{G,\Ran},\star)$ is non-unital monoidal.

\appendix
\section{Sheaves of categories}
\label{Sect_Sheaves of categories}

\ssec{} In this appendix we recall some basics of the theory of sheaves of categories attached to any of our sheaf theory from \cite{GLys}, which is a right-lax symmetric monoidal functor  $Shv: (\Sch^{aff}_{ft})^{op}\to \DGCat_{cont}$, $S\mapsto Shv(S)$, $(S_1\toup{f} S_2)\mapsto (f^!: Shv(S_2)\to Shv(S_1))$. 

\sssec{} The right Kan extension of $Shv$ along $(\Sch^{aff}_{ft})^{op}\hook{} \PreStk_{lft}^{op}$ is also denoted $Shv: \PreStk_{lft}^{op}\to\DGCat_{cont}$. The latter functor preserves limits. 

  Consider the functor 
$$
(\Sch^{aff}_{ft})^{op}\to 1-\Cat, \; S\mapsto Shv(S)-mod
$$ 
It sends $S'\to S$ to the functor
$$
Shv(S)-mod\to Shv(S')-mod, \; M\mapsto M\otimes_{Shv(S)} Shv(S')
$$
  
  We right-Kan-extend it along $(\Sch^{aff}_{ft})^{op}\subset (\PreStk_{lft})^{op}$ to a functor \begin{equation}
\label{def_ShvCat}
ShvCat: (\PreStk_{lft})^{op}\to 1-\Cat, \; Y\mapsto ShvCat(Y).
\end{equation} 
For a morphism $f: Y'\to Y$ in $\PreStk_{lft}$ this gives a restriction functor denoted 
\begin{equation}
\label{functor_f^!_for_sheaves_of_cat}
f^!: ShvCat(Y)\to ShvCat(Y')
\end{equation} 
or by $C\mid_{Y'}$ for $C\in ShvCat(Y)$. Note that (\ref{def_ShvCat}) preserves limits. 

 So, for $C\in ShvCat(Y)$ and $S\in(\Sch^{aff}_{ft})_{/Y}$ we get $\Gamma(S, C)\in Shv(S)-mod$, and for $S'\to S$ in $\Sch_{ft}^{aff}$ the equivalence 
$$
\Gamma(S, C)\otimes_{Shv(S)} Shv(S')\,\iso\, \Gamma(S', C)
$$ 
is a part of data of $C$. 

\index{Sheaf of categories on a prestack}

 For $Y\in\PreStk_{lft}$, $ShvCat(Y)$ admits colimits. In the case of $\cD$-modules it also admits limits, because for $T\to S$ in $\Sch_{ft}$, $Shv(T)$ is dualizable as a $Shv(S)$-module category in this case. In the constructible context this is not clear. 
 
\sssec{}  The category $ShvCat(Y)$ carries a symmetric monoidal structure given by the compo\-nent-wise tensor product. Namely, the limit
$$
ShvCat(Y)\,\iso\,  \lim_{(S\to Y)\in ((\Sch^{aff}_{ft})/Y)^{op}} Shv(S)-mod
$$ 
can be understood in $CAlg(1-\Cat)$. Write $Shv_{/Y}$ for the unit of $ShvCat(Y)$. This is the sheaf of categories given by $\Gamma(S, Shv_{/Y}))=Shv(S)$ for $S\in(\Sch^{aff}_{ft})_{/Y}$. 
 
\sssec{} 
\label{Sect_A.1.3_now}
The global sections are defined in the usual way. Namely, given $Y\in\PreStk_{lft}$ and $C\in ShvCat(Y)$, the right Kan extension of the functor 
$$
(\Sch^{aff}_{/Y})^{op}\to \DGCat_{cont}, \; (S\to Y)\mapsto \Gamma(S, C)
$$ 
along $(\Sch^{aff}_{/Y})^{op}\hook{} ({\PreStk_{lft}}_{/Y})^{op}$ defines a functor 
$$
\Gamma(\cdot, C): ({\PreStk_{lft}}_{/Y})^{op}\to\DGCat_{cont}, (Z\to Y)\mapsto \Gamma(Z, C)
$$

 For $h: Z\to Z'$ in ${\PreStk_{lft}}_{/Y}$ we also write $h^!_C: \Gamma(Z', C)\to \Gamma(Z, h^!C)$ for the corresponding restriction functor. The functor $\Gamma(\cdot, C)$ preserves limits.
 
 
\sssec{} The functor $\Gamma(Y, \cdot): ShvCat(Y)\to \DGCat_{cont}$ is right-lax symmetric monoidal, so naturally upgrades to a functor 
$$
\Gamma^{enh}_Y: ShvCat(Y)\to Shv(Y)-mod
$$ 
The latter functor is also right-lax symmetric monoidal. 
  
  Namely, for $F,F'\in ShvCat(Y)$, 
$$
\Gamma(Y, F\otimes F')\,\iso\, \lim_{S\to Y}\Gamma(S, F\otimes F')\,\iso\, \lim_{S\to Y} \Gamma(S, F)\otimes_{Shv(S)} \Gamma(S, F')
$$
For each $S\to Y$ in $(\Sch^{aff}_{ft})_{/Y}$ we have the pojection
$$
\Gamma(Y, F)\otimes \Gamma(Y, F')\to \Gamma(S, F)\otimes\Gamma(S, F')\to  \Gamma(S, F)\otimes_{Shv(S)} \Gamma(S, F')
$$
they are compatible with the transition maps, so yield the desired morphism 
$$
\Gamma(Y, F)\otimes \Gamma(Y, F')\to \Gamma(Y, F\otimes F').
$$ 
It factors through $\Gamma(Y, F)\otimes_{Shv(Y)} \Gamma(Y, F')\to \Gamma(Y, F\otimes F')$.

\sssec{} 
\label{Sect_A.1.5}
The functor $\Gamma^{enh}_Y$ has a left adjoint 
$$
\Loc_Y: Shv(Y)-mod\to ShvCat(Y)
$$ 
sending $C$ to the sheaf of categories whose sections over $S\to Y$ are $\Gamma(S, \Loc_Y(C))=C\otimes_{Shv(Y)} Shv(S)$. The functor $\Loc_Y$ is symmetric monoidal. 

 We say that $Y\in\PreStk_{lft}$ is 1-affine (for our sheaf theory!) if $\Loc_Y$ is an equivalence. It is known that each ind-scheme of ind-finite type is 1-affine. 
 
 One may show also that for $S, T\in \Sch_{ft}$ and a Zariski cover $T\to S$ the base change functor $Shv(S)-mod\to Shv(T)-mod$ is conservative. 

\sssec{}  Let $Y\in \PreStk_{lft}$, $C\in ShvCat(Y)$. One may show that the functor 
$$
\Gamma(\cdot, C): (\Sch_{ft}/Y)^{op}\to \DGCat_{cont}
$$ 
satisfies Zarizki and proper descent, hence h-descent. Recall that $Shv: (\PreStk_{lft})^{op}\to \DGCat_{cont}$ satisfies the proper descent, it also satisfies the \'etale descent (as in \cite{Ga1}). 

\sssec{} If $C\in Shv(Y)-mod$ is dualizable in $Shv(Y)-mod$ then the natural map $C\to \Gamma(Y, \Loc_Y(C))$ is an isomorphism, as in (\cite{Ga1}, 1.3.5). 

\sssec{} For $Y, Z\in \PreStk_{lft}$ we have the functor 
$$
ShvCat(Y)\times ShvCat(Z)\to ShvCat(Y\times Z)
$$ 
sending $(C,D)$ to $C\boxtimes D:=\pr_1^!C\otimes \pr_2^!D$ for 
$$
Y\getsup{\pr_1} Y\times Z\toup{\pr_2} Z
$$    

\sssec{} Let $C\in ShvCat(Y), D\in ShvCat(Z)$.  We have a natural functor 
\begin{equation}
\label{ext_product_sheaves}
\Gamma(Y, C)\otimes \Gamma(Z, D)\to \Gamma(Y\times Z, C\boxtimes D)
\end{equation}
Indeed, $Y\,\iso\,\mathop{\colim}\limits_{S\to Y} S$ in $\PreStk_{lft}$, the colimit taken over $(\Sch^{aff}_{ft})_{/Y}$. So, 
$$
Y\times Z\,\iso\, \mathop{\colim}\limits_{(S\to Y, T\to Z)} S\times T
$$ 
in $\PreStk_{lft}$, the colimit taken over $(\Sch^{aff}_{ft})_{/Y}\times (\Sch^{aff}_{ft})_{/Z}$. So, 
$$
\Gamma(Y\times Z, C\boxtimes D)\,\iso\, \lim_{(S\to Y, T\to Z)} \Gamma(S\times T, C\boxtimes D)
$$
Now $\Gamma(S\times T, C\boxtimes D)\,\iso\, (\Gamma(S, C)\otimes\Gamma(T, D))\otimes_{Shv(S)\otimes Shv(T)} Shv(S\times T)$, and we have a natural map 
$$
\Gamma(Y, C)\otimes \Gamma(Z, D)\to \Gamma(S, C)\otimes \Gamma(T, D)\to \Gamma(S\times T, C\boxtimes D)
$$ 
These maps are compatible with the transition maps, so define the morphism (\ref{ext_product_sheaves}). 

\sssec{} For a morphism $f: Y'\to Y$ in $\PreStk_{lft}$ it is not clear in general if (\ref{functor_f^!_for_sheaves_of_cat}) has a right adjoint. 

If $Y\in\PreStk_{lft}$ is 1-affine then this right adjoint exist and is denoted 
$$
coind_f: ShvCat(Y')\to ShvCat(Y)
$$ 
Indeed, $f^!$ factors as $Shv(Y)-mod\to Shv(Y')-mod\toup{\Loc_{Y'}} ShvCat(Y')$, where the first functor sends $C$ to $Shv(Y')\otimes_{Shv(Y)} C$. So its right adjoint is the composition
$$
ShvCat(Y')\toup{\Gamma^{enh}_{Y'}}\; Shv(Y')-mod\to Shv(Y)-mod,
$$
where the second arrow is the restriction of scalars along $Shv(Y)\toup{f^!} Shv(Y')$. 

 For $S\in(\Sch_{ft})_{/Y}$ we have a natural map 
$$
\Gamma(Y', C)\otimes_{Shv(Y)} Shv(S)\to \Gamma(Y'\times_Y S, C).
$$ 
This is not an equivalence in general. Already in the case $Y=\Spec k$ and $C=Shv_{/Y'}$ the above map $Shv(Y')\otimes Shv(S)\to Shv(Y'\times S)$ is not an equivalence in the constructible setting.

\sssec{} 
\label{Sect_A.1.11}
Let $f: Y\to Z$ be a map in $\PreStk_{lft}$ and $C\in ShvCat(Z)$. Assume $Y, Z$ are pseudo-indschemes, and $f: Y\to Z$ is pseudo-indproper in the sense of (\cite{R}, 7.15.1). Then $f^!_C$ admits a left adjoint $f_{!, C}: \Gamma(Y, f^! C)\to \Gamma(Z, C)$ in $\DGCat_{cont}$.

Namely, let $Z=\colim_{j\in J} Z_j$, where the transition maps $\alpha: Z_j\to Z_{j'}$ for $j\to j'$ in $J$ are proper, and each $Z_i\in \Sch_{ft}$ is separated. Recall that $Z_i$ is 1-affine, and we have the adjoint pair $\alpha_!: Shv(Z_j)\rightleftarrows Shv(Z_{j'}): \alpha^!$ in $Shv(Z_{j'})-mod$. Tensoring by $\Gamma(Z_{j'}, C)$, we get an adjoint pair 
$$
\alpha_{!, C}: \Gamma(Z_j, C)\rightleftarrows \Gamma(Z_{j'}, C): \alpha^!_C
$$ 
Assume now $I\to J$ is a diagram, and $Y=\colim_{i\in I} Y_i$, here $Y_i\in \Sch_{ft}$ is  separated, and the transition maps $Y_i\to Y_{i'}$ are proper. Then $\Gamma(Y, C)\,\iso\colim_{i\in I} \Gamma(Y_i, C)$. The desired functor $f_{!, C}$ is obtained from the compatible system of functors 
$\beta_{!, C}: \Gamma(Y_i, C)\to \Gamma(Z_{j(i)}, C)$. Here the corresponding morphism $\beta: Y_i\to Z_{j(i)}$ is proper. 

 Compare with (\cite{Ly}, 9.2.21).

\section{Generalities about sheaf theories}
\label{appendix_some_generalities}

\ssec{} We collect in this section some properties of any sheaf theorie $Shv$ from \cite{GLys}. 

\sssec{} 
\label{Sect_B.1.1}
Let $S\in\Sch_{ft}$. Recall that $Shv(S)$ is compactly generated. We view it by default as equipped with $\otimes^!$-symmetric monoidal structure.

Consider a closed immersion $i: Y\to Z$ in $\Sch_{ft}$. Then $Shv(Y)$ is dualizable and canonically self-dual in $Shv(Z)-mod$ by (\cite{Ly}, 3.1.10), it is a retract of $Shv(Y)$. 
If in addition $Z'\to Z$ is a map in $\PreStk_{lft}$ and $Y'=Y\times_Z Z'$ then the natural map $Shv(Z')\otimes_{Shv(Z)} Shv(Y)\to Shv(Y')$ is an equivalence by (\cite{Ly4}, 0.3.1). 

Consider an open embedding $j: U\to Z$ in $\Sch_{ft}$. Then $Shv(U)$ is dualizable and canonically self-dual in $Shv(Z)-mod$ by (\cite{Ly}, 3.1.10). For a map $Z'\to Z$ in $\PreStk_{lft}$ and $U'=U\times_Z Z'$ the natural map $Shv(U)\otimes_{Shv(Z)} Shv(Z')\to Shv(U')$ is an equivalence by (\cite{Ly4}, 0.3.2). 

\sssec{} 
\label{Sect_B.1.2}
Let $C\in Shv(S)-mod$, $U\subset S$ be an open subscheme. Then one has canonically 
$$
\Fun_{Shv(S)}(C, Shv(S))\otimes_{Shv(S)} Shv(U)\,\iso\,\Fun_{Shv(U)}(C\otimes_{Shv(S)} Shv(U), Shv(U))
$$

\begin{Lm} 
\label{Lm_dualizability_is_local}
Let $C\in Shv(S)-mod$.  Then $C$ is dualizable in $Shv(S)-mod$ iff it is dualizable locally on $S$ is Zariski topology.
\end{Lm}
\begin{proof}
Consider the canonical pairing $C\otimes_{Shv(S)}\Fun_{Shv(S)}(C, Shv(S))\to Shv(S)$. For each $D\in Shv(S)-mod$ it induces a map
$$
D\otimes_{Shv(S)} C\to \Fun_{Shv(S)}(\Fun_{Shv(S)}(C, Shv(S)), D)
$$
Recall that $C$ is dualizable in $Shv(S)-mod$ iff the latter map is an equivalence for any $D\in Shv(S)-mod$. This condition is local in Zariski topology, cf. Section~\ref{Sect_B.1.2}.
\end{proof}

\sssec{} In the case of $\cD$-modules the following holds. Given $C\in Shv(S)-mod$, $C$ is dualizable in $Shv(S)-mod$ iff it is dualizable in $\DGCat_{cont}$. This property is not known in the constructible context.

\sssec{} 
Let $S\in\Sch_{ft}$. Let $C\in Shv(S)-mod$. Following S. Raskin (\cite{Ras2}, B.5.1), we adapt the following.

\begin{Def} An object $c\in C$ is ULA iff the functor $\und{\HOM}_C(c,-): C\to Shv(S)$ is continuous and its right-lax $Shv(S)$-structure is strict. Here $\und{\HOM}_C$ denotes the inner hom with respect to the action of $Shv(S)$. Since $Shv(S)$ is presentable, this inner hom automatically exists. Moreover, for any $x\in C$, $K\in Shv(S)$ we have a canonical map 
\begin{equation}
\label{map_for_ULA}
K\otimes^! \und{\HOM}_C(c,x)\to \und{\HOM}_C(c, x\otimes K)
\end{equation}
coming from the natural morphism $K\otimes^! \und{\HOM}_C(c,x)\otimes c\to x\otimes K$. The above strictness requirement means that (\ref{map_for_ULA}) is an isomorphism for any $K\in Shv(S)$. 
\end{Def}

\index{ULA objects of a category}

The following is obtained as in (\cite{Ras2}, B.5.1).
\begin{Rem} 
\label{Rem_3.6.3_now}
i) Let $C\in Shv(S)-mod$, $c\in C$. If $c$ is ULA then for any $M\in Shv(S)-mod$ and $m\in M^c$, the product $c\boxtimes_{Shv(S)} m$ is compact in $C\otimes_{Shv(S)} M$. \\
ii) If $L: \cC\to \cD$ is a map in $Shv(S)-mod$ admitting a $Shv(S)$-linear continuous right adjoint then $L$ sends ULA objects to ULA objects.
\end{Rem}

\begin{Lm} Let $j: U\hook{} S$ be an open immersion, $S\in\Sch_{ft}$, $C\in Shv(S)-mod$, $F\in C$ be ULA over $S$. Let $G\in Shv(U)$ such that $j_! G$ is defined. Then 
$$
j_!(G)\otimes^! F\,\iso\, j_!(G\otimes^! j^!F).
$$ 
In particular, for the partially defined left adjoint $j_!: C_U:=C\otimes_{Shv(S)} Shv(U)\to C$ to $j^!$ the object $j_!j^!F$ is defined. 
\end{Lm}
\begin{proof} As in \cite{Ras2}, B.5.3. For $K\in C$, 
$$
j_*j^!\und{\HOM}_C(F, K)\,\iso\, (j_*\omega_U)\otimes^! \und{\HOM}_C(F,K)\,\iso\,  \und{\HOM}_C(F, (j_*\omega_U)\otimes^! K)
$$ 
in $C$. This gives $j^!\und{\HOM}_C(F, K)\,\iso\, \und{\HOM}_{C_U}(j^! F, j^! K)$.
Now for any $\tilde F\in C$, 
\begin{multline*}
\HOM_C(j_!G\otimes^! F, \tilde F)\,\iso\, \HOM_{Shv(S)}(j_!G, \und{\HOM}_C(F, \tilde F))\,\iso\,
\HOM_{Shv(U)}(G, j^!\und{\HOM}_C(F, \tilde F))\\ \iso\, \HOM_{Shv(U)}(G, \und{\HOM}_{C_U}(j^!F, j^!\tilde F))\,\iso\, \HOM_{C_U}(G\otimes^! j^!F, j^!\tilde F)\,\iso\,\HOM_C(j_!(G\otimes^! j^!F), \tilde F)
\end{multline*}
as desired. 
\end{proof}

 For $C\in Shv(S)-mod$ write $C^{ULA}\subset C$ for the full subcategory of ULA objects. This is a stable subcategory of $C$, closed under the action of $Shv(S)^{dualizable}$. Here $Shv(S)^{dualizable}\subset Shv(S)$ is the full subcategory of dualizable objects. Inspired by (\cite{Ras2}, B.6.1), we adapt the following.
\begin{Def}
\label{Def_ULA_module-category}
Let $S\in \Sch_{ft}$, $C\in Shv(S)-mod$. Say that $C$ is ULA if $C$ is generated as a $Shv(S)$-module category by ULA objects. That is, $C$ is generated by objects of the form $c\otimes m$ with $c\in C^{ULA}$ and $m\in Shv(S)^c$.
\end{Def}

\index{ULA category over some base}

 We have $C^{ULA}\subset C^c$. Indeed, for $\cD$-modules this is (\cite{Ras2}, B.4.2), and in the constructible context this follows from the fact that $\omega_S$ is compact. Moreover, If $C$ is ULA over $Shv(S)$ then for any $c\in C^{ULA}, K\in Shv(S)^c$, $K\otimes c\in C^c$. We see also that if $C$ is ULA then $C$ is compactly generated.  

\begin{Pp} 
\label{Pp_3.6.6}
Let $\cC\in Shv(S)-mod$ be ULA and $F: \cC\to\cD$ be a map in $Shv(S)-mod$. Then $F$ has a $Shv(S)$-linear continuous right adjoint iff $F(\cC^{ULA})\subset \cD^{ULA}$.

 More generally, assume $C_0\subset \cC^{ULA}$ is a full subcategory such that the objects of the form $c\otimes_{Shv(S)} F$ for $c\in C_0, F\in Shv(S)^c$ generate $\cC$. If $F(C_0)\subset \cD^{ULA}$ then $F$ has a $Shv(S)$-linear continuous right adjoint.
\end{Pp}
\begin{proof} For $\cD$-modules this is (\cite{Ras2}, B.7.1), and the proof of \select{loc.cit.} holds for constructible context also.
\end{proof}

\begin{Pp} 
\label{Pp_3.7.8_devissage_using_ULA_preservation}
i) Let $S\in\Sch_{ft}$, $j: U\hook{} S$ be an open subscheme, the complement to the closed immersion $i: Z\to S$. Let $f: C\to D$ be a morphism in $Shv(S)-mod$. Assume $f$ admits a $Shv(S)$-linear continuous right adjoint. Then $f$ is an isomorphism iff $f$ induces equivalences
$$
C\otimes_{Shv(S)} Shv(U)\to D\otimes_{Shv(S)} Shv(U), \;\;  C\otimes_{Shv(S)} Shv(Z)\to D\otimes_{Shv(S)} Shv(Z)
$$

\smallskip\noindent
ii) If $C$ is ULA over $Shv(S)$ and $f$ preserves ULA objects then the conclusion of i) holds.
\end{Pp}
\begin{proof}
as in (\cite{Ras2}, B.8.1)
\end{proof}

\sssec{} Let $C\in Shv(S)-mod$. Since $C^{ULA}\subset C$ is closed under the action of $\Vect^c$, $\Ind(C^{ULA})\in\DGCat_{cont}$. Let $\epsilon: \Ind(C^{ULA})\to C$ be the continuous functor extending the inclusion $C^{ULA}\hook{} C$. Note that $\epsilon$ is fully faithful. Let $h: Shv(S)\otimes\Ind(C^{ULA})\to C$ be the $Shv(S)$-linear functor sending $\omega_S\boxtimes M$ to $\epsilon(M)$. 

 Note that $Shv(S)\otimes \Ind(C^{ULA})$ is compactly generated by objects of the form $K\boxtimes c$ with $K\in Shv(S)^c, c\in C^{ULA}$, and $h(K\boxtimes c)=K\otimes c\in C^c$. So, $h$ has a continuous right adjoint $h^R$. 
 
 Write $\Fun_{ex, e}((C^{ULA})^{op}, Shv(S))$ for the category of exact $\Vect^c$-linear functors from $(C^{ULA})^{op}$ to $Shv(S)$. One has canonically
\begin{multline*}
Shv(S)\otimes \Ind(C^{ULA})\,\iso\,\Fun_{cont, e}(\Ind((C^{ULA})^{\vee}, Shv(S))\\
\iso\, \Fun_{ex, e}((C^{ULA})^{op}, Shv(S)),
\end{multline*}
here $\Fun_{cont, e}$ denotes the category of exact continious $\Vect$-linear functors.  
The functor 
$$
h^R: C\to \Fun_{ex, e}((C^{ULA})^{op}, Shv(S))
$$ 
sends $c$ to the functor $c_1\mapsto \und{\HOM}_C(c_1, c)$, it is $Shv(S)$-linear. Note that $C$ is ULA over $Shv(S)$ iff $h^R$ is conservative.

 Applying $\Fun_{Shv(S)}(?, Shv(S))$ to the adjoint pair $(h, h^R)$ in $Shv(S)-mod$, we get an adjoint pair
$$
\bar h: \Ind((C^{ULA})^{op})\otimes Shv(S)\leftrightarrows \Fun_{Shv(S)}(C, Shv(S)): \bar h^R
$$
 
 Note that if $c\in C^{ULA}$ then $\bar h(c\boxtimes \omega_S)$ is the functor $c_1\mapsto\und{\HOM}_C(c, c_1)$. It is easy to see that $\bar h(c\boxtimes \omega_S)$ is a ULA object of $\Fun_{Shv(S)}(C, Shv(S))$. 
 
\begin{Lm} 
\label{Lm_C^vee(X)_is_still_ULA}
If $C$ is ULA over $Shv(S)$ then $\Fun_{Shv(S)}(C, Shv(S))$ is also ULA over $Shv(S)$.
\end{Lm}
\begin{proof}
It suffices to show that $\bar h$ generates the target under colimits, that is, $\bar h^R$ is conservative. Rewrite $\bar h^R$ as a functor $\Fun_{Shv(S)}(C, Shv(S))\to \Fun_{ex, e}(C^{ULA}, Shv(S))$ sending $f: C\to Shv(S)$ to its restriction to $C^{ULA}$. 
Let now $f\in \Fun_{Shv(S)}(C, Shv(S))$ such that $\bar h^R(f)=0$, that is, $f(c)=0$ for any $c\in C^{ULA}$. Since $C$ is generated by objects of the form $K\otimes c$ with $K\in Shv(S), c\in C^{ULA}$, we see that $f=0$. Thus, $\bar h^R$ is conservative.
\end{proof} 
 
\sssec{} 
\label{Sect_B.1.12}
We need the following fact. Let $Y$ be an ind-scheme of ind-finite type. Assume given a filtration by closed ind-subschemes $\ldots \bar Y_2\hook{}\bar Y_1\hook{}\bar Y_0=Y$. Set $Y_m=\bar Y_m-\bar Y_{m+1}$ for $m\ge 0$. Assume $\bar Y_m$ is the closure of $Y_m$ for all $m\ge 0$. So, $Y_0$ is open and dense in $Y$. Assume the inclusion $j_m: Y_m\to Y$ is affine. 

 We equip $Shv(Y)$ with the perverse t-structure. In details, assume $Y\,\iso\, \colim_{i\in I} Y_i$, where $I$ is small filtered, $Y_i\in\Sch_{ft}$ and for $i\to j$ the map $Y_i\to Y_j$ is a closed immersion. Then $Shv(Y)^{\le 0}\subset Shv(Y)$ is the smallest full subcategory containing $Shv(Y_i)^{\le 0}$ for all $i$, closed under extensions and colimits. 

Set $Y_{[m-1,m]}=\bar Y_{m-1}-\bar Y_{m+1}$. Let $\cF\in Shv(Y)^{\heartsuit}$ such that $j_m^!\cF[m]$ is in the heart of the t-structure for any $m\ge 0$. Set 
$$
\cF_m=(j_m)_*j_m^!\cF[m]\in Shv(Y)^{\heartsuit}
$$ 
The exact triangle 
$$
(j_m)_*j_m^!\cF\to \cF\mid^!_{Y_{[m-1,m]}} \to (j_{m-1})_*j_{m-1}^!\cF,
$$
which we extend to $Y$ by $*$-extension, gives the transition maps in the following sequence
$$
0\to \cF\to \cF_0\to \cF_1\to\ldots
$$
in $Shv(Y)^{\heartsuit}$. Then the latter sequence is exact.

\sssec{} 
\label{Sect_B.1.15}
Let $Y\toup{f} Z\gets Y'$ be a diagram in $\Sch_{ft}$, where $f$ is finite. Then by (\cite{Ly4}, Theorem~0.4.18), the natural functor $Shv(Y)\otimes_{Shv(Z)} Shv(Y')\to Shv(Y\times_Z Y')$ is an equivalence (compare also with Theorem~\ref{Thm_Kunneth_formula}). This is easier for $\cD$-modules than in the constructible context. This implies formally that the same K\"unneth formula holds if we only assume $Y'\in\PreStk_{lft}$. 

\sssec{} 
\label{Sect_B.1.16}
Let $f: Y\to Z$ be a finite morphism in $\Sch_{ft}$. Then $Shv(Y)$ is canonically self-dual in $Shv(Z)-mod$. 

 The counit is the functor $c: Shv(Y)\otimes_{Shv(Z)} Shv(Y)\to Shv(Z)$ sending $K_1, K_2\in Shv(Y)$ to $f_*(K_1\otimes^! K_2)$. The unit is the functor 
$$
u: Shv(Z)\to Shv(Y\times_Z Y)\,\iso\, Shv(Y)\times_{Shv(Z)} Shv(Y)
$$ 
sending $K$ to $(q^!K)\otimes^!\!\vartriangle_*\!\omega$, where $q: Y\times_Z Y\to Z$ is the projection, and $\vartriangle: Y\to Y\times_Z Y$ is the diagonal. This is well-known for $\cD$-modules, and in the constructible context this follows from the previous subsection. 

\ssec{Factorization categories} 
\label{Sect_B2_factorization}

\sssec{} For the convenience of the reader, we recall the notion of a (weak) factorization sheaf of categories on $\Ran$. 

 For $J\in fSets$ let $\Ran^J_d\subset \Ran$ be the open locus defined as follows. For $S\in\Sch_{ft}$ an $S$-point of $\Ran^J$ given by $I_j\subset \Map(S, X), j\in J$ belongs to $\Ran^J_d$ if for every $j_1\ne j_2$ and any $i_1\in I_{j_1}, i_2\in I_{j_2}$ the corresponding two maps $S\to X$ have non-intersecting images.
 
 For $I\in fSets$ we have the sum map $u_I: \Ran^I_d\to \Ran$. For a map $I\to J$ in $fSets$ we have the commutative diagram
$$
\begin{array}{ccc} 
\Ran^I_d & \toup{\prod_j u_{I_j}} & \Ran^J_d\\
& \searrow\lefteqn{\scriptstyle u_I} & \downarrow\lefteqn{\scriptstyle u_J}\\
&& \Ran
\end{array}
$$
Let $C$ be a sheaf of categories on $\Ran$. We say that it is equipped with a \select{weak factorization structure} if we are given 
\begin{itemize}
\item for any $J\in fSets$ a functor in $ShvCat(\Ran^J_d)$
\begin{equation}
\label{functor_def_of_weak_fact_cat}
\underset{j\in J}{\boxtimes} C\mid_{\Ran^J_d}\to u_J^!C
\end{equation} 
\item for a map $(I\to J)$ in $fSets$ a commutativity datum for the diagram
$$
\begin{array}{ccc}
\underset{i\in I}{\boxtimes} C\mid_{\Ran^I_d} & \to & ((\prod_j u_{I_j})^!(\underset{j\in J}\boxtimes C))\mid_{\Ran^I_d}\\
& \searrow &\downarrow\\
&& u_I^! C\mid_{\Ran^I_d}
\end{array}
$$
in $ShvCat(\Ran^I_d)$;
\item a homotopy-coherent system of compatibilities for higher order compositions.
\end{itemize}
(So, our weak factorization structure is not unital). A more precise definition is given in (\cite{R}, Section~6). 

\index{weak factorization sheaf of categories on $\Ran$}

 We say that a weak factrorization structure as above is a \select{factorization structure} if for any $J\in fSets$ the map (\ref{functor_def_of_weak_fact_cat}) is an isomorphism. 
 
\sssec{} Let $Z\to\Ran$ be a factorization prestack over $\Ran$ in the sense of (\cite{GLys}, 2.2.1) with $Z\in\PreStk_{lft}$. We define a sheaf of categories $Shv(Z)_{/\Ran}$ by a compatible system of objects $Shv(X^I\times_{\Ran} Z)\in Shv(X^I)-mod$ for $I\in fSets$. Since $\Ran\,\iso\, \underset{I\in fSets^{op}}{\colim} X^I$, and for any map $(I\to J)$ in $fSets$ the transition map $X^J\to X^I$ is a closed immersion, this is a compatible system by Section~\ref{Sect_B.1.1}. We used that
$$
ShvCat(\Ran)\,\iso\, \underset{I\in fSets}{\lim} ShvCat(X^I)
$$
and $Shv(X^I)-mod\,\iso\, ShvCat(X^I)$. 

\sssec{} The sheaf $Shv(Z)_{/\Ran}$ has a natural weak factorization structure. Namely, for $J\in fSets$ we have the factorization isomorphisms
\begin{equation}
\label{iso_factorization_of_Z_prestack}
Z^J\times_{\Ran^J} \Ran^J_d\,\iso\, Z\times_{\Ran} \Ran^J_d
\end{equation}
Now for any $S\in\Sch_{ft}$ and $S\to\Ran^J_d$ given by $I_j\subset \Map(S,X), j\in J$
we get the functor
$$
\underset{j, Shv(S)}{\otimes} Shv(S\times_{I_j, \Ran} Z)\to Shv(\underset{j, S}{\prod} (S\times_{I_j, \Ran} Z))= Shv(S\times_{\Ran^J} Z^J)\toup{(\ref{iso_factorization_of_Z_prestack})} Shv(S\times_{I,\Ran} Z).
$$
They are equipped with the corresponding higher compatibilities. 

\sssec{} In (\cite{GLys}, 2.2.3) we noticed the following. If in addition the sheaf theory is that of $\cD$-modules, assume that for any $I\in fSets$, $Shv(X^I\times_{\Ran} Z)$ is dualizable in $\DGCat_{cont}$. Then the above weak factorization structure on $Shv(Z)_{/\Ran}$ is strict. This does not seem to be the case in the constructible context.

\section{Some categorical notions}
\label{Sect_Some categorical notions}

\begin{Def} 
\label{Def_zero-cofinal}
Let $j: C^0\subset C$ be a full embedding in $1-\Cat$. We say that $j$ is \select{zero-cofinal} if for any $c\in C$ not lying in $C^0$ the category $C^0\times_C (C_{/c})$ is empty. 
\end{Def}

\index{Zero-cofinal functor}

\sssec{} 
\label{Sect_C.0.2}
Let $E\in 1-\Cat$ admit an initial object $0\in E$. Then for any $h^0: C^0\to E$ the left Kan extension $h$ of $h^0$ along $j$ exists. Moreover, for $c\in C$ not lying in $C^0$ one has $h(c)\,\iso\, 0$. 

So, if $f: C\to E$ is a functor such that for each $c\in C$ not lying in $C^0$ one has $f(c)\,\iso\, 0$ then the natural map $\underset{c\in C^0}{\colim} f(c)\to \underset{c\in C}{\colim} f(c)$ is an isomorphism (wherever either side exists).  


\begin{Lm} 
\label{Lm_C.0.3}
Let $C\in 1-\Cat$, $C_0, C_1\subset C$ be full suncategories and $C_{01}=C_0\cap C_1$. Assume each object of $C$ lies either in $C_0$ or in $C_1$. Assume the following: given a map $c\to c'$ in $C$ if $c'\in C_0$ (resp., $c'\in C_1$) then $c\in C_0$ (resp., $c\in C_1$). Then $C$ is the colimit of the diagram $C_0\gets C_{01}\to C_1$ in $1-\Cat$.
\end{Lm}
\begin{proof} The argument is due to Nick Rozenblyum\footnote{We are grateful to N. Rozenblyum for explaining us this proof.}. Write $\Cart_{/C}$ for the $\infty$-category, whose objects are cartesian fibrations $Y\to C$ over $C$, and morphisms from $Y_1\to C$ to $Y_1\to C$ are functors $Y_1\to Y_2$ over $C$ sending a $C$-cartesian arrow to a $C$-cartesian arrow. By our assumption, the inclusions $C_0\to C, C_1\to C, C_{01}\to C$ are cartesian fibrations. The forgetful functor $1-\Cat_{/C}\to C$ preserves all limits and contractible colimits. The forgetful functor $\Cart_{/C}\to 1-\Cat_{/C}$ preserves all limits and colimits. So, it suffices to show that $C$ is the colimit of the diagram $C_0\gets C_{01}\to C_1$ in $\Cart_{/C}$. 

 By straightening, one has an equivalence $\Cart_{/C}\,\iso\, \Fun(C^{op}, 1-\Cat)$. The functor $C_0\hook{} C$ corresponds to the functor $C^{op}\to 1-\Cat$ sending $c$ to $*$ (resp., to $\emptyset$) for  $c\in C_0$ (resp., $c\notin C_0$), and similarly for $C_1$. Our claim follows.
\end{proof}

\begin{Lm} 
\label{Lm_binary_products_give_contractibility}
Let $C\in 1-\Cat$ admit binary products. Then $C$ is contractible.
\end{Lm}
\begin{proof}
First, the geometric realization functor $1-\Cat\to \Spc$ extends to the 2-category $1-\Cat$ as defined in (\cite{G}, A.1, 2.4). Let $z\in C$. Consider the functor $h: C\to C$, $x\mapsto x\times z$. The projections $x\gets x\times z\to z$ define natural transofrmations $\id\gets h\to h_z$, where $h_z$ is the contant functor with value $z$. This implies our claim.
\end{proof}

\sssec{Generalization of the Day convolution} 
\label{Sect_C.0.3}
Let $A\in CAlg(\DGCat_{cont})$. Let $B,D\in \Alg(A-mod)$. Assume that the product map $m: B\otimes_A B\to B$ admits a right adjoint $m^R$ in $A-mod$, and $1_B: A\to B$ admits a right adjoint in $A-mod$. Assume $B$ 
dualizable in $A-mod$, write $B^{\vee}$ for its dual in $A-mod$. 

 Then $B$ is naturally an object of $CoCAlg(A-mod)$ with the coproduct $m^R: B\to B\otimes_A B$ and the counit $1_B^R: B\to A$. In turn, $B^{\vee}\in Alg(A-mod)$ with the product and the unit
$$
(m^R)^{\vee}: B^{\vee}\otimes_A B^{\vee}\to B^{\vee},\;\;\;\;\; (1_B^R)^{\vee}: A\to B^{\vee}.
$$ 
Since $Alg(A-mod)$ is symmetric monoidal, 
$$
\Fun_A(B, D)\,\iso\, B^{\vee}\otimes_A D\in Alg(A-mod).
$$ 

 This structure looks like the Day convolution on $\Fun_A(B,D)$. Namely, given $f_i\in \Fun_A(B,D)$ their product $f_1\ast f_2$ is obtained from $B\otimes_A B\toup{f_1\otimes f_2} D\otimes_A D\toup{m_D} D$ by applying the functor 
$$
\Fun_A(B\otimes_A B, D)\to \Fun_A(B, D)
$$ 
left adjoint to the functor $\Fun_A(B, D)\to \Fun_A(B\otimes_A B, D)$ given by composing with the product $m_B: B\otimes_A B\to B$. 
 
 The unit of $\Fun_A(B,D)$ is obtained from $1_D: A\to D$ by applying the functor 
$$
\Fun_A(A, D)\to \Fun_A(B, D)
$$
 left adjoint to the functor $\Fun_A(B, D)\to \Fun_A(A, D)$ given by the composition with $1_B: A\to B$.
 
\index{Day convolution on $\Fun_A(B,D)$} 
 
  An important observation here is that 
$$ 
\Fun_A^{rlax}(B, D)\,\iso\, \Alg(\Fun_A(B, D)),
$$
where $\Fun_A^{rlax}(B, D)$ is the category of those $A$-linear exact continuous functors, which are right-lax monoidal.   

\ssec{Twist of a symmetric monoidal category}
\label{Sect_Twist}

\sssec{} Let $A$ be a finite abelian group, $C\in CAlg(\DGCat_{cont})$. Write $\Fun_{e, cont}^{\otimes}(C,C)$ for the category of continuous exact $e$-linear symmetric monoidal endo-functors on $C$. Recall that an action of $A$ on $C$ (as a symmetric monoidal category) by automoprhisms of the identity functor is, by definition, a monoidal functor $B(A)\to \Fun_{e, cont}^{\otimes}(C,C)$.

 Let $\epsilon\in A$ be a 2-torsion element. In Section~\ref{Sect_Twist} we reformulate the construction of the category $C^{\epsilon}\in CAlg(\DGCat_{cont})$ from (\cite{GLys}, 8.2.4). 
 
\index{Twist of a symmetric monoidal category} 

\sssec{} The category $\DGCat_{cont}$ being cocomplete, it is tensored over $\Spc$ in the terminology of \cite{HA}. By (\cite{Ly}, 9.2.20) the functor $\Spc\to \DGCat_{cont}$, $Z\mapsto Z\otimes\Vect$ is symmetric monoidal and preserves colimits. So, it extends to a functor $CAlg(\Spc)\to CAlg(\DGCat_{cont})$. 

\sssec{} Set $\Gamma=\Hom(A, e^*)$. Viewing $B(A)$ as a scheme over $e$, write $\QCoh_e(B(A))$ for the corresponding category of quasi-coherent sheaves. We get $\QCoh_e(B(A))\,\iso\,\Vect^{\Gamma}$, where $\Vect^{\gamma}=\underset{\gamma\in\Gamma}\oplus \Vect\otimes\gamma$ is the category of $\Gamma$-graded complexes. Write $*$ for the convolution monoidal structure on $\QCoh_e(B(A))$. It is given by $\gamma\ast \gamma\,\iso\,\gamma$, and $\gamma\ast\gamma'=0$ for $\gamma\ne\gamma'$. 

The above $A$-action on $C$ extends to a monoidal functor $B(A)\otimes\Vect\to \Fun_{e, cont}^{\otimes}(C,C)$. It is easy to see that $B(A)\otimes\Vect\in CAlg(\DGCat_{cont})$ identifies with $(\Vect^{\Gamma}, *)$. Each $c\in C$ has a canonical grading 
$$
c\,\iso\, \underset{\gamma\in\Gamma}{\oplus} c_{\gamma}, 
$$
where $c_{\gamma}$ is the result of the action of $\gamma$ on $c$. We refer to it as $\Gamma$-grading. 

\sssec{} Consider the object $\cC\in CAlg(\Spc)$ defined in (\cite{GLys}, 4.8.2). Recall that, as a monoidal groupoid, $\cC=(\ZZ/2\ZZ)\times B(\ZZ/2\ZZ)$. The brading on $\cC$ is given by the unique bilinear form $b'$ on $\ZZ/2\ZZ$ with valued in $\ZZ/2\ZZ$ such that $b'(1,1)=1$. 

 Consider $\cC\otimes\Vect\in CAlg(\DGCat_{cont})$. We get 
$$
\cC\,\iso\, \underset{\ZZ/2\ZZ}{\oplus} B(\ZZ/2\ZZ)\otimes\Vect,
$$
we refer to this decomposition as `parity' (given by isomorphism classes of objects of $\cC$). In the other hand, we have an action of $\ZZ/2\ZZ$ by automorphisms of the identity functor on $\cC$, hence by functoriality, an action of $\ZZ/2\ZZ$ by automorphisms of the identity functor on $\cC\otimes\Vect$. 

 Let $\Vect^{s}\subset \cC\otimes\Vect$ be the full subcategory of those objects, on which the parity concides with the value of the $\ZZ/2\ZZ$-action by the automorphisms of the identity functor. We refer to it as the \select{$\DG$-category of super-vector spaces}. It inherits a symmetric monoidal structure from $\cC\otimes\Vect$. Our $\Vect^{s}$ inherits a grading $\Vect^{\epsilon} \,\iso\, \underset{\ZZ/2\ZZ}{\oplus} \Vect$ by parity. The commutativity constraint on $\Vect^{s}$ is the usual one for the super vector spaces. 
 
\index{$\DG$-category of super-vector spaces $\Vect^{s}$}
 
\sssec{} Consider $\Vect^s\otimes_{\Vect} C\in CAlg(\DGCat_{cont})$. It has a parity grading (coming from that of $\Vect^s$) and also the $\Gamma$-grading coming from that on $C$. Let $C^{\epsilon}\subset \Vect^s\otimes_{\Vect} C$ be the full subcategory of those objects on which the two gradings are compatible, that is, $\gamma(\epsilon)$ coincides with the parity of an object. It inherits the symmetric monoidal structure from $\Vect^s\otimes_{\Vect} C$, so $C^{\epsilon}\in CAlg(\DGCat_{cont})$. 
 
 Note that $C\,\iso\, C^{\epsilon}$ in $Alg(\DGCat_{cont})$, so we only changed the commutativity constraint. 

\section{Proof of Proposition~\ref{Pp_4.1.10_Raskin_dualizability}}
\label{Sect_proof_Pp_4.1.10_Raskin_dualizability}

\ssec{Preliminaries} 
\label{Sect_D.1}

\sssec{} In Section~\ref{Sect_D.1} we assume $A\in CAlg(\DGCat_{cont})$. We start with some lemmas.

\begin{Lm} 
\label{Lm_D.1.2}
Let $C, E\in A-mod$. There is a canonical equivalence in $A-mod$
\begin{equation}
\label{map_for_Sect_9.2.71}
\Fun_A(\Fun([1], C), E)\to \Fun([1], \Fun_A(C, E))
\end{equation} 
\end{Lm}
\begin{proof}
Write an object of $\Fun([1], C)$ as $(c_0, c_1, \eta)$, where $c_i\in C$ and $\eta:  c_0\to c_1$ in $C$. For $c\in C$ we have a fibre sequence in $\Fun([1], C)$
$$
(0, c, 0)\to (c, c,\id)\to (c, 0,0)\toup{\delta} (0, c[1], 0),
$$
here $\delta$ is the boundary morphism. Let $\Theta: \Fun([1], C)\to E$ be a continuous $A$-linear functor. It gives the functors $\xi_i: C\to E$ in $\Fun_A(C, E)$ for $i=0,1$ given by $\xi_0(c)=\Theta(c,0,0)$, $\xi_1(c)=\Theta(0, c[1], 0)$ and a map $\bar\eta: \xi_0\to \xi_1$ in $\Fun_A(C,E)$ given by applying $\Theta$ to $\delta$. We used that $\delta$ is a map in $\Fun_A(C, \Fun([1], C))$. 

 The inverse map to (\ref{map_for_Sect_9.2.71}) is as follows. For any object $(c_0, c_1, \eta)\in \Fun([1], C)$ we have a commutative diagram, where the arrows are fibre sequences in $\Fun([1], C)$
$$
\begin{array}{ccccccc}
(0, c_1, 0) & \to & (c_0, c_1, \alpha) & \to & (c_0, 0,0) & \to & (0, c_1[1], 0)\\
\downarrow\lefteqn{\scriptstyle\id} && \downarrow\lefteqn{\scriptstyle\eta\times\id} &&\downarrow\lefteqn{\scriptstyle\eta} &&\downarrow\lefteqn{\scriptstyle\id}\\
(0, c_1, 0) & \to & (c_1, c_1, \id) & \to & (c_1, 0,0) & \toup{\delta} & (0, c_1[1], 0) 
\end{array}
$$

 Now $\eta$ yields maps 
$$
\xi_0(c_0)\toup{\xi_0(\eta)}\xi_0(c_1)\toup{\bar\eta(c_1)} \xi_1(c_1), 
$$
and $\Theta(c_0, c_1,\eta)\in E$ is recovered as the fibre in $E$ of
$\bar\eta(c_1)\xi_0(\eta): \xi_0(c_0)\to \xi_1(c_1)$.  
\end{proof}

\begin{Cor} 
\label{Cor_D.1.3}
For $D, C\in A-mod$ the natural map 
$$
\Fun([1], C)\otimes_A D\to \Fun([1], C\otimes_A D)
$$ 
is an equivalence in $A-mod$.
\end{Cor}
\begin{proof} For $E\in A-mod$ we check that the composition
$$
\Fun_A(\Fun([1], C\otimes_A D), E)\to \Fun_A(\Fun([1], C)\otimes_A D, E)
$$
is an equivalence. By Lemma~\ref{Lm_D.1.2}, the LHS identifies with
$$
\Fun([1], \Fun_A(C\otimes_A D, E)),
$$
and the RHS identifies with
$$
\Fun_A((\Fun([1], C), \Fun_A(D, E))\,\iso\, \Fun([1], \Fun_A(C, \Fun_A(D, E))
$$ 
Both sides are the same.
\end{proof}

\sssec{} The purpose of  Section~\ref{Sect_D.1} is to prove the following.

\begin{Lm} 
\label{Lm_D.1.5}
Consider a diagram 
$$
\begin{array}{ccc}
C_3 & \getsup{m} & C_2\\
\uparrow\lefteqn{\scriptstyle j^*}\\
C_1
\end{array}
$$
in $A-mod$, where $m$ (resp., $j^*$) has a continuous $A$-linear right adjoint $m^R$ (resp., $j_*$). Assume $j_*$ fully faithful. Let $C=C_1\times_{C_3} C_2$. Then

\smallskip
\noindent
i) For any $D\in A-mod$ the canonical functor
\begin{equation}
\label{functor_for_Lm_D.1.5}
C\otimes_A D\to (C_1\otimes_A D)\times_{C_3\otimes_A D} (C_2\otimes_A D)
\end{equation}
is an equivalence.

\smallskip\noindent
ii) If, in addition, each $C_i$ is dualizable in $A-mod$ then $C$ is also dualizable in $A-mod$, and $C^{\vee}$ identifies canonically with the colimit of the dual diagram 
$C_1^{\vee}\gets C_3^{\vee}\to C_2^{\vee}$. 
\end{Lm}

 In the rest of Section~\ref{Sect_D.1} we prove Lemma~\ref{Lm_D.1.5}. 

\sssec{} Set 
$$
\Glue=C_2\times_{C_3} \Fun([1], C_3)\times_{C_3} C_1, 
$$
the limit taken in $A-mod$. So, $\Glue\in A-mod$ classifies $\cF\in C_1, \cG\in C_2$ and a map $\eta: m(\cG)\to j^*(\cF)$ in $C_3$. We write $(\cG, \cF, \eta)\in\Glue$. 

\begin{Lm} 
\label{Lm_D.1.7_right_adjoint}
The full embedding $C\hook{} \Glue$ is a map in $A-mod$. It has a continuous right adjoint in $A-mod$.
\end{Lm}
\begin{proof}
The right adjoint is constructed as follows. Given $(\cG, \cF, \eta)\in \Glue$, define $\tilde\cF\in C_1$ by the cartesian square in $C_1$
$$
\begin{array}{ccc}
\tilde\cF & \to & j_*m(\cG)\\
\downarrow && \downarrow{\lefteqn{\scriptstyle j_*(\eta)}}\\
\cF & \to & j_*j^*(\cF)
\end{array}
$$
Applying $j^*$ to the above square one gets an isomorphism $\epsilon: j^*(\tilde\cF)\,\iso\, m(\cG)$, as $j_*$ is fully faithful. The desired right adjoint sends $(\cG, \cF, \eta)$ to $(\cG, \tilde\cF, \epsilon)\in C$. This functor is $A$-linear by construction.
\end{proof} 

\begin{Lm} 
\label{Lm_D.1.8}
For $E\in A-mod$ one has canonically in $A-mod$
\begin{equation}
\label{map_for_LmD.1.8_description}
\Fun_A(\Glue, E)\,\iso\, \Fun_A(C_2, E)\times_{\Fun_A(C_1, E)} \Fun([1], \Fun_A(C_1, E)),
\end{equation}
where the map $\Fun_A(C_2, E)\to \Fun_A(C_1, E)$ is the composition with $m^Rj^*$, and $$
\Fun([1], \Fun_A(C_1, E))\to \Fun_A(C_1, E)
$$ 
is the evaluation at $0\in [1]$. 
\end{Lm}
\begin{proof}
In the case when both $j^*, m$ are equivalences this is precisely Lemma~\ref{Lm_D.1.2}, whose proof we are going to generalize. 

 For $\cF\in C_1$ one has a fibre sequence in $\Glue$
$$
(0, \cF,0)\to (m^Rj^*(\cF), \cF, \eta_{\cF})\to (m^Rj^*(\cF), 0, 0)\toup{\delta} (0, \cF[1],0),  
$$
here $\eta_{\cF}: mm^R(j^*\cF)\to j^*\cF$ is the natural map, and $\delta$ is the boundary morphism.

Let $\Theta\in \Fun_A(\Glue, E)$. The map (\ref{map_for_LmD.1.8_description}) sends $\Theta$ to the collection: $\xi_0\in\Fun_A(C_2, E), \xi_1\in \Fun_A(C_1, E)$ together with a morphism
$$
\bar\eta: \xi_0 m^Rj^*\to \xi_1
$$
in $\Fun_A(C_1, E)$ defined as follows. We set $\xi_0(\cG)=\Theta(\cG, 0,0)$ for $\cG\in C_2$. Set $\xi_1(\cF)=\Theta(0, \cF[1], 0)$ for $\cF\in C_1$. Now $\bar\eta$ is obtained by applying $\Theta$ to $\delta$. We used that $\delta$ is a map in $\Fun_A(C_1, \Glue)$. 

 The inverse map to (\ref{map_for_LmD.1.8_description}) is constructed as follows. For any $(\cG,\cF,\eta)\in\Glue$, we have a commutative diagram, where the arrows are fibre sequences in $\Glue$
$$
\begin{array}{ccccccc}
(0, \cF,0) &\to & (\cG, \cF,\eta) & \to & (\cG, 0,0) & \to & (0, \cF[1], 0)\\
\downarrow\lefteqn{\scriptstyle \id} && \downarrow && \downarrow\lefteqn{\scriptstyle \eta} && \downarrow\\
(0, \cF,0) &\to &(m^Rj^*(\cF), \cF, \eta_{\cF}) & \to & (m^Rj^*(\cF), 0, 0)& \toup{\delta} & (0, \cF[1],0)\\ 
\end{array}
$$
It yields morphisms
$$
\xi_0(\cG)\toup{\xi_0(\eta)} \xi_0(m^Rj^*(\cF))\toup{\bar\eta(\cF)} \xi_1(\cF)
$$

The inverse map to (\ref{map_for_LmD.1.8_description}) sends a collection $(\xi_0,\xi_1,\bar\eta)$ to the functor $\Theta$, where $\Theta(\cG, \cF,\eta)$ is the 
fibre in $E$ of $\bar\eta(\cF)\xi_0(\eta): \xi_0(\cG)\to \xi_1(\cF)$.
\end{proof} 

\sssec{} For $D\in A-mod$ define $\Glue_D\in A-mod$ using the diagram
\begin{equation}
\label{diag_for_Glue_tensored_by_D}
\begin{array}{ccc}
C_3\otimes_A D & \getsup{m\otimes\id} C_2\otimes_A D\\
\uparrow\lefteqn{\scriptstyle j^*\otimes\id}\\
C_1\otimes_A D
\end{array}
\end{equation}
by 
$$
\Glue_D=(C_2\otimes_A D)\times_{C_3\otimes_A D}\Fun([1], C_3\otimes_A D)\times_{C_3\otimes_A D} (C_1\otimes_A D),
$$
the limit taken in $A-mod$. By Corollary~\ref{Cor_D.1.3}, $\Fun([1], C_3)\otimes_A D\,\iso\, \Fun([1], C_3\otimes_A D)$, and we have a canonical functor
\begin{equation}
\label{functor_between_Glues}
\Glue\otimes_A D\to \Glue_D
\end{equation}

\begin{Lm} The functor (\ref{functor_between_Glues}) is an equivalence.
\end{Lm}
\begin{proof}
It suffices to show that for $E\in A-mod$ the canonical map 
$$
\Fun_A(\Glue_D, E)\to \Fun_A(\Glue\otimes_A D, E)
$$
is an equivalence. This follows immediately from Lemma~\ref{Lm_D.1.8} and the fact that for $D, E\in A-mod$, $\Fun_A(D, E)\in A-mod$ is the inner hom in the symmetric monoidal category $A-mod$. 
\end{proof}

\sssec{Proof of Lemma~\ref{Lm_D.1.5}} i) Tensoring the adjoint pair of Lemma~\ref{Lm_D.1.7_right_adjoint} by $D$, one gets an adjoint pair 
$$
l\otimes\id: C\otimes_A D\leftrightarrows \Glue\otimes_A D: r\otimes\id
$$ 
in $A-mod$, where $l\otimes\id$ is fully faithful. Applying Lemma~\ref{Lm_D.1.7_right_adjoint} to the diagram (\ref{diag_for_Glue_tensored_by_D}), one gets an adjoint pair in $A-mod$
$$
l_D: (C_1\otimes_A D)\times_{C_3\otimes_A D} (C_2\otimes_A D)\leftrightarrows \Glue_D: r_D,
$$
where $l_D$ is fully faithful. The diagram commutes
$$
\begin{array}{ccc}
C\otimes_A D & \to &\Glue\otimes_A D\\
\downarrow\lefteqn{\scriptstyle (\ref{functor_for_Lm_D.1.5})} && \downarrow\lefteqn{\scriptstyle (\ref{functor_between_Glues})}\\
(C_1\otimes_A D)\times_{C_3\otimes_A D} (C_2\otimes_A D) &\toup{l_D} &\Glue_D
\end{array}
$$
So, (\ref{functor_for_Lm_D.1.5}) is fully faithful. We claim that the composition
$$
\Glue_D\,\iso\, \Glue\otimes_A D\toup{r\otimes\id} C\otimes_A D\toup{(\ref{functor_for_Lm_D.1.5})} (C_1\otimes_A D)\times_{C_3\otimes_A D} (C_2\otimes_A D)\toup{l_D}\Glue_D
$$
is isomorphic to $l_Dr_D$. Indeed,both these functors send a collection $(\cG,\cF,\eta)\in\Glue_D$ to $(\cG, \tilde\cF,\epsilon)$, where $\tilde\cF\in C_1\otimes_A D$ is defined by the cartesian square
$$
\begin{array}{ccc}
\tilde \cF & \to & (j\otimes\id)_*(m\otimes\id)(\cG)\\
\downarrow && \downarrow\lefteqn{\scriptstyle (j\otimes\id)_*(\eta)}\\
\cF & \to & (j\otimes\id)_*(j\otimes\id)^*(\cF)
\end{array}
$$
in $C_1\otimes_A D$, and $\epsilon: (j\otimes\id)^*(\tilde\cF)\,\iso\, (m\otimes\id)(\cG)$ is obtained by applying $(j\otimes\id)^*$ to the above square. 

\smallskip\noindent
ii) is immediate from i).
\QED

\ssec{Proof of Proposition~\ref{Pp_4.1.10_Raskin_dualizability}} This is a version of (\cite{Ras2}, Lemma~6.18.1). We make induction by $\mid I\mid$, the case of $\mid I\mid=1$ is evident. Let $U=X^I-X$, let $j: U\hook{} X^I$ be the open embedding.
Consider the cartesian square 
$$
\begin{array}{ccc}
\bar C_{X^I}\mid_U &\gets & \bar C_{X^I}\\
\downarrow\lefteqn{\scriptstyle m} && \downarrow\\
C\otimes Shv(U) & \getsup{j^*} & C\otimes Shv(X^I)
\end{array}
$$
from Corollary~\ref{Cor_4.1.3} ii). The functor $j^*$ has a continuous $Shv(X^I)$-linear fully faithful right adjoint $j_*$. 

 We claim that $m$ admits a continuous $Shv(X^I)$-linear right adjoint $m^R$. Indeed, by induction hypothesis,  over each open subset $X^I_{\phi, d}\subset U$ for $(I\toup{\phi} I')\in \cTw(fSets)$ with $\mid I'\mid >1$ the map (\ref{functor_product_for_barC_X^I}) is an equivalence, because for any $\tilde I\in fSets$ the category $\bar C_{X^{\tilde I}}$ is dualizable. The existence of the desired right adjoint over $X^I_{\phi, d}$ follows from the fact that the product  $C^{\otimes J}\to C$ admits a right adjoint in $\DGCat_{cont}$. Now conclude the construction of $m^R$ using Zariski descent for sheaves of categories.

 Now apply Lemma~\ref{Lm_D.1.5} with $A=Shv(X^I)$. The category $\bar C_{X^I}\mid_U$ is dualizable in $Shv(X^I)-mod$ by Lemma~\ref{Lm_dualizability_is_local}, as it is dualizable locally in Zariski topology by induction hypothesis. We conclude that $\bar C_{X^I}$ is dualizable in $Shv(X^I)-mod$, and for any $D\in Shv(X^I)-mod$ the square is cartesian
$$
\begin{array}{ccc}
(\bar C_{X^I}\mid U)\otimes_{Shv(X^I)} D & \gets & \bar C_{X^I}\otimes_{Shv(X^I)} D\\
\downarrow && \downarrow\\
C\otimes Shv(U)\otimes_{Shv(X^I)} D & \gets & C\otimes D
\end{array}
$$

 To see that (\ref{map_for_Pp_4.1.10}) is an equivalence, first apply Corollary~\ref{Cor_2.3.3}. It shows that the RHS of (\ref{map_for_Pp_4.1.10}) identifies with the limit in $Shv(X^I)-mod$ of the diagram
$$
\begin{array}{ccc}
\underset{(I\toup{p} J\to K)\in (\Tw(I)^{>1})^{op}}{\lim} C^{\otimes K}\otimes (D\mid_{X^I_{p, d}})\\
\downarrow\\
\underset{(I\toup{p} J\to *)\in (\Tw(I)^{>1}\cap \Tw(I)^f)^{op}}{\lim} C\otimes (D\mid_{X^I_{p, d}})& 
\gets &C\otimes D
\end{array}
$$
It suffices to show that the latter diagram identifies with
$$
\begin{array}{ccc}
\bar C_{X^I}\mid_U\otimes_{Shv(X^I)} D\\
\downarrow\lefteqn{\scriptstyle m\otimes\id} && \downarrow\\
C\otimes (D\mid_U) & \getsup{j^*\otimes\id} & C\otimes D
\end{array}
$$
The canonical map
$$
D\otimes_{Shv(X^I)} \underset{(I\toup{p} J\to *)\in (\Tw(I)^{>1}\cap \Tw(I)^f)^{op}}{\lim} Shv(X^I_{p, d})\to \underset{(I\toup{p} J\to *)\in (\Tw(I)^{>1}\cap \Tw(I)^f)^{op}}{\lim} D\mid_{X^I_{p, d}}
$$
is an equivalence, by etale descent for the sheaf of categories $D\mid_U$ on $U$.  
Since $C$ is dualizable in $\DGCat_{cont}$, we see that the left down corners in the two diagrams are the same.

 Finally, the canonical map 
$$
\bar C_{X^I}\mid_U\otimes_{Shv(X^I)} D\to \underset{(I\toup{p} J\to K)\in (\Tw(I)^{>1})^{op}}{\lim} C^{\otimes K}\otimes (D\mid_{X^I_{p, d}})
$$
is also an equivalence. Indeed, it suffices to show this after restriction to any $X^I_{\phi, d}$ for $(I\toup{\phi} I')\in\cTw(fSets)$ with $\mid I'\mid >1$. Over that locus it follows from the induction hypothesis, factorization (already established over $U$) and cofinality argument as in Section~\ref{Sect_4.1.8}. \QED

\section{Graded version of $\Fact(C)$}
\label{Sect_E}

\sssec{}
\label{Sect_E.0.1}
 In some applications the construction of this section is useful (for example, cf. \cite{Gai19Ran}, Section 2.5.4(ii)). It also plays a key role in our construction of the so called \select{graded Hecke functors} in \cite{DL2}. 
 
 Let $C(X)\in CAlg(Shv(X)-mod)$. Let $\Lambda$ be a semigroup isomorphic to $\ZZ_+^n$ for some $n>0$. Assume that $C(X)$ is $\Lambda$-graded in a way compatible with the symmetric monoidal structure $C(X)\,\iso\, \underset{\lambda\in\Lambda}{\oplus} C(X)_{\lambda}$. 

In this case $\Fact(C)$ carries an additional structure that we describe below.

\index{Graded commutative factorization category}

\sssec{} For $0\ne \lambda\in\Lambda$ let $X^{\lambda}$ be the moduli scheme of $\Lambda-\{0\}$-valued divisors on $X$ of degree $\lambda$. For $\lambda=0$ we let $X^0=\Spec k$. 

Recall the functor $\cF_{\Ran, C}: \cTw(fSets)\to Shv(\Ran)-mod$ given by (\ref{functor_cF_Ran_C}). For each $(J\to K)\in \cTw(fSets)$, $\underset{k\in K}{\boxtimes} C^{\otimes J_k}(X)$ carries a $\Lambda$-grading coming from that of $C(X)$ via the composition
$$
\underset{k\in K}{\boxtimes} C^{\otimes J_k}(X)\toup{\vartriangle^!} C^{\otimes J}(X)\toup{m} C(X)
$$ 
for $\vartriangle: X\to X^K$. Here $m$ is the product map. This grading is respected by the transition maps in $\cF_{\Ran, C}$. Thus, $\underset{\cTw(fSets)}{\colim} \cF_{\Ran, C}\,\iso\, \Fact(C)$ inherits a $\Lambda$-grading in $Shv(\Ran)-mod$
$$
\Fact(C)\,\iso\,\underset{\lambda\in\Lambda}{\oplus} \Fact(C)_{\lambda}.
$$  

\sssec{}  Let us describe the component $\Fact(C)_{\lambda}$. Write $\cTw(fSets)_{\lambda}$ for the category, whose objects are collections: a map $J\toup{\phi} K$ in $fSets$, a map $\und{\lambda}: J\to \Lambda$ with $\sum_j \und{\lambda}(j)=\lambda$. A map $\tau$ from $(\und{\lambda}^1, J_1\to K_1)$ to $(\und{\lambda}^2, J_2\to K_2)$ is a map from $(J_1\to K_1)$ to $(J_2\to K_2)$ in $\cTw(fSets)$ with $\phi_*\und{\lambda}^1=\und{\lambda}^2$, here $\phi: J_1\to J_2$ is the corresponding map. 

 Given $(\und{\lambda}, J\to K)$ in $\cTw(fSets)_{\lambda}$, we get the corresponding direct summand 
$$
(C^{\otimes\phi})_{\und{\lambda}}=(\underset{k\in K}{\boxtimes} C^{\otimes J_k}(X))_{\und{\lambda}}\subset\underset{k\in K}{\boxtimes} C^{\otimes J_k}(X).
$$ 
For a map $\tau$ as above the corresponding transition morphism
$$
\underset{k\in K_1}{\boxtimes} C^{\otimes (J_1)_k}(X)\to \underset{k\in K_2}{\boxtimes} C^{\otimes (J_2)_k}(X)
$$
for $\cF_{\Ran, C}$ yields a morphism
$$
(\underset{k\in K_1}{\boxtimes} C^{\otimes (J_1)_k}(X))_{\und{\lambda}^1}\to (\underset{k\in K_2}{\boxtimes} C^{\otimes (J_2)_k}(X))_{\und{\lambda}^2}
$$
so defining a functor $\cF_{\Ran, C, \lambda}: \cTw(fSets)_{\lambda}\to Shv(\Ran)-mod$.
We get
$$
\Fact(C)_{\lambda}\,\iso\, \underset{\cTw(fSets)_{\lambda}}{\colim}\cF_{\Ran, C,\lambda}.
$$

\sssec{} Let $\Theta: \Lambda\to\ZZ_+$ be the morphism of semigroups sending each generator of $\Lambda$ to $1$. For $d\in\ZZ_+$ write $X^{(d)}$ for the $d$-th symmetric power of $X$.
For $\lambda\in\Lambda$ and $d=\Theta(\lambda)$ the map $\Theta$ yields a morphism $X^{\lambda}\to X^{(d)}$, which we denote also by $D\mapsto \Theta(D)$. For $S\in\Sch^{aff}$ and an $S$-point $D\in\Map(S, X^{\lambda})$ we denote by $S\times X-\supp(D)$ the scheme $S\times X-\supp(\Theta(D))$. Here $\supp(\Theta(D))$ is the pullback under $\Theta(D)\times\id: S\times X\to X^{(d)}\times X$ of the incidence divisor.

\sssec{} 
\label{Sect_E.0.5}
For $S\in\Sch^{aff}$ and an $S$-point $\cI$ of $\Ran$ given $i\in \cI$ we write $\Gamma_i\subset S\times X$ for the graph of the corresponding map, and $\Gamma_{\cI}$ for the union of these $\Gamma_i$ for $i\in\cI$. 

 For $\lambda\in\Lambda$ define $(X^{\lambda}\times\Ran)^{\subset}$ as the prestack sending $S\in\Sch^{aff}$ to the set to those $S$-points $(D, \cI)$ of $X^{\lambda}\times\Ran$ for which $(S\times X)-\Gamma_{\cI}\subset (S\times X)-\supp(D)$. 

\sssec{} 
\label{Sect_E.0.5_now}
For $(\und{\lambda}, J\toup{\phi} K)\in \cTw(fSets)_{\lambda}$ let $K^*=\{k\in K\mid \sum_{j\in J_k} \und{\lambda}(j)\ne 0\}$. 
Consider the map $X^K\to (X^{\lambda}\times\Ran)^{\subset}$ sending $(x_k)$ to $(D, (x_k))$ with
$$
D=\underset{k\in K*}{\sum} x_k\underset{j\in J_k}{\sum}\und{\lambda}(j).
$$ 
In this way we view $\underset{k\in K}{\boxtimes} C^{\otimes J_k}(X)_{\und{\lambda}}\in Shv((X^{\lambda}\times\Ran)^{\subset})-mod$. For a map $\tau$ as above the corresponding transition map for $\cF_{\Ran, C,\lambda}$ is $Shv((X^{\lambda}\times\Ran)^{\subset})$-linear. So,
$\underset{\cTw(fSets)_{\lambda}}{\colim}\cF_{\Ran, C,\lambda}$ can be understood in $Shv((X^{\lambda}\times\Ran)^{\subset})-mod$. Finally, $\Fact(C)_{\lambda}$ is naturally an object of $Shv((X^{\lambda}\times\Ran)^{\subset})-mod$. 

\sssec{} The commutative chiral product is compatible with the $\Lambda$-grading of $\Fact(C)$. Namely, for $\mu_1,\mu_2\in\Lambda$ the commutative chiral product yields a functor 
\begin{equation}
\label{chiral_product_graded_Sect_E.0.7}
\Fact(C)_{\mu_1}\boxtimes \Fact(C)_{\mu_2}\to u^!\Fact(C)_{\mu}
\end{equation} 
in $ShvCat(X^{\mu_1}\times X^{\mu_2})$, where $\mu=\mu_1+\mu_2$ and
$u: X^{\mu_1}\times X^{\mu_2}\to X^{\mu}$ is the sum map.

   Indeed, we have the functor $\cTw(fSets)_{\mu_1}\times \cTw(fSets)_{\mu_2}\to \cTw(fSets)_{\mu}$ sending $(\und{\mu}^1, J_1\to K_1), (\und{\mu}^2, J_2\to K_2)$ to $(\und{\mu}, J\to K)$, where $J=J_1\sqcup J_2, K=K_1\sqcup K_2$, and $\und{\mu}$ is obtained from $\und{\mu}_i$. For each $(\und{\mu}^1, J_1\to K_1), (\und{\mu}^2, J_2\to K_2)$ as above we have
$$
(\underset{k\in K_1}{\boxtimes} C^{\otimes (J_1)_k}(X)_{\und{\mu}^1})\boxtimes (\underset{k\in K_2}{\boxtimes} C^{\otimes (J_2)_k}(X)_{\und{\mu}^2})\iso \underset{k\in K}{\boxtimes} C^{\otimes J_k}(X)_{\und{\mu}}
$$
an isomorphism in $Shv(X^{\mu_1}\times X^{\mu_2})-mod$. The colimit over $\cTw(fSets)_{\mu_1}\times \cTw(fSets)_{\mu_2}$ in the above maps yields the morphism $\Fact(C)_{\mu_1}\boxtimes \Fact(C)_{\mu_2}\to \Fact(C)_{\mu}$ in $Shv(X^{\mu})-mod$, where on the LHS it acts through $u^!: Shv(X^{\mu})\to Shv(X^{\mu_1}\times X^{\mu_2})$. By Section~\ref{Sect_B.1.16}, $Shv(X^{\mu_1}\times X^{\mu_2})$ is canonically self-dual in $Shv(X^{\mu})-mod$. So, by adjointness given !n (\cite{Ly}, 3.2.1), this yields the desired morphism (\ref{chiral_product_graded_Sect_E.0.7}). 

This implies that the non-unital symmetric monoidal structure $\star$ on $\Gamma(\Ran, \Fact(C))$ defined in Section~\ref{Sect_Non-unital symmetric monoidal structure_Fact(C)(Ran)} is compatible with the $\Lambda$-grading of $\Fact(C)$. 

\sssec{} For $I\in fSets$ let us describe the category $\Fact(C)_{\lambda}\otimes_{Shv(\Ran)} Shv(X^I)$. 

 Let $\Tw(I)_{\lambda}$ be the category, whose objects are diagrams $(I\to J\to K)\in\Tw(I)$ and $\und{\lambda}: J\to\Lambda$ with $\sum_{j\in J} \und{\lambda}(j)=\lambda$. A maps from the first object to the second is given by a morphism
$$
\begin{array}{ccccc}
I & \to & J_1 & \to & K_1\\
\Vert && \downarrow\lefteqn{\scriptstyle \tau} && \uparrow\\
I & \to & J_2 & \to & K_2
\end{array}
$$
in $\Tw(I)$ such that $\tau_*\und{\lambda}^1=\und{\lambda}^2$. 

 Consider the functor 
$$
\cF_{I, C,\lambda}: \Tw(I)_{\lambda}\to Shv((X^{\lambda}\times X^I)^{\subset})-mod
$$ 
sending $(I\to J\toup{\phi} K, \und{\lambda})$ to $\vartriangle_!^{(I/K)}(C^{\otimes\phi})_{\und{\lambda}}$. The transition maps are exactly those as for the functor $\cF_{I, C}$ from Section~\ref{Sect_2.1.3_now}.  
\begin{Lm} If $\lambda\in\Lambda, I\in fSets$ then 
$$
\Fact(C)_{\lambda}\otimes_{Shv(\Ran)} Shv(X^I)\,\iso\, \underset{\Tw(I)_{\lambda}}{\colim} \cF_{I, C,\lambda}
$$ 
canonically. In particular, $\Fact(C)_{\lambda}\otimes_{Shv(\Ran)} Shv(X)\,\iso\, C(X)_{\lambda}$ canonically in $Shv(X)-mod$.
\end{Lm}
\begin{proof}
Consider the category $\cE_{\lambda}$ whose objects are: a diagram in $fSets$
 $$
 (I\to \bar K\gets K\getsup{\phi} J)
 $$
 and $\und{\lambda}: J\to \Lambda$ with $\sum_{j\in J} \und{\lambda}(j)=\lambda$. A map from the first object to the second is given by a commutative diagram in $fSets$
 $$
 \begin{array}{ccccccc}
 I & \to & \bar K_1 & \gets & K_1 & \getsup{\phi_1} & J_1\\
 & \searrow & \uparrow && \uparrow && \downarrow\lefteqn{\scriptstyle \tau}\\
 && \bar K_2 & \gets & K_2 & \getsup{\phi_2} & J_2
\end{array}
$$ 
for which $\tau_*\und{\lambda}^1=\und{\lambda}^2$. As in Section 2.1.12 of the paper, we get
\begin{equation}
\label{category_Fact(C)_lambda_base_changed}
\Fact(C)_{\lambda}\otimes_{Shv(\Ran)} Shv(X^I)\,\iso\, \underset{(I\to \bar K\gets K\getsup{\phi} J, \und{\lambda})\in \cE_{\lambda}}{\colim} \vartriangle^{(I/\bar K)}_!\vartriangle^{(K/\bar K)!} (C^{\otimes\phi})_{\und{\lambda}}
\end{equation}
taken in $Shv(X^I)-mod$. Here $X^I \getsup{\vartriangle^{(I/\bar K)}} X^{\bar K} \toup{\vartriangle^{(K/\bar K)}} X^K$ are the corresponding maps. 

  Write $\bar\phi$ for the composition $J\toup{\phi} K\to \bar K$ for an object of $\cE_{\lambda}$. Then 
$$
{\vartriangle^{(K/\bar K)!}(C^{\otimes\phi})_{\und{\lambda}}}\,\iso\, (C^{\otimes\bar\phi})_{\und{\lambda}}.
$$   

 Let $r_0: \cE^0_{\lambda}\subset \cE_{\lambda}$ be the full subcategory given by the property that $K\to \bar K$ is an isomorphism. Think of $\cE^0_{\lambda}$ as the category classifying $(I\to \bar K\gets J, \und{\lambda}: J\to\Lambda)$, where 
$$
(J\to \bar K, \und{\lambda})\in \cTw(fSets)_{\lambda}.
$$  

 Let $l_0: \cE_{\lambda}\to \cE^0_{\lambda}$ be the functor sending $(I\to \bar K\gets K\gets J, \und{\lambda})$ to $(I\to \bar K\gets J, \und{\lambda})$. Then $r_0$ is left adjoint to $l_0$, so $l_0$ is cofinal. Now (\ref{category_Fact(C)_lambda_base_changed}) identifies with 
\begin{equation}
\label{expression_2}
\underset{(I\to K\getsup{\phi} J, \und{\lambda})\in \cE^0_{\lambda}}{\colim}  \vartriangle^{(I/K)}_! (C^{\otimes\phi})_{\und{\lambda}}
\end{equation}

Consider the functor $q_{\lambda}: \Tw(I)_{\lambda}\to \cE^0_{\lambda}$ sending $(I\toup{a} J\toup{\phi} K, \und{\lambda})$ to $(I\toup{\phi\comp a} K\gets J, \und{\lambda})$. 
Our claim follows now from Lemma~\ref{Lm_E.0.10} below.
\end{proof}
 
\begin{Lm} 
\label{Lm_E.0.10}
The functor $q_{\lambda}: \Tw(I)_{\lambda}\to \cE^0_{\lambda}$ is cofinal.
\end{Lm}
\begin{proof} Pick an object $c=(I\to K\gets J, \und{\lambda}: J\to\Lambda)\in \cE^0_{\lambda}$. We show that the category $\Tw(I)_{\lambda}\times_{\cE^0_{\lambda}} (\cE^0_{\lambda})_{c/}$ is contractible. An object of the latter category is given by $(I\toup{a_1} J_1\toup{\phi_1} K_1, \und{\lambda}^1: J_1\to\Lambda)\in \Tw(I)_{\lambda}$ together with a commutative diagram in $fSets$
$$
\begin{array}{ccccc}
I &\to &K &\gets &J\\
\Vert && \uparrow\lefteqn{\scriptstyle\eta} && \downarrow\lefteqn{\scriptstyle \tau}\\
I &\toup{\phi_1a_1} & K_1 & \getsup{\phi_1} & J_1
\end{array}
$$
such that $\tau_*\und{\lambda}=\und{\lambda}^1$. 

 Let $j: C\hook{} \Tw(I)_{\lambda}\times_{\cE^0_{\lambda}} (\cE^0_{\lambda})_{c/}$ be the full subcategory given by the property that $\eta$ is an isomorphism. Denote by $j^R: \Tw(I)_{\lambda}\times_{\cE^0_{\lambda}} (\cE^0_{\lambda})_{c/}\to C$ the functor sending a point of the source as above to $(I\toup{a_1} J_1\toup{\eta\phi_1} K, \und{\lambda}^1)\in \Tw(I)_{\lambda}$ together with the diagram
$$
\begin{array}{ccccc}
I &\to &K &\gets &J\\
\Vert && \uparrow\lefteqn{\scriptstyle\id} && \downarrow\lefteqn{\scriptstyle\tau}\\
I &\toup{\eta\phi_1a_1} & K & \getsup{\eta\phi_1} & J_1
\end{array}
$$
Then $j$ is left adjoint to $j^R$. So, they yield equiavlences $\mid C\mid\,\iso\, \mid \Tw(I)_{\lambda}\times_{\cE^0_{\lambda}} (\cE^0_{\lambda})_{c/}\mid$. It remains to show that $C$ is contractible.
 
 An object of $C$ is a commutative diagram in $fSets$
$$
\begin{array}{ccccc}
I &\to &K & \gets & J\\
& \searrow &\uparrow & \swarrow\lefteqn{\scriptstyle \tau}\\
&& J_1,
\end{array}
$$
so that $\und{\lambda}^1=\tau_*\und{\lambda}$. The category $C$ has a final object given by $(J_1\to K)$ equal to $(K\toup{\id} K)$.
\end{proof}  
 
\sssec{Graded factorization algebras} 
\label{Sect_Graded factorization algebras}
Let now $A\in CAlg^{nu}(C(X))$. Assume we have a grading $A=\underset{\lambda\in\Lambda}{\oplus} A_{\lambda}$, $A_{\lambda}\in C(X)_{\lambda}$. Then $\Fact(A)\in \Fact(C)$ inherits a $\Lambda$-grading $\Fact(A)\,\iso\, \underset{\lambda\in\Lambda}{\oplus} \Fact(A)_{\lambda}$, where $\Fact(A)_{\lambda}$ is as follows.

\index{Graded factorization algebras}

 Recall the functor $\cF_{\Ran, A}: \cTw(fSets)\to\Fact(C)$ from Section~\ref{Sect_2.4.4_now}. For $\lambda\in\Lambda$ define similarly the functor 
$$
\cF_{\Ran, A,\lambda}: \cTw(fSets)_{\lambda}\to \Fact(C)_{\lambda}
$$ 
as follows. It sends $(\und{\lambda}, J\toup{\phi} K)$ to the image of $\underset{k\in K}{\boxtimes} (\underset{j\in J_k}{\otimes} A_{\und{\lambda}(j)})$ under $\underset{k\in K}{\boxtimes} (\underset{j\in J_k}{\otimes} C(X)_{\und{\lambda}(j)})\to \Fact(C)_{\lambda}$. It sends a map $\tau$ from $(\und{\lambda}^1, J_1\to K_1)$ to $(\und{\lambda}^1, J_1\to K_1)$ to the image of the composition
$$
m(\underset{k\in K_1}{\boxtimes} (\underset{j\in (J_1)_k}{\otimes} A_{\und{\lambda}^1(j)}))\to \underset{k\in K_1}{\boxtimes} (\underset{j\in (J_2)_k}{\otimes} A_{\und{\lambda}^2(j)})\toup{\vartriangle_!} \underset{k\in K_2}{\boxtimes} (\underset{j\in (J_2)_k}{\otimes} A_{\und{\lambda}^2(j)})
$$
under the functor $\underset{k\in K_2}{\boxtimes} (\underset{j\in (J_2)_k}{\otimes} C(X)_{\und{\lambda}^2(j)})\to \Fact(C)_{\lambda}$. Here we used the transition map
$$
\underset{k\in K_1}{\boxtimes} (\underset{j\in (J_1)_k}{\otimes} C(X)_{\und{\lambda}^1(j)})\toup{m} 
\underset{k\in K_1}{\boxtimes} (\underset{j\in (J_2)_k}{\otimes} C(X)_{\und{\lambda}^2(j)})\toup{\vartriangle_!} \underset{k\in K_2}{\boxtimes} (\underset{j\in (J_2)_k}{\otimes} C(X)_{\und{\lambda}^2(j)})
$$
attached to $\tau$ by $\cF_{\Ran, C,\lambda}$. 

\ssec{Non-unital version}
\label{Sect_E1}

\sssec{} 
\label{Sect_E.1.1}
Assume $C(X)\in CAlg^{nu}(Shv(X)-mod)$. Assume $C(X)=\underset{\lambda\in\Lambda^*} C(X)_{\lambda}$ is graded by $\Lambda^*:=\Lambda-\{0\}$ in a way compatible with the symmetric monoidal structure. One has the evident non-unital version of the above construction.

 Namely, for $(J\to K)\in \cTw(fSets)$, $\underset{k\in K}{\boxtimes} C^{\otimes J_k}(X)$ carries a similar grading by $\Lambda^*$ respected by the transition maps in $\cF_{\Ran, C}$. Thus, $\colim \cF_{\Ran, C}\,\iso\, \Fact(C)$ inherits a $\Lambda^*$-grading $\Fact(C)\,\iso\,\underset{\lambda\in \Lambda^*}{\oplus}\Fact(C)_{\lambda}$ in $Shv(\Ran)-mod$.
 
\sssec{} 
\label{Sect_E.1.2_now}
For $\lambda\in\Lambda^*$ the component $\Fact(C)_{\lambda}$ is described as follows.

For $\lambda\in\Lambda^*$ write $\cTw(fSets)^{\lambda}$ for the category, whose objects are collections: a map $(J\toup{\phi} K)$ in $fSets$ and a map $\und{\lambda}: J\to \Lambda^*$ with $\sum_j \und{\lambda}(j)=\lambda$. A map $\tau$ in $\cTw(fSets)^{\lambda}$ from $(\und{\lambda}^1, J_1\to K_1)$ to $(\und{\lambda}^2, J_2\to K_2)$ is a map (\ref{map_in_cTw(fSets)}) in $\cTw(fSets)$ with $\phi_*\und{\lambda}^1=\und{\lambda}^2$, here $\phi: J_1\to J_2$ is the corresponding map. 

 Given $(\und{\lambda},  J\to K)\in \cTw(fSets)^{\lambda}$, we get the corresponding direct summand
\begin{equation}
\label{summand_for_Sect_E.1.2}
\underset{k\in K}{\boxtimes} C^{\otimes J_k}(X)_{\und{\lambda}}
\end{equation}
in $\underset{k\in K}{\boxtimes} C^{\otimes J_k}(X)$. One similarly gets a functor $\cF_{\Ran, C}^{\lambda}: \cTw(fSets)^{\lambda}\to Shv(\Ran)-mod$ sending $(\und{\lambda}, J\to K)$ to (\ref{summand_for_Sect_E.1.2}). We get
$$
\Fact(C)_{\lambda}\,\iso\, \underset{\cTw(fSets)^{\lambda}}{\colim} \cF_{\Ran, C}^{\lambda}.
$$

\sssec{} 
\label{Sect_E.1.3}
For $(\und{\lambda},  J\to K)\in \cTw(fSets)^{\lambda}$ we have the map $X^K\to (X^{\lambda}\times\Ran)^{\subset}$ sending $(x_k)$ to $(D, (x_k))$ with
\begin{equation}
\label{divisor_D_for_Sect_E.1.3}
D=\sum_{k\in K} x_k\sum_{j\in J_k}\und{\lambda}(j).
\end{equation} 
This way we view (\ref{summand_for_Sect_E.1.2}) as an object of $Shv((X^{\lambda}\times\Ran)^{\subset})-mod$. For a map $\tau$ as above the corresponding transition map in $\cF_{\Ran, C}^{\lambda}$ is $Shv((X^{\lambda}\times\Ran)^{\subset})$-linear, so $\underset{\cTw(fSets)^{\lambda}}{\colim} \cF_{\Ran, C}^{\lambda}$ can be understood in $Shv((X^{\lambda}\times\Ran)^{\subset})-mod$, and $\Fact(C)_{\lambda}\in Shv((X^{\lambda}\times\Ran)^{\subset})-mod$.

\sssec{} Write $fSets^{\lambda}$ for the category, whose objects are pairs $(\und{\lambda}, K)$, where $K\in fSets$ and $\und{\lambda}: K\to \Lambda^*$ is a map such that $\sum_k \und{\lambda}(k)=\lambda$. A map from $(K_1, \und{\lambda}^1)$ to $(K_2, \und{\lambda}^2)$ in $fSets^{\lambda}$  is a map $\tau: K_1\to K_2$ in $fSets$ such that $\tau_*\und{\lambda}^1=\und{\lambda}^2$. 

\begin{Lm} 
\label{Lm_E.1.5}
For $\lambda\in\Lambda^*$ one has canonically $\Fact(C)_{\lambda}\otimes_{Shv(X^{\lambda})} Shv(X)\,\iso\, C(X)_{\lambda}$ in $Shv(X)-mod$.
\end{Lm}
\begin{proof} We have the adjoint pair $l: (fSets)^{\lambda}\leftrightarrows \cTw(fSets)^{\lambda}: r$, where $r(\und{\lambda}, J\to K)=(\und{\lambda}, J)$, and $l$ sends $(\und{\lambda}, J)$ to $(\und{\lambda}, J\to *)$. By definition, 
$$
\Fact(C)_{\lambda}\otimes_{Shv(X^{\lambda})} Shv(X)\,\iso\,\underset{(\und{\lambda}, J\to K)\in\cTw(fSets)^{\lambda}}{\colim} (\underset{j\in J}{\otimes} C(X)_{\und{\lambda}(j)})
$$ 
in $Shv(X)-mod$. For any functor $e: fSets^{\lambda}\to Shv(X)-mod$ the composition $e\comp r$ is the left Kan extension of $e$ under $l$ by (\cite{Ly}, 2.2.39). So, the above colimit identifies with
$$
\underset{(\und{\lambda}, J)\in fSets^{\lambda}}{\colim} (\underset{j\in J}{\otimes} C(X)_{\und{\lambda}(j)})\,\iso\, C(X)_{\lambda},
$$
as $fSets^{\lambda}$ has a final object.  
\end{proof}

\begin{Rem} 
\label{Rem_E.1.6}
If $\lambda$ is one of generators of $\Lambda$ then $X^{\lambda}\,\iso\, X$. In this case Lemma~\ref{Lm_E.1.5} specifies to an isomorphism $\Fact(C)_{\lambda}\,\iso\, C(X)_{\lambda}$.
\end{Rem}

\sssec{} As in the unital case, the commutative chiral product on $\Fact(C)$ is compatible with the $\Lambda^*$-grading. 

\sssec{} For $\lambda\in\Lambda^*$ let $\Tw(I)^{\lambda}$ be the category of collections $(I\to J\to K)\in \Tw(I)$ and a map $\und{\lambda}: J\to \Lambda^*$. A morphism from $(I\to J_1\to K_1, \und{\lambda}^1)$ to $(I\to J_2\to K_2, \und{\lambda}^2)$ is a morphism (\ref{morphism_in_Tw_v2}) in $\Tw(I)$ such that $\phi_*\und{\lambda}^1=\und{\lambda}^2$ for the underlying map $\phi: J_1\to J_2$. 

 For $I\in fSets$. The category $C_{X^I}\,\iso\, \underset{\lambda\in\Lambda^*}{\oplus} C_{X^I, \lambda}$ inherits a grading from that of $\Fact(C)$. For $\lambda\in\Lambda^*$ we get a functor 
$$
\cF_{I, C}^{\lambda}: \Tw(I)^{\lambda}\to Shv((X^{\lambda}\times X^I)^{\subset})
$$ 
sending $(I\to J\to K, \und{\lambda})$ to $\underset{k\in K}{\boxtimes} (\underset{j\in J_k}{\otimes} C(X)_{\und{\lambda}(j)})$, the transition maps being defined as in Section~\ref{Sect_2.1.3_now} for $\cF_{I,C}$. Here we used the map $X^K\to (X^{\lambda}\times X^I)^{\subset}$ sending $(x_k)$ to $(D, \vartriangle_!^{(I/K)}(x_k))$, where $D$ is given by (\ref{divisor_D_for_Sect_E.1.3}). Then $C_{X^I, \lambda}\,\iso\, \underset{\Tw(I)^{\lambda}}{\colim} \cF_{I, C}^{\lambda}$.

\sssec{Graded factorization algebras: non-unital case} 
\label{Sect_Gr_fact_alg_non-unital}
Let $A\in CAlg^{nu}(C(X))$. Assume given a grading $A=\underset{\lambda\in\Lambda^*}{\oplus} A_{\lambda}$ with $A_{\lambda}\in C(X)_{\lambda}$. As in Section~\ref{Sect_Graded factorization algebras}, $\Fact(A)$ inherits a grading 
$$
\Fact(A)=\underset{\lambda\in\Lambda^*}{\oplus} \Fact(A)_{\lambda}
$$ 
with $\Fact(A)_{\lambda}\in \Fact(C)_{\lambda}$ described as follows. For $\lambda\in\Lambda^*$ we define the functor
$$
\cF^{\lambda}_{\Ran, A}: \cTw(fSets)^{\lambda}\to \Fact(C)_{\lambda}
$$
as follows. It sends $(\und{\lambda}, J\to K)$ to the image of $\underset{k\in K}{\boxtimes} (\underset{j\in J_k}{\otimes} A_{\und{\lambda}(j)})$ under the functor $\underset{k\in K}{\boxtimes} (\underset{j\in J_k}{\otimes} C(X)_{\und{\lambda}(j)})\to \Fact(C)_{\lambda}$ coming from the definition of $\cF^{\lambda}_{\Ran, C}$ in Section~\ref{Sect_E.1.2_now}. 
It sends a map $\tau$ in $\cTw(fSets)^{\lambda}$ from $(\und{\lambda}^1, J_1\to K_1)$ to $(\und{\lambda}^2, J_2\to K_2)$ to the image of the composition
$$
m(\underset{k\in K_1}{\boxtimes} (\underset{j\in (J_1)_k}{\otimes} A_{\und{\lambda}^1(j)}))\to \underset{k\in K_1}{\boxtimes} (\underset{j\in (J_2)_k}{\otimes} A_{\und{\lambda}^2(j)})\toup{\vartriangle_!} \underset{k\in K_2}{\boxtimes} (\underset{j\in (J_2)_k}{\otimes} A_{\und{\lambda}^2(j)})
$$
under the functor $\underset{k\in K_2}{\boxtimes} (\underset{j\in (J_2)_k}{\otimes} C(X)_{\und{\lambda}^2(j)})\to \Fact(C)_{\lambda}$. Here we used the transition map
$$
\underset{k\in K_1}{\boxtimes} (\underset{j\in (J_1)_k}{\otimes} C(X)_{\und{\lambda}^1(j)})\toup{m} 
\underset{k\in K_1}{\boxtimes} (\underset{j\in (J_2)_k}{\otimes} C(X)_{\und{\lambda}^2(j)})\toup{\vartriangle_!} \underset{k\in K_2}{\boxtimes} (\underset{j\in (J_2)_k}{\otimes} C(X)_{\und{\lambda}^2(j)})
$$
attached to $\tau$ by $\cF_{\Ran, C}^{\lambda}$. Then $\Fact(A)_{\lambda}\,\iso\, \underset{\cTw(fSets)^{\lambda}}{\colim} \; \cF^{\lambda}_{\Ran, A}$. 

For $I\in fSets$ the object $A_{X^I}\in C_{X^I}$ inherits a grading $A_{X^I}\,\iso\, \underset{\lambda\in\Lambda^*}{\oplus} A_{X^I, \lambda}$ from that of $\Fact(A)$. Here $A_{X^I, \lambda}\in C_{X^I, \lambda}$ is described as follows. We get a functor $\cF_{I, A}^{\lambda}: \Tw(I)^{\lambda}\to C_{X^I,\lambda}$ sending $(I\to J\to K, \und{\lambda})$ to the image of $\underset{k\in K}{\boxtimes}(\underset{j\in J_k}{\otimes} A_{\und{\lambda}(j)})$ under $\underset{k\in K}{\boxtimes}(\underset{j\in J_k}{\otimes} C(X)_{\und{\lambda}(j)})\to C_{X^I, \lambda}$. Then $A_{X^I, \lambda}\,\iso\, \underset{\Tw(I)^{\lambda}}{\colim} \cF^{\lambda}_{I, A}$ in $C_{X^I,\lambda}$. 

\sssec{Example} The following result appeared as (\cite{Gai19Ran}, 2.5.4(ii)) without a proof.

\begin{Pp} 
\label{Pp_E.1.5}
Take $C(X)=\underset{\lambda\in\Lambda^*}{\oplus} Shv(X)\in CAlg^{nu}(Shv(X)-mod)$, the category of $\Lambda^*$-graded sheaves on $X$. For each $\lambda\in\Lambda^*$ there is a canonical equivalence in $Shv(X^{\lambda})-mod$
$$
\Fact(C)_{\lambda}\,\iso\, Shv(X^{\lambda}).
$$ 
\end{Pp}

In Sections~\ref{Sect_E.1.9}-\ref{Sect_E.1.19} we prove Proposition~\ref{Pp_E.1.5}. We underline that it is a surprise that $Shv(X^{\lambda})$ is a module over $Shv(\Ran)$. 

\sssec{} 
\label{Sect_E.1.9}
 We have the functor 
\begin{equation}
\label{functor_h_Sect_E.1.9}
h: \cTw(fSets)^{\lambda}\to (fSets^{\lambda})^{op}
\end{equation} 
sending $(\und{\lambda}, J\toup{g} K)$ to $(K, g_*\und{\lambda})$. Consider the functor 
$$
\eta: (fSets^{\lambda})^{op}\to\DGCat_{cont}
$$ 
sending $(\und{\lambda}, K)$ to $Shv(X^K)$. It sends a map $\tau: K_1\to K_2$ as above to $\vartriangle_!: Shv(X^{K_2})\to Shv(X^{K_1})$ for the diagonal $\vartriangle: X^{K_2}\to X^{K_1}$. We are calculating  
$$
\Fact(C)_{\lambda}\,\iso\,\underset{\cTw(fSets)^{\lambda}}{\colim} \eta\comp h.
$$
 
\begin{Lm} 
\label{Lm_E.1.7}
The functor $h$ is cofinal. So,
$$
\Fact(C)_{\lambda}\,\iso\, \underset{(fSets^{\lambda})^{op}}{\colim} \eta.
$$
\end{Lm}
\begin{proof}
Given $\xi:=(K', \und{\mu}: K'\to \Lambda^*)\in fSets^{\lambda}$, consider the category 
$$
\cTw(fSets)^{\lambda}\times_{(fSets^{\lambda})^{op}} ((fSets^{\lambda})^{op})_{\xi/}.
$$ 
It suffices to show that the latter category of contractible. It classifies an object $(\und{\lambda}, J\toup{\phi} K)\in \cTw(fSets)^{\lambda}$ together with a surjection $\phi': K\to K'$ such that $\phi'_*\phi_*\und{\lambda}=\und{\mu}$. This category identifies with 
$$
\prod_{k\in K'} \cTw(fSets)^{\und{\mu}(k)}.
$$ 

 Namely, for each $k\in K'$ the fibre over $k$ is an object $(J_k\to K_k, \und{\lambda}^k)\in \cTw(fSets)^{\und{\mu}(k)}$. So, our claim follows from the lemma below.
\end{proof} 

\begin{Lm} 
\label{Lm_E.1.8}
The category $\cTw(fSets)^{\lambda}$ is contractible. 
\end{Lm} 
\begin{proof}
We have an adjoint pair $l: fSets^{\lambda}\leftrightarrows \cTw(fSets)^{\lambda}: r$, where $l$ sends 
$
(J, \und{\lambda})
$ 
to 
$$
(J\to *, \und{\lambda})\in \cTw(fSets)^{\lambda},
$$ 
and $r$ sends $(J\to K, \und{\lambda})$ to $(J, \und{\lambda})$. By (\cite{Ly}, 2.2.106), $l$ and $r$ induce isomorphisms between the geometric realizations $\mid fSets^{\lambda}\mid\,\iso\, \mid \cTw(fSets)^{\lambda}\mid$.  The category $fSets^{\lambda}$ has a final object, so is contractible.
\end{proof} 

\sssec{} 
\label{Sect_E.1.12_now}
We may view $\eta$ as a functor still denoted $\eta: (fSets^{\lambda})^{op}\to Shv(X^{\lambda})-mod$ by abuse of notations. We have the natural functor 
$$
F: \underset{(fSets^{\lambda})^{op}}{\colim} \eta\to Shv(X^{\lambda})
$$ 
in $Shv(X^{\lambda})-mod$. Namely, for $(\und{\lambda}, J)\in fSets^{\lambda}$ let $s^{\und{\lambda}}: X^J\to X^{\lambda}$ be the map sending $(x_j)_{j\in J}$ to $\sum_{j\in J} \und{\lambda}(j)x_j$. The functors $(s^{\und{\lambda}})_!: Shv(X^J)\to Shv(X^{\lambda})$ are compatible with the transition functors in the diagram $\eta$, hence give rise to the desired functor $F$. We will show that $F$ is an equivalence. 

\sssec{} One may pass to right adjoints in $\eta$ and get a functor $\eta^R: fSets^{\lambda}\to Shv(X^{\lambda})-mod$. So, $\colim \eta\,\iso\, \lim \eta^R$, where both are calculated in $Shv(X^{\lambda})-mod$. 

 From (\cite{Ly}, 9.2.6) we derive the following. For $(\und{\lambda}, J)\in fSets^{\lambda}$ the functors $(s^{\und{\lambda}})^!: Shv(X^{\lambda})\to Shv(X^J)$ are compatible with the transition functors in $\eta^R$, so yield a functor $F^R: Shv(X^{\lambda})\to \lim \eta^R$ in $Shv(X^{\lambda})-mod$, which is the right adjoint to $F$. So, the adjoint pair
\begin{equation}
\label{adj_pair_Sect_E.1.10}
F: \underset{(fSets^{\lambda})^{op}}{\colim} \eta\leftrightarrows Shv(X^{\lambda}): F^R
\end{equation}
takes place in $Shv(X^{\lambda})-mod$.  

\sssec{} 
\label{Sect_E.1.16_now}
Write $\gU(\lambda)$ for a way to write $\lambda=\sum_{k\in K} d_k\lambda_k$, where $K\in fSets$, $d_k>0$ and $\lambda_k\in \Lambda^*$ are pairwise distinct. Recall that $X^{\lambda}$ is stratified by locally closed subschemes $\oo{X}{}^{\gU(\lambda)}$. Namely, let $X^{\gU(\lambda)}=\prod_{k\in K} X^{(d_k)}$. Then $\oo{X}{}^{\gU(\lambda)}\subset X^{\gU(\lambda)}$ is the complement to all the diagonals. 

 For $J\in fSets$ write $\oo{X}{}^J\subset X^J$ for the complement to all the diagonals. 

\sssec{} Write $Q(J)$ for the set of equivalence relations on a finite nonempty set $J$. We think of an element of $Q(J)$ as a map $J\toup{\phi} K$ in $fSets$. We identify such objects $(J\toup{\phi} K)$ and $(J\toup{\phi'} K')$ if there is a compatible bijection $K\,\iso\, K'$. 

\sssec{} Let $(\und{\lambda}, J)\in fSets^{\lambda}$, pick $\gU(\lambda)$ as above. Write $Q(J)^{\gU(\lambda)}\subset Q(J)$ for the subset of equivalence relations $\phi: J\to \tilde K$ such that $\phi_*\und{\lambda}$ is of type $\gU(\lambda)$. This means that there is a surjection $\tilde K\to K$ such that for $k\in K$ the set $\tilde K_k$ consists of $d_k$ elements, and the value of $\phi_*\und{\lambda}$ on $\tilde K_k$ is constant equal to $\lambda_k$. 

For an element $(J\toup{\phi} \tilde K)$ of $Q(J)^{\gU(\lambda)}$ the composition $\oo{X}{}^{\tilde K}\to X^J\toup{s^{\und{\lambda}}} X^{\lambda}$ factors uniquely through $\oo{X}{}^{\gU(\lambda)}\hook{} X^{\lambda}$. In the diagram
$$
\begin{array}{ccc}
X^J & \toup{s^{\und{\lambda}}} & X^{\lambda}\\
\uparrow && \uparrow\\
\oo{X}{}^{\tilde K} & \toup{s^{\phi}} & \oo{X}{}^{\gU(\lambda)}
\end{array}
$$
the map $s^{\phi}$ is a Galois covering with Galois group $\prod_{k\in K} S_{d_k}$, where $S_r$ is the symmetric group on $r$ elements. 

\begin{Lm} 
\label{Lm_E.1.14_now}
Given $\gU(\lambda)$ and $(\und{\lambda}, J)\in fSets^{\lambda}$, the fibre product $X^J\times_{X^{\lambda}} \oo{X}{}^{\gU(\lambda)}$ identifies with 
$$
\underset{(J\toup{\phi} \tilde K)\in Q(J)^{\gU(\lambda)}}{\sqcup} \oo{X}{}^{\tilde K}.
$$
We used the map $s^{\und{\lambda}}$ to form the fibred product. The corresponding map to $\oo{X}{}^{\gU(\lambda)}$ is given by $s^{\phi}$ on $\oo{X}{}^{\tilde K}$ for any $(J\toup{\phi} \tilde K)\in Q(J)^{\gU(\lambda)}$.
\QED
\end{Lm}

\begin{Lm} 
\label{Lm_E.1.21_now}
Pick $\gU(\lambda)$. The functor $F$ becomes an equivalence after the base change $Shv(\oo{X}{}^{\gU(\lambda)})\otimes_{Shv(X^{\lambda})}\cdot$. 
\end{Lm}
\begin{proof}
Write $fSets^{\gU(\lambda)}$ for the category, whose objects are collections: $(\und{\lambda}, J)\in fSets^{\lambda}$, $(J\toup{\phi} \tilde K)\in Q(J)^{\gU(\lambda)}$. A maps from the first object to the second in $fSets^{\gU(\lambda)}$ is given by a morphism $\tau: J_1\to J_2$ in $fSets$ such that $\tau_*\und{\lambda}^1=\und{\lambda}^2$ and the composition $J_1\toup{\tau} J_2\toup{\phi_2} \tilde K_2$ coincides with $(J_1\toup{\phi_1}\tilde K_1)$ in the set $Q(J_1)^{\gU(\lambda)}$. Note that there is a unique isomorphism $\tilde\tau$ in $fSets$ making the diagram below commutative
\begin{equation}
\label{diag_with_tilde}
\begin{array}{ccc}
J_1 & \toup{\phi_1} & \tilde K_1\\
\downarrow\lefteqn{\scriptstyle\tau} && \downarrow\lefteqn{\scriptstyle\tilde\tau}\\
J_2 & \toup{\phi_2} & \tilde K_2
\end{array}
\end{equation}

 Using (\cite{Ly4}, Section~0.3), we get
\begin{equation}
\label{iso_1_for_Pp_answer}
(\colim \eta)\otimes_{Shv(X^{\lambda})} Shv(\oo{X}{}^{\gU(\lambda)})\,\iso\, \underset{(\und{\lambda}, J, J\toup{\phi} \tilde K)\in (fSets^{\gU(\lambda)})^{op}}{\colim} Shv(\oo{X}{}^{\tilde K})
\end{equation}

  Consider the full subcategory $fSets^{\gU(\lambda)}_0\subset fSets^{\gU(\lambda)}$ spanned by triples $(\und{\lambda}, J, J\toup{\phi} \tilde K)$ such that $\phi$ is a bijection. We may also think of $fSets^{\gU(\lambda)}_0$ as a full subcategory of $fSets^{\lambda}$ spanned by those $(\und{\lambda}, J)$ for which $\und{\lambda}$ is of type $\gU(\lambda)$. 

 Consider the functor 
$$
\bar h: fSets^{\gU(\lambda)}\to fSets^{\gU(\lambda)}_0
$$ 
sending $(\und{\lambda}, J, J\toup{\phi} \tilde K)$ to $(\tilde K, \phi_*\und{\lambda})$. It sends a map 
in $fSets^{\gU(\lambda)}$ given by (\ref{diag_with_tilde}) to the morphism $\tilde\tau: \tilde K_1\,\iso\,\tilde K_2$. 

 Consider the functor 
$$
\bar\eta: (fSets^{\gU(\lambda)}_0)^{op}\to Shv(\oo{X}{}^{\gU(\lambda)})-mod
$$ 
sending $(\und{\lambda}, J)$ to $Shv(\oo{X}{}^J)$. It sends an isomorphism $\tilde\tau: J_1\,\iso\, J_2$ in $fSets^{\gU(\lambda)}_0$ to $(\vartriangle_{\tilde\tau})_!: Shv(\oo{X}{}^{J_2})\,\iso\, Shv(\oo{X}{}^{J_1})$, here $\vartriangle_{\tilde\tau}: \oo{X}{}^{J_2}\,\iso\, \oo{X}{}^{J_1}$ is the corresponding map. 

 Now (\ref{iso_1_for_Pp_answer}) identifies with $\underset{(fSets^{\gU(\lambda)})^{op}}{\colim} \bar\eta \comp \bar h$. We claim that 
\begin{equation}
\label{functor_bar_h_proof_Lm_E.1.21}
\bar h: (fSets^{\gU(\lambda)})^{op}\to (fSets^{\gU(\lambda)}_0)^{op}
\end{equation}
is cofinal. Indeed, Let $\xi=(\und{\lambda}^0, \tilde K_0)\in fSets^{\gU(\lambda)}_0$. We must show that 
$$
fSets^{\gU(\lambda)}\times_{fSets^{\gU(\lambda)}_0} (fSets^{\gU(\lambda)}_0)_{/\xi}
$$
is contractible. This category has a final object given by $\xi\in fSets^{\gU(\lambda)}$ together with the isomorphism $\id: \bar h(\xi)\,\iso\,\xi$. We used the fact that the full embedding $fSets^{\gU(\lambda)}_0\subset fSets^{\gU(\lambda)}$ is a section of $\bar h$. 
Thus, $\bar h$ is cofinal. 

 Now (\ref{iso_1_for_Pp_answer}) identifies with
$$
\underset{(fSets^{\gU(\lambda)}_0)^{op}}{\colim} \bar\eta.
$$

By (\cite{Ly}, 2.7.24), for a finite group $G$ acting on $Y\in \Sch_{ft}$ one has $\colim_{B(G)} Y\,\iso\, Y/G$, here $Y/G$ is the prestack quotient. This gives 
$$
Shv(Y/G)\,\iso\, \lim_{B(G)^{op}} Shv(Y)\,\iso\, \underset{B(G)}{\colim} Shv(Y),
$$ 
where in the latter formula we passed to the left adjoints. 

This gives finally 
$$
\underset{(fSets^{\gU(\lambda)}_0)^{op}}{\colim} \bar\eta\;\iso\; Shv(\oo{X}{}^{\gU(\lambda)}).
$$
\end{proof}

\sssec{End of the proof of Proposition~\ref{Pp_E.1.5}} 
\label{Sect_E.1.19}
Since the adjoint pair (\ref{adj_pair_Sect_E.1.10}) takes place in $Shv(X^{\lambda})-mod$, and $F$ is an isomorphism after restriction to each stratum $\oo{X}{}^{\gU(\lambda)}$, our claim follows from Proposition~\ref{Pp_3.7.8_devissage_using_ULA_preservation}. \QED

\begin{Pp} 
\label{Pp_E.1.23_now}
Let $C(X)$ be as in Proposition~\ref{Pp_E.1.5}. Set $A=\underset{\lambda\in\Lambda^*}{\oplus} \omega_X\in CAlg^{nu}(C(X))$ graded by $\Lambda^*$ compatibly with the non-unital commutative algebra structure. Then for $\lambda\in\Lambda^*$ one has canonically $\Fact(A)_{\lambda}\,\iso\,\omega_{X^{\lambda}}$ in $Shv(X^{\lambda})$.
\end{Pp}
\begin{proof} The argument is similar to that of Proposition~\ref{Pp_E.1.5}. As in Section~\ref{Sect_Gr_fact_alg_non-unital}, 
$$
\Fact(A)_{\lambda}\,\iso\, \underset{(\und{\lambda}, J\to K)\in  \cTw(fSets)^{\lambda}}{\colim} \omega_{X^K}
$$ 
taken in $\Fact(C)_{\lambda}$. Consider the functor $\eta_A: (fSets^{\lambda})^{op}\to Shv(X^{\lambda})$ sending $(K, \und{\lambda})$ to $s^{\und{\lambda}}_!\omega_{X^K}$, here $s^{\und{\lambda}}: X^K\to X^{\lambda}$ is defined in Section~\ref{Sect_E.1.12_now}. 
It sends a morphism $(\und{\lambda}^1, K_1)\to (\und{\lambda}^2, K_2)$ in $fSets^{\lambda}$ to the map 
$$
s^{\und{\lambda}^1}_!\vartriangle_!\omega_{X^{K_2}}\to s^{\und{\lambda}^1}_!\omega_{X^{K_1}},
$$
here $\vartriangle: X^{K_2}\to X^{K_1}$ is the closed immersion. Since (\ref{functor_h_Sect_E.1.9}) is cofinal, we get
$$
\Fact(A)_{\lambda}\,\iso\, \underset{(fSets^{\lambda})^{op}}{\colim}\eta_A
$$ 

 For $(\und{\lambda}, K)\in fSets^{\lambda}$ we have the natural map 
$s^{\und{\lambda}}_!\omega_{X^K}\to \omega_{X^{\lambda}}$ in $Shv(X^{\lambda})$. These maps form a compatible system for $(\und{\lambda}, K)\in (fSets^{\lambda})^{op}$, so give rise to a morphism $b: \Fact(A)_{\lambda}\to \omega_{X^{\lambda}}$. We check that $b$ is an isomorphism. 

 Pick $\gU(\lambda)$ as in Section~\ref{Sect_E.1.16_now}. It suffices to show that the $!$-restriction of $b$ under $\oo{X}{}^{\gU(\lambda)}\hook{} X^{\lambda}$ is an isomorphism. The latter is the map 
$$
\underset{(\und{\lambda}, J, J\toup{\phi} \tilde K)\in (fSets^{\gU(\lambda)})^{op}}{\colim} s^{\phi}_!\omega_{\oo{X}{}^{\tilde K}}\to \omega_{\oo{X}{}^{\gU(\lambda)}}
$$
As in the proof of Proposition~\ref{Pp_E.1.5}, it identifies with
$$
\underset{(\und{\lambda}, J)\in (fSets^{\gU(\lambda)}_0)^{op}}{\colim} s^{\phi}_!\omega_{\oo{X}{}^{\tilde K}}\,\iso\, \omega_{\oo{X}{}^{\gU(\lambda)}}.
$$
\end{proof}

\sssec{} 
\label{Sect_E.1.24_now}
For $\lambda\in\Lambda^*$ set $\und{X}^{\lambda}=\underset{(J, \und{\lambda})\in (fSets^{\lambda})^{op}}{\colim} X^J$ taken in $\PreStk$. This is a finitary pseudo-scheme in the sense of (\cite{Ga}, 7.4.6). 

\begin{Lm} For $S\in \Sch^{aff}$ the groupoid $\Map(S, \und{X}^{\lambda})$ is the set of pairs $(\cI, \und{\lambda})$, where $\cI\subset \Map(S, X)$ is a nonempty finite subset, and $\und{\lambda}: \cI\to\Lambda^*$ is a map such that $\sum_{i\in \cI} \und{\lambda}(i)=\lambda$.
\end{Lm}
\begin{proof}
Similar to Lemma~\ref{Lm_about_Ran_A}. 
\end{proof}
\begin{Rem} In (\cite{Ga_Eis_and_Quantum_groups}, 4.1.3) it is claimed that the natural map 
$\und{X}^{\lambda}\to X^{\lambda}$ induces an isomorphism of sheafifications in the topology generated by finite surjective maps.
\end{Rem}

\begin{Rem} 
\label{Rem_E.1.25_now}
Let $C(X)$ be as in Section~\ref{Sect_E.1.1}. We have the canonical map $\und{X}^{\lambda}\to (X^{\lambda}\times\Ran)^{\subset}$ in $\PreStk$ as in Section~\ref{Sect_E.1.3}. The $!$-pullback under this map is a symmetric monoidal functor 
\begin{equation}
\label{functor_for_Rem_E.1.22}
Shv((X^{\lambda}\times\Ran)^{\subset}, \otimes^!)\to Shv(X^{\lambda}, \otimes^!).
\end{equation} 
We may add now that for $\lambda\in\Lambda^*$ the $Shv((X^{\lambda}\times\Ran)^{\subset})$-module structure on $\Fact(C)_{\lambda}$ is obtained from the structure of a $Shv(X^{\lambda})$-module via restriction of scalars with respect to (\ref{functor_for_Rem_E.1.22}). 
\end{Rem}

\sssec{} Our immediate purpose is to generalize Proposition~\ref{Pp_E.1.5} as follows. Let $E(X)\in CAlg^{nu}(Shv(X)-mod)$. Set $C(X)=\underset{\lambda\in\Lambda^*}{\oplus} E(X)$, which we view as $\Lambda^*$-graded non-unital commutative algebra in $Shv(X)-mod$.

\begin{Pp} 
\label{Pp_E.1.22}
Let $\lambda\in\Lambda^*$. One has canonically 
\begin{equation}
\label{iso_of_Pp_E.1.22}
\Fact(E)\otimes_{Shv(\Ran)} Shv(X^{\lambda})\,\iso\, \Fact(C)_{\lambda}.
\end{equation}
\end{Pp}
\begin{proof}
{\bf Step} 1. By Proposition~\ref{Pp_E.1.5}, $Shv(X^{\lambda})\,\iso\, \underset{(I, \und{\lambda})\in (fSets^{\lambda})^{op}}{\colim} Shv(X^I)$, where the transition maps are the !-direct images. So, 
\begin{multline}
\label{complex_first_rewriting_AppE}
\Fact(E)\otimes_{Shv(\Ran)} Shv(X^{\lambda})\,\iso\, \underset{(I, \und{\lambda})\in (fSets^{\lambda})^{op}}{\colim} \Fact(E)\otimes_{Shv(\Ran)} Shv(X^I),\iso\,\\
\underset{(I, \und{\lambda})\in (fSets^{\lambda})^{op}}{\colim} E_{X^I}\,\iso\, \underset{(I, \und{\lambda})\in (fSets^{\lambda})^{op}}{\colim}\;\; \underset{(I\to J\to K)\in \Tw(I)}{\colim} \underset{k\in K}{\boxtimes} E^{\otimes J_k}(X)
\end{multline}

 Write $\cB^{\lambda}$ for the category, whose objects are collections: a diagram $I\toup{p} J\to K$ in $fSets$ and $\und{\lambda}: I\to\Lambda^*$ with $\sum_{i\in I} \und{\lambda}(i)=\lambda$. A morphism from the first object to the second is given by a commutative diagram in $fSets$
$$
\begin{array}{ccccc}
I_1 & \to & J_1 & \to & K_1\\
\uparrow\lefteqn{\scriptstyle \nu} && \downarrow && \uparrow\\
I_2 & \to & J_2 & \to & K_2   
\end{array}
$$
such that $\nu_*\und{\lambda}^2=\und{\lambda}^1$.

 Let $\alpha: \cB^{\lambda}\to (fSets^{\lambda})^{op}$ be the functor sending an object $(I\toup{p} J\to K, \und{\lambda})$ to $(I,\und{\lambda})$. One checks that $\alpha$ is a cocartesian fibration, so we apply (\cite{G}, ch. I.1, 2.2.4) to calculate the LKE along $\alpha$. This shows that (\ref{complex_first_rewriting_AppE}) identifies with
\begin{equation}
\label{complex_second_rewriting_AppE}
\underset{(I\toup{p} J\to K, \,\und{\lambda})\in \cB^{\lambda}}{\colim} \;\underset{k\in K}{\boxtimes} E^{\otimes J_k}(X)
\end{equation}

 Let $b: \cB^{\lambda}\to \cTw(fSets)^{\lambda}$ be the functor sending the above collection to $(p_*\und{\lambda}, {J\to K)}$. Recall the functor $\cF^{\lambda}_{\Ran, C}: \cTw(fSets)^{\lambda}\to Shv(X^{\lambda})-mod$ from Section~\ref{Sect_E.1.2_now}. Now (\ref{complex_second_rewriting_AppE}) is 
$$
\underset{\cB^{\lambda}}{\colim} \; \cF^{\lambda}_{\Ran, C}\comp b.
$$ 

\noindent
{\bf Step} 2. It suffices to show that the LKE of $\cF^{\lambda}_{\Ran, C}\comp b$ along $b$ identifies with $\cF^{\lambda}_{\Ran, C}$.

 One checks that $b$ is a cocartesian fibration. An arrow 
 from $(I_1\to J_1\to K_1, \und{\lambda}^1)$ to $(I_2\to J_2\to K_2, \und{\lambda}^2)$ in $\cB^{\lambda}$ is cocartesian iff $\nu$ is an isomorphism. Now by (\cite{G}, ch. I.1, 2.2.4) for any $\xi=(J\to K, \und{\mu})\in \cTw(fSets)^{\lambda}$ the natural functor
$$
(\cB^{\lambda})_{\xi}\to \cB^{\lambda}\times_{\cTw(fSets)^{\lambda}} (\cTw(fSets)^{\lambda})_{/\xi}
$$
is cofinal, where $(\cB^{\lambda})_{\xi}$ denotes the fibre of $\cB^{\lambda}$ over $\xi$. 
So, the value of the LKE of $\cF^{\lambda}_{\Ran, C}\comp b$ along $b$ at $\xi$ is
$$
\underset{(I\toup{p} J\to K, \und{\lambda})\in (\cB^{\lambda})_{\xi}}{\colim} \; \cF^{\lambda}_{\Ran, C}\comp b
$$
In turn, $(\cB^{\lambda})_{\xi}\to \mid (\cB^{\lambda})_{\xi}\mid$ is cofinal. The category $(\cB^{\lambda})_{\xi}$ has an initial object $(J\toup{\id} J\to K, \und{\mu})$, hence $(\cB^{\lambda})_{\xi}$ is contractible. Our claim follows.
\end{proof}

\sssec{} 
\label{Sect_E.1.28_now}
Assume for a moment that $C\in CAlg^{nu}(\DGCat_{cont})$ is $\Lambda^*$-graded compatibly with the non-unital symmetric monoidal structure $C=\underset{\lambda\in\Lambda^*}{\oplus} C_{\lambda}$. Assume that $C(X)=C\otimes Shv(X)$. As in Section~\ref{Sect_E.1.1}, $\ov{\Fact}(C)$ inherits a $\Lambda^*$-grading, and the canonical arrow $\Fact(C)\to \ov{\Fact}(C)$ is $\Lambda^*$-graded.

 To describe it let $I\in fSets, \lambda\in\Lambda^*$. We get
\begin{equation}
\label{equivalence_for_Sect_E.1.30_now}
\bar C_{X^I, \lambda}\,\iso\, \underset{(I\toup{p} J\to K)\in \Tw(I)^{op}}{\lim} (C^{\otimes K})_{\lambda}\otimes Shv(X^I_{p, d}),
\end{equation}
where $(C^{\otimes K})_{\lambda}$ denotes the $\lambda$-component of $C^{\otimes K}$, the transition functors being defined as for the functor $\bar\cG_{I, C}$ in Section~\ref{Sect_4.1.1_now}. Then $\bar C_{X^I}\,\iso\, \underset{\lambda\in\Lambda^*}{\oplus}\bar C_{X^I, \lambda}$ in $Shv(X^I)-mod$.  
 
 As in Lemma~\ref{Lm_4.1.2_restrictions}, the above gradings are compatible for a map $I\to I'$ in $fSets$ with the !-restrictions under $X^{I'}\to X^I$. So, for $\lambda\in\Lambda^*$ we get the sheaf of categories $\ov{\Fact}(C)_{\lambda}$ on $\Ran$ given by the compatible collection $\{\bar C_{X^I, \lambda}\}$ for $I\in fSets$. One gets canonically
$$
\ov{\Fact}(C)\,\iso\,\underset{\lambda\in\Lambda^*}{\oplus} \ov{\Fact}(C)_{\lambda}.
$$
\begin{Rem} Let $_{\lambda}\Tw(I)$ be the category whose objects are collections $(I\to J\to K, \und{\lambda})$, where $(I\to J\to K)\in \Tw(I)$ and $\und{\lambda}: K\to\Lambda^*$. A map from $(I\to J_1\to K_1, \und{\lambda}^1)$ to $(I\to J_2\to K_2, \und{\lambda}^2)$ is a morphism from $(I\to J_1\to K_1)$ to $(I\to J_2\to K_2)$ in $\Tw(I)$ such that $\phi_*\und{\lambda}^2=\und{\lambda}^1$ for the underlying map $\phi: K_2\to K_1$. The projection $\xi: {_{\lambda}\Tw(I)}\to \Tw(I)$ is a cartesian fibration. We get a functor $_{\lambda}\bar\cG_{I,C}: {_{\lambda}\Tw(I)^{op}}\to Shv(X^I)-mod$ sending $(I\toup{p} J\to K, \und{\lambda})$ to 
$$
Shv(X^I_{p,d})\otimes (\underset{k\in K}{\otimes} C_{\und{\lambda}(k)}),
$$ 
the transition maps being defined as for $\bar\cG_{I, C}$. It is not true that
$\underset{_{\lambda}\Tw(I)^{op}}{\lim} {_{\lambda}\bar\cG_{I,C}}\,\iso\, \bar C_{X^I, \lambda}$. The reason is that $\xi$ is not cocartesian in general.
\end{Rem}

\begin{Rem} In the situation of Section~\ref{Sect_E.1.28_now} assume also that $C$ is dualizable in $\DGCat_{cont}$, and $m: C^{\otimes 2}\to C$ has a continuous right adjoint. Then a graded version of Proposition~\ref{Pp_4.1.10_Raskin_dualizability} holds. Namely, for $\lambda\in\Lambda^*$, $I\in fSets, D\in Shv(X^I)-mod$ the canonical arrow
$$
\bar C_{X^I, \lambda}\otimes_{Shv(X^I)} D\to \underset{(I\toup{p} J\to K)\in \Tw(I)^{op}}{\lim} (C^{\otimes K})_{\lambda}\otimes Shv(X^I_{p, d})\otimes_{Shv(X^I)} D
$$
is an equivalence in $Shv(X^I)-mod$. Here $(C^{\otimes K})_{\lambda}$ is the $\lambda$-component of the corresponding grading on $C^{\otimes K}$. This follows from (\cite{Ly}, 9.2.77).
\end{Rem}

\sssec{} 
\label{Sect_E.1.33}
Assume in the situation of Section~\ref{Sect_E.1.28_now} that for any $\lambda\in\Lambda^*$ the category $C_{\lambda}$ is compactly generated and equipped with a compactly generated t-structure (cf. Section~\ref{Sect_4.1.12_now_t-str}). We equip $C$ with the t-structure given by $C^{\le 0}=\prod_{\lambda\in\Lambda^*} C_{\lambda}^{\le 0}$. Assume for any $\lambda,\lambda'\in\Lambda^*$ the product $C_{\lambda}\otimes C_{\lambda'}\to C_{\lambda+\lambda'}$ is t-exact. 

 Recall that for $I\in fSets$ the categories $\bar C_{X^I}$ and $\ov{\Fact}(C)$ are equipped with t-structures as in Section~\ref{Sect_4.1.12_now_t-str}. These t-structures are compatible with gradings of $\bar C_{X^I}$ and $\ov{\Fact}(C)$ respectively, so each $\bar C_{X^I,\lambda}$ and $\ov{\Fact}(C)_{\lambda}$ inherits a t-structure. All these t-structures are compatible with filtered colimits. 
  
\begin{Rem} In the situation of Section~\ref{Sect_E.1.33} assume that the canonical arrow $\Fact(C)\to\ov{\Fact}(C)$ is an equivalence. Let $A=\underset{\lambda\in\Lambda^*}{\oplus} A_{\lambda}\in CAlg^{nu}(C)$ with $A_{\lambda}\in C_{\lambda}$. Assume that for any $\lambda\in\Lambda^*$ one has $A_{\lambda}\in C_{\lambda}^{\heartsuit}$. Recall that for 
$I\in fSets$ the object $A_{X^I}\,\iso\, \underset{\lambda\in\Lambda^*}{\oplus} A_{X^I, \lambda}$ is $\Lambda^*$-graded in $\bar C_{X^I}$. 

 Then, as in Remark~\ref{Rem_4.1.13_now}, for $\lambda\in\Lambda^*$ the image of $A_{X^I, \lambda}$ in $\bar C_{X^I,\lambda}$ is placed in degrees $<0$. 
\end{Rem} 

\begin{Lm} Let $C(X)$ be as in Proposition~\ref{Pp_E.1.5}. The canonical arrow $\Fact(C)\to\ov{\Fact}(C)$ is an equivalence (compatible with $\Lambda^*$-gradings). 
\end{Lm}
\begin{proof}
{\bf Step 1}. Set $C=\underset{\lambda\in \Lambda^*}{\oplus} \Vect$ and $D=\underset{\lambda\in\Lambda}{\oplus}\Vect\in CAlg(Shv(X)-mod)$.
We claim that the natural arrow $\Fact(C)\to \Fact(D)$ has a continuous $Shv(\Ran)$-linear right adjoint.  Indeed, we have a functor $a: \cTw(fSets)\times[1]\to Shv(\Ran)-mod$ sending $(J\to K)$ to the natural map 
$$
C^{\otimes J}\otimes Shv(X^K)\to D^{\otimes J}\otimes Shv(X^K).
$$ 
This is a natural transformation $\cF_{\Ran, C}\to \cF_{\Ran, D}$ of functors $\cTw(fSets)\to Shv(\Ran)-mod$. We may pass to right adjoints in the diagram $a$ and get a functor $a^R: (\cTw(fSets)\times[1])^{op}\to Shv(\Ran)-mod$. Our claim follows by (\cite{Ly}, 9.2.39). 

\smallskip\noindent
{\bf Step 2}. The natural functor $\ov{\Fact}(C)\to \ov{\Fact}(D)$ is $Shv(\Ran)$-linear and fully faithful. Indeed, it is obtained by passing to the limit over $I\in fSets$
in the functors, say $\gamma_I: \bar C_{X^I}\to \bar D_{X^I}$. So, it suffices to show that $\gamma_I$ is fully faithful. The functor $\gamma_I$ is obtained by passing to the limit over $(I\toup{p} J\to K)\in \Tw(I)^{op}$ in the fully faithful functors 
$$
C^{\otimes K}\otimes Shv(X^I_{p, d})\to D^{\otimes K}\otimes Shv(X^I_{p, d}),
$$ 
so $\gamma_I$ is fully faithful.

\smallskip\noindent
{\bf Step 3} The functor $\Fact(D)\to \ov{\Fact}(D)$ is an equivalence by Proposition~\ref{Pp_4.1.19}. So, the composition $\Fact(C)\to \ov{\Fact}(C)\hook{}\ov{\Fact}(D)$ has a continuous $Shv(\Ran)$-linear right adjoint. Our claim follows now from Lemma~\ref{Lm_C.3.6_now_cont_right_adjoint} below.
\end{proof}
\begin{Lm} 
\label{Lm_C.3.6_now_cont_right_adjoint}
Let $A\in CAlg(\DGCat_{cont})$ and 
$D\toup{L^0} C^0\hook{j} C$ be a diagram in $A-mod$. Set $L=j\comp L^0$. Assume $L$ has a right adjoint $R: C\to D$ in $A-mod$, and $j$ is fully faithful. Then $Rj$ is the right adjoint of $L^0$ in $A-mod$, so $L^0$ admits a right adjoint in $A-mod$.
\end{Lm}
\begin{proof} Let $j^R: C\to C^0$ be the right adjoint to $j$, $R^0: C^0\to D$ be the right adjoint to $L^0$. Then $R^0\comp j^R\,\iso\, R$ and $R\comp j\,\iso\, R^0\comp j^R\comp j\,\iso\, R^0$.
\end{proof}

\sssec{} Let $C(X)$ be as in Proposition~\ref{Pp_E.1.5}. The functor $C(X)\to Shv(X)$ forgetting the grading yields by functoriality the functor $o: \Fact(C)\to Shv(\Ran)$ in $Shv(\Ran)-mod$. For $\lambda\in\Lambda$ write $o_{\lambda}: \Fact(C)_{\lambda}\to Shv(\Ran)$ for its restriction to the summand $\Fact(C)_{\lambda}$.

\begin{Lm} The functor $o: \Fact(C)\to Shv(\Ran)$ is conservative.
\end{Lm}
\begin{proof}
It is obtained by passing to the limit over $I\in fSets$ in the functors $\bar C_{X^I}\to Shv(X^I)$. It remains to show that $\bar C_{X^I}\to Shv(X^I)$ is conservative. The latter is obtained by passing to the limit over $(I\toup{p} J\to K)\in \Tw(I)^{op}$ in the conservative functors $C^{\otimes K}\otimes Shv(X^I_{p, d})\to \Vect^{\otimes K}\otimes Shv(X^I_{p, d})$. By (\cite{Ly}, 2.5.3), we are done.
\end{proof}

\begin{Lm} Let $C=\underset{\lambda\in \Lambda^*}{\oplus} \Vect$, it is equipped with a natural t-structure. For $\lambda\in\Lambda^*$ equip
$\ov{\Fact}(C)_{\lambda}$ with a t-structure as in Section~\ref{Sect_E.1.33}. Then this t-structure coincides with the perverse t-structure on $\Fact(C)_{\lambda}\,\iso\, Shv(X^{\lambda})$, where the latter isomorphism is that of Proposition~\ref{Pp_E.1.5}.
\end{Lm}
\begin{proof}
{\bf Step 1}. Recall the equivalence (\ref{equivalence_for_Sect_E.1.30_now}). For any $(I\toup{p}  J\to K)\in\Tw(I)$ consider the natural map 
$$
(I\toup{p} J\to *)\to (I\toup{p}  J\to K)
$$ 
in $\Tw(I)$, it gives the transition map 
$$
\gamma: \underset{\und{\lambda}: K\to \Lambda^*, \sum_k \und{\lambda}(k)=\lambda}{\oplus} Shv(X^I_{p, d})\,\iso\, (C^{\otimes K})_{\lambda}\otimes Shv(X^I_{p, d})\to C_{\lambda}\otimes Shv(X^I_{p, d})=Shv(X^I_{p, d})
$$ 
in the above diagram. Here $\gamma$ comes from the coproduct over $\und{\lambda}$ as above of the identity maps. For $\cK\in (C^{\otimes K})_{\lambda}\otimes Shv(X^I_{p, d})$, the object $\cK$ is connective/coconnective iff $\gamma(\cK)$ is connective/coconnective. 

 Thus, $\cK\in \bar C_{X^I, \lambda}$ is connective/coconnective iff for any $(I\toup{p} J\to *)\in\Tw(I)$ its image in $C_{\lambda}\otimes Shv(X^I_{p, d})=Shv(X^I_{p, d})$ is connective/coconnective (cf. \cite{Ly}, 10.1.6). 

 For any $(I\toup{p} J\to *)\in\Tw(I)$ we have a canonical map $(I\toup{p} J\to *)\to (I\to *\to *)$ in $\Tw(I)$, which gives the transition map $C_{\lambda}\otimes Shv(X^I)\to C_{\lambda}\otimes Shv(X^I_{p, d})$ given by restriction to this open subset. Thus, $\cK\in \bar C_{X^I, \lambda}$ is connective/coconnective iff its image in $C_{\lambda}\otimes Shv(X^I)$ is connective/coconnective. 
 
\smallskip\noindent
{\bf Step 2}. Given a map $I\to I'$ in $fSets$, the diagram commutes naturally
$$
\begin{array}{ccc}
\bar C_{X^I,\lambda} & \toup{\vartriangle^!} & \bar C_{X^{I'}, \lambda}\\
\downarrow && \downarrow\\
C_{\lambda}\otimes Shv(X^I) & \toup{\vartriangle^!} & C_{\lambda}\otimes Shv(X^{I'}),
\end{array}
$$
here $\vartriangle: X^{I'}\to X^I$. They organize into a morphism $\bar C_{X^I,\lambda}\to C_{\lambda}\otimes Shv(X^I)$ functorial in $I\in fSets$. Passing to the limit over $I\in fSets$, this gives the functor 
$$
o_{\lambda}: Shv(X^{\lambda})\,\iso\,\Fact(C)_{\lambda}\to Shv(\Ran).
$$ 

 Now $\cK\in (\Fact(C)_{\lambda})^{\ge 0}$ iff for any $I\in fSets$, its image in $\bar C_{X^I,\lambda}$ is placed in degrees $\ge 0$, that is, its image in $Shv(X^I)$ is placed in degrees $\ge 0$. So, $\cK\in (\Fact(C)_{\lambda})^{\ge 0}$ iff $o_{\lambda}(\cK)\in Shv(\Ran)^{\ge 0}$. 

  Recall the prestack $\und{X}^{\lambda}$ define in Section~\ref{Sect_E.1.24_now}. Consider the map $\eta_{\lambda}: \und{X}^{\lambda}\to \Ran$ defined as in Section~\ref{Sect_E.1.3}. Namely, for $(J, \und{\lambda})\in fSets^{\lambda}$ the composition $X^J\to \und{X}^{\lambda}\toup{\eta_{\lambda}} \Ran$ is the natural map $X^J\to\Ran$. Note that $\eta_{\lambda}$ is pseudo-proper. The functor $o_{\lambda}$ identifies with $(\eta_{\lambda})_*$. 

 Finally, one checks $\cK\in Shv(X^{\lambda})$ is placed in perverse degrees $\ge 0$ iff $(\eta_{\lambda})_*\cK$ is placed in perverse degrees $\ge 0$. Indeed, recall  the stratification of $X^{\lambda}$ by locally closed subschemes $\oo{X}{}^{\gU(\lambda)}$ from Section~\ref{Sect_E.1.16_now}. Use that $\cK\in Shv(X^{\lambda})$ is placed in perverse degrees $\ge 0$ iff its !-restriction to each stratum is placed in perverse degrees $\ge 0$.  
\end{proof}

\sssec{} We finish this subsection with an analog of Proposition~\ref{Pp_E.1.22} for commutative factorization algebras.

\begin{Pp} Let $\lambda\in\Lambda^*$. Under the assumptions of Proposition~\ref{Pp_E.1.22} let $A\in CAlg^{nu}(E(X))$. Set $B=\underset{\lambda\in\Lambda^*}{\oplus} A$ viewed as a $\Lambda^*$-graded object of $CAlg^{nu}(C(X))$. The image of $\Fact(A)\boxtimes \omega_{X^{\lambda}}$ under the equivalence (\ref{iso_of_Pp_E.1.22}) identifies canonically with $\Fact(B)_{\lambda}$. 
\end{Pp}
\begin{proof}
As in the proof of Proposition~\ref{Pp_E.1.23_now} 
one has
\begin{equation}
\label{expression_1_addendum_for_Pp_E.1.29}
\Fact(A)\boxtimes \omega_{X^{\lambda}}\,\iso\, \underset{(\und{\lambda}, I)\in (fSets^{\lambda})^{op}}{\colim} \Fact(A)\boxtimes s^{\und{\lambda}}_!\omega_{X^I}\,\iso\, \underset{(\und{\lambda}, I)\in (fSets^{\lambda})^{op}}{\colim} s^{\und{\lambda}}_! A_{X^I}
\end{equation}
Recall from Section~\ref{Sect_2.4.2_v4} that 
$$
A_{X^I}\,\iso\, \mathop{\colim}\limits_{(I\to J\to K)\in Tw(I)} \; (\mathop{\boxtimes}\limits_{k\in K} A^{\otimes J_k})
$$
in $\Fact(E)$. So, (\ref{expression_1_addendum_for_Pp_E.1.29}) becomes
$$
\underset{(\und{\lambda}, I)\in (fSets^{\lambda})^{op}}{\colim} \mathop{\colim}\limits_{(I\to J\to K)\in Tw(I)} \; s^{\und{\lambda}}_! (\mathop{\boxtimes}\limits_{k\in K} A^{\otimes J_k})
$$
in $\Fact(C)_{\lambda}$. As in the proof of Proposition~\ref{Pp_E.1.22}, 
the latter expression identifies with
$$
\underset{(I\toup{p} J\to K, \und{\lambda})\in \cB^{\lambda}}{\colim} s^{\und{\lambda}}_! (\mathop{\boxtimes}\limits_{k\in K} A^{\otimes J_k})
\,\iso\, \underset{\cB^{\lambda}}{\colim} \; \cF^{\lambda}_{\Ran, B}\comp b\,\iso\,
\colim \cF^{\lambda}_{\Ran, B}\,\iso\, \Fact(B)_{\lambda},
$$
here $\cF^{\lambda}_{\Ran, B}: \cTw(fSets)^{\lambda}\to \Fact(C)_{\lambda}$.
\end{proof}

\ssec{Unital version}
\label{Sect_E2}

\sssec{} Assume we are in the situation of Section~\ref{Sect_E.0.1}, so in the unital setting.
 
\begin{Pp} 
\label{Pp_E.2.3}
Take $C(X)=\underset{\lambda\in\Lambda}{\oplus} Shv(X)$, the category of $\Lambda$-graded sheaves on $X$. For $\lambda\in\Lambda$ one has a canonical equivalence $\Fact(C)_{\lambda}\,\iso\, Shv((X^{\lambda}\times\Ran)^{\subset})$
in $Shv(X^{\lambda})-mod$.
\end{Pp}
In Sections~\ref{Sect_E.2.3_now}-\ref{Lm_E.2.13_almost_last} we prove Proposition~\ref{Pp_E.2.3}. 
 
\sssec{} 
\label{Sect_E.2.3_now}
Fix $\lambda\in\Lambda$. Write $fSets_{\lambda}$ for the category whose objects are pairs $(J, \und{\lambda}: J\to \Lambda)$ with $\sum_{j} \und{\lambda}(j)=\lambda$. A map from $(J_1, \und{\lambda}^1)$ to $(J_2,\und{\lambda}^2)$ is a map $\phi: J_1\to J_2$ in $fSets$ with $\phi_*\und{\lambda}^1=\und{\lambda}^2$. Write 
$$
h^u: \cTw(fSets)_{\lambda}\to (fSets_{\lambda})^{op}
$$ 
for the functor sending $(J\toup{\phi} K, \und{\lambda}: J\to \Lambda)$ to $(K, \phi_*\und{\lambda})$. As in Lemma~\ref{Lm_E.1.7} one shows that $h^u$ is cofinal. As in as in Lemma~\ref{Lm_E.1.8} one checks that $\cTw(fSets)_{\lambda}$ is contractible.

 Consider the functor 
$$
\eta^u: (fSets_{\lambda})^{op}\to \DGCat_{cont}, \;\; (J, \und{\lambda})\mapsto Shv(X^J),
$$ 
where the transition functors are given by $!$-direct images. By the above,
$$
\Fact(C)_{\lambda}\,\iso\, \underset{(J, \und{\lambda})\in (fSets_{\lambda})^{op}}{\colim} Shv(X^J)=\colim \eta^u. 
$$ 

\sssec{} For $(J, \und{\lambda})\in fSets_{\lambda}$ consider the map $\pi^{\und{\lambda}}: X^J\to (X^{\lambda}\times\Ran)^{\subset}$ sending $\{x_j\}_{j\in J}\in X^J$ to $\cI=\{x_j\}\in\Ran$, $D\in X^{\lambda}$ with 
$$
D=\underset{j\in J^*}{\sum} x_j \und{\lambda}(j)
$$
These maps form a compatible system yielding the map denoted
$$
\pi: \underset{(J, \und{\lambda})\in (fSets_{\lambda})^{op}}{\colim} X^J\to (X^{\lambda}\times\Ran)^{\subset}.
$$ 
Here the colimit is taken in $\PreStk$, and $J^*=\{j\in J\mid \und{\lambda}(j)\ne 0\}$. 

We will show that the functor 
$$
\pi^!: Shv((X^{\lambda}\times\Ran)^{\subset})\to Shv(\underset{(J, \und{\lambda})\in (fSets_{\lambda})^{op}}{\colim} X^J)\,\iso\,\underset{(J, \und{\lambda})\in (fSets_{\lambda})}{\lim} Shv(X^J)
$$ 
is an equivalence. Here the limit is calculated in $\DGCat_{cont}$. 

\begin{Lm} 
\label{Lm_E.2.5_now}
For $(J,\und{\lambda})\in fSets_{\lambda}$ the map $\pi^{\und{\lambda}}$ is pseudo-proper.
\end{Lm}
\begin{proof} Fix $I\in fSets$ and consider the base change of $\pi^{\und{\lambda}}$ by $X^I\to \Ran$. Recall that $X^J\times_{\Ran} X^I\,\iso\, \underset{(I\to I'\gets J)}{\colim} X^{I'}$, where the colimit is over $(fSets_{I/}\times_{fSets} fSets_{J/})^{op}$. 

 Let now $S\in\Sch^{aff}$, and $S\to (X^{\lambda}\times X^I)^{\subset}$ be given. Then the prestack 
$$
X^J\times_{(X^{\lambda}\times\Ran)^{\subset}} S
$$ 
identifies with 
$$
\underset{(I\to I'\gets J)}{\colim} X^{I'}\times_{X^{\lambda}\times X^I} S
$$
taken over the same category. Each scheme $X^{I'}\times_{X^{\lambda}\times X^I} S$ is proper over $S$.
\end{proof}

\sssec{} By (\cite{Ga}, 1.5.4), the functors $\pi^{\und{\lambda}}_!: Shv(X^J)\to Shv((X^{\lambda}\times\Ran)^{\subset})$ are defined. They form a compatible family for $(J, \und{\lambda})\in fSets_{\lambda}$, hence give rise to a functor 
$$
\pi_!: \underset{(J,\und{\lambda})\in fSets_{\lambda}}{\colim} Shv(X^J)\to Shv((X^{\lambda}\times\Ran)^{\subset}).
$$ 
in $\DGCat_{cont}$. From (\cite{Ly}, 9.2.6) we see that we get an adjoint pair
\begin{equation}
\label{adj_pair_Sect_E_1.1.7}
\pi_!: \underset{(J,\und{\lambda})\in fSets_{\lambda}}{\colim} Shv(X^J)\leftrightarrows Shv((X^{\lambda}\times\Ran)^{\subset}): \pi^!
\end{equation}
in $Shv((X^{\lambda}\times\Ran)^{\subset})-mod$. 

\sssec{Marked Ran space} 
\label{Sect_Marked Ran space}
The following definition has appeared in (\cite{Ga_contractibility_Ran}, 3.5.2). Let $A$ be a finite nonempty set. Consider the category $fSets_A$ whose objects are finite nonempty sets $I$ equipped with a map $\alpha: A\to I$. A map from $(\alpha_1, I_1)$ to $(\alpha_2, I_2)$ is a morphism $\phi: I_1\to I_2$ in $fSets$ compatible with $\alpha_i$, that is, $\alpha_2=\phi\alpha_1$.

 Set $\Ran_A=\underset{(I,\alpha)\in (fSets_A)^{op}}{\colim} X^I$ in $\PreStk$. 
 
\index{Marked Ran space} 
 
\begin{Lm} 
\label{Lm_about_Ran_A}
$\Ran_A$ is the prestack sending $S\in\Sch^{aff}$ to the set of pairs $(\alpha, \cI)$, where $\alpha: A\to \cI$ is a map, and $\cI\subset \Map(S, X)$ is a finite nonempty subset.
\end{Lm}
\begin{proof}
We have $\Map(S, \Ran_A)\,\iso\,\underset{(I,\alpha)\in (fSets_A)^{op}}{\colim} \Map(I, \Map(S, X))$ taken in $\Spc$. Here $\cA=\Map(S, X)$ is a set.
 
 For any set $\cA$, $\underset{(I,\alpha)\in (fSets_A)^{op}}{\colim} \Map(I, \cA)$ identifies with the set of pairs $(\beta, \cI)$, where $\cI\subset \cA$ is a finite subset, and $\beta: A\to \cI$ is a map.
 
 Indeed, consider the functor $\kappa: (fSets_A)^{op}\to\Spc$ sending $(I,\alpha)$ to $\Map(I, \cA)$. Let $\cC\to (fSets_A)^{op}$ be the cocartesian fibration in spaces attached to $\kappa$. Recall that $\colim\kappa$ identifies with the geometric realization of $\cC$.
The category $\cC$ is the disjoint union of categories $\cC_{(\cI,\beta)}$ indexed by pairs: a finite nonempty subset $\cI\subset \cA$ and a map $\beta: A\to \cI$. Namely, $\cC_{(\cI,\beta)}$ is the category classifying $(I, \alpha)\in (fSets_A)^{op}$ and a surjection $\phi: I\to \cI$ such that $\phi\alpha=\beta$.  The category $\cC_{(\cI,\beta)}$ has an initial object, so is contractible. 
\end{proof}

We have the projection $\Ran_A\to X^A$. 

\sssec{Example} 
\label{Sect_Example_E.2.9}
If $\lambda\in\Lambda$ is one of the generators of the semi-group 
$\Lambda$ then the category $fSets_{\lambda}$ is the category $fSets_*$ for $A=\{*\}$ consisting of one element. So, 
$$
\underset{(J,\und{\lambda})\in (fSets_{\lambda})^{op}}{\colim} X^J\,\iso\, \Ran_*
$$
for $A=*$. In this case the natural map $\pi: \Ran_*\to (X\times\Ran)^{\subset}$ induces an isomorphism $\pi^!: Shv((X\times\Ran)^{\subset})\,\iso\, Shv(\Ran_*)$, though $\pi$ is not an isomorphism of prestacks. 

\sssec{Proof of Proposition~\ref{Pp_E.2.3}} Since the adjoint pair (\ref{adj_pair_Sect_E_1.1.7}) takes place in $Shv(X^{\lambda})-mod$, we are reduced in view of Proposition~\ref{Pp_3.7.8_devissage_using_ULA_preservation} to Lemma~\ref{Lm_E.2.13_almost_last} below. \QED

\sssec{}  Pick $\gU(\lambda)$ given by $\lambda=\sum_{k\in K} d_k\lambda_k$, where $K\in fSets$, $d_k>0$, and $\lambda_k\in\Lambda^*$ are pairwise distinct.

\begin{Lm} 
\label{Lm_E.2.13_almost_last}
After the base change $Shv(\oo{X}{}^{\gU(\lambda)})\otimes_{Shv(X^{\lambda})}\cdot$ the functors (\ref{adj_pair_Sect_E_1.1.7}) become equivalences.
\end{Lm}
\begin{proof} By (\cite{Ly4}, 0.3) the category 
$$
\underset{(J,\und{\lambda})\in (fSets_{\lambda})^{op}}{\colim} Shv(X^J)\otimes_{Shv(X^{\lambda})} Shv(\oo{X}{}^{\gU(\lambda)})\,\iso\, \underset{(J,\und{\lambda})\in (fSets_{\lambda})^{op}}{\colim} Shv(X^J\times_{X^{\lambda}} \oo{X}{}^{\gU(\lambda)})
$$
is the category of sheaves on 
\begin{equation}
\label{colimit_for_proof_of_Lm_E_1.1.12_now}
\underset{(J,\und{\lambda})\in (fSets_{\lambda})^{op}}{\colim} X^J\times_{X^{\lambda}} \oo{X}{}^{\gU(\lambda)},
\end{equation}
where the colimit is taken in $\PreStk$. Note that 
$$
X^J\times_{X^{\lambda}} \oo{X}{}^{\gU(\lambda)}\,\iso\, X^{J-J^*}\times (X^{J^*}\times_{X^{\lambda}} \oo{X}{}^{\gU(\lambda)}),
$$ 
and 
$$
X^{J^*}\times_{X^{\lambda}} \oo{X}{}^{\gU(\lambda)}\,\iso\,
\underset{(J^*\toup{\phi}\tilde K)\in Q(J^*)^{\gU(\lambda)}}{\sqcup} \oo{X}{}^{\tilde K}
$$
by Lemma~\ref{Lm_E.1.14_now}.  

  Write $fSets_{\gU(\lambda)}$ for the category whose objects are collections: $(\und{\lambda}, J)\in fSets_{\lambda}$, $(J^*\toup{\phi} \tilde K)\in Q(J^*)^{\gU(\lambda)}$. A map from the first object to the second in $fSets_{\gU(\lambda)}$ is given by a morphism $\tau: J_1\to J_2$ in $fSets$ such that $\tau_*\und{\lambda}^1=\und{\lambda}^2$, and the composition $J_1^*\toup{\tau} J_2^*\toup{\phi_2} \tilde K_2$ coincides with $J_1^*\toup{\phi_1} \tilde K_1$ in the set $Q(J_1^*)^{\gU(\lambda)}$. Then there is a unique isomorphism $\tilde\tau: \tilde K_1\,\iso\, \tilde K_2$ making the diagram commutative
$$
\begin{array}{ccccc}
J_1 & \supset & J_1^* & \toup{\phi_1} & \tilde K_1\\
\downarrow\lefteqn{\scriptstyle\tau} &&
\downarrow\lefteqn{\scriptstyle\tau} && \downarrow\lefteqn{\scriptstyle \tilde\tau}\\
J_1 & \supset & J_2^* & \toup{\phi_2} & \tilde K_2
\end{array}
$$
 
 The projection $fSets_{\gU(\lambda)}\to fSets_{\lambda}$ is a cartesian fibration. Using (\cite{Ly}, 2.7.24),  
one identifies (\ref{colimit_for_proof_of_Lm_E_1.1.12_now}) with
\begin{equation}
\label{colim__for_proof_of_Lm_E_1.1.12_now_2}
\underset{(\und{\lambda}, J, \; J^*\toup{\phi}\tilde K)\in (fSets_{\gU(\lambda)})^{op}}{\colim} X^{J-J^*}\times \oo{X}{}^{\tilde K}
\end{equation}

 Consider the full subcategory $fSets_{\gU(\lambda)}^0\subset fSets_{\gU(\lambda)}$ spanned by those objects for which the map $\phi: J^*\to\tilde K$ is bijective. We may also view $fSets_{\gU(\lambda)}^0$ as a full subcategory of $fSets_{\lambda}$ spanned by those $(J, \und{\lambda})$ for which $(J^*, \und{\lambda}^*)$ is of type $\gU(\lambda)$. Here $\und{\lambda}^*: J^*\to \Lambda^*$ is the restriction of $\und{\lambda}$. 
 
  The above embedding admits a left adjoint $\bar h^u: fSets_{\gU(\lambda)}\to fSets_{\gU(\lambda)}^0$ sending $(\und{\lambda}, J, \phi: J^*\to\tilde K)$ to $(J',\und{\lambda}')$, where $J'=(J-J^*)\sqcup \tilde K$, and $\und{\lambda}': J'\to \Lambda$ is the map $\phi_*\und{\lambda}^*: \tilde K\to \Lambda$ extended by zero from $\tilde K$ to $J'$. Given a map from $(\und{\lambda}^1, J_1, \phi_1: J_1^*\to\tilde K_1)$ to $(\und{\lambda}^2, J_2, \phi_2: J_2^*\to\tilde K_2)$ in $fSets_{\gU(\lambda)}$ as above, the composition $J_1\toup{\tau} J_2\to J'_2$ factors uniquely through a morphism $\tau': J'_1\to J'_2$. The functor $\bar h^u$ sends the map $\tau$ to $\tau'$. 
  
  Consider the functor 
$$
\mu: (fSets_{\gU(\lambda)}^0)^{op}\to \PreStk_{lft}
$$ 
obtained by restricting the diagram (\ref{colim__for_proof_of_Lm_E_1.1.12_now_2}) to this full subcategory. Then (\ref{colim__for_proof_of_Lm_E_1.1.12_now_2}) is
$$
\underset{(fSets_{\gU(\lambda)})^{op}}{\colim} \mu (\bar h^u)^{op}
$$
The functor $\mu (\bar h^u)^{op}$ is the LKE of $\mu$ along the inclusion $(fSets_{\gU(\lambda)}^0)^{op}\subset (fSets_{\gU(\lambda)})^{op}$ by (\cite{Ly}, 2.2.39). 
So, (\ref{colim__for_proof_of_Lm_E_1.1.12_now_2}) identifies with
$$
\underset{(fSets_{\gU(\lambda)}^0)^{op}}{\colim}  \mu.
$$
The latter colimit is the prestack sending $S\in\Sch^{aff}$ to the set of collections: a finite subset $\cI\subset \Map(S,X)$ given by points $(x_i)_{i\in \cI}$ together with subsets $\cI_k\subset \cI$ for $k\in K$ such that $\sum_{i\in \cI_k} x_i\in \Map(S, \oo{X}{}^{(d_k)})$, and the collection $((x_i)_{i\in \cI_k}$ for $k\in K)$ is a point of $\oo{X}{}^{\gU(\lambda)}$.  
The category of sheaves on the latter prestack identifies with 
$$
Shv(\oo{X}{}^{\gU(\lambda)}\times_{X^{\lambda}} (X^{\lambda}\times\Ran)^{\subset}).
$$
\end{proof}

\sssec{} Let $\cA$ be a set. One has the following version of Lemma~\ref{Lm_about_Ran_A}. For $(J,\und{\lambda})\in fSets_{\gU(\lambda)}^0$ write
$\Map(J, \cA)_0=\{f: J\to \cA\mid $ the restriction of $f$ to $J^*$ is injective$\}$. 

\begin{Lm} The colimit $\underset{(J,\und{\lambda})\in (fSets_{\gU(\lambda)}^0)^{op}}{\colim} \Map(J, \cA)_0$ in $\Spc$ identifies with the set of collections: a finite subset $\cI\subset \cA$, and a collection of disjoint subsets $\cI_k\subset \cJ$ for $k\in K$ such that $\cJ_k$ contains $d_k$ elements. \QED
\end{Lm}

\ssec{Adding and removing units}
\label{Sect_Adding and removing units}

\sssec{} Let us be in the setting of Section~\ref{Sect_E.0.1}. So, $C(X)\in CAlg(Shv(X)-mod)$ is $\Lambda$-graded. Assume the natural map $Shv(X)\to C(X)_0$ is an isomorphism. 
Set 
$$
C_{>0}=\underset{\lambda\in\Lambda^*}{\oplus} C(X)_{\lambda}\in CAlg^{nu}(Shv(X)-mod).
$$ 
By functoriality, we have a canonical map $\Fact(C_{>0})\to \Fact(C)$ in $Shv(\Ran)-mod$. 
compatible with the commutative chiral products. For $\lambda\in\Lambda^*$ it restricts to a map
\begin{equation}
\label{map_for_E.2.15}
\Fact(C_{>0})_{\lambda}\to \Fact(C)_{\lambda}
\end{equation}
in $Shv((X^{\lambda}\times\Ran)^{\subset})-mod$. The full embedding $j: \cTw(fSets)^{\lambda}\hook{} \cTw(fSets)_{\lambda}$ is zero-cofinal (cf. Section~\ref{Def_zero-cofinal}), so (\ref{map_for_E.2.15}) is the natural map
$$
\underset{\cTw(fSets)^{\lambda}}{\colim} \cF_{\Ran, C,\lambda}\comp j\to 
\underset{\cTw(fSets)_{\lambda}}{\colim} \cF_{\Ran, C,\lambda}.
$$

\index{Adding and removing units for commutative factorization categories}

 From Proposition~\ref{Pp_E.2.3} we see that $\Fact(C)_0\,\iso\, Shv(\Ran)$ canonically. 

\sssec{Example} 
\label{Sect_E.3.2}
Recall the canonical projection $\pr: \Ran_*\to X$. If $\lambda$ is one of generators of $\Lambda$ then 
$$
\Fact(C)_{\lambda}\,\iso\, C(X)_{\lambda}\otimes_{Shv(X)} Shv(\Ran_*),
$$ 
where we used the $Shv(X)$-action on $Shv(\Ran_*)$ via $\pr^!: Shv(X)\to Shv(\Ran_*)$.
 
 Indeed, $h^u: \cTw(fSets)_{\lambda}\to (fSets_{\lambda})^{op}$ is cofinal (cf. Section~\ref{Sect_E.2.3_now}). So, 
$$
\Fact(C)_{\lambda}\,\iso\, \underset{(K, *)\in (fSets_*)^{op}}{\colim} C(X)_{\lambda}\otimes_{Shv(X)} Shv(X^K),
$$
where $Shv(X)$ acts on $Shv(X^K)$ by $\pr_K^!$ for the projection $\pr_K: X^K\to X$ attached to $*\in K$. For a map $(K_1, *)\to (K_2, *)$ in $fSets_*$ the corresponding transition map is $\id\otimes\vartriangle_!$ for $\vartriangle: X^{K_2}\to X^{K_1}$.  

 In this case the map (\ref{map_for_E.2.15}) idenifies with $\id\otimes q_!: C(X)_{\lambda}\to 
C(X)_{\lambda}\otimes_{Shv(X)} Shv(\Ran_*)$, where $q: X\to\Ran_*$ sends $x\in X$ to the pointed set $\{x\}$.  

\sssec{} For $\lambda\in\Lambda^*$ we have the diagram of functors
$$
\begin{array}{ccc}
\cTw(fSets)^{\lambda} & \hook{j} & \cTw(fSets)_{\lambda}\\
 & \searrow\lefteqn{\scriptstyle \id} & \downarrow\lefteqn{\scriptstyle \xi}\\
 && \cTw(fSets)^{\lambda}, 
\end{array}
$$ 
where $\xi$ sends $(\und{\lambda}, J\toup{\phi} K)$ to $(\und{\lambda}^*, J^*\to K^*)$. Here $J^*=\{j\in J\mid\und{\lambda}(j)\ne 0\}$, $K^*=\phi(J^*)$, and $\und{\lambda}^*$ is the restriction of $\und{\lambda}$.

\begin{Lm} 
\label{Lm_xi_is_cofinal}
The functor $\xi$ is cofinal.
\end{Lm}
\begin{proof}
Let $a=(\und{\mu}, \bar J\to \bar K)\in \cTw(fSets)^{\lambda}$. We must show that
$$
D:=\cTw(fSets)_{\lambda}\times_{\cTw(fSets)^{\lambda}} (\cTw(fSets)^{\lambda})_{a/}
$$ 
is contractible.

 Write $\cH_a$ for the category classifying $J\in fSets$ and a map $\phi: \bar J\to J$ such that for $J^*:=Im(\phi)$ there is a (automatically unique) surjection $J^*\to \bar K$ making the diagram commutative
$$
\begin{array}{ccc}
\bar J & \toup{\phi} & J^*\\
& \searrow & \downarrow\\
&& \bar K.
\end{array}
$$ 
In other words, we require that $\bar K\le J^*$ as elements of $Q(\bar J)$. The morphisms in $\cH_a$ are surjections $J_1\to J_2$ compatible with the maps $\phi_i: \bar J\to J_i$.

For an object of $\cH_a$ as above we let $\und{\lambda}=\phi_*\und{\mu}$.  We have the functor 
$$
\nu: D\to \cH_a
$$ 
sending $(\und{\lambda}, J\toup{\tau} K, \bar J\toup{\phi} J^*\to K^*\to \bar K)$ to $\bar J\toup{\phi} J$. We claim that $\nu$ is a cartesian fibration. 
Indeed, let $(J_1\toup{\gamma_J} J_2)$ be a map in $\cH$ from $(\bar J\toup{\phi_1} J_1)$ to $(\bar J\toup{\phi_2} J_2)$. Assume the object $(\bar J\toup{\phi_2} J_2)$ is lifted to an object 
$$
(\und{\lambda}^2, J_2\toup{\tau_2} K_2, \bar J\toup{\phi_2} J_2^*\to K_2^*\to \bar K)$$
of $D$. We get an object $(\und{\lambda}^1, J_1\toup{\tau_1} K_1, \bar J\toup{\phi_1} J_1^*\to K_1^*\to \bar K)\in D$, where $K_1=K_2$, and $\tau_1$ is the composition $J_1\toup{\gamma_J} J_2\toup{\tau_2} K_2$, so $K_2^*=K_1^*$, and the composition $\bar J\to J_1^*\to K_1^*\to \bar K$ is the original map $\bar J\to \bar K$. We get the morphism $\gamma$ in $D_a$ given by the diagram
$$
\begin{array}{ccccc}
\bar J & \toup{\phi_1} & J_1 & \toup{\tau_1} & K_1\\
 & \searrow\lefteqn{\scriptstyle \phi_2} & \downarrow\lefteqn{\scriptstyle\gamma_J} && \uparrow\lefteqn{\scriptstyle\id}\\
 && J_2 & \toup{\tau_2} & K_2
\end{array}
$$
over $\gamma_J$. Then $\gamma$ is a $\cH$-cartesian arrow. 
   
 Let us show that $\nu$ is cofinal. By (\cite{HTT}, 4.1.3.2) , it suffices to show that the fibre of $\nu$ over $(J^*\to J)$ is contractible. This fibre classifies a surjection $\tau: J\to K$ such that $(\bar J\to \bar K)$ factors through 
$\bar J\to J^*\toup{\tau^*} K^*$. Here $K^*=\tau(J^*)$, and $\tau^*$ is the restriction of $\tau$. This fibre has an initial object, so is contractible.
 
 Now it suffices to show that $\cH_a$ is contractible. The category $\cH_a$ admits binary products. Namely, given $\phi_i: \bar J\to J_i$ their product is $(\phi_1, \phi_2): \bar J\to J_1\times J_2$. So, $\cH_a$ is contractible by Lemma~\ref{Lm_binary_products_give_contractibility}. We are done.
\end{proof}

\sssec{} Let us construct a similar diagram of functors in $Shv(X^{\lambda})-mod$
\begin{equation}
\label{diag_for_Sect_E.2.17}
\begin{array}{ccc}
\Fact(C_{>0})_{\lambda} &\toup{(\ref{map_for_E.2.15})}& \Fact(C)_{\lambda}\\
& \searrow\lefteqn{\scriptstyle\id} & \downarrow\lefteqn{\scriptstyle \xi_{C,\lambda}}\\
&& \Fact(C_{>0})_{\lambda} 
\end{array}
\end{equation}

 For $(\und{\lambda}, J\to K)\in \cTw(fSets)_{\lambda}$ consider the functor
\begin{equation}
\label{functor_for_Sect_E.3.3}
(\underset{k\in K^*}{\boxtimes} (\underset{j\in J^*_k}{\otimes} C(X)_{\und{\lambda}(j)}))\otimes_{Shv(X^{K^*})} Shv(X^K)\to \underset{k\in K^*}{\boxtimes} (\underset{j\in J^*_k}{\otimes} C(X)_{\und{\lambda}(j)})
\end{equation}
given by the direct image $Shv(X^K)\to Shv(X^{K^*})$ under the projection $X^K\to X^{K^*}$. These functors are compatible with the corresponding transition functors, so organize into a morphism 
$$
\cF_{\Ran, C,\lambda}\to  \cF^{\lambda}_{\Ran, C_{>0}}\comp\xi
$$ 
of functors from $\cTw(fSets)_{\lambda}$ to $Shv(X^{\lambda})-mod$. Passing to the colimit over $\cTw(fSets)_{\lambda}$ we obtain the desired functor $\xi_{C,\lambda}$. By construction, the diagram (\ref{diag_for_Sect_E.2.17}) commutes. The functor $\xi_{C,\lambda}$ is not $Shv(\Ran)$-linear. 
 
 For example, for $C(X)=\underset{\lambda\in\Lambda}{\oplus} Shv(X)$ the functor $\xi_{C,\lambda}$ identifies with 
$$
(\pr_{\lambda})_!: Shv((X^{\lambda}\times\Ran)^{\subset})\to Shv(X^{\lambda}),
$$ 
where $\pr_{\lambda}: (X^{\lambda}\times\Ran)^{\subset}\to X^{\lambda}$ is the projection. 
 


\sssec{} Let us construct a natural $Shv((X^{\lambda}\times\Ran)^{\subset})$-linear functor
$$
\bar\xi_{C,\lambda}^L: \Fact(C)_{\lambda}\to \Fact(C_{>0})_{\lambda}\otimes_{Shv(X^{\lambda})} Shv((X^{\lambda}\times\Ran)^{\subset})
$$
For $(\und{\lambda}, J\to K)\in\cTw(fSets)_{\lambda}$
let 
$$
\Gamma_{\und{\lambda}}: X^K\to X^{K^*}\times_{X^{\lambda}}(X^{\lambda}\times\Ran)^{\subset}
$$ 
be the map, whose first component is the projection and the second one is the map $X^K\to (X^{\lambda}\times\Ran)^{\subset}$ from Section~\ref{Sect_E.0.5_now}. 
From Lemma~\ref{Lm_E.2.5_now} one easily derives that $\Gamma_{\und{\lambda}}$ is pseudo-proper, so 
$$
(\Gamma_{\und{\lambda}})_!: Shv(X^K)\to Shv(X^{K^*}\times_{X^{\lambda}}(X^{\lambda}\times\Ran)^{\subset})
$$ 
is well-defined.

 For $(\und{\lambda}, J\to K)\in\cTw(fSets)_{\lambda}$ consider the functors
\begin{multline*}
\id\otimes(\Gamma_{\und{\lambda}})_!:  
(\underset{k\in K^*}{\boxtimes} (\underset{j\in J^*_k}{\otimes} C(X)_{\und{\lambda}(j)}))\otimes_{Shv(X^{K^*})} Shv(X^K)\to \\  (\underset{k\in K^*}{\boxtimes} (\underset{j\in J^*_k}{\otimes} C(X)_{\und{\lambda}(j)}))\otimes_{Shv(X^{K^*})} Shv(X^{K^*}\times_{X^{\lambda}}(X^{\lambda}\times\Ran)^{\subset})
\end{multline*}
They are functorial in $\cTw(fSets)_{\lambda}$ and $Shv((X^{\lambda}\times\Ran)^{\subset})$-linear.

 The \select{key observation} is that, since $X^{K^*}\to X^{\lambda}$ is finite, from Section~\ref{Sect_B.1.15} we get canonically
$$
Shv(X^{K^*}\times_{X^{\lambda}}(X^{\lambda}\times\Ran)^{\subset})\,\iso\, Shv(X^{K^*})\otimes_{Shv(X^{\lambda})} Shv((X^{\lambda}\times\Ran)^{\subset}).
$$
We get a diagram
\begin{multline*}
\id\otimes(\Gamma_{\und{\lambda}})_!:  
(\underset{k\in K^*}{\boxtimes} (\underset{j\in J^*_k}{\otimes} C(X)_{\und{\lambda}(j)}))\otimes_{Shv(X^{K^*})} Shv(X^K)\to \\ 
(\underset{k\in K^*}{\boxtimes} (\underset{j\in J^*_k}{\otimes} C(X)_{\und{\lambda}(j)}))\otimes_{Shv(X^{\lambda})} Shv((X^{\lambda}\times\Ran)^{\subset})
\end{multline*} 
in $\Fun(\cTw(fSets)_{\lambda}, \, Shv((X^{\lambda}\times\Ran)^{\subset}))$. Passing to the colimit over $\cTw(fSets)_{\lambda}$ and using the cofinality of $\xi$, we get the desired morphism $\bar\xi_{C,\lambda}^L$.

\sssec{} For the rest of Section~\ref{Sect_Adding and removing units} assume in addition that for each $(I, \und{\lambda})\in fSets^{\lambda}$ the product functor $\underset{i\in I}{\otimes} C(X)_{\und{\lambda}(i)}\to C(X)_{\lambda}$ admits a $Shv(X)$-linear continuous right adjoint. 

 Then we may pass to right adjoints in the functor $\cF_{\Ran, C, \lambda}$ and get the functor 
$$
\cF_{\Ran, C, \lambda}^R: (\cTw(fSets)_{\lambda})^{op}\to Shv(X^{\lambda})-mod.
$$ 
So, $\Fact(C)_{\lambda}\,\iso\,\underset{(\cTw(fSets)_{\lambda})^{op}}{\lim} \cF_{\Ran, C, \lambda}^R$. Similarly, we may pass to the right adjoints in the functor $\cF^{\lambda}_{\Ran, C_{>0}}$ and get the functor 
$$
\cF^{\lambda, R}_{\Ran, C_{>0}}: (\cTw(fSets)^{\lambda})^{op}\to Shv(X^{\lambda})-mod.
$$ 
So, $\Fact(C_{>0})_{\lambda}\,\iso\,\underset{(\cTw(fSets)^{\lambda})^{op}}{\lim} \cF^{\lambda, R}_{\Ran, C}$.

 Now we may pass to the $Shv(X^{\lambda})$-linear continuous right adjoints in the diagram (\ref{functor_for_Sect_E.3.3}) viewed as a diagram of functors from $\cTw(fSets)_{\lambda}$ to $Shv(X^{\lambda})-mod$. Passing to the limit over $(\cTw(fSets)_{\lambda})^{op}$ the right adjoints to (\ref{functor_for_Sect_E.3.3}) provide the functor $\xi_{C, \lambda}^R: \Fact(C_{>0})_{\lambda}\to \Fact(C)_{\lambda}$ in $Shv(X^{\lambda})-mod$. By (\cite{Ly}, 9.2.39), $\xi_{C, \lambda}^R$ is the right adjoint to $\xi_{C,\lambda}$. 
 
  If we view $\Fact(C)_{\lambda}$ as a $Shv((X^{\lambda}\times\Ran)^{\subset})$-module then the structure of a $Shv(X^{\lambda})$-module on $\Fact(C)_{\lambda}$ is obtained via restriction of scalars through $\pr_{\lambda}^!: Shv(X^{\lambda})\to Shv((X^{\lambda}\times\Ran)^{\subset})$. So, by adjointness the functor $\xi_{C, \lambda}^R$ gives rise to the functor
$$
\bar\xi_{C,\lambda}: \Fact(C_{>0})_{\lambda}\otimes_{Shv(X^{\lambda})}Shv((X^{\lambda}\times\Ran)^{\subset}) \to \Fact(C)_{\lambda},
$$
where the $Shv(X^{\lambda})$-module structure on $Shv((X^{\lambda}\times\Ran)^{\subset})$ is given by $\pr_{\lambda}^!$.

\begin{Rem} In fact,  $\Fact(C_{>0})_{\lambda}\otimes_{Shv(X^{\lambda})}Shv((X^{\lambda}\times\Ran)^{\subset})$ has two different $Shv((X^{\lambda}\times\Ran)^{\subset})$-module structures. The first comes from the $Shv((X^{\lambda}\times\Ran)^{\subset})$-action on the second factor. The second comes from the $Shv((X^{\lambda}\times\Ran)^{\subset})$-action on the factor $\Fact(C_{>0})_{\lambda}$. The functor $\bar\xi_{C,\lambda}$ is $Shv((X^{\lambda}\times\Ran)^{\subset})$-linear if we use the first $Shv((X^{\lambda}\times\Ran)^{\subset})$-module structure on the source.
\end{Rem}

\sssec{} Equivalently, $\bar\xi_{C,\lambda}$ is obtained as follows. Given $(\und{\lambda}, J\to K)\in \cTw(fSets)_{\lambda}$, consider the maps
\begin{multline*}
(\underset{k\in K^*}{\boxtimes} (\underset{j\in J^*_k}{\otimes} C(X)_{\und{\lambda}(j)}))\otimes_{Shv(X^{\lambda})} Shv((X^{\lambda}\times\Ran)^{\subset})\to\\ (\underset{k\in K^*}{\boxtimes} (\underset{j\in J^*_k}{\otimes} C(X)_{\und{\lambda}(j)}))\otimes_{Shv(X^{K^*})} Shv(X^K)
\end{multline*}
given by the !-pullbacks $Shv((X^{\lambda}\times\Ran)^{\subset})\to Shv(X^K)$ and $Shv(X^{\lambda})\to Shv(X^{K^*})$ for the diagram
$$
\begin{array}{ccc}
X^K & \to & (X^{\lambda}\times\Ran)^{\subset}\\
\downarrow && \downarrow\lefteqn{\scriptstyle \pr_{\lambda}}\\
X^{K^*} & \to & X^{\lambda},
\end{array}
$$
where the vertical arrows are the projections. These maps are functorial in $(\und{\lambda}, J\to K)\in (\cTw(fSets)_{\lambda})^{op}$. Passing to the limit over $(\cTw(fSets)_{\lambda})^{op}$ they yield $\bar\xi_{C,\lambda}$.

 By (\cite{Ly}, 9.2.39), $(\bar\xi_{C,\lambda}^L, \bar\xi_{C,\lambda})$ is an adjoint pair in $Shv((X^{\lambda}\times\Ran)^{\subset})-mod$. 

\sssec{} By (\cite{Ly}, 9.2.6), the functor (\ref{map_for_E.2.15}) has a continuous $Shv(X^{\lambda})$-linear right adjoint, which is the natural map $\underset{\cTw(fSets)_{\lambda}^{op}}{\lim}\cF^R_{\Ran, C, \lambda}\to \underset{(\cTw(fSets)^{\lambda})^{op}}{\lim}\cF^R_{\Ran, C, \lambda}\comp j^{op}$. 

\sssec{} The following is one of our main technical results, its proof is given in Sections~\ref{Sect_E.3.13_begin_proof_Thm_Con_E.3.8}-\ref{Lm_the_functor_zeta_is_cofinal_for_Con}.

\begin{Thm} 
\label{Con_E.3.8}
For any $\lambda\in\Lambda^*$ the functor $\bar\xi_{C,\lambda}$ is an equivalence.
\end{Thm}

\begin{Rem} i) For $C(X)$ as in Proposition~\ref{Pp_E.2.3} Theorem~\ref{Con_E.3.8} is evident.\\
ii) If $\lambda$ is one of generators of $\Lambda$ then Theorem~\ref{Con_E.3.8} follows by combining Remark~\ref{Rem_E.1.6} and Example~\ref{Sect_E.3.2}. 
\end{Rem} 

\sssec{Proof of Theorem~\ref{Con_E.3.8}} 
\label{Sect_E.3.13_begin_proof_Thm_Con_E.3.8}
Since the adjoint pair $(\bar\xi_{C,\lambda}^L, \bar\xi_{C,\lambda})$ takes place in $Shv((X^{\lambda}\times\Ran)^{\subset})-mod$, we are reduced in view of Proposition~\ref{Pp_3.7.8_devissage_using_ULA_preservation} to Proposition~\ref{Pp_E.3.14_for_Thm_Con_E.3.8} below. \QED

\sssec{} 
\label{Sect_E.3.15_now}
Pick $\lambda\in\Lambda^*$ and $\gU(\lambda)$ given by $\lambda=\sum_{m\in M} d_m\lambda_m$ with $d_m>0$ and $\lambda_m\in\Lambda^*$ pairwise distinct, here $M\in fSets$. 

\begin{Pp} 
\label{Pp_E.3.14_for_Thm_Con_E.3.8}
After the base change $Shv(\oo{X}{}^{\gU(\lambda)})\otimes_{Shv(X^{\lambda})}$ the functor $\bar\xi^L_{C,\lambda}$ becomes an equvalence.
\end{Pp}
\begin{proof}
Recall the full embedding $fSets_{\gU(\lambda)}^0\subset fSets_{\gU(\lambda)}$ and its left adjoint $\bar h^u$ from the proof of Lemma~\ref{Lm_E.2.13_almost_last}. 
Set 
$$
\cTw(fSets)_{\gU(\lambda)}=\cTw(fSets)_{\lambda}\times_{(fSets_{\lambda})^{op}} (fSets_{\gU(\lambda)})^{op},
$$
where we used the map $h^u$ to form the fibred product. So, and object of this category is a collection $(\und{\lambda}, J\toup{\tau} K)\in \cTw(fSets)_{\lambda}$ and $(K^*\toup{\phi} \tilde M)\in Q(K^*)^{\gU(\lambda)}$. That is, if we write $\und{\mu}$ for the restriction of $\tau_*\und{\lambda}$ to $K^*$ then $\phi_*\und{\mu}$ is required to be of type $\gU(\lambda)$. A map from $(\und{\lambda}^1, J_1\to K_1, \phi_1)$ to $(\und{\lambda}^2, J_2\to K_2, \phi_2)$ is given by a map $(\und{\lambda}^1, J_1\to K_1)\to (\und{\lambda}^2, J_2\to K_2)$ in $\cTw(fSets)_{\lambda}$ such that the composition $K_2^*\to K_1^*\to \tilde M_1$ coincides with $K_2^*\to\tilde M_2$ in the set $Q(K_2^*)^{\gU(\lambda)}$. For such a map there is a unique isomorphism $\tilde\tau$ making the diagram commutative
$$
\begin{array}{ccc}
K_1^* & \to & \tilde M_1\\
\uparrow && \uparrow\lefteqn{\scriptstyle \tilde\tau}\\
K_2^* & \to & \tilde M_2
\end{array}
$$
 
 The projection $\cTw(fSets)_{\gU(\lambda)}\to \cTw(fSets)_{\lambda}$ is a cocartesian fibration. We get
\begin{multline}
\label{first_expression_proof_Con_E.3.9}
\Fact(C)_{\lambda}\otimes_{Shv(X^{\lambda})} Shv(\oo{X}{}^{\gU(\lambda)})\,\iso\\ 
\underset{(\und{\lambda}, J\to K)\in\cTw(fSets)_{\lambda}}{\colim}
(\underset{k\in K^*}{\boxtimes}(\underset{j\in J^*_k}{\otimes} C(X)_{\und{\lambda}(j)}))\otimes_{Shv(X^{K^*})} Shv(X^K)\otimes_{Shv(X^{\lambda})}Shv(\oo{X}{}^{\gU(\lambda)})\,\iso\,\\
\underset{(\und{\lambda}, J\to K, K^*\toup{\phi}\tilde M)\in \cTw(fSets)_{\gU(\lambda)}}{\colim} (\underset{k\in K^*}{\boxtimes}(\underset{j\in J^*_k}{\otimes} C(X)_{\und{\lambda}(j)}))\otimes_{Shv(X^{K^*})} Shv(X^K)\otimes_{Shv(X^{K^*})} Shv(\oo{X}{}^{\tilde M})\\
\iso\, \underset{(\und{\lambda}, J\to K, K^*\toup{\phi}\tilde M)\in \cTw(fSets)_{\gU(\lambda)}}{\colim} (\underset{m\in \tilde M}{\boxtimes} (\underset{j\in J^*_m}{\otimes} C(X)_{\und{\lambda}(j)}))\otimes_{Shv(X^{\tilde M})} Shv(\oo{X}{}^{\tilde M}\times X^{K-K^*})
\end{multline}

 Let
$$
\cTw(fSets)_{\gU(\lambda)}^0=\cTw(fSets)_{\lambda}\times_{(fSets_{\lambda})^{op}} (fSets_{\gU(\lambda)}^0)^{op} 
$$
Then $\cTw(fSets)_{\gU(\lambda)}^0\subset \cTw(fSets)_{\lambda}$ is the full subcategory of those objects $(\und{\lambda}, J\toup{\tau} K)$ for which $(\und{\mu}, K^*)$ is of type $\gU(\lambda)$. We have an adjoint pair
$$
\cTw(fSets)_{\gU(\lambda)}^0\leftrightarrows \cTw(fSets)_{\gU(\lambda)},
$$
where the right adjoint sends $(\und{\lambda}, J\to K, K^*\to \tilde M)$ to $(\und{\lambda}, J\to K')$, where $K'$ is the colimit of the diagram $(\tilde M\gets K^*\to K)$ in the diagram of sets. So, as in the proof of Lemma~\ref{Lm_E.2.13_almost_last}, 
(\ref{first_expression_proof_Con_E.3.9}) identifies with 
$$
\underset{(\und{\lambda}, J\to K)\in \cTw(fSets)_{\gU(\lambda)}^0}{\colim}(\underset{k\in K^*}{\boxtimes} (\underset{j\in J^*_k}{\otimes} C(X)_{\und{\lambda}(j)}))\otimes_{Shv(X^{K^*})} Shv(\oo{X}{}^{K^*}\times X^{K-K^*})
$$

 Write
$$
r: {^0\cTw(fSets)_{\gU(\lambda)}}\subset \cTw(fSets)_{\gU(\lambda)}^0
$$ 
for the full subcategory of those objects $(\und{\lambda}, J\to K)$ for which the restriction $J^*\to K^*$ is an isomorphism. This embedding admits a left adjoint 
$$
l: \cTw(fSets)_{\gU(\lambda)}^0\to {^0\cTw(fSets)_{\gU(\lambda)}}
$$ 
sending $(\und{\lambda}, J\to K)$ to the collection $(\und{\lambda}', J'\to K)$. Here $J'$ is the push-out of the diagram $K^*\gets J^*\to J$ in the category of sets, and $\und{\lambda}'$ is the direct image of $\und{\lambda}$ along $J\to J'$. 

 We have used here the following. Given a morphism from $(\und{\lambda}^1, J_1\to K_1)$ to $(\und{\lambda}^2, J_2\to K_2)$ in $\cTw(fSets)_{\gU(\lambda)}^0$
the corresponding map $K_2^*\to K_1^*$ is an isomorphism, and its \select{inverse} fits into a commutative diagram
$$
\begin{array}{ccccc}
K_1^* & \gets &J_1^*&\to &J_1\\
\downarrow && \downarrow && \downarrow\\
K_2^* & \gets &J_2^*&\to &J_2,
\end{array}
$$
which gives a natural map $J'_1\to J'_2$. For this reason $l$ is well-defined. 
 
  By (\cite{G}, I.1, 2.2.3), $r$ is cofinal. So, (\ref{first_expression_proof_Con_E.3.9}) identifies with
\begin{equation}
\label{second_expression__proof_Con_E.3.9}
\underset{(\und{\lambda}, J\to K)\in {^0\cTw(fSets)_{\gU(\lambda)}}}{\colim}(\underset{k\in K^*}{\boxtimes} C(X)_{\und{\lambda}(k)})\otimes_{Shv(X^{K^*})} Shv(\oo{X}{}^{K^*}\times X^{K-K^*})  
\end{equation}

 Consider the functor 
$$
h^{\gU(\lambda)}: {^0\cTw(fSets)_{\gU(\lambda)}}\to (fSets_{\gU(\lambda)}^0)^{op}
$$ 
sending $(\und{\lambda}, J\toup{\tau} K)$ to $(K, \tau_*\und{\lambda})$. By Lemma~\ref{Lm_first_for_proof_Con_E.3.9} below it is cofinal. So, (\ref{second_expression__proof_Con_E.3.9}) identifies with
$$
\underset{(\und{\lambda}, K)\in (fSets_{\gU(\lambda)}^0)^{op}}{\colim} 
(\underset{k\in K^*}{\boxtimes} C(X)_{\und{\lambda}(k)})\otimes_{Shv(X^{K^*})} Shv(\oo{X}{}^{K^*}\times X^{K-K^*})  
$$

 Consider now
\begin{multline}
\label{third_expression__proof_Con_E.3.9}
\Fact(C_{>0})\otimes_{Shv(X^{\lambda})} Shv((X^{\lambda}\times\Ran)^{\subset})\otimes_{Shv(X^{\lambda})} Shv(\oo{X}{}^{\gU(\lambda)})\,\iso\\
(\underset{(\und{\lambda}, J\to K)\in\cTw(fSets)^{\lambda}}{\colim} 
(\underset{k\in K}{\boxtimes}(\underset{j\in J_k}{\otimes} C(X)_{\und{\lambda}(j)})))\otimes_{Shv(X^{\lambda})} Shv((X^{\lambda}\times\Ran)^{\subset})\otimes_{Shv(X^{\lambda})} Shv(\oo{X}{}^{\gU(\lambda)})
\end{multline}

 Recall the category $fSets^{\gU(\lambda)}$ defined in the proof of Lemma~\ref{Lm_E.1.21_now}.
Set
$$
\cTw(fSets)^{\gU(\lambda)}=\cTw(fSets)^{\lambda}\times_{(fSets^{\lambda})^{op}} (fSets^{\gU(\lambda)})^{op},
$$
where we used the map (\ref{functor_h_Sect_E.1.9}) to form the fibred product. An objects of the latter category is $(\und{\lambda}, J\to K\to \tilde K)$ with $(\und{\lambda}, J\to K)\in \cTw(fSets)^{\lambda}$ and $(K\to \tilde K)\in Q(K)^{\gU(\lambda)}$. As above, (\ref{third_expression__proof_Con_E.3.9}) identifies with
\begin{multline*}
(\underset{(\und{\lambda}, J\to K\to \tilde K)\in\cTw(fSets)^{\gU(\lambda)}}{\colim} 
(\underset{k\in K}{\boxtimes}(\underset{j\in J_k}{\otimes} C(X)_{\und{\lambda}(j)})))\otimes_{Shv(X^K)} Shv(\oo{X}{}^{\tilde K})\otimes_{Shv(X^{\lambda})}  
Shv((X^{\lambda}\times\Ran)^{\subset})\\
\iso\, (\underset{(\und{\lambda}, J\to K\to \tilde K)\in\cTw(fSets)^{\gU(\lambda)}}{\colim} 
(\underset{k\in K}{\boxtimes}(\underset{j\in J_k}{\otimes} C(X)_{\und{\lambda}(j)})))\otimes_{Shv(X^K)} Shv((\oo{X}{}^{\tilde K}\times\Ran)^{\subset}),
\end{multline*}
where we have set $(\oo{X}{}^{\tilde K}\times\Ran)^{\subset}=\oo{X}{}^{\tilde K}\times_{X^{\lambda}} (X^{\lambda}\times\Ran)^{\subset}$.

 The latter colimit identifies with
\begin{equation}
\label{third_half_expression_proof_Con_E.3.9}
(\underset{(\und{\lambda}, J\to K\to \tilde K)\in\cTw(fSets)^{\gU(\lambda)}}{\colim} 
(\underset{k\in \tilde K}{\boxtimes} (\underset{j\in J_k}{\otimes} C(X)_{\und{\lambda}(j)})))\otimes_{Shv(X^{\tilde K})} Shv((\oo{X}{}^{\tilde K}\times\Ran)^{\subset})
\end{equation}
Recall the full embedding $fSets_0^{\gU(\lambda)}\subset fSets^{\gU(\lambda)}$ from the proof of Lemma~\ref{Lm_E.1.21_now}. Set
$$
\cTw(fSets)^{\gU(\lambda)}_0=\cTw(fSets)^{\lambda}\times_{(fSets^{\lambda})^{op}} (fSets_0^{\gU(\lambda)})^{op}.  
$$
An object of the latter category is $(\und{\lambda}, J\toup{\tau} K)\in \cTw(fSets)^{\lambda}$ such that $(\tau_*\und{\lambda}, K)$ is of type $\gU(\lambda)$. Consider the functor
$$
\tilde h: \cTw(fSets)^{\gU(\lambda)}\to \cTw(fSets)^{\gU(\lambda)}_0
$$
sending $(\und{\lambda}, J\to K\to \tilde K)$ to $(\und{\lambda}, J\to \tilde K)$. We have an adjoint pair
$$
\cTw(fSets)^{\gU(\lambda)}_0\leftrightarrows \cTw(fSets)^{\gU(\lambda)}: \tilde h,
$$
where the left adjoint is the inclusion. So, $\tilde h$ is cofinal, and 
(\ref{third_half_expression_proof_Con_E.3.9}) identifies with
$$
(\underset{(\und{\lambda}, J\to K)\in\cTw(fSets)^{\gU(\lambda)}_0}{\colim} 
(\underset{k\in K}{\boxtimes} (\underset{j\in J_k}{\otimes} C(X)_{\und{\lambda}(j)})))\otimes_{Shv(X^K)} Shv((\oo{X}{}^K\times\Ran)^{\subset})
$$

 The full subcategory of $\cTw(fSets)^{\gU(\lambda)}_0$ spanned by those $(\und{\lambda}, J\toup{\tau} K)$ for which $\tau$ is an isomorphism identifies with $(fSets^{\gU(\lambda)}_0)^{op}$. We get an adjoint pair
$$
\cTw(fSets)^{\gU(\lambda)}_0\leftrightarrows (fSets^{\gU(\lambda)}_0)^{op},
$$
where the left adjoint sends $(\und{\lambda}, J\toup{\tau} K)$ to $(\tau_*\und{\lambda}, K)$, and the right adjoint is the above inclusion. Thus, the right adjoint being cofinal, (\ref{third_half_expression_proof_Con_E.3.9}) rewites as
$$
(\underset{(\und{\lambda}, K)\in (fSets^{\gU(\lambda)}_0)^{op}} {\colim} 
(\underset{k\in K}{\boxtimes}  C(X)_{\und{\lambda}(k)})))\otimes_{Shv(X^K)} Shv((\oo{X}{}^K\times\Ran)^{\subset})
$$

 Now we have the functor $\zeta: fSets^0_{\gU(\lambda)}\to fSets^{\gU(\lambda)}_0$ sending $(\und{\lambda}, K)$ to $(\und{\lambda}^*, K^*)$, here $\und{\lambda}^*$ is the restriction of $\und{\lambda}$. For $(\und{\lambda}, K)\in fSets^0_{\gU(\lambda)}$ we have the morphism
\begin{multline}
\label{forth_expression__proof_Con_E.3.9}
(\underset{k\in K^*}{\boxtimes} C(X)_{\und{\lambda}(k)})\otimes_{Shv(X^{K^*})} Shv(\oo{X}{}^{K^*}\times X^{K-K^*})\to \\
(\underset{k\in K^*}{\boxtimes}  C(X)_{\und{\lambda}(k)})))\otimes_{Shv(X^{K^*})} Shv((\oo{X}{}^{K^*}\times\Ran)^{\subset})
\end{multline}
given by the direct image along the map $\oo{X}{}^{K^*}\times X^{K-K^*}\to (\oo{X}{}^{K^*}\times\Ran)^{\subset}$ over $\oo{X}{}^{K^*}$. The map (\ref{forth_expression__proof_Con_E.3.9}) is functorial in $(fSets^0_{\gU(\lambda)})^{op}$. 

 By Lemma~\ref{Lm_the_functor_zeta_is_cofinal_for_Con}, $\zeta$ is cofinal. Now the desired map from (\ref{first_expression_proof_Con_E.3.9}) to (\ref{third_expression__proof_Con_E.3.9}) is obtained from (\ref{forth_expression__proof_Con_E.3.9}) by passing to the colimit over $(fSets^0_{\gU(\lambda)})^{op}$. The colimit over $(fSets^0_{\gU(\lambda)})^{op}$ of the maps (\ref{forth_expression__proof_Con_E.3.9}) is an isomorphism, this is similar to the argument in Lemma~\ref{Lm_E.2.13_almost_last}. We are done.
\end{proof}

\begin{Lm} 
\label{Lm_first_for_proof_Con_E.3.9}
The functor $h^{\gU(\lambda)}$ is cofinal.
\end{Lm}
\begin{proof}
Pick $\xi=(\und{\nu}, K)\in fSets_{\gU(\lambda)}^0$. We show that the category
$$
{^0\cTw(fSets)_{\gU(\lambda)}}\times_{(fSets_{\gU(\lambda)}^0)^{op}} ((fSets_{\gU(\lambda)}^0)^{op})_{\xi/}
$$
is contractible. It classifies an object $(\und{\lambda}^1, J_1\toup{\tau} K_1)\in {^0\cTw(fSets)_{\gU(\lambda)}}$ and a map $\eta: (K_1 \tau_*\und{\lambda}^1)\to (K, \und{\nu})$ in $fSets_{\gU(\lambda)}^0$. This category identifies with
$$
(\underset{k\in K-K^*}{\prod} \cTw(fSets))\times (\underset{k\in K^*}{\prod} \cTw(fSets_*))
$$
Here we have denoted by $fSets_*$ the category of pointed finite sets and surjective morphisms. Each of the categories $\cTw(fSets)$, $\cTw(fSets_*)$ is contractible, for the second one this is a particular case of a claim from Section~\ref{Sect_E.2.3_now}. 
\end{proof}

\begin{Lm} 
\label{Lm_the_functor_zeta_is_cofinal_for_Con}
The functor $\zeta: (fSets^0_{\gU(\lambda)})^{op}\to (fSets^{\gU(\lambda)}_0)^{op}$ is cofinal.
\end{Lm}
\begin{proof}
The argument here is essentially the one that has already appeared in Lemma~\ref{Lm_xi_is_cofinal}. Given $\xi:=(\und{\mu}, T)\in fSets^{\gU(\lambda)}_0$, we must show that 
$$
fSets^0_{\gU(\lambda)}\times_{fSets^{\gU(\lambda)}_0} (fSets^{\gU(\lambda)}_0)_{/\xi}
$$ 
is contractible. An object of the latter category is a datum of $(\und{\lambda}, K)\in fSets^0_{\gU(\lambda)}$ together with an isomorphism $(\und{\lambda}^*, K^*)\,\iso\, \xi$ in $fSets^{\gU(\lambda)}_0$. This category admits binary products, so is contractible by Lemma~\ref{Lm_binary_products_give_contractibility}.
\end{proof} 

\sssec{} In fact, $\xi_{C,\lambda}$ is the composition of $\bar\xi^L_{C,\lambda}$ with 
$$
(\pr_{\lambda})_!: \Fact(C_{>0})\otimes_{Shv(X^{\lambda})} Shv((X^{\lambda}\times\Ran)^{\subset})\to \Fact(C_{>0}).
$$ 

\sssec{Adding units for graded factorization algebras} 
\label{Sect_Adding units for graded factorization algebras}
Let $A\in CAlg(C(X))$ with a grading $A=\underset{\lambda\in\Lambda}{\oplus} A_{\lambda}$, where $A_{\lambda}\in C(X)_{\lambda}$. Assume that the canonical map $\omega_X\to A_0$ is an isomorphism in $C(X)_0\,\iso\, Shv(X)$. 

Set $A_{>0}=\underset{\lambda\in\Lambda^*}{\oplus} A_{\lambda}$. Then $A_{>0}\in CAlg^{nu}(C_{>0})$. Let $\lambda\in\Lambda^*$. Let us construct a canonical morphism 
\begin{equation}
\label{map_from_Fact(A)_lambda_strengthened_trace}
\bar\xi^L_{C, \lambda}(\Fact(A)_{\lambda})\to \pr_{\lambda}^! \Fact(A_{>0})
\end{equation}
where by abuse of notations $\pr_{\lambda}^!$ denotes the functor
$$
\Fact(C_{>0})\to \Fact(C_{>0})\otimes_{Shv(X^{\lambda})} Shv((X^{\lambda}\times\Ran)^{\subset})
$$
given by the !-pullback along $\pr_{\lambda}: (X^{\lambda}\times\Ran)^{\subset}\to X^{\lambda}$. 

\index{Adding units for graded factorization algebras} 

\sssec{} Let $(\und{\lambda}, J\to K)\in \cTw(fSets)_{\lambda}$. Set
$$
(X^{K^*}\times\Ran)^{\subset}=X^{K^*}\times_{X^{\lambda}}(X^{\lambda}\times\Ran)^{\subset},
$$ 
so $\Gamma_{\und{\lambda}}: X^K\;\to (X^{K^*}\times\Ran)^{\subset}$. 
 For $M\in Shv(X^{K^*})-mod$ we have a canonical map 
$$
(\id\otimes (\Gamma_{\und{\lambda}})_!: M\otimes_{Shv(X^{K^*})} Shv(X^K)\to M\otimes_{Shv(X^{K^*})} Shv(X^{K^*}\times_{X^{\lambda}}(X^{\lambda}\times\Ran)^{\subset})
$$
and for $K\in M$ a canonical morphism 
$$
(\id\otimes \Gamma_{\und{\lambda}})_!(\id\otimes \Gamma_{\und{\lambda}})^!\pr_{\lambda}^! K\to \pr_{\lambda}^! K
$$
in $M\otimes_{Shv(X^{K^*})} Shv(X^{K^*}\times_{X^{\lambda}}(X^{\lambda}\times\Ran)^{\subset})$. We apply this to 
$M= \underset{k\in K^*}{\boxtimes}(\underset{j\in J^*_k}{\otimes} C(X)_{\und{\lambda}(j)})$ and 
$$
K=\underset{k\in K^*}{\boxtimes}(\underset{j\in J^*_k}{\otimes} A_{\und{\lambda}(j)}))
$$
The result is a canonical map
\begin{equation}
\label{map_for_graded_fact_alg_adding_units_1}
(\id\otimes \Gamma_{\und{\lambda}})_!(\underset{k\in K}{\boxtimes}(\underset{j\in J_k}{\otimes} A_{\und{\lambda}(j)}))\to \pr_{\lambda}^!(\underset{k\in K^*}{\boxtimes}(\underset{j\in J^*_k}{\otimes} A_{\und{\lambda}(j)}))
\end{equation}
in 
$$
(\underset{k\in K^*}{\boxtimes} (\underset{j\in J^*_k}{\otimes} C(X)_{\und{\lambda}(j)}))\otimes_{Shv(X^{K^*})} Shv((X^{K^*}\times\Ran)^{\subset}).
$$
It is formal to check that the maps (\ref{map_for_graded_fact_alg_adding_units_1}) are functorial in $(\und{\lambda}, J\to K)\in\cTw(fSets)_{\lambda}$. Passing to the colimit over $(\und{\lambda}, J\to K)\in\cTw(fSets)_{\lambda}$ the maps (\ref{map_for_graded_fact_alg_adding_units_1}) give the desired morphism (\ref{map_from_Fact(A)_lambda_strengthened_trace}), here we used the fact that $\xi$ is cofinal by Lemma~\ref{Lm_xi_is_cofinal}. 

\begin{Pp} 
\label{Pp_E.3.22_now}
In the setting of Section~\ref{Sect_Adding units for graded factorization algebras} the map (\ref{map_from_Fact(A)_lambda_strengthened_trace}) is an isomorphism.
\end{Pp}
\begin{proof} It is immediately reduced to Lemma~\ref{Lm_map_from_Fact(A)_lambda_strengthened_trace_over_stratum} below.
\end{proof}

\begin{Lm} 
\label{Lm_map_from_Fact(A)_lambda_strengthened_trace_over_stratum}
Pick $\gU(\lambda)$ as in Section~\ref{Sect_E.3.15_now}. After the base change by $\oo{X}{}^{\gU(\lambda)}\to X^{\lambda}$ the map (\ref{map_from_Fact(A)_lambda_strengthened_trace}) becomes an isomorphism in 
$$
\Fact(C_{>0})\otimes_{Shv(X^{\lambda})} Shv((X^{\lambda}\times\Ran)^{\subset})\otimes_{Shv(X^{\lambda})} Shv(\oo{X}{}^{\gU(\lambda)})
$$
\end{Lm}
\begin{proof}
We follow the pattern of the proof of Proposition~\ref{Pp_E.3.14_for_Thm_Con_E.3.8}. The LHS of (\ref{map_from_Fact(A)_lambda_strengthened_trace}) in 
$$
\Fact(C_{>0})\otimes_{Shv(X^{\lambda})} Shv((X^{\lambda}\times\Ran)^{\subset}).
$$
is given by
$$
\underset{(\und{\lambda}, J\to K)\in\cTw(fSets)_{\lambda}}{\colim} 
(\id\otimes \Gamma_{\und{\lambda}})_!((\underset{k\in K^*}{\boxtimes}(\underset{j\in J^*_k}{\otimes} A_{\und{\lambda}(j)}))\boxtimes \omega_{X^{K-K^*}}).
$$  

 Recall that for $(\und{\lambda}, J\to K)\in \cTw(fSets)_{\lambda}$, 
$$
X^{K^*}\times_{X^{\lambda}} \oo{X}{}^{\gU(\lambda)}\,\iso\, \underset{(K^*\toup{\phi} \tilde M)\in Q(K^*)^{\gU(\lambda)}}{\sqcup} \oo{X}{}^{\tilde M}
$$
For an object $(\und{\lambda}, J\to K, K^*\toup{\phi}\tilde M)\in \cTw(fSets)_{\gU(\lambda)}$ let 
$$
\Gamma_{\und{\lambda}, \phi}: \oo{X}{}^{\tilde M}\times X^{K-K^*}\to (\oo{X}{}^{\tilde M}\times\Ran)^{\subset}
$$
be the base change of $\Gamma_{\und{\lambda}}$ by $\oo{X}{}^{\tilde M}\to X^{K^*}$. 

 The $!$-restriction of the LHS of (\ref{map_from_Fact(A)_lambda_strengthened_trace}) under $\oo{X}{}^{\gU(\lambda)}\to X^{\lambda}$ is
$$
\underset{(\und{\lambda}, J\to K, K^*\toup{\phi}\tilde M)\in \cTw(fSets)_{\gU(\lambda)}}{\colim}
(\Gamma_{\und{\lambda}, \phi})_! ((\underset{m\in \tilde M}{\boxtimes}(\underset{j\in J^*_m}{\otimes} A_{\und{\lambda}(j)}))\mid_{\oo{X}{}^{\tilde M}}\boxtimes \omega_{X^{K-K^*}})
$$
As in the proof of Theorem~\ref{Con_E.3.8}, the latter expression identifies using the adjoint pair
$$
\cTw(fSets)_{\gU(\lambda)}^0\leftrightarrows \cTw(fSets)_{\gU(\lambda)}
$$
with
$$
\underset{(\und{\lambda}, J\to K)\in \cTw(fSets)_{\gU(\lambda)}^0}{\colim}
(\oo{\Gamma}_{\und{\lambda}})_! ((\underset{k\in K^*}{\boxtimes}(\underset{j\in J^*_k}{\otimes} A_{\und{\lambda}(j)}))\mid_{\oo{X}{}^{K^*}}
\boxtimes \omega_{X^{K-K^*}}),
$$
where we have denoted by $\oo{\Gamma}_{\und{\lambda}}: \oo{X}{}^{K^*}\times X^{K-K^*}\to (\oo{X}{}^{K^*}\times\Ran)^{\subset}$ the restriction of $\Gamma_{\und{\lambda}}$. 

 As in the proof of Theorem~\ref{Con_E.3.8}, the latter expression identifies with
$$
\underset{(\und{\lambda}, J\to K)\in {^0\cTw(fSets)_{\gU(\lambda)}}}{\colim}(\oo{\Gamma}_{\und{\lambda}})_! ((\underset{k\in K^*}{\boxtimes}(A_{\und{\lambda}(k)}))\mid_{\oo{X}{}^{K^*}}
\boxtimes \omega_{X^{K-K^*}}),
$$ 
In turn, by the cofinality of $h^{\gU(\lambda)}$, the latter expression identifies with
\begin{equation}
\label{expressions_LHS_of_map_from_Fact(A)_lambda_strengthened_trace}
\underset{(\und{\lambda}, K)\in (fSets_{\gU(\lambda)}^0)^{op}}{\colim}(\oo{\Gamma}_{\und{\lambda}})_! ((\underset{k\in K^*}{\boxtimes}(A_{\und{\lambda}(k)}))\mid_{\oo{X}{}^{K^*}}
\boxtimes \omega_{X^{K-K^*}}),
\end{equation}  

 Let us now analyze the !-restriction under $\oo{X}{}^{\gU(\lambda)}\to X^{\lambda}$ of the RHS of (\ref{map_from_Fact(A)_lambda_strengthened_trace}). By definition, this is
$$
\underset{(\und{\lambda}, J\to K)\in\cTw(fSets)^{\lambda}}{\colim} \pr_{\lambda}^!(\underset{k\in K}{\boxtimes} 
(\underset{j\in J_k}{\otimes} A_{\und{\lambda}(j)}))\mid_{\oo{X}{}^{\gU(\lambda)}}
$$
over $(\oo{X}{}^{\gU(\lambda)}\times\Ran)^{\subset}$. As in the proof of Theorem~\ref{Con_E.3.8}, the latter expression identifies with
$$
\underset{(\und{\lambda}, J\to K\toup{\phi} \tilde K)\in\cTw(fSets)^{\gU(\lambda)}}{\colim} \pr_{\tilde K}^!((\underset{k\in \tilde K}{\boxtimes} 
(\underset{j\in J_k}{\otimes} A_{\und{\lambda}(j)}))\mid_{\oo{X}{}^{\tilde K}})
$$
where we have denoted by $\pr_{\tilde K}: (\oo{X}{}^{\tilde K}\times\Ran)^{\subset}\to \oo{X}{}^{\tilde K}$ the projection. As in the proof of Theorem~\ref{Con_E.3.8}, the latter expression identifies with
$$
\underset{(\und{\lambda}, J\to K)\in\cTw(fSets)^{\gU(\lambda)}_0}{\colim}\pr_K^!((\underset{k\in K}{\boxtimes} 
(\underset{j\in J_k}{\otimes} A_{\und{\lambda}(j)}))\mid_{\oo{X}{}^K})
$$
and in turn with
\begin{equation}
\label{expressions2_LHS_of_map_from_Fact(A)_lambda_strengthened_trace} 
\underset{(\und{\lambda}, K)\in (fSets_0^{\gU(\lambda)})^{op}}{\colim}
\pr_K^!((\underset{k\in K}{\boxtimes} A_{\und{\lambda}(k)})\mid_{\oo{X}{}^K})
\end{equation}

 Recall the functor $\zeta: fSets^0_{\gU(\lambda)}\to fSets^{\gU(\lambda)}_0$ sending $(\und{\lambda}, K)$ to $(\und{\lambda}^*, K^*)$. For $(\und{\lambda}, K)\in fSets^0_{\gU(\lambda)}$ we have the natural morphism
\begin{equation}
\label{expressions3_LHS_of_map_from_Fact(A)_lambda_strengthened_trace} 
(\oo{\Gamma}_{\und{\lambda}})_! ((\underset{k\in K^*}{\boxtimes}(A_{\und{\lambda}(k)}))\mid_{\oo{X}{}^{K^*}}
\boxtimes \omega_{X^{K-K^*}})\to \pr_{K^*}^!((\underset{k\in K^*}{\boxtimes} A_{\und{\lambda}(k)})\mid_{\oo{X}{}^{K^*}})
\end{equation}
functorial in $(\und{\lambda}, K)\in (fSets^0_{\gU(\lambda)})^{op}$. 

 Finally the desired morphism from (\ref{expressions_LHS_of_map_from_Fact(A)_lambda_strengthened_trace}) to (\ref{expressions2_LHS_of_map_from_Fact(A)_lambda_strengthened_trace}) is obtained from (\ref{expressions3_LHS_of_map_from_Fact(A)_lambda_strengthened_trace}) by passing to the colimit over $(fSets^0_{\gU(\lambda)})^{op}$, the latter colimit is an isomorphism. We are done. 
\end{proof}

\ssec{Applications}
\label{Sect_E.3}

\sssec{} 
\label{Sect_E.3.1}
Let $G$ be a connected reductive group over $k$. Let $\Lambda$ be as in Section~\ref{Sect_E.0.1}. Fix $\lambda\in\Lambda^*$. Denote by $\Gr_{G, X^{\lambda}}$ the prestack classifying a $G$-torsor $\cF_G$ on $X$, $D\in X^{\lambda}$ and a trivilization $\cF_G\,\iso\, \cF^0_G\mid_{X-\supp(D)}$.

\index{$\Conf$-version of the affine Grassmannian}

For $(J, \und{\lambda})\in fSets^{\lambda}$ consider the map $s^{\und{\lambda}}_{\Gr}: \Gr_{G, X^J}\to \Gr_{G, X^{\lambda}}$ over the map $s^{\und{\lambda}}: X^J\to X^{\lambda}$ from Section~\ref{Sect_E.1.12_now}. It sends $(\cF_G, \cI, \beta: \cF_G\,\iso\,\cF^0_G\mid_{X-\Gamma_{\cI}})$ to $(\cF_G, D, \beta: \cF_G\,\iso\,\cF^0_G\mid_{X-\supp(D)}$. Here $D$ is the image of $\cI$ under $s^{\und{\lambda}}$, and we used the equality $X-\supp(D)=X-\Gamma_{\cI}$. 

 The maps $s^{\und{\lambda}}_{\Gr}$ form a compatible system of maps for $(J, \und{\lambda})\in fSets^{\lambda}$, so yield a map $s_{\Gr}: \underset{(J, \und{\lambda})\in (fSets^{\lambda})^{op}}{\colim} \Gr_{G, X^J}\to \Gr_{G, X^{\lambda}}$. 
  
\begin{Lm} 
\label{Lm_E.3.2}
The functor
$
(s_{\Gr})_!: \underset{(J, \und{\lambda})\in (fSets^{\lambda})^{op}}{\colim} Shv(\Gr_{G, X^J})\to Shv(\Gr_{G, X^{\lambda}})
$ 
is an equivalence.
\end{Lm}
\begin{proof}
Similar to Proposition~\ref{Pp_E.1.5}.
\end{proof}

\begin{Cor} 
\label{Cor_E.3.3}
For $\lambda\in\Lambda^*$ one has $Shv(\Gr_{G,\Ran})\otimes_{Shv(\Ran)} Shv(X^{\lambda})\,\iso\, Shv(\Gr_{G, X^{\lambda}})$ canonically.
\end{Cor}
\begin{proof} One has
\begin{multline*}
Shv(\Gr_{G,\Ran})\otimes_{Shv(\Ran)} Shv(X^{\lambda})\,\iso\, \underset{(J, \und{\lambda})\in (fSets^{\lambda})^{op}}{\colim} Shv(\Gr_{G,\Ran})\otimes_{Shv(\Ran)} Shv(X^J)\\
\iso\,\underset{(J, \und{\lambda})\in (fSets^{\lambda})^{op}}{\colim} Shv(\Gr_{G, X^J}),
\end{multline*}
where for the second isomorphisms we used the fact that $\Ran$ is 1-affine. Our claim follows now from Lemma~\ref{Lm_E.3.2}.
\end{proof}

\sssec{} For $S\in\Sch^{aff}$, $I\in fSets$, $\cI\in \Map(S, X^I)$ denote by $\hat\cD_{\cI}$ the formal completion of $S\times X$ along $\Gamma_{\cI}$ viewed as a formal scheme. Let $\cD_{\cI}$ denote the affine scheme obtained from $\hat\cD_{\cI}$, the image of $\hat\cD_{\cI}$ under $\colim: \Ind(\Sch^{aff})\to \Sch^{aff}$. Set $\oo{\cD}_{\cI}=\cD_{\cI}-\Gamma_{\cI}$. 

 Let $\gL^+(G)_{\Ran}$ be the group scheme over $\Ran$ whose $S$-points are
$\cI\in\Map(S,\Ran)$ and a map $\cD_{\cI}\to G$. For $I\in fSets$ let $\gL^+(G)_I=\gL^+(G)_{\Ran}\times_{\Ran} X^I$. 

 For $\lambda\in\Lambda^*$ and an $S$-point $D\in\Map(S, X^{\lambda})$ denote by $\hat\cD_D$ the formal completion of $S\times X$ along $\supp(\Theta(D))$. Write $\cD_D$ for the affine scheme, the image of $\hat\cD_D$ under $\colim: \Ind(\Sch^{aff})\to \Sch^{aff}$. 
 
 We get the group scheme $\gL^+(G)_{\lambda}$ on $X^{\lambda}$ whose $S$-points are $D\in\Map(S, X^{\lambda})$ and a section $\cD_D\to G$. 
 
\sssec{} For $I\in fSets$ write $\gL^+(G)_I\backslash \Gr_{G, X^I}$ for the corresponding stack quotient (for the etale topology). It classifies $\cI\in X^I$, $G$-torsors $\cF_G, \cF'_G$ on $\cD_{\cI}$ and a trivialization $\cF_G\,\iso\,\cF'_D\mid_{\oo{\cD}_{\cI}}$. 

 Similarly, we consider the stack quotient $\gL^+(G)_{\lambda}\backslash\Gr_{G, X^{\lambda}}$. As in Section~\ref{Sect_E.3.1}, one gets the morphism
$$
s_{\Gr}^{eq}: \underset{(J, \und{\lambda})\in (fSets^{\lambda})^{op}}{\colim} \gL^+(G)_J\backslash\Gr_{G, X^J}\to \gL^+(G)_{\lambda}\backslash\Gr_{G, X^{\lambda}},   
$$
here $eq$ stands for `equivariant'. As in Lemma~\ref{Lm_E.3.2}, one gets the following.  
\begin{Lm} The functor 
$$
(s_{\Gr}^{eq})_!: \underset{(J, \und{\lambda})\in (fSets^{\lambda})^{op}}{\colim} Shv(\Gr_{G, X^J})^{\gL^+(G)_J}\to Shv(\Gr_{G, X^{\lambda}})^{\gL^+(G)_{\lambda}}
$$
is an equivalence. \QED
\end{Lm} 

 Set $\Sph_{G,\Ran}=Shv(\gL^+(G)_{\Ran}\backslash \Gr_{G,\Ran})$.
\begin{Cor}
\label{Cor_E.3.7}
For $\lambda\in\Lambda^*$ one has canonically
$$
\Sph_{G,\Ran}\otimes_{Shv(\Ran)} Shv(X^{\lambda})\,\iso\, Shv(\gL^+(G)_{\lambda}\backslash\Gr_{G, X^{\lambda}}).
$$
\end{Cor}
\begin{proof}
Argue as in Corollary~\ref{Cor_E.3.3}. We use here the fact that for $J\in fSets$, 
$$
Shv(\gL^+(G)_{\Ran}\backslash \Gr_{G,\Ran})\otimes_{Shv(\Ran)} Shv(X^J)\,\iso\, Shv(\gL^+(G)_J\backslash \Gr_{G, X^J}),
$$ 
as $\Ran$ is 1-affine for our sheaf theories. 
\end{proof}

\sssec{} 
\label{Sect_E.4.8}
Let $\Conf$ be the moduli scheme of $\Lambda^*$-valued divisors on $X$, so $\Conf=\underset{\lambda\in\Lambda^*}{\sqcup} X^{\lambda}$. Here $\Conf$ stands for the configurations space. Set $\und{\Conf}=\underset{\lambda\in\Lambda^*}{\sqcup} \und{X}^{\lambda}$. The canonical map $\und{\Conf}\to \Conf$ induces $Shv(Conf)\,\iso\, Shv(\und{\Conf})$. 

\index{Configuration space $\Conf$}
\index{Combinatorial configuration space $\und{\Conf}$}

Write $\check{G}$ for the Langlands dual group of $G$ over $e$. Let
$$
E(X)=\Rep(\check{G})\otimes Shv(X)\in CAlg(Shv(X)-mod)
$$ 
and $C(X)=\underset{\lambda\in\Lambda^*}{\oplus} E(X)$ viewed as a $\Lambda^*$-graded non-unital commutative algebra in $Shv(X)-mod$. By Proposition~\ref{Pp_E.1.22},
$\Fact(E)\otimes_{Shv(\Ran)} Shv(\Conf)\,\iso\, \Fact(C)$ canonically in $Shv(\Conf)-mod$. 

 Set 
$$
\Sph_{G,\Conf}=\prod_{\lambda\in\Lambda^*} Shv(\gL^+(G)_{\lambda}\backslash\Gr_{G, X^{\lambda}}).
$$
Using Corollary~\ref{Cor_E.3.7} apply the base change $\cdot\otimes_{Shv(\Ran)} Shv(\Conf)$ to the Satake functor 
$$
\Sat_{G,\Ran}: \Fact(\Rep(\check{G}))\to \Sph_{G,\Ran}
$$ 
given by (\ref{functor_Sat_G_Ran}). The result is the functor
\begin{equation}
\label{Satake_functor_Conf}
\Sat_{G,\Conf}: \Fact(C)\to \Sph_{G,\Conf},
\end{equation}
which is a version of the Satake functor for the configuration space. 

\index{Satake functor for the configuration space}

\sssec{} Recall that $(Shv(\Ran), \otimes^{ch})$ is an object of $CAlg^{nu}(\DGCat_{cont})$, and similarly for $Shv(\Conf)$. 

 Namely, for $I\in fSets$ write $\Ran^I_d\subset\Ran^I$ for the open part classifying collections $\cI_i, i\in I$ such that if $i\ne i'\in I$ then $\cI_i$ and $\cI_{i'}$ are disjoint. The subscript $d$ stands for `disjoint'. Write $\can:\und{\Conf}\to\Ran$ for the natural map defined  as in Remark~\ref{Rem_E.1.25_now}. Set 
$$
\und{\Conf}^I_d= \und{\Conf}^I\times_{\Ran^I} \Ran^I_d.
$$ 
For $I\in fSets$ we have the commutative diagram
$$
\begin{array}{ccccc}
\Ran^I & \getsup{v} & \Ran^I_d &\toup{u} & \Ran\\
\uparrow\lefteqn{\scriptstyle\can^I} && \uparrow && \uparrow\lefteqn{\scriptstyle\can}\\
\und{\Conf}^I & \getsup{v_c} & \und{\Conf}^I_d & \toup{u_c} & \und{\Conf},
\end{array}
$$
where both squares are cartesian. Here $u$ is the union of sets, $u_c$ is the sum of divisors, and $v, v_c$ are the natural inclusions (the subscript $c$ stands for configurations). 

 For $K_i\in Shv(\Ran)$, $i\in I$ their chiral product is given by $\underset{i\in I}{\otimes}^{ch} K_i=u_*v^!(\underset{i\in I}{\boxtimes} K_i)$. For $F_i\in Shv(\und{\Conf})$ their chiral product is given by $\underset{i\in I}{\otimes}^{ch} F_i=(u_c)_*(v_c)^!(\underset{i\in I}{\boxtimes} F_i)$. We see that $\can^!(\underset{i\in I}{\otimes}^{ch} K_i)\,\iso\, \underset{i\in I}{\otimes}^{ch}(\can^! K_i)$ canonically. This yields the following.
 
\begin{Lm} The functor $\can^!: (Shv(\Ran), \otimes^{ch})\to (Shv(\Conf), \otimes^{ch})$ is non-unital symmetric monoidal. \QED
\end{Lm} 
 
\sssec{} Recall the definition of the chiral product on $\Sph_{G,\Ran}$. View  $\gL^+(G)_{\Ran}\backslash \Gr_{G,\Ran}$ as the stack classifying $\cI\in \Ran$, $G$-torsors $\cF_G, \cF'_G$ on $\cD_{\cI}$ together with an isomorphism 
$$
\cF_G\,\iso\, \cF'_G\mid_{\oo{\cD}_{\cI}}. 
$$

\index{chiral product on $\Sph_{G,\Ran}$}

 Let $I\in fSets$. Set 
$$
(\gL^+(G)_{\Ran}\backslash\Gr_{G,\Ran})^I_d=(\gL^+(G)_{\Ran}\backslash\Gr_{G,\Ran})^I\times_{\Ran^I} \Ran^I_d.
$$ 
Consider the diagram
$$
(\gL^+(G)_{\Ran}\backslash\Gr_{G,\Ran})^I\;\getsup{v}\; (\gL^+(G)_{\Ran}\backslash\Gr_{G,\Ran})^I_d\;\toup{u} \; \gL^+(G)_{\Ran}\backslash\Gr_{G,\Ran},
$$
where by abuse of notations, we write $v$ for the natural map, and $u$ is the morphism sending 
$$
(\cF^i_G, {\cF'}^i_G, \cI_i, \; \beta_i: \cF^i_G\,\iso\, {\cF'}^i_G\mid_{\oo{\cD}_{\cI_i}})_{i\in I}
$$ 
to the collection
$(\cI, \cF_G, \cF'_G, \beta: \cF_G\,\iso\,\cF'_G\mid_{\oo{\cD}_{\cI}})$, where $\cI=\cup_{i\in I} \cI_i$ so that $\cD_{\cI}\,\iso\, \sqcup_{i\in I} \cD_{\cI_i}$ and 
$$
\oo{\cD}_{\cI}\,\iso\, \sqcup_{i\in I} \oo{\cD}_{\cI_i}.
$$ 
Here $\cF_G$ (resp., $\cF'_G$) is the $G$-torsor on $\cD_{\cI}$ whose restriction to $\cD_{\cI_i}$ is $\cF^i_G$ (resp., ${\cF'}^i_G$), and $\beta\mid_{\oo{\cD}_{\cI_i}}=\beta_i$ for $i\in I$.  Given $K_i\in \Sph_{G, \Ran}$ for $i\in I$ their chiral product is defined as 
$$
\underset{i\in I}{\otimes}^{ch} \, K_i=u_*v^!(\underset{i\in I}{\boxtimes} K_i).
$$ 

One similarly defines the chiral non-unital symmetric monoidal structure $\otimes^{ch}$ on the category 
$$
\Shv(\gL^+(G)_{\und{\Conf}}\backslash \Gr_{G, \und{\Conf}})\,\iso\, \Sph_{G,\Conf}.
$$

\index{chiral non-unital symmetric monoidal structure $\otimes^{ch}$ on $\Sph_{G,\Conf}$}

  Denote by
$$
\can_{\Gr}: \gL^+(G)_{\und{\Conf}}\backslash \Gr_{G, \und{\Conf}}\to \gL^+(G)_{\Ran}\backslash\Gr_{G,\Ran}
$$ 
the map obtained from $\can: \und{\Conf}\to \Ran$ by the base change $\gL^+(G)_{\Ran}\backslash\Gr_{G,\Ran}\to\Ran$.
 
\begin{Pp} 
\label{Pp_E.4.12}
The functor $\can^!_{\Gr}: (\Sph_{G,\Ran}, \otimes^{ch})\to (\Sph_{G,\Conf}, \otimes^{ch})$ is nonunital symmetric monoidal.
\end{Pp}
\begin{proof} Let $I\in fSets$. We have the commutative diagram, where the right square is cartesian
$$
\begin{array}{ccccc}
(\gL^+(G)_{\Ran}\backslash\Gr_{G,\Ran})^I & \getsup{v}&  (\gL^+(G)_{\Ran}\backslash\Gr_{G,\Ran})^I_d & \toup{u} & \gL^+(G)_{\Ran}\backslash\Gr_{G,\Ran}\\
\uparrow && \uparrow && \uparrow\lefteqn{\scriptstyle\can_{\Gr}}\\
(\gL^+(G)_{\und{\Conf}}\backslash\Gr_{G,\und{\Conf}})^I & \getsup{v_c} & (\gL^+(G)_{\und{\Conf}}\backslash\Gr_{G,\und{\Conf}})^I_d &\toup{u_c} & \gL^+(G)_{\und{\Conf}}\backslash\Gr_{G,\und{\Conf}}
\end{array}
$$
So for $K_I\in \Sph_{G,\Ran}$ one has canonically 
$$
\can_{\Gr}^!
(\underset{i\in I}{\otimes}^{ch} K_i)\,\iso\, \underset{i\in I}{\otimes}^{ch}(\can_{\Gr}^! K_i).
$$ 
\end{proof}

\sssec{} Let us describe the version of the Chevalley-Cousin complex for $\Gr_{G,\und{\Conf}}$.
Let $B\in \Perv(\Gr_{G, X})^{\gL^+(G)_X}$ be $\Lambda^*$-graded $B=\oplus_{\lambda\in\Lambda^*} B_{\lambda}$. Assume $B$ is a chiral algebra on $\Gr_{G, X}$ in a way compatible with $\Lambda^*$-grading. So, for a set $I$ of two elements, we have for the diagram 
$$
(\Gr_{G, X})^I\times_{X^I} \oo{X}{}^I\,\hook{\bar j^{(I)}}\, \Gr_{G, X^I}\,\getsup{\vartriangle} \,\Gr_{G, X}
$$ 
the chiral pairing 
$$
\bar j^{(I)}_*((B[1])^{\boxtimes I}\mid_{(\Gr_{G, X})^I\times_{X^I} \oo{X}{}^I})\to \vartriangle_*B[2]
$$
compatible with $\Lambda^*$-grading and satisfying the Jacobi identity.   

\index{Chevalley-Cousin complex for $\Gr_{G,\und{\Conf}}$}

 To this datum one associates the graded version of the Chevalley-Cousin complex $\cC(B)_{\und{\Conf}}$ on $\Gr_{G, \und{\Conf}}$ as follows. This is a collections of objects $\cC(B)_{X^{\lambda}}\in \Sph_{G, X^{\lambda}}$ for $\lambda\in\Lambda^*$. Since 
$$
\Gr_{G, \und{X}^{\lambda}}\,\iso\, \underset{(J, \und{\lambda})\in (fSets^{\lambda})^{op}}{\colim} \Gr_{G, X^J}, 
$$
the object $\cC(B)_{X^{\lambda}}$ is given by a compatible collection of objects $\cC(B)_{J, \und{\lambda}}\in \Sph_{G, J}$ for $(J, \und{\lambda})\in fSets^{\lambda}$. So, for maps $\phi: (J_1,\und{\lambda}^1)\to (J_2,\und{\lambda}^2)$ in $fSets^{\lambda}$ we must be given a compatible system of isomorphisms 
$$
(\vartriangle^{(J_1/J_2)})^! \cC(B)_{J_1, \und{\lambda}^1}\,\iso\, \cC(B)_{J_2, \und{\lambda}^2}.
$$ 
in $\Sph_{G, J_2}$. Given $(J, \und{\lambda})\in fSets^{\lambda}$, we set
\begin{equation}
\label{Chevalley-Cousin complex_Conf}
\cC(B)_{J, \und{\lambda}}=\underset{(I\toup{\phi} T)\in Q(I)}{\oplus} \vartriangle^{(J/T)}_*\bar j^{(T)}_*((\underset{t\in T}{\boxtimes} B_{\und{\mu}(t)}[1])\mid_{(\Gr_{G, X})^T\times_{X^T} \oo{X}{}^T})
\end{equation}
where $\und{\mu}=\phi_*\und{\lambda}$, and the corresponding maps are
$$
(\Gr_{G, X})^T\times_{X^T} \oo{X}{}^T\;\toup{\bar j^{(T)}}\;
\Gr_{G, X^T}\;\hook{\vartriangle^{(J/T)}} \;\Gr_{G, X^J}.
$$ 

 The differential in (\ref{Chevalley-Cousin complex_Conf}) is defined exactly as for the usual
Chevalley-Cousin complex in (\cite{BD_chiral}, 3.4.11). This concludes the construction of $\cC(B)_{X^{\lambda}}$.

\sssec{} Now using the notations of Section~\ref{Sect_E.4.8} take $B$ such that for each $\lambda\in\Lambda^*$, $B_{\lambda}=\cT_A$ with $A=\cO(\check{G})$ as in Section~\ref{Sect_7.3.7_now}. The resulting object $\cC(B)_{\und{\Conf}}\in\Sph_{G, \und{\Conf}}$ is the version of the chiral Hecke algebra for $\Gr_{G, \und{\Conf}}$. 

 Recall that the chiral Hecke algebra $\cC(\cT_A)\in \Fact(\Rep(\check{G}))\otimes_{Shv(\Ran)} \Sph_{G,\Ran}$ for $\Gr_{G,\Ran}$ is defined in 
Section~\ref{Sect_gap_described}. 

 We conclude that $\can_{\Gr}^! \cC(\cT_A)\in\Sph_{G,\Conf}$ identifies canonically with $\cC(B)_{\Conf}$. It gives rise to the version of the Satake functor (\ref{Satake_functor_Conf}) for $\Conf$ described above.
 
\sssec{} Similarly to Section~\ref{Sect_ext_convolution_precisely}, we define the exteriour convolution for $\Sph_{G,\Conf}$ as follows. 

\index{exteriour convolution for $\Sph_{G,\Conf}$}

 Consider the diagram
$$
(\gL^*(G)_{\Conf}\backslash \Gr_{G, \Conf})\times (\gL^*(G)_{\Conf}\backslash \Gr_{G, \Conf})\;\getsup{p}\; \Conv_{G,\Conf^2}\;\toup{m}\; \gL^*(G)_{\Conf}\backslash \Gr_{G, \Conf}
$$
Here the stack $\Conv_{G,\Conf^2}$ classifies $D, D'\in \Conf$, $G$-torsors $\cF_G,\cF'_G,\cF''_G$ on $\cD_{D+D'}$ together with isomorphisms 
$$
\beta: \cF_G\,\iso\,\cF'_D\mid_{\cD_{D+D'}-\supp(D)}, \;\;\; \beta': \cF'_G\,\iso\,\cF''_D\mid_{\cD_{D+D'}-\supp(D')}
$$
The map $m$ sends this point to $D+D'\in \Conf$, $G$-torsors $\cF_G,\cF''_G$ on $\cD_{D+D'}$ together with the isomorphism $\beta'\beta: \cF_G\,\iso\,\cF''_D\mid_{\oo{\cD}_{D+D'}}$. The map $p$ sends the above point to
$$
(D, \cF_G\mid_{\cD_D}, \cF'_G\mid_{\cD_D}, \bar\beta: \cF_G\,\iso\,\cF'_D\mid_{\oo{\cD}_D})\in \gL^*(G)_{\lambda}\backslash \Gr_{G, X^{\lambda}},
$$
$$
(D', \cF'_G\mid_{\cD_{D'}}, \cF''_G\mid_{\cD_{D'}}, \bar\beta: \cF'_G\,\iso\,\cF''_D\mid_{\oo{\cD}_{D'}})\in \gL^*(G)_{\lambda'}\backslash \Gr_{G, X^{\lambda'}}.
$$

 The convolution of $K, K'\in \Sph_{G,\Conf}$ is defined as
$$
K\star K'=m_*p^!(K\boxtimes K'). 
$$
One may also replace $\Conf$ by $\und{\Conf}$ in the above definition.

\begin{Pp} 
\label{Pp_E.4.16}
The functor $\can^!_{\Gr}: (\Sph_{G,\Ran}, \star)\to (\Sph_{G,\Conf}, \star)$ is nonunital monoidal.
\end{Pp}
\begin{proof}
The base change of 
$$
\Conv_{G,\Ran^2}\toup{m} \gL^+(G)_{\Ran}\backslash \Gr_{G,\Ran}
$$
by $\can: \und{\Conf}\to\Ran$ is denoted
$$
\Conv_{G,\und{\Conf}^2}\toup{m} \gL^+(G)_{\und{\Conf}}\backslash\Gr_{G,\und{\Conf}}
$$
Then the base change of
$$
\Conv_{G,\Ran^2}\toup{p}
(\gL^+(G)_{\Ran}\backslash \Gr_{G,\Ran})\times (\gL^+(G)_{\Ran}\backslash \Gr_{G,\Ran})
$$
by $\can^2: \und{\Conf}^2\to\Ran^2$ identifies canonically with
$$
\Conv_{G,\und{\Conf}^2}\toup{p} (\gL^+(G)_{\und{\Conf}}\backslash\Gr_{G,\und{\Conf}})\times (\gL^+(G)_{\und{\Conf}}\backslash\Gr_{G,\und{\Conf}})
$$
For $K, K'\in\Sph_{G,\Ran}$ one gets $\can^!_{\Gr}(K\star K')\,\iso\, (\can^!_{\Gr}K)\star (\can^!_{\Gr} K')$ canonically. 
\end{proof}
 
\ssec{Factorization coalgebras}
\label{Sect_Factorization coalgebras}

\sssec{} Let $C\in CAlg^{nu}(\DGCat_{cont})$ and $C(X)=C\otimes Shv(X)$. Let $A\in CoCAlg^{nu}(C)$. Consider the functor 
$$
\cF_{A,\Ran}: \cTw(fSets)^{op}\to \Fact(C)
$$ 
defined as follows. 

 For $J\in fSets$ we write $A^{\otimes J}\in C^{\otimes J}$ for the exteriour power of $A$. For a map $I\to J$ in $fSets$ we get the coproduct map $com: A^{\otimes J}\to m(A^{\otimes I})$ in $C^{\otimes J}$. Now $\cF_{A,\Ran}$ sends $(J\to K)\in \cTw(fSets)$ to the image of $A^{\otimes J}\boxtimes e_{X^K}$ under $C^{\otimes J}\otimes Shv(X^K)\to \Fact(C)$. Given a morphism (\ref{map_in_cTw(fSets)}) in $\cTw(fSets)$ recall the transition map in $\cF_{\Ran, C}$ given as 
$$
m\otimes \vartriangle_*: C^{\otimes J_1}\otimes Shv(X^{K_1})\to C^{\otimes J_2}\otimes Shv(X^{K_2})
$$ 
for $\vartriangle: X^{K_1}\to X^{K_2}$. We get a morphism in $C^{\otimes J_2}\otimes Shv(X^{K_2})$ given as the composition
$$
A^{\otimes J_2}\boxtimes e_{X^{K_2}}\to 
A^{\otimes J_2}\boxtimes \vartriangle_*e_{X^{K_1}}\toup{com} m(A^{\otimes J_1})\boxtimes \vartriangle_*e_{X^{K_1}},
$$  
hence also a morphism in $\Fact(C)$. Set
$$
\Fact^{coalg}(A)=\underset{(J\to K)\in\cTw(fSets)^{op}}{\lim} \cF_{A,\Ran}
$$ 
taken in $\Fact(C)$. 

\index{Factorization coalgebras in commutative factorization categories}

\sssec{Example} 
\label{Sect_example_constant_sheaf}
Take $C=\Vect$ and $A=e\in C$ with the evident structure of a coalgebra. The projection $\cTw(fSets)\to fSets^{op}$, $(J\to K)\mapsto K$ is cofinal. So, we get 
$$
\Fact^{coalg}(e)\,\iso\, \underset{K\in fSets}{\lim} e_{X^K},
$$
in $Shv(\Ran)$, where the functor $fSets\to \Fact(C)$ is as follows. If $K\to K'$ is a map in $fSets$ then for $\vartriangle: X^{K'}\to X^K$ the corresponding ransition map is $e_{X^K}\to \vartriangle_*e_{X^{K'}}$. 

\sssec{} Let $I\in fSets$. Consider the functor $\cF_{A, I}: \Tw(I)^{op}\to C_{X^I}$ defined as follows. It sends $(I\to J\to K)$ to the image of $A^{\otimes J}\boxtimes e_{X^K}$ under $C^{\otimes J}\otimes Shv(X^K)\to C_{X^I}$. Here we view $C_{X^I}$ as $\underset{\Tw(I)}{\colim} \cF_{I, C}$. The transition maps are defined as for $\cF_{A,\Ran}$.

 Set $A_{X^I}=\underset{Tw(I)^{op}}{\lim} \cF_{A, I}$. 

\sssec{} Let $f: I\to I'$ be a map in $fSets$ and $\vartriangle: X^{I'}\to X^I$ the corresponding diagonal. Let us construct a natural morphism 
$$
A_{X^I}\to \vartriangle_* A_{X^{I'}}
$$ 
in $C_{X^I}$. 
 We have the full embedding $\Tw(I')\subset \Tw(I)$ sending $(I'\to J'\to K')$ to $(I\to J'\to K')$ as in Section~\ref{Sect_2.1.5_now}. Since $\vartriangle_*: C_{X^{I'}}\to C_{X^I}$ is an exact functor, we get a canonical morphism
$$
A_{X^I}\to \underset{(I'\to J'\to K')\in\Tw(I')}{\lim} \vartriangle_*(A^{\otimes J'}\boxtimes e_{X^{K'}})\,\iso\,\vartriangle_* A_{X^{I'}}
$$ 
in $C_{X^I}$ by functoriality of right Kan extension. 

 As a result, we get a functor 
$$
\cF_{A, fSets}: fSets\to \Fact(C)
$$ 
sending $I$ to the image of $A_{X^I}$ under $C_{X^I}\to \Fact(C)$. It sends a morphism $f: I\to I'$ to the above map $A_{X^I}\to \vartriangle_* A_{X^{I'}}$.

\begin{Lm} There is a canonical isomoprhism
\begin{equation}
\label{map_from_Fact^coalg(A)}
\Fact^{coalg}(A)\,\iso\, \underset{fSets}{\lim} \cF_{A, fSets}
\end{equation}
in $\Fact(C)$. 
\end{Lm}
\begin{proof}  This is similar to Lemma~\ref{Lm_2.1.2_about_Fact(C)}. As in the proof of Lemma~\ref{Lm_2.1.2_about_Fact(C)}, we have the cartesian fibration $\zeta^{op}: \cY_{Tw}^{op}\to \cTw(fSets)^{op}$. The right Kan extension of $\cF_{A,\Ran}\comp \zeta^{op}$ along $\zeta^{op}$ identifies canonically with $\cF_{A,\Ran}$, so 
$$
\underset{\cY_{Tw}^{op}}{\lim} \cF_{A,\Ran}\comp \zeta^{op}\,\iso\, \underset{\cTw(fSets)}{\lim}  \cF_{A, \Ran}=\Fact^{coalg}(A). 
$$

Since $\cY_{Tw}\to fSets^{op}$, $(I\to J\to K)\mapsto I$ is a cocartesian fibration, we get using (\cite{G}, I.1, 2.2.4)
$$
\underset{\cY_{Tw}^{op}}{\lim} \, \cF_{A,\Ran}\comp \zeta^{op}\,\iso\, \underset{I\in fSets}{\lim}\;\;\,  \underset{(I\to J\to K)\in \Tw(I)^{op}}{\lim} \cF_{A,\Ran}(J\to K)
$$
as desired.
\end{proof}

\begin{Rem} Let $\phi: I\to I'$ be a map in $fSets$. Arguing as in Section~\ref{Sect_2.1.11_now}, one shows that $A_{X^I}\mid_{X^I_{\phi, d}}$ factorizes canonically, that is, one has a canonical isomorphism
$$
A_{X^I}\mid_{X^I_{\phi, d}}\,\iso\, (\underset{i'\in I'}{\boxtimes} A_{X^{I_{i'}}})\mid_{X^I_{\phi, d}}.
$$
in $C_{X^I}\mid_{X^I_{\phi, d}}$. 
\end{Rem}

\sssec{} In (\cite{Ga}, Section 7) Gaitsgory introduced the Verdier duality functor 
$$
\DD:Shv(\Ran)\to Shv(\Ran)^{op},
$$ 
which works for both constructible context and $\cD$-modules. This functor preserves colimits.

  For $I\in fSets$ let $\vartriangle^I: X^I\to\Ran$ be the structure map, so $\omega_{\Ran}\,\iso\, \underset{I\in fSets^{op}}{\colim} (\vartriangle^I)_!\omega_{X^I}$. Recall that $\Ran$ is a pseudo-scheme with finitary diagonal in the sense of (\cite{Ga}, 7.4.9), so (\cite{Ga}, 7.5.2) applies and yields an isomorphism
$$
\DD(\omega_{\Ran})\,\iso\, \underset{I\in fSets}{\lim} (\vartriangle^I)_! e_{X^I}.
$$
So, the object $\Fact^{coalg}(e)$ from Section~\ref{Sect_example_constant_sheaf} identifies with $\DD(\omega_{\Ran})$.  By (\cite{Ga}, Section 0.4.1, footnote 2), $\DD(\omega_{\Ran})=0$. So, in the example of Section~\ref{Sect_example_constant_sheaf} we get $\Fact^{coalg}(e)=0$. 
 
\sssec{} The commutative chiral coproduct for $\Fact^{coalg}(A)$ in this generality does not seem to exist. 

 In details, let us equip $\cTw(fSets)$ with the non-unital symmetric monoidal structure sending $(J_1\to K_1), (J_2\to K_2)$ to $(J\to K)$ with $J=J_1\sqcup J_2, K=K_1\sqcup K_2$. Let us equip $\Fact(C)$ with the $\star$-nonunital symmetric monoidal structure from Section~\ref{Sect_2.3.6_now}. Then the functor $\cF_{A,\Ran}$ is non-unital symmetric monoidal with respect to these structures. So, $\cF_{A,\Ran}^{op}: \cTw(fSets)\to \Fact(C)^{op}$ is also non-unital symmetric monoidal for the induced structures. 
 
 We can not conclude that $\colim \cF_{A,\Ran}^{op}$ is naturally an object of $CAlg^{nu}(\Fact(C)^{op},\star)$, as we need to know that the $\star$-product in $\Fact(C)^{op}$ preserves $\cTw(fSets)$-indexed colimits separately in each variable. 
 
  We think $\Fact^{coalg}(A)$ in this generality is not a reasonable object, there seems no reason for it to factorize over $\Ran$. 
  
\sssec{} 
\label{Sect_E.5.9_now}
For the rest of Section~\ref{Sect_Factorization coalgebras} assume that $C\in CAlg^{nu}(\DGCat_{cont})$ is graded $C=\underset{\lambda\in\Lambda^*}{\oplus} C_{\lambda}$ by $\Lambda^*$ in a way compatible with the non-unital symmetric monoidal structure. Our purpose is to show that under this additional assumption for $A\in CoCAlg^{nu}(C)$ the object $\Fact^{coalg}(A)\in\Fact(C)$ is reasonable. Namely, we equip it with a structure of a factorization coalgebra in $\Fact(C)$ in the sense of \cite{FG}. 

\sssec{} Write $A=\underset{\lambda\in\Lambda^*}{\oplus} A_{\lambda}$ with $A_{\lambda}\in C_{\lambda}$. Then $\Fact^{coalg}(A)$ inherits a $\Lambda^*$-grading. Namely, for $\lambda\in\Lambda^*$ we get the functor 
$$
^{\lambda}\cF_{A,\Ran}: \cTw(fSets)^{op}\to \Fact(C)_{\lambda}
$$ 
sending $(J\to K)$ to 
$$
\underset{\und{\lambda}: J\to \Lambda^*, \sum_j \und{\lambda}(j)=\lambda}{\oplus} (\underset{j\in J}{\otimes} A_{\und{\lambda}(j)})\boxtimes e_{X^K}.
$$ 
This is precisely the composition $\cTw(fSets)^{op}\to \Fact(C)\toup{\pr}\Fact(C)_{\lambda}$, where $\pr$ is the projection on the $\lambda$-component. Now
$$
\Fact^{coalg}(A)_{\lambda}\,\iso\,\underset{(J\to K)\in\cTw(fSets)}{\lim} {^{\lambda}\cF_{A,\Ran}}.
$$

\sssec{} For $\lambda\in\Lambda^*$ recall the category $\cTw(fSets)^{\lambda}$ from Section~\ref{Sect_E.1.2_now}. Define the functor
$$
\cF^{\lambda}_{A,\Ran}: (\cTw(fSets)^{\lambda})^{op}\to \Fact(C)_{\lambda}
$$
sending $(\und{\lambda}, J\to K)$ to $(\underset{j\in J}{\otimes} A_{\und{\lambda}(j)})\boxtimes e_{X^K}$. The transition map for a morphism $(\und{\lambda}^1, J_1\to K_1)\to (\und{\lambda}^2, J_2\to K_2)$ is  the composition
$$
(\underset{j\in J_2}{\otimes} A_{\und{\lambda}^2(j)})\otimes e_{X^{K_2}}\to (\underset{j\in J_2}{\otimes} A_{\und{\lambda}^2(j)})\otimes \vartriangle_*e_{X^{K_1}}
\to m(\underset{j\in J_1}{\otimes} A_{\und{\lambda}^1(j)})\otimes \vartriangle_*e_{X^{K_1}},
$$
where the second map is the comultiplication composed with the projection on the summand attached to $\und{\lambda}^1$ in $m(A^{\otimes J_1})$. 

\begin{Lm} One has canonically
$$
\Fact^{coalg}(A)_{\lambda}\,\iso\, \underset{(\cTw(fSets)^{\lambda})^{op}}{\lim} \cF^{\lambda}_{A,\Ran}.
$$
\end{Lm}
\begin{proof}
The projection $\cTw(fSets)^{\lambda}\to \cTw(fSets)$ forgetting $\und{\lambda}$ is a cocartesian fibration. So, by (\cite{G}, ch. I.1, 2.2.4), for any 
$(J\to K)\in\cTw(fSets)$ the functor 
$$
(\cTw(fSets)^{\lambda})_{(J\to K)}\to \cTw(fSets)^{\lambda}\times_{\cTw(fSets)} \cTw(fSets)_{/(J\to K)}
$$ 
is cofinal, where the LHS stands for the fibre. This shows that the right Kan extension of $\cF^{\lambda}_{A,\Ran}$ along $(\cTw(fSets)^{\lambda})^{op}\to \cTw(fSets)^{op}$ is the functor sending $(J\to K)$ to
$$
\underset{\und{\lambda}: J\to \Lambda^*, \sum_j \und{\lambda}(j)=\lambda}{\oplus} (\underset{j\in J}{\otimes} A_{\und{\lambda}(j)})\boxtimes e_{X^K}=(A^{\otimes J})_{\lambda}\boxtimes e_{X^K}.
$$
Thus, the right Kan extension of $\cF^{\lambda}_{A,\Ran}$ along $(\cTw(fSets)^{\lambda})^{op}\to \cTw(fSets)^{op}$ identifies with $^{\lambda}\cF_{A,\Ran}$. 
\end{proof}

\sssec{Example} 
\label{Sect_E.5.13_now}
Take $C=\underset{\lambda\in\Lambda^*}{\oplus} \Vect$. For $\lambda\in\Lambda^*$ let $X^{\lambda}$ be the moduli scheme of $\Lambda^*$-valued divisors of degree $\lambda$. By Proposition~\ref{Pp_E.1.5}, $\Fact(C)_{\lambda}\,\iso\, Shv(X^{\lambda})$ canonically. 

 Take $A=\underset{\lambda\in\Lambda^*}{\oplus} e\in C$. We may think of $A$ as the space of functions on the semigroup $\Lambda^*$. The product $\Lambda^*\times\Lambda^*\to \Lambda^*$ yields via restriction the cocommutative comultiplication map $\com: A\to A\otimes A$, so $A\in CoCAlg^{nu}(C)$. 
 
\begin{Lm} For $\lambda\in\Lambda^*$ in the example of Section~\ref{Sect_E.5.13_now} one has canonically
$$
\Fact^{coalg}(A)_{\lambda}\,\iso\, e_{X^{\lambda}}
$$ 
in $Shv(X^{\lambda})$. 
\end{Lm}
\begin{proof} Recall the functor $h: \cTw(fSets)^{\lambda}\to (fSets^{\lambda})^{op}$ given by (\ref{functor_h_Sect_E.1.9}). By Lemma~\ref{Lm_E.1.7}, $h$ is cofinal. For $(\und{\mu}, K)\in fSets^{\lambda}$ recall the map $s^{\und{\mu}}: X^K\to X^{\lambda}$ from Section~\ref{Sect_E.1.12_now}. 

 For $\lambda\in\Lambda^*$ we get
$$
\Fact^{coalg}(A)\,\iso\, \underset{(\und{\lambda}, J\toup{\phi} K)\in (\cTw(fSets)^{\lambda})^{op}}{\lim} (s^{\phi_*\und{\lambda}})_!e_{X^K}\,\iso\,
\underset{(\und{\mu}, K)\in fSets^{\lambda}}{\lim} s^{\und{\mu}}_!e_{X^K}
$$ 
in $Shv(X^{\lambda})$. The category $fSets^{\lambda}$ is finite. For $(\und{\mu}, K)\in fSets^{\lambda}$, 
$$
(s^{\und{\mu}})_!e_{X^K}\,\iso\, \DD(s^{\und{\mu}}_*\omega_{X^K})
$$ 
in $Shv(X^{\lambda})^c$. Thus, 
$$
\Fact^{coalg}(A)\,\iso\, \DD(\underset{(\und{\mu}, K)\in (fSets^{\lambda})^{op}}
{\colim} s^{\und{\mu}}_!\omega_{X^K}).
$$
in $Shv(X^{\lambda})^c$. By 
Proposition~\ref{Pp_E.1.23_now}, $\underset{(\und{\mu}, K)\in (fSets^{\lambda})^{op}}
{\colim} s^{\und{\mu}}_!\omega_{X^K}\,\iso\, \omega_{X^{\lambda}}$.
\end{proof}

\sssec{Commutative chiral coproduct} The big advantage of the graded setting is that for $\lambda\in\Lambda^*$ the category $\cTw(fSets)^{\lambda}$ is finite, and exact functors preserve finite limits. 

In the setting of Section~\ref{Sect_E.5.9_now} define the commutative chiral coproduct on $\Fact^{coalg}(A)$ as follows. Given $\lambda\in\Lambda^*$ we need to define a morphism
$$
\Fact^{coalg}(A)_{\lambda}\to \underset{\lambda_1+\lambda_2=\lambda, \lambda_i\in\Lambda^*}{\oplus} \Fact^{coalg}(A)_{\lambda_1}\star \Fact^{coalg}(A)_{\lambda_2},
$$
where $\star$ is the non-unital symmetric monoidal structure on $\Fact(C)$ from Section~\ref{Sect_2.3.6_now}. 

\index{commutative chiral coproduct on $\Fact^{coalg}(A)$}

 In a stable category a finite coproduct is the finite product, so we may replace in the latter formula $\oplus$ by $\prod$. For $\lambda_i\in\Lambda^*$ and $\lambda=\lambda_1+\lambda_2$ it suffices to define the morphism
\begin{equation}
\label{chiral_coproduct_for_coalgebras_map}
\Fact^{coalg}(A)_{\lambda}\to \Fact^{coalg}(A)_{\lambda_1}\star \Fact^{coalg}(A)_{\lambda_2}.
\end{equation}

Since $\star$ preserves finite limits separately in each variable,
\begin{equation}
\label{formula_for_star_product_Sect_E.5.15}
\Fact^{coalg}(A)_{\lambda_1}\star \Fact^{coalg}(A)_{\lambda_2}\,\iso\, 
\underset{
\begin{array}{cc}
\scriptstyle{(\und{\lambda}^1, J_1\to K_1)\in \cTw(fSets)^{\lambda_1, op}}\\
\scriptstyle{(\und{\lambda}^2, J_2\to K_2)\in \cTw(fSets)^{\lambda_2, op}}
\end{array}}
{\lim}
(\underset{j\in J}{\otimes} A_{\und{\lambda}(j)})\boxtimes e_{X^K}
\end{equation}
taken in $\Fact(C)$, where we have set $J=J_1\sqcup J_2, K=K_1\sqcup K_2$, $\und{\lambda}=\und{\lambda}^1\sqcup \und{\lambda}^2$. 

 Now
$$
\Fact^{coalg}(A)_{\lambda}=\underset{(\und{\lambda}, J\to K)\in \cTw(fSets)^{\lambda, op}}{\lim} (\underset{j\in J}{\otimes} A_{\und{\lambda}(j)})\boxtimes e_{X^K}
$$
We have the functor $\cTw(fSets)^{\lambda_1}\times \cTw(fSets)^{\lambda_2}\to \cTw(fSets)^{\lambda}$ sending 
$$
((\und{\lambda}^1, J_1\to K_1), (\und{\lambda}^2, J_2\to K_2))\mapsto (\und{\lambda}, J\to K)
$$ 
given by the disjoint union $J=J_1\sqcup J_2, K=K_1\sqcup K_2$, and $\und{\lambda}=\und{\lambda}_1\sqcup\und{\lambda}_2$. So, we get the desired morphism (\ref{chiral_coproduct_for_coalgebras_map}) by functoriality of the limit. 
This equips $\Fact^{coalg}(A)$ with a structure of an object of $CoCAlg^{nu}(\Fact(C),\star)$. 

\sssec{} Consider the open subscheme $(X^{\lambda_1}\times X^{\lambda_2})_d\subset X^{\lambda_1}\times X^{\lambda_2}$ given by the property that the divisors have disjoint supports. The sum map $m_d: (X^{\lambda_1}\times X^{\lambda_2})_d\to X^{\lambda}$ is \'etale, here $\lambda=\lambda_1+\lambda_2$.
The map (\ref{chiral_coproduct_for_coalgebras_map}) yields a morphism
\begin{equation}
\label{factorization_map_coalg}
m_d^* \Fact^{coalg}(A)_{\lambda}\to \Fact^{coalg}(A)_{\lambda_1}\boxtimes \Fact^{coalg}(A)_{\lambda_2}
\end{equation}
in 
\begin{equation}
\label{category_over_open_part_graded_version}
(\Fact(C)_{\lambda_1}\boxtimes \Fact(C)_{\lambda_2})\otimes_{Shv(X^{\lambda_1}\times X^{\lambda_2})} Shv((X^{\lambda_1}\times X^{\lambda_2})_d).
\end{equation} 

\begin{Pp} The map (\ref{factorization_map_coalg}) is an isomorphism. 
\end{Pp}
\begin{proof}
Recall that for $(\und{\lambda}, J\toup{\phi} K)\in \cTw(fSets)$, $C^{\otimes J}\otimes Shv(X^K)$ is viewed as an object of $Shv(X^{\lambda})-mod$ via the map $s^{\phi_*\und{\lambda}}: X^K\to X^{\lambda}$. 
  
  For $(\mu, K)\in fSets^{\lambda}$ we have a cartesian square
$$
\begin{array}{ccc}
X^K & \toup{s^{\und{\mu}}} & X^{\lambda}\\
\uparrow && \uparrow\\
\underset{
\begin{array}{cc}
\scriptstyle{(\und{\mu}_1, K_1)\in fSets^{\lambda_1},}\\ 
\scriptstyle{(\und{\mu_2}, K_2)\in fSets^{\lambda_2}}\\
\scriptstyle{(\und{\mu}_1, K_1)\sqcup (\und{\mu_2}, K_2)\,\iso\,(\und{\mu}, K)}
\end{array}}
{\sqcup} (X^{K_1}\times X^{K_2})_d & \to & (X^{\lambda_1}\times X^{\lambda_2})_d
\end{array}
$$
where the vertical arrow is given by $(X^{K_1}\times X^{K_2})_d\subset X^K$ for $K=K_1\sqcup K_2$ and $\und{\mu}=\und{\mu}_1\sqcup \und{\mu}_2$. 

 Let now $(\und{\lambda}, J\toup{\phi} K)\in\cTw(fSets)^{\lambda}$. Assume given a decomposition 
$$
(\phi_*\und{\lambda}, K)=(\und{\mu}_1, K_1)\sqcup (\und{\mu}_2, K_2)
$$ 
with $(\und{\mu}_i, K_i)\in fSets^{\lambda_i}$. We then get $J_i=\phi^{-1}(K_i)$, and the objects $(\und{\lambda}^i, J_i\to K_i)\in \cTw(fSets)^{\lambda_i}$ equipped with 
\begin{equation}
\label{decomp_for_cTw(fSets)^lambda}
(\und{\lambda}, J\to K)=(\und{\lambda}^1, J_1\to K_1)\sqcup (\und{\lambda}^2, J_2\to K_2).
\end{equation}
For this reason
$$
m_d^* \Fact^{coalg}(A)_{\lambda}\,\iso\, \underset{
\begin{array}{cc}
\scriptstyle{(\und{\lambda}^1, J_1\to K_1)\in \cTw(fSets)^{\lambda_1, op}}\\
\scriptstyle{(\und{\lambda}^2, J_2\to K_2)\in \cTw(fSets)^{\lambda_2, op}}
\end{array}}
{\lim}
(\underset{j\in J}{\otimes} A_{\und{\lambda}(j)})\boxtimes e_{X^K}
$$
in (\ref{category_over_open_part_graded_version}), where $(\und{\lambda}, J\to K)$ is given by (\ref{decomp_for_cTw(fSets)^lambda}). Our claim follows now from (\ref{formula_for_star_product_Sect_E.5.15}). 
\end{proof}

\sssec{} We conclude that in the setting of Section~\ref{Sect_E.5.9_now} the object $\Fact^{coalg}(A)$ is a factorization coalgebra in $(\Fact(C), \star)$ as in (\cite{FG}, Definition~2.4.7). 

\index{Factorization coalgebra in $(\Fact(C), \star)$}

 \section{Safe pseudo-schemes}
 \label{Sect_Safe pseudo-schemes}
 
 \sssec{} In this section we introduce a notion of safe pseudo-schemes and prove some version of the K\"unneth formula for them. While for $\cD$-modules this is easy, in the constructible context this is a more difficult new result of independent interest. The mains results of this section are Theorem~\ref{Thm_Kunneth_formula}, Corollary~\ref{Cor_F.0.9}.
In the case of $\cD$-modules (resp., in the constructible context) they are used to get in addition Proposition~\ref{Cor_F.1.3} (resp., Proposition~\ref{Pp_F.2.19}) below.
 
\sssec{}  As in \cite{Ga}, by a pseudo-scheme $Y$ we mean a prestack of the form $Y\,\iso\, \colim_{i\in I} Y_i$ taken in $\PreStk$, where $I\in 1-\Cat$ is small, for $i\in Y$ we have $Y_i\in\Sch_{ft}$ with $Y_i$ separated, and for $\alpha: i\to j$ in $I$ the transition map $f_{\alpha}: Y_i\to Y_j$ is proper. 

\index{pseudo-scheme}
 
\sssec{} Let $Y$ be a pseudo-scheme. Then $Shv(Y)$ is naturally self-dual, and the corresponding pairing is $Shv(Y)\otimes Shv(Y)\to \Vect$, $K_1\otimes K_2\mapsto \RG(Y, K_1\otimes^! K_2)$.   

 Let $f_i: Y_i\to Y$ be the natural map. Then $f_i$ is pseudo-proper, and the diagonal map $Y\to Y\times Y$ is pseudo-proper by (\cite{Ga}, 7.4.2). Besides, $Shv(Y)$ is compactly generated by objects of the form $(f_i)_*K$ for $i\in I, K\in Shv(Y_i)^c$. 
 
\sssec{} Define $Shv(Y)^{constr}$ as the colimit $\colim_{i\in I} Shv(Y_i)^c$ taken in $\DGCat^{non-cocmpl}$. Since $\Ind: \DGCat^{non-cocmpl}\leftrightarrows \DGCat_{cont}: \oblv$ is an adjoint pair, we get 
$$
\Ind(Shv(Y)^{constr})\,\iso\, \colim_{i\in I} Shv(Y_i)\,\iso\, Shv(Y).
$$ 
In particular, $Shv(Y)^{constr}\subset Shv(Y)$ is a full embedding. In fact, $Shv(Y)^{constr}\subset Shv(Y)^c$ is a full subcategory by (\cite{HTT}, 5.3.5.5). So, the natural map $ \colim_{i\in I} Shv(Y_i)^c\to Shv(Y)^c$ induces an equivalence after applying $\Ind$. 

 Recall that for $K\in Shv(Y_I)^c$ we have canonically $\DD(f_{\alpha !}(K))\,\iso\, f_{\alpha !}\DD(K)$ by (\cite{Ga}, 7.2.3). So, 
$$
\DD: Shv(Y_i)^c\,\iso\, \Shv(Y_i)^{c, op}
$$
is functorial in $i\in I$. Passing to the colimit over $i$ in $\DGCat^{non-cocmpl}$, we get an equivalence
$$
\DD: \underset{i\in I}{\colim} \, Shv(Y_i)^c\,\iso\, \underset{i\in I}{\colim} \, \Shv(Y_i)^{c, op}\,\iso\, (\underset{i\in I}{\colim} \, \Shv(Y_i)^c)^{op}
$$
in $\DGCat^{non-cocmpl}$. We used the fact that $E\mapsto E^{op}$ is an involution on $\DGCat^{non-cocmpl}$. We constructed an equivalence
$$
\DD: Shv(Y)^{constr}\,\iso\, Shv(Y)^{constr, op}
$$
which we also refer to as the Verdier duality. 

\sssec{} Recall that if $i,j\in I$ then $Y_i\times_Y Y_j$ is a pseudo-scheme, consider the diagram 
\begin{equation}
\label{diag_projections_for_Y_i,j}
Y_i\;\getsup{\pr_i}\; Y_i\times_Y Y_j\;\toup{\pr_j}\; Y_j.
\end{equation} 
Then $\pr_i$ are pseudo-proper (cf. \cite{Ga}, 7.4.2). 

\begin{Def} 
\label{Def_safe_pseudo-scheme}
Say that the pseudo-scheme $Y$ as above is \select{safe} if for any $i, j\in I$, $Y_i\times_Y Y_i\in \Sch_{ft}$ and both projections in $\pr_i, \pr_j$ in (\ref{diag_projections_for_Y_i,j}) are finite.
\end{Def}

\index{Safe pseudo-scheme}

\sssec{} For the rest of Section~\ref{Sect_Safe pseudo-schemes} we assume that $Y$ is a safe pseudo-scheme. For $i,j\in I$ let 
$$
\cF: Shv(Y_i)\otimes_{Shv(Y)} Shv(Y_j)\to Shv(Y_i\times_Y Y_j)
$$
be the functor obtained from $\pr_i^!\otimes^! \pr_j^!: Shv(Y_i)\otimes Shv(Y_j)\to Shv(Y_i\times_Y Y_j)$ via the universal property.

 The main results of this section are the following.
\begin{Thm} 
\label{Thm_Kunneth_formula}
For $i,j\in I$ the functor $\cF$ is an equivalence.
\end{Thm}

The proof of Theorem~\ref{Thm_Kunneth_formula} for $\cD$-modules (resp., in the constructible context) is found in Section~\ref{Sect_D-modules_proof_of_Thm_Kunneth_formula} (resp., in Section~\ref{Sect_F.2.17}). 

\begin{Cor} 
\label{Cor_F.0.9}
The prestack $Y$ is 1-affine.
\end{Cor}
\begin{proof}
Let $E\in ShvCat(Y)$. For a map $\alpha: i\to j$ in $I$ we have the adjoint pair $(f_{\alpha})_{!, E}: \Gamma(Y_i, E)\leftrightarrows \Gamma(Y_j, E): (f_{\alpha})^!_E$. So, we may pass to the left adjoints in the diagram $\underset{i\in I^{op}}{\lim} \Gamma(Y_i, E)$, and 
$$
\Gamma(Y, E)\,\iso\, \underset{i\in I^{op}}{\lim} \Gamma(Y_i, E)\,\iso\, \underset{i\in I}{\colim} \, \Gamma(Y_i, E).
$$

 Now given $D\in Shv(Y)-mod$ the above yields an isomorphism $\Gamma(Y, \Loc_Y(D))\,\iso\, D$ in $Shv(Y)-mod$, so $\Loc_Y: Shv(Y)-mod\to ShvCat(Y)$ is fully faithful. 

 Let us show that $\Loc_{\Ran}$ is essentially surjective. For $i\in I$ it suffices to show that the natural map
$$
\Gamma(Y, E)\otimes_{Shv(Y)} Shv(Y_i)\to \Gamma(Y_i, E)
$$
is an equivalence. We have 
\begin{equation}
\label{first_iso_for_Pp_1-affineness_any_Y}
\Gamma(Y, E)\otimes_{Shv(Y)} Shv(Y_i)\,\iso\, \underset{j\in I}{\colim} \, (\Gamma(Y_j, E)\otimes_{Shv(Y)} Shv(Y_i)).
\end{equation}
Since 
\begin{multline*}
\Gamma(Y_j, E)\otimes_{Shv(Y)} Shv(Y_i)\,\iso\, \Gamma(Y_j, E)\otimes_{Shv(Y_j)} (Shv(Y_j)\otimes_{Shv(Y)} Shv(Y_i))\,\iso\\  \Gamma(Y_j, E)\otimes_{Shv(Y_j)} Shv(Y_j\times_Y Y_i)\,\iso\, \Gamma(Y_j\times_Y Y_i, E),
\end{multline*}
(\ref{first_iso_for_Pp_1-affineness_any_Y}) identifies canonically with $\underset{j\in I}{\colim} \, \Gamma(Y_j\times_Y Y_i, E)\,\iso\,\Gamma(Y_i, E)$ as desired.
\end{proof}
 
\sssec{} Let $\vartriangle: Y\to Y\times Y$ be the diagonal. For $i,j\in I$ consider the natural maps 
$$
Y\;\getsup{p_{ij}}\;Y_i\times_Y Y_j\;\toup{q_{ij}} \;Y_i\times Y_j.
$$ 
Since $\pr_j$ is finite and $Y_i$ is separated, $q_{ij}$ is proper.  
 
\ssec{Case of $\cD$-modules}

\sssec{} In this subsection our sheaf theory is $\cD$-modules. Note that $\boxtimes: Shv(Y)\otimes Shv(Y)\to Shv(Y\times Y)$ is an equivalence. 

\sssec{Proof of Theorem~\ref{Thm_Kunneth_formula}}
\label{Sect_D-modules_proof_of_Thm_Kunneth_formula}
We have the adjoint pair $(q_{ij})_!: Shv(Y_i\times_Y Y_j)\leftrightarrows Shv(Y_i\times Y_j): q_{ij}^!$ with $(q_{ij})_!$ fully faithful. 

We have the adjoint pair $\vartriangle_!: Shv(Y)\leftrightarrows Shv(Y\times Y): \vartriangle^!$ in $Shv(Y\times Y)$-mod with $\vartriangle_!$ fully faithful. It yields the adjoint pair
$$
L: Shv(Y_i)\otimes Shv(Y_j)\otimes_{Shv(Y\times Y)} Shv(Y)\leftrightarrows Shv(Y_i)\otimes Shv(Y_j): R
$$
in $\DGCat_{cont}$ with $L$ fully faihtful. Besides, 
$$
Shv(Y_i)\otimes Shv(Y_j)\otimes_{Shv(Y\times Y)} Shv(Y)\,\iso\, Shv(Y_i)\otimes_{Shv(Y)} Shv(Y_j)
$$ 
canonically. It remains to show that the endofunctor $LR$ on $Shv(Y_i)\otimes Shv(Y_j)$ is naturally identified with the endofunctor
$(q_{ij})_!q_{ij}^!$. This follows from the compatibility between the exterior products and base change isomorphisms. \QED

\begin{Cor} 
\label{Cor_F.1.3}
Let $i,j\in I$ and $S\to Y_i, S'\to Y_j$ be maps in $\Sch_{ft}$. Then the natural functor 
$$
Shv(S)\otimes_{Shv(Y)} Shv(S')\to Shv(S\times_Y S')
$$
is an equivalence.
\end{Cor}
\begin{proof}
We have
\begin{multline*}
Shv(S)\otimes_{Shv(Y)} Shv(S')\,\iso\, Shv(S)\otimes_{Shv(Y_i)} Shv(Y_i)\otimes_{Shv(Y)} 
Shv(Y_j)\otimes_{Shv(Y_j)} Shv(S')\\ 
\iso\, Shv(S)\otimes_{Shv(Y_i)} Shv(Y_i\times_Y Y_j)\otimes_{Shv(Y_j)} Shv(S')\,\iso\, Shv(S\times_Y S'),
\end{multline*}
where the last isomorphism is given by (\cite{Chen}, A.4.5).
\end{proof}

\ssec{Constructible context}

\sssec{} In this subsection we work in the constructible context. 

\sssec{} 
\label{Sect_F.2.2}
For $i\in I$ define the functor $f_i^*: Shv(Y)\to Shv(Y_i)$ via the compatible family of functors 
$$
Shv(Y_j)\to Shv(Y_i), \, K\mapsto (\pr_i)_!\pr_j^*K
$$ 
indexed by $j\in I$ by passing to the colimit over $j\in I$. Here the transition maps are $f_{\alpha !}$ for $\alpha: j\to j'$ in $I$. The family is compatible because of the base change: for a map $\alpha: j\to j'$ in $I$ and the diagram
$$
\begin{array}{ccc}
Y_i\times_Y Y_j &\toup{\pr_j} & Y_j\\
\uparrow\lefteqn{\scriptstyle \id\times f_{\alpha}} && \uparrow\lefteqn{\scriptstyle f_{\alpha}}\\
Y_i\times_Y Y_{j'} & \toup{\pr_{j'}} & Y_{j'}
\end{array}
$$
we have $\pr_j^*(f_{\alpha})_!\,\iso\, (\id\times f_{\alpha})_!(\pr_{j'})^*$. Indeed, $\pr_j$ is a map in $\Sch_{ft}$. 

Using base change for pseudo-proper maps from (\cite{Ga}, 1.5.4), we see that 
\begin{equation}
\label{adj_pair_f_i^*_and_f_i_*}
f_i^*: Shv(Y)\leftrightarrows Shv(Y_i): (f_i)_*
\end{equation} 
is an adjoint pair in $\DGCat_{cont}$.

\begin{Rem} For a map $\alpha: i\to j$ in $I$ we have canonically $f_{\alpha}^*f_j^*\,\iso\, f_i^*$ as functors $Shv(Y)\to Shv(Y_i)$. So, we get a functor 
$$
\Theta: Shv(Y)\to \underset{i\in I^{op}, f_{\alpha}^*}{\lim} Shv(Y_i),
$$ 
where in the RHS we used the diagram
$I^{op}\to \DGCat_{cont}$, $i\mapsto Shv(Y_i)$ sending $\alpha: i\to j$ in $I$ to $f_{\alpha}^*$. We do not know if $\Theta$ is an equivalence.
\end{Rem}
\begin{Rem} 
\label{Rem_left_adjoint_to_(id_times_f_i)_*}
For $i\in I, S\in \Sch_{ft}$ one similarly shows that for $\id\times f_i: S\times Y_i\to S\times Y$ the functor $(\id\times f_i)_*$ has a left adjoint denoted $(\id\times f_i)^*$. 
\end{Rem}

\begin{Lm}
\label{Lm_existence_of_vartriangle^*}
 i) The functor $\vartriangle_*: Shv(Y)\to Shv(Y\times Y)$ preserves limits, so admits a left adjoint denoted $\vartriangle^*$. Besides, one has canonically
$$
\vartriangle^*(f_i\times f_j)_*\,\iso\, (p_{ij})_*q_{ij}^*.
$$

\smallskip\noindent
ii) The functor $\boxtimes: Shv(Y)\otimes Shv(Y)\to Shv(Y\times Y)$ is fully faithful.
\end{Lm}
\begin{proof}
i) Let $i,j\in I$. It suffices to show that the functor $(f_i\times f_j)^!\!\vartriangle_*: Shv(Y)\to Shv(Y_i\times Y_j)$ preserves limits. We have $(f_i\times f_j)^!\vartriangle_*\,\iso\, (q_{ij})_*p_{ij}^!$. Since $q_{ij}$ is a map in $\Sch_{ft}$, $(q_{ij})_*$ has the left adjoint, so $(q_{ij})_*p_{ij}^!$ preserves limits.

 The second claim is obtained by passing to the left adjoints in the isomorphism $(f_i\times f_j)^!\vartriangle_*\,\iso\, (q_{ij})_*p_{ij}^!$. 

\smallskip\noindent
ii) Since $Shv(Y)$ is dualizable, we get
$$
Shv(Y)\otimes Shv(Y)\,\iso\, \lim_{i,j\in I\times I} Shv(Y_i)\otimes Shv(Y_j).
$$
So, $\boxtimes$ is obtained by passing to the limit over $(i,j)\in I\times I$ in the fully faithful functors $Shv(Y_i)\otimes Shv(Y_j)\to Shv(Y_i\times Y_j)$.
\end{proof}

\sssec{} 
\label{Sect_F.2.6}
We define the non-unital symmetric monoidal structure $\otimes$ on $Shv(Y)$ so that the product is the composition $Shv(Y)\otimes Shv(Y)\toup{\boxtimes} Shv(Y\times Y)\toup{\vartriangle^*}Shv(Y)$. 

 In more details, consider the 1-full subcategory $\PreStk_{ind-sch}\subset\PreStk_{lft}$, where we restrict 1-morphisms to be ind-schematic. Then we have a well-defined functor 
$$
Shv_{\PreStk_{ind-sch}}: \PreStk_{ind-sch}\to \DGCat_{cont}
$$ 
sending $Y$ to $Shv(Y)$ and a morphism $f: Y\to Y'$ to $f_*: Shv(Y)\to Shv(Y')$. Moreover, this functor is right-lax symmetric monoidal. 

 Consider the 1-full subcategory 
$$
\PreStk_{ind-sch, R}\subset \PreStk_{ind-sch},
$$ 
where we restrict the morphism to those $f: Y\to Y'$ for which $f_*: Shv(Y)\to Shv(Y')$ has a left adjoint. Let 
$$
Shv_{\PreStk_{ind-sch, R}}: \PreStk_{ind-sch, R}\to \DGCat_{cont}
$$ 
be the restriction of $Shv_{\PreStk_{ind-sch}}$. Then we may pass to the left adjoints in the functor $\PreStk_{ind-sch, R}$ and get the functor 
$$
\PreStk_{ind-sch}^*: \PreStk_{ind-sch, R}^{op}\to \DGCat_{cont}
$$ 
sending $Y$ to $Shv(Y)$ and sending $f: Y\to Y'$ to $f^*: Shv(Y')\to Shv(Y)$. The functor $\PreStk_{ind-sch}^*$ is still right-lax symmetric monoidal.

 We have $Y\in CoCAlg(\PreStk_{lft})$. This gives
$$
Y\in CAlg^{nu}(\PreStk_{ind-sch, R}^{op}),
$$
so $Shv(Y)\in CAlg^{nu}(\DGCat_{cont})$ naturally. This is the $\otimes$-symmetric monidal structure on $Shv(Y)$. 

 Let us underline the following. For the map $p: Y\to \Spec k$ we do not know if $p_*: Shv(Y)\to \Vect$ has a left adjoint. For this reason, $Y$ is a \select{non-unital} commutative algebra in $\PreStk_{ind-sch, R}^{op}$. Thus, the $\otimes$-symmetric monoidal structure on $Shv(Y)$ is not necessarily unital. 
\begin{Lm} 
\label{Lm_F.2.7}
i) For $i\in I$ the functor $f_i^*: (Shv(Y), \otimes)\to (Shv(Y_i),\otimes)$ is non-unital symmetric monoidal.\\
ii) The right-lax $(Shv(Y),\otimes)$-structure on $(f_i)_*$ is strict. So, the adjoint pair (\ref{adj_pair_f_i^*_and_f_i_*}) takes place in $(Shv(Y),\otimes)-mod$. 
\end{Lm}
\begin{proof}
i) Consider the diagram
$$
\begin{array}{ccc}
Y & \toup{\vartriangle} & Y\times Y\\
\uparrow\lefteqn{\scriptstyle f_i} && \uparrow\lefteqn{\scriptstyle f_i\times f_i}\\
Y_i & \toup{\vartriangle} & Y_i\times Y_i.
\end{array}
$$
Passing to the left adjoints in the isomoprhism $\vartriangle_* (f_i)_* \,\iso\,(f_i\times f_i)_*\vartriangle_*$ we get $f_i^*\vartriangle^*\,\iso\, \vartriangle^*(f_i\times f_i)^*$. For $K, K'\in Shv(Y)$ we have 
$$
(f_i^*K)\boxtimes (f_i^*K')\,\iso\, (f_i\times f_i)^*(K\boxtimes K')
$$
in $Shv(Y_i\times Y_i)$. Our claim follows.

\smallskip\noindent
ii) Let $K\in Shv(Y_i), L\in Shv(Y)$. We show that the natural map $((f_i)_*K)\otimes L\to (f_i)_*(K\otimes f_i^*L)$ is an equivalence. We may and do assume that there is $j\in I$ and $L\,\iso\, (f_j)_*M$ for some $M\in Shv(Y_j)$. By Lemma~\ref{Lm_existence_of_vartriangle^*} i), one has 
$$
((f_i)_*K)\otimes ((f_j)_*M)\,\iso\, (p_{ij})_*q_{ij}^*(K\boxtimes M)
$$
On the other hand, 
\begin{multline*}
(f_i)_*(K\otimes f_i^*(f_j)_*M)\,\iso\, (f_i)_*(K\otimes (\pr_i)_!(\pr_j)^*M)\,\iso\\  (f_i)_*(\pr_i)_!(\pr_i^*K\otimes \pr_j^*M)\,\iso\, (p_{ij})_*(q_{ij}^*(K\boxtimes M)
\end{multline*}
by $(!, ^*)$-base change as desired.
\end{proof}

\begin{Lm} 
\label{Lm_F.2.8}
The functor $(f_i)_!: Shv(Y_i)\to Shv(Y)$ is conservative and monadic. The corresponding monad is $Shv(Y)\to Shv(Y), K\mapsto (f_i)_*e\otimes K$. 
\end{Lm}
\begin{proof}
Let $K\in Shv(Y_i)$ with $(f_i)_!K=0$. Let $j\in I$. For the diagram (\ref{diag_projections_for_Y_i,j}) the functor $(\pr_j)_!$ is conservative.  
So, $(\pr_j)_!\pr_i^!K\,\iso\, f_j^!(f_i)_!K=0$, hence $\pr_i^!K=0$ in $Shv(Y_i\times_Y Y_j)$. This gives $K=0$. The monadicity follows by (\cite{G}, I.1, 3.7.7). 
\end{proof}

\sssec{} Note that  $(f_i)_*: (Shv(Y_i), \otimes)\to (Shv(Y),\otimes)$ is right-lax non-unital symmetric monoidal, so $(f_i)_*e\in CAlg^{nu}(Shv(Y),\otimes)$.

Denote by 
$$
Shv(Y_i)\otimes_{(Shv(Y), \otimes)} Shv(Y_j)
$$ 
the tensor product, where we view $Shv(Y_i), Shv(Y_j)$ as $(Shv(Y), \otimes)$-modules via the monoidal functors $f_i^*, f_j^*$. Let 
$$
\cF': Shv(Y_i)\otimes_{(Shv(Y), \otimes)} Shv(Y_j)\to Shv(Y_i\times_Y Y_j)
$$
be the functor obtained from $Shv(Y_i)\otimes Shv(Y_j)\to Shv(Y_i\times_Y Y_j)$, $K\boxtimes K'\mapsto \pr_i^*K\otimes \pr_j^*K'$ via the universal property.

\begin{Lm} 
\label{Lm_functor_cF'_is_equivalence}
The functor $\cF'$ is an equivalence.
\end{Lm}
\begin{proof} Since $\pr_j$ is a finite morphism in $\Sch_{ft}$, $(\pr_j)_!: Shv(Y_i\times_Y Y_j)\to Shv(Y_j)$ is conservative and continuous, hence monadic. The corresponding monad is $K\mapsto (\pr_j)_!(\pr_j)^*K\,\iso\, (\pr_j)_!e\otimes K$. Note that $f_j^*(f_i)_*e\,\iso\, (\pr_j)_!e$ in $CAlg^{nu}(Shv(Y_j), \otimes)$. By Lemma~\ref{Lm_F.2.8} andd
(\cite{G}, I.1, 8.5.7), we get canonically
$$
Shv(Y_i)\otimes_{(Shv(Y), \otimes)} Shv(Y_j)\,\iso\, f_j^*(f_i)_*e-mod(Shv(Y_j)).
$$
Our claim follows.
\end{proof}

The following is straightforward from definitions.
\begin{Lm} 
\label{Lm_the_dual_of_f_*_is_f^!}
Let $h: Z\to Z'$ be a map in $\PreStk_{lft}$, which is ind-schematic of ind-finite type. We have a canonical evalution map $\ev: Shv(Z)\to \Fun(Shv(Z), \Vect)$ sending $K$ to the functor $\RG(Z, K\otimes^!\!\_)$. The diagram commutes
$$
\begin{array}{ccc}
\Fun(Shv(Z'),\Vect) & \getsup{f^{\vee}} & \Fun(Shv(Z),\Vect)\\
\uparrow\lefteqn{\scriptstyle\ev} && \uparrow\lefteqn{\scriptstyle\ev} \\
Shv(Z') & \getsup{f_*} & Shv(Z)
\end{array}
$$
In particular, if $\ev$ is a self-duality on $Shv(Z)$ and on $Shv(Z')$ then $f_*\,\iso\, (f^!)^{\vee}$. \QED
\end{Lm}

\sssec{} Denote by $\bfitDelta_s\subset \bfitDelta$ the subcategory, where we keep all objects and only injective morphisms $[n]\to [m]$ for $n,m\ge 0$. Recall that  $\bfitDelta_s^{op}\to\bfitDelta^{op}$ is cofinal by (\cite{HTT}, 6.5.3.7).

For $i\in I$ write $\Gamma_i: Y_i\to Y_i\times Y$ for the graph of $f_i$. 

\begin{Lm} 
\label{Lm_dual_of_the_relative_tensor_product}
The dual $(Shv(Y_i)\otimes_{Shv(Y)} Shv(Y_j))^{\vee}$ in $\DGCat_{cont}$ identifies canonically with 
$$
Shv(Y_i)\otimes_{(Shv(Y), \otimes)} Shv(Y_j).
$$
\end{Lm}
\begin{proof} The right adjoint to the composition
$$
Shv(Y_i)\otimes Shv(Y)\toup{\boxtimes} Shv(Y_i\times Y)\toup{\Gamma_i^!} Shv(Y_i)
$$
is $\boxtimes^R\comp (\Gamma_i^!)^R$, here $\boxtimes^R\,\iso\, \boxtimes^{\vee}$ is the right adjoint to $\boxtimes$, and $\Gamma_i^R$ is the right adjoint to $\Gamma_i$. Since $\Gamma_i$ is the composition $Y_i\to Y_i\times Y_i\toup{\id\times f_i} Y_i\times Y$, from Remark~\ref{Rem_left_adjoint_to_(id_times_f_i)_*} we see that $(\Gamma_i)_*$ has the left adjoint $\Gamma_i^*$. So, 
$$
(((\Gamma_i)_*)^{\vee},
(\Gamma_i^*)^{\vee})
$$ 
is an adjoint pair in $\DGCat_{cont}$. By Lemma~\ref{Lm_the_dual_of_f_*_is_f^!}, $((\Gamma_i)_*)^{\vee}\,\iso\, \Gamma_i^!$. So, $(\Gamma_i^!)^R\,\iso\, (\Gamma_i^*)^{\vee}$. 

For the product map 
$$
Shv(Y)\otimes Shv(Y)\toup{\boxtimes} Shv(Y\times Y)\toup{\vartriangle^!} Shv(Y)
$$
the right adjoint is $\boxtimes^R\comp (\vartriangle^!)^R$, and $\boxtimes^R\,\iso\, \boxtimes^{\vee}$. Since $(\vartriangle_*)^{\vee}\,\iso\, \vartriangle^!$, we get similarly $(\vartriangle^!)^R\,\iso\, (\vartriangle^*)^{\vee}$. 

 Let $C_{\bfitDelta^{op}}: \bfitDelta^{op}\to \DGCat_{cont}$ be the bar complex defining $Shv(Y_i)\otimes_{Shv(Y)} Shv(Y_j)$, it sends $[n]$ to $Shv(Y_i)\otimes Shv(Y)^{\otimes n} \otimes Shv(Y_j)$. For a map $\alpha: [n]\to [m]$ in $\bfitDelta$ write
$$
h_{\alpha}: Shv(Y_i)\otimes Shv(Y)^{\otimes m} \otimes Shv(Y_j)\to Shv(Y_i)\otimes Shv(Y)^{\otimes n} \otimes Shv(Y_j)
$$
for the corresponding transition map in $C_{\bfitDelta^{op}}$. If $\alpha$ is injective then
$h_{\alpha}$ is given by !-pullbacks. Let $C^{\vee}_{\bfitDelta}: \bfitDelta\to\DGCat_{cont}$ be the diagram obtained from $C_{\bfitDelta^{op}}$ by passing to the duals. Denote by 
$$
C_{\bfitDelta^{op}_s}: \bfitDelta^{op}_s\to \DGCat_{cont}\;\;\; \mbox{and}\;\;\; C^{\vee}_{\bfitDelta_s}: \bfitDelta_s\to\DGCat_{cont}
$$ 
the restrictions of $C_{\bfitDelta^{op}}$ and $C^{\vee}_{\bfitDelta}$ respectively.
 
  Note that if $\alpha: [n]\to [m]$ is a map in $\bfitDelta$, which is injective (resp., not injective) then the right adjoint $h_{\alpha}^R$ of $h_{\alpha}$ is continuous (resp., maybe discontinuous).  Let 
$$
C_{\bfitDelta_s}^R: \bfitDelta_s\to\DGCat_{cont}
$$ 
be the diagram obtained from $C_{\bfitDelta^{op}_s}$ by passing to right adjoints. Note that passing to the right adjoint in the whole of $C_{\bfitDelta^{op}}$ one gets only a functor $C_{\bfitDelta}^R: \bfitDelta\to \DGCat$.
 
 The bar complex defining $Shv(Y_i)\otimes_{(Shv(Y),\otimes)} Shv(Y_j)$ is denoted by 
$$
\tilde C_{\bfitDelta^{op}_s}:  \bfitDelta^{op}_s\to \DGCat_{cont}.
$$ 
We conclude that $C_{\bfitDelta_s}^R$ is obtained from $\tilde C_{\bfitDelta_s^{op}}$ by passing to the duals. 

Applying (\cite{G}, I.1, 6.3.4) to the diagam $C_{\bfitDelta^{op}_s}$ we see that
$$
Shv(Y_i)\otimes_{Shv(Y)} Shv(Y_j)\,\iso\, \colim C_{\bfitDelta^{op}_s}
$$ 
is dualizable, and 
$$
(Shv(Y_i)\otimes_{Shv(Y)} Shv(Y_j))^{\vee}\,\iso\, \lim C_{\bfitDelta_s}^{\vee}.
$$ 
Passing to the left adjoints in $\lim C_{\bfitDelta_s}^{\vee}$ we get 
$$
\lim C_{\bfitDelta_s}^{\vee}\,\iso\, \colim \tilde C_{\bfitDelta_s^{op}}\,\iso\, Shv(Y_i)\otimes_{(Shv(Y),\otimes)} Shv(Y_j).
$$
as desired.
\end{proof}

\sssec{} Note that $Shv(Y_i)\otimes_{Shv(Y)} Shv(Y_j)$ is compactly generated by images of the objects $K\boxtimes K'\in Shv(Y_i)\otimes Shv(Y_j)$ for $K\in Shv(Y_i)^c, K'\in Shv(Y_j)^c$. 

  If $K\in Shv(Y_j)^c, K'\in Shv(Y_j)^c$ then $q_{ij}^!(K\boxtimes K')\in Shv(Y_i\times_Y Y_j)^c$. So, the right adjoint $\cF^R$ of $\cF$ is continuous.

\sssec{} Write $\delta: Shv(Y_i)\otimes Shv(Y_j)\to Shv(Y_i)\otimes_{Shv(Y)} Shv(Y_j)$ and 
$$
\delta_{\otimes}: Shv(Y_i)\otimes Shv(Y_j)\to Shv(Y_i)\otimes_{(Shv(Y),\otimes)} Shv(Y_j)
$$
for the natural functors. Let $\delta^R$ be the right adjoint to $\delta$. From the construction we get $(\delta^R)^{\vee}\,\iso\, \delta_{\otimes}$.

\begin{Lm}
\label{Lm_relation_between_cF_and_cF'}
One has canonically $\cF^{\vee}\,\iso\, (\cF')^R$ and $(\cF')^{\vee}\,\iso\, \cF^R$, where $R$ stands for the right adjoint.
\end{Lm}
\begin{proof} Use the same notations as in the proof of Lemma~\ref{Lm_dual_of_the_relative_tensor_product}. For $n\ge 0$ denote by $f_n$ the composition 
$$
Shv(Y_i)\otimes Shv(Y))^{\otimes n} \otimes Shv(Y_j)\toup{ins_n} Shv(Y_j)\otimes_{Shv(Y)} Shv(Y_j)\toup{\cF} Shv(Y_i\times_Y Y_j),
$$
where $ins_n$ is a structure morphism for $\colim C_{\bfitDelta^{op}}$. Recall that $C^R_{\bfitDelta_s}: \bfitDelta_s\to\DGCat_{cont}$ is obtained from $C_{\bfitDelta_s^{op}}$ by passing to right adjoints, so the right adjoints $ins_n^R$ and $f_n^R$ are continuous. 

 The map $\cF^R: Shv(Y_i\times_Y Y_j)\to \lim C^R_{\bfitDelta_s}$ in $\DGCat_{cont}$
 is obtained from the compatible system of functors $f_n^R$ for $[n]\in \bfitDelta_s$ by passing to the limit over $\bfitDelta_s$.
   
 As in the proof of Lemma~\ref{Lm_dual_of_the_relative_tensor_product}, $Shv(Y_i)\otimes_{(Shv(Y),\otimes)} Shv(Y_j)\,\iso\, \lim C^{\vee}_{\bfitDelta_s}$. The functor
$$
\cF^{\vee}: Shv(Y_i\times_Y Y_j)\to\lim C^{\vee}_{\bfitDelta_s}
$$
is obtained from the compatible family of functors $f_n^{\vee}$ for $[n]\in \bfitDelta_s$ by passing to the limit over $\bfitDelta_s$.

 We have the adjoint pairs $((f_n^R)^{\vee}, f_n^{\vee})$ and $((\cF^R)^{\vee}, \cF^{\vee})$ in $\DGCat_{cont}$. Let
$$
(\cF^{\vee})^L: \colim \tilde C_{\bfitDelta_s^{op}}\to Shv(Y_i\times_Y Y_j) 
$$
be the left adjoint to $\cF^{\vee}$. It is obtained by passing to the colimit over $\bfitDelta_s^{op}$ in the compatible system of functors $(f_n^{\vee})^L$. 

 For $n\ge 0$ denote by $\tilde f_n$ the composition
$$
Shv(Y_i)\otimes Shv(Y)^{\otimes n}\otimes Shv(Y_j)\toup{\wt{ins}_n} Shv(Y_i)\otimes_{(Shv(Y),\otimes)} Shv(Y_j)\toup{\cF'} Shv(Y_i\times_Y Y_j),
$$
where $\wt{ins}_n$ is the structure map for $\tilde C_{\bfitDelta_s^{op}}$. Now it suffices to show that for $n\ge 0$ we have canonically $(f_n^{\vee})^L\,\iso\, \tilde f_n$ in a way compatible with the transition maps in $\tilde C_{\bfitDelta_s^{op}}$. 

  Denote by $\tau_n$ the projection
$$
Y_i\times_Y Y_j\,\iso\, Y\times_{Y^{n+2}}(Y_i\times Y^n\times Y_j)\to Y_i\times Y^n\times Y_j.
$$
Then $f_n$ is the composition 
$$
Shv(Y_i)\otimes Shv(Y)^{\otimes n}\otimes Shv(Y_j)\toup{\boxtimes} Shv(Y_i\times Y^n\times Y_j)\toup{\tau_n^!} Shv(Y_i\times_Y Y_j).
$$
So, $f_n^R\,\iso\, \boxtimes^R\comp(\tau_n^!)^R$.  As in Lemma~\ref{Lm_existence_of_vartriangle^*}, $(\tau_n)_*$ admits a left adjoint $\tau_n^*$. 

 We have $((\tau_n)_*)^{\vee}\,\iso\, \tau_n^!$ as in Lemma~\ref{Lm_the_dual_of_f_*_is_f^!}, so $(\tau_n^!)^R\,\iso\, (\tau_n^*)^{\vee}$. This gives 
$$
(f_n^{\vee})^L\,\iso\,(f_n^R)^{\vee}\,\iso\, \tau_n^*\comp\boxtimes\,\iso\, \tilde f_n
$$
as desired. 
\end{proof}

\sssec{Proof of Theorem~\ref{Thm_Kunneth_formula}} 
\label{Sect_F.2.17}
By Lemma~\ref{Lm_relation_between_cF_and_cF'}, $\cF^R\,\iso\, (\cF')^{\vee}$. Since $\cF'$ is an equivalence by Lemma~\ref{Lm_functor_cF'_is_equivalence}, $\cF^R$ is an equivalence. \QED

\sssec{} We generalize Lemma~\ref{Lm_functor_cF'_is_equivalence} as follows. Pick $i, j\in J$, let us be given a map $\beta: S\to Y_j$ with $S\in\Sch_{ft}$. View $Shv(S)$ as an object of $(Shv(Y),\otimes)-mod$ via the non-unital symmetric monoidal functor $\beta^*f_j^*: (Shv(Y),\otimes)\to (Shv(S), \otimes)$. Viewing also $Shv(Y_i)$ as an object of $Shv(Y)-mod$ via the functor $f_i^*$, we form the tensor product
$Shv(Y_i)\otimes_{(Shv(Y),\otimes)} Shv(S)$. Note that $Y_i\times_Y S\in \Sch_{ft}$, it is obtained from $Y_i\times_Y Y_j$ by the base change $S\to Y_j$. Consider the natural functor
\begin{equation}
\label{map_cF'_generalized}
Shv(Y_i)\otimes_{(Shv(Y),\otimes)} Shv(S)\to Shv(Y_i\times_Y S)
\end{equation}
coming from $Shv(Y_i)\otimes Shv(S)\to Shv(Y_i\times_Y S)$, $K\otimes K'\mapsto g_{Y_i}^*K\otimes g_S^*K'$ via the universal property. Here the maps are given by the diagram of projections
$$
Y_i\getsup{g_{Y_i}} Y_i\times_Y S \toup{g_S} S.
$$
Note that $g_S$ is a finite morphism.

\begin{Pp} 
\label{Pp_F.2.19}
The functor (\ref{map_cF'_generalized}) is an equivalence.
\end{Pp}
\begin{proof} Since $g_S$ is finite, the functor $(g_S)_!: Shv(Y_i\times_Y S)\to Shv(S)$ is monadic. The corresponding monad is $K\mapsto ((g_S)_!e)\otimes K$ for $K\in Shv(S)$. 

By (\cite{G}, I.1, 8.5.7), we get canonically
$$
Shv(Y_i)\otimes_{(Shv(Y),\otimes)} Shv(S)\,\iso\, \beta^*f_j^*(f_i)_*e-mod(Shv(S)).
$$
By construction, $\beta^*f_j^*(f_i)_*e\,\iso\, (g_S)_!e$ in $CAlg^{nu}(Shv(S))$. Our claim follows. 
\end{proof}

\section{Relatively rigid monoidal categories}
\label{Sect_Appendix_G}

In (\cite{G}, ch. I.1, 9.1.2) for an object of $Alg(\DGCat_{cont})$ a notion of a rigid monoidal category was introduced. In this section we study a relative version of this notion over some base. 

\ssec{Basic properties of relatively rigid categories}
\sssec{} 
\label{Sect_1.0.1G}
Let $A\in CAlg(\DGCat_{cont})$. Let $B\in Alg(A-mod)$, assume that the unit map $u: A\to B$ maps to the center of $B$, that is, we are given an identidication between $A$-actions on $B$ by left and right translations. Write $u^R$ for the right adjoint of $u$. Write $B^{left-dualiz}$ and $B^{right-dualiz}$ for the full subcategories of $B$ spanned by left dualizable and right dualizable objects respectively.

We want to study the following generalization of a notion of a rigid monoidal category.

\begin{Def} 
\label{Def_relatively_rigid_G}
Let $B$ be as in Section~\ref{Sect_1.0.1G}. Say it is rigid relative to $A$ is the following properties hold:
\begin{itemize}
\item The unit map $u: A\to B$ has a continuous $A$-linear right adjoint $u^R$;
\item The product map $m: B\otimes_AB\to B$ has a continuous $B\otimes_AB^{rm}$-linear right adjoint $m^R: B\to B\otimes_AB$.
\end{itemize}
\end{Def} 

\index{Relatively rigid monoidal category}

\begin{Rem} i) If $A=\Vect$ then the above notion becomes precisely the notion of a rigid monoidal category as an object of $\Alg(\DGCat_{cont})$.

\medskip\noindent
ii) If $B$ as above is rigid relative to $A$ then $B$ is not necessarily rigid as an object of $\DGCat_{cont}$. Indeed, $B=A$ is always rigid relative to $A$.

 For example, if $S$ is a scheme of finite type over an algebraically closed field $k$ then $(Shv(S),\otimes^!)$ in the constructible context or for $\cD$-modules is not rigid. 
 
\smallskip\noindent
iii) $B$ is rigid relative to $A$ iff $B^{rm}$ is rigid relative to $A$. 
\end{Rem} 

\sssec{} For $M\in A-mod$ write $M^c_A\subset M$ for the full subcategory of objects compact relative to $A$ as defined in (\cite{G}, ch. I.1, 8.8.2). So, $m\in M^c_A$ iff the functor $M\to A, x\mapsto \und{\HOM}_A(m, x)$ is continious. Here for $m, x\in M$ we denote by $\und{\HOM}_A(m, x)\in A$ the inner hom with respect to the $A$-action. 

 Recall that $m\in M$ is ULA over $A$ if $m\in M^c_A$ and moreover the right-lax $A$-structure on the functor $\und{\HOM}_A(m,\_): M\to A$ is strict. That is, for $a\in A, x\in M$ the natural map
$$
a\otimes \und{\HOM}_A(m, x)\to \und{\HOM}_A(m, a\otimes x)
$$
is an isomorphism. We write $M^{ULA}\subset M$ for the full subcategory of ULA objects over $A$. So, $M^{ULA}\subset M^c_A$. Recall that $M^{ULA}\subset M$ is a stable subcategory closed under the action of $A^{dualiz}$. Here $A^{dualiz}\subset A$ is the full subcategory of dualizable objects. 

\sssec{} 
\label{Sect_1.0.6G}
Let $B$ be as in Section~\ref{Sect_1.0.1G}, $b\in B$. For $b, b'\in B$ write $\und{\HOM}_B(b, b')\in B$ for the inner hom in $B$, it exists because $B$ is presentable. Note that $b\in B$ is left-dualizable iff the functor $\und{\HOM}_B(b, \_): B\to B$ is a strict morphism of left $B$-modules (with respect to the left multiplication). That is, for any $y\in B$ the natural map
$$
y\otimes \und{\HOM}_B(b, 1)\to \und{\HOM}_B(b, y)
$$
is an isomorphism. In the latter case $b^{\vee, L}=\und{\HOM}_B(b, 1)$ is the left dual of $b$. 

 For $b\in B$ we denote by $\und{\HOM}_A(b,\_): B\to A$ the inner hom with respect to the $A$-action. We have $B^{ULA}\subset B^c_A\subset B$. Note that $u^R=\und{\HOM}_A(1_B, \_)$. 

\sssec{} 
\label{Sect_1.0.7G}
For $b,b'\in B$ we get the relative inner hom $\und{\HOM}_{B^{rm}}(b,b')\in B$. It is characterized by the functorial isomorphisms for $x\in B$
$$
\HOM(x, \und{\HOM}_{B^{rm}}(b,b'))\,\iso\,\HOM(b\otimes x, b')
$$
in $\Vect$. We have an adjoint pair $\cL^{rm}: B\leftrightarrows B: \und{\HOM}_{B^{rm}}(b,\_)$, where $\cL^{rm}(x)=b x$. The functor $\cL^{rm}$ in $B^{rm}$-linear, so $\und{\HOM}_{B^{rm}}(b,\_)$ has a right lax $B^{rm}$-structure. It is given by a canonical map for $x, b'\in B$
$$
\und{\HOM}_{B^{rm}}(b, x)\otimes b'\to \und{\HOM}_{B^{rm}}(b, x\otimes b')
$$ 
As in Section~\ref{Sect_1.0.6G}, $b\in B$ is right-dualizable iff the functor $\und{\HOM}_{B^{rm}}(b,\_)$ is a strict morphism of $B^{rm}$-modules, that is, the canonical map
$$
\und{\HOM}_{B^{rm}}(b, 1)\otimes b'\to \und{\HOM}_{B^{rm}}(b, b')
$$ 
is an isomorphism in $B$. In the latter case $b^{\vee, R}=\und{\HOM}_{B^{rm}}(b, 1)$. 

\begin{Lm} 
\label{Lm_1.0.8G}
Let $b,y\in B$ then one has canonically $u^R\und{\HOM}_B(b, y)\,\iso\,\und{\HOM}_A(b, y)$.  
\end{Lm}
\begin{proof}
For $a\in A$ one has 
\begin{multline*}
\HOM_A(a, \und{\HOM}_A(b, y))\,\iso\,
\HOM_B(u(a)b, y)\,\iso\, 
\HOM_B(u(a), \und{\HOM}_B(b, y))\\ 
\iso\, \HOM_A(a, u^R\und{\HOM}_B(b, y)).
\end{multline*}
\end{proof}

\begin{Lm} Let $B$ be rigid relative to $A$. Then 
\begin{equation}
\label{iso_for_Lm_1.0.8G}
B^{ULA}=B^{right-dualiz}\cap B^{left-dualiz}=B^{left-dualiz}
\end{equation}
In fact, $B^{left-dualiz}=B^{right-dualiz}$.
\end{Lm}
\begin{proof} 
{\bf Step 1}. First, we show the inclusion $B^{ULA}\subset B^{right-dual}\cap B^{left-dual}$. Let $b\in B^{ULA}$.

\smallskip
\noindent
i) First, we show that $b$ is left dualizable. Let $L: A\to B$ be the functor $a\mapsto a\otimes b$, its right adjoint is $\und{\HOM}_A(b,\_)$. Our assumption gives an adjoint pair $L: A\leftrightarrows B: \und{\HOM}_A(b,\_)$ in $A-mod$. Tensoring by the right $A$-module $B$, we get an adjoint pair
$$
\id\otimes L: B\otimes_A A\leftrightarrows B\otimes_A B: \id\otimes\und{\HOM}_A(b,\_)
$$
in $B-mod$, where $B$ acts by left multiplication. Here $\id\otimes L: B\to B\otimes_A B$ is $b'\mapsto b'\otimes b$. So, the functor $m\comp (\id\otimes L): B\to B, b'\mapsto b'b$ has the right adjoint $\und{\HOM}_B(b, \_)$ given by the composition 
$$
B\toup{m^R} B\otimes_A B\toup{\id\otimes \und{\HOM}_A(b,\_)} B
$$
The latter composition is continuous and $B$-linear with respect to the $B$-action by left multiplication. Thus, $b$ is left-dualizable in $B$ by Section~\ref{Sect_1.0.6G}. 

\medskip\noindent
ii) Let us show that $b$ is right dualizable. Tensoring the adjoint pair $L: A\leftrightarrows B: \und{\HOM}_A(b,\_)$ by the left $A$-module $B$, we get an adjoint pair
$$
L\otimes\id: A\otimes_A B\leftrightarrows B\otimes_A B: \und{\HOM}_A(b,\_)\otimes\id
$$
in $B^{rm}-mod$, where $B$ acts by right multiplications. Here $L\otimes\id: B\to B\otimes_A B$ is $b'\mapsto b\otimes b'$. So, the functor $m\comp (L\otimes\id): B\to B, b'\mapsto bb'$ has the right adjoint $\und{\HOM}_{B^{rm}}(b,\_)$ given by the composition
$$ 
B\toup{m^R} B\otimes_A B \toup{\und{\HOM}_A(b,\_)\otimes\id} B
$$
Since $m^R$ is $B^{rm}$-linear, we see that $\und{\HOM}_{B^{rm}}(b,\_)$ in $B^{rm}$-linear. By Section~\ref{Sect_1.0.7G}, $b$ is right dualizable.

\smallskip\noindent
{\bf Step 2}. Let $b\in B^{left-dualiz}$. We claim that $b\in B^{ULA}$. Indeed, by Lemma~\ref{Lm_1.0.8G}, we have for $y\in B$, $\und{\HOM}_A(b, y)\,\iso\, u^R(y\otimes b^{\vee, L})$. Since $u^R$ is continuous and $A$-linear, the functor $\und{\HOM}_A(b,\_): B\to A$ is continuous and $A$-linear. This gives the equality (\ref{iso_for_Lm_1.0.8G}). In particular, $B^{left-dualiz}\subset B^{right-dualiz}$.

\smallskip\noindent
{\bf Step 3} Recall that $B^{rm}$ is rigid relative to $A$. We get $B^{right-dualiz}\subset B^{left-dualiz}$ applying (\ref{iso_for_Lm_1.0.8G}) to $B^{rm}$.
\end{proof}

\begin{Pp} 
\label{Pp_1.0.9G}
Let $B$ be rigid relative to $A$. Then the maps
$$
\epsilon: B\otimes_A B\toup{m} B\toup{u^R} A,\;\;\; \mu: A\toup{u} B\toup{m^R} B\otimes_A B
$$
provide a counit and unit of a duality datum in $A-mod$, hence an isomorphism
$$
\phi_B: B\,\iso\, B^{\vee},
$$
where $B^{\vee}=\Fun_A(B, A)$ is the dual of $B$ in $A-mod$.
\end{Pp}
\begin{proof}
It is precisely the same as for usual rigid categories in (\cite{G}, ch. I.1, 9.2). It uses the fact that for $1_B\in B$ its dual in $B$ is $1_B^{\vee, R}\,\iso\, 1$, and similarly for $1_B^{\vee, L}$.
\end{proof}

\sssec{} Our normalization is that $\phi_B$ sends $a\in B$ to the functor $B\to A, b\mapsto \und{\HOM}_A(1_B, ba)$. With this normalization, $\phi_B$ is a morphism of \select{left} $B$-modules. The left $B$-action on $B^{\vee}$ is such that $b'\in B$ sends $L\in \Fun_A(B, A)$ to the functor $b'L: B\to A, b\mapsto L(bb')$. Our $\phi_B$ is not a morphism of right $B$-modules in general. 

 The map $\phi_B^{-1}: \Fun_A(B, A)\to B$ sends $L$ to the object $(\id\otimes L)m^R(1)$, here $\id\otimes L: B\otimes_A B\to B$. 

\begin{Pp}  Let $B$ be rigid relative to $A$. The diagram commutes
$$
\begin{array}{ccc}
B^{\vee} & \toup{m^{\vee}} & B^{\vee}\otimes_A B^{\vee}\\
\uparrow\lefteqn{\scriptstyle \phi_B} && \uparrow\lefteqn{\scriptstyle \phi_B\otimes\phi_B}\\
B & \toup{m^R} & B\otimes_A B
\end{array}
$$
\end{Pp}
\begin{proof}
The same as (\cite{G}, ch. I.1, 9.2.6). 
\end{proof}

\begin{Pp} 
\label{Pp_1.0.12_strict_morphismG}
Let $M, N\in B-mod$, let $f: M\to N$ be a morphism in $A-mod$. Assume that this structure is lifted to a right-lax (resp., left-lax) morphism $f: M\to N$ between $B$-module categories. Then $f$ is a strict morphism of $B$-module categories.
\end{Pp}
\begin{proof}
The proof is almost the same as (\cite{G}, ch. I.1, 9.3.6), we give the details as in \select{loc.cit.} some of them are left to a reader. Assume $f$ is right-lax morphism of $B$-module categories. Write $\act: B\otimes_A M\to M, \act: B\otimes_A N\to N$ for the action maps over $A$. We need to show that the natural transformation from
\begin{equation}
\label{LHS_transform_1.0.12G}
B\otimes_A M \toup{\id\otimes f} B\otimes_A N\toup{\act} N
\end{equation}
to
\begin{equation}
\label{RHS_transform_1.0.12G}
B\otimes_A M\toup{\act} M\toup{f} N
\end{equation}
is an equivalence. We construct an explicit inverse natural transformation as follows.

Consider two more functors $B\otimes_A M\to N$. One is
\begin{multline}
\label{functor_one_1.0.12G}
B\otimes_A M\toup{m^R(1)\otimes\id} B\otimes_A B\otimes_A B\otimes_A M\toup{\id\otimes\act} B\otimes_A B\otimes_A M\toup{\id\otimes f} \\ B\otimes_A B\otimes_A  N\toup{\id\otimes\act}B\otimes_A N \toup{\otimes\act}N
\end{multline}
The other is
\begin{multline}
\label{functor_two_1.0.12G}
B\otimes_A M\toup{m^R(1)\otimes\id} B\otimes_A B\otimes_A B\otimes_A M\toup{\id\otimes\act} B\otimes_A B\otimes_A M\toup{\id\otimes\act} \\ B\otimes_A M\toup{\id\otimes f}  A\otimes N\toup{\act} N
\end{multline}
The desired natural transformation is a composition $(\ref{RHS_transform_1.0.12G})\to (\ref{functor_one_1.0.12G})\to (\ref{functor_two_1.0.12G})\to (\ref{LHS_transform_1.0.12G})$ of natural transformations.

The dual pair $m: B\otimes_A B\leftrightarrows B: m^R$ gives a diagram $1\toup{unit} m m^R(1)\toup{counit} 1$, whose composition is $\id$. For $b'\in B, z\in N$ denote by $b'z$ the action of $b'$ on $z$, and similarly for the $B$-action on $M$. The map $unit$ gives a natural transformation in $N$
$$
f(bx)\to (mm^R(1))f(bx)=\act(\id\otimes\act)(\id\otimes f)(m^R(1)\otimes bx)
$$
functorial in $b\in B, x\in M$, this is a natural transformation from (\ref{RHS_transform_1.0.12G}) to (\ref{functor_one_1.0.12G}).

 Next, we have a natural transformation in $N$
$$
\act(\id\otimes\act)(\id\otimes f)(m^R(1)\otimes bx)\to \act(\id\otimes f)(\id\otimes\act)(m^R(1)\otimes bx)
$$
functorial in $b\in B, x\in M$ and coming from the right-lax structure on $f$. This is a natural transformation from (\ref{functor_one_1.0.12G}) to  (\ref{functor_two_1.0.12G}).

 The construction of a natural transformation from (\ref{functor_two_1.0.12G}) to 
(\ref{LHS_transform_1.0.12G}) uses the fact that $m^R$ is a morphism of $B$-bimodule categories. Namely, 
$$
(\id\otimes\act)(m^R(1)\otimes bx)\,\iso\,(\id\otimes\act)(m^R(1)\cdot (1\otimes b))\otimes x\,\iso\, (b\otimes \id)\cdot(\id\otimes\act)(m^R(1)\otimes x),
$$
in $B\otimes_A M$, where $\cdot$ stands for the products in $B\otimes_A B$ and in $B$. Thus, we need to construct a morphism in $N$
$$
\act(\id\otimes f)(b\otimes \id)\cdot(\id\otimes\act)(m^R(1)\otimes x)\to bf(x)
$$
In the LHS we may put $(b\otimes\id)$ outside, hence it remains to construct a morphism
$$
\act(\id\otimes f)(\id\otimes\act)(m^R(1)\otimes x)\to f(x)
$$
in $N$. It is the composition
$$
\act(\id\otimes f)(\id\otimes\act)(m^R(1)\otimes x)\to f\act(mm^R(1)\otimes x)\to f(x),
$$
where the first map comes from the right-lax structure on $f$, and the second one from $counit: mm^R(1)\to 1$.
\end{proof}

\begin{Lm} Let $B_1, B_2$ be rigid relative to $A$. Let $f: B_1\to B_2$ be a map in $Alg(A-mod)$. Then its right adjoint $f^R: B_2\to B_1$ is a map in $A-mod$, and the diagram commutes
$$
\begin{array}{ccc}
B_2^{\vee} & \toup{f^{\vee}} & B_1^{\vee}\\
\uparrow\lefteqn{\scriptstyle \phi_{B_2}} && \uparrow\lefteqn{\scriptstyle \phi_{B_1}}\\
B_2 & \toup{f^R} & B_1.
\end{array}
$$
\end{Lm}
\begin{proof}
By Proposition~\ref{Pp_1.0.12_strict_morphismG}, $f^R$ is a strict morphism of $B_1$-module categories. Since $f$ is a map of right $B_1$-module categories, $f^{\vee}$ is a map in $B_1-mod$ also. Let $b_2\in B_2, b_1\in B_1$ the functor $f^{\vee}\phi_{B_2}(b_2): B_1\to A$ is $b_1\mapsto \und{\HOM}_A(1_{B_2}, f(b_1)b_2)$. The functor $\phi_{B_1}f^R(b_2): B_1\to A$ is $\und{\HOM}_A(1_{B_1}, b_1f^R(b_2)$. We must establish an isomorphism in $A$
$$
\und{\HOM}_A(1_{B_1}, b_1f^R(b_2))\,\iso\, \und{\HOM}_A(1_{B_2}, f(b_1)b_2)
$$
For $a\in A$ one has 
$$
\HOM_{B_1}(a\cdot 1_{B_1}, b_1f^R(b_2))\,\iso\, \HOM_{B_2}(a\cdot 1_{B_2}, f(b_1)b_2)
$$
in $\Vect$ by adjointness, as desired. Here $\HOM$ stands for the inner hom with respect to the $\Vect$-action.
\end{proof}

\ssec{Modules over relatively rigid categories}
\sssec{} Let $B$ be rigid relative to $A$, and $M\in B-mod$. Let $\act: B\otimes_A M\to M$ be the action map.

\begin{Lm} The functor $\act: B\otimes_A M\to M$ admits a continuous right adjoint given by
\begin{equation}
\label{act^R_for_Lm_1.1.2G}
M\toup{\mu\otimes\id} B\otimes_A B\otimes_A M\toup{\id\otimes\act} B\otimes_A M.
\end{equation}
Here $\mu$ is given in Proposition~\ref{Pp_1.0.9G}. In fact, (\ref{act^R_for_Lm_1.1.2G}) is a map in $B-mod$. 
\end{Lm}
\begin{proof} The proof of (\cite{G}, ch. I.1, 9.3.2) applies in our situation. Let us indicate the corresponding unit and counit. The composition
$$
M\toup{(\ref{act^R_for_Lm_1.1.2G})} B\otimes_A M\toup{\act} M
$$
identifies with the multiplication map by $mm^R(1)\in B$. So, $counit: mm^R(1)\to 1$ gives the desired counit natural transformation $\act \comp (\ref{act^R_for_Lm_1.1.2G})\to \id$.

 The composition
$$
B\otimes_A M\toup{\act} M\toup{(\ref{act^R_for_Lm_1.1.2G})} B\otimes_A M
$$  
is $B$-linear and sends $1_B\otimes x\in B\otimes_A M$ to $(\id\otimes\act)(m^R(1_B)\otimes x)$. The map $1_B\otimes 1_B\to m^B(1_B)$ yields the desired unit natural transformation $\id\to (\ref{act^R_for_Lm_1.1.2G})\comp\act$.

 Since (\ref{act^R_for_Lm_1.1.2G}) is $A$-linear, its right-lax $B$-structure is strict by Proposition~\ref{Pp_1.0.12_strict_morphismG}.  
\end{proof}

\begin{Lm} Let $M\in B-mod$, $N\in B^{rm}-mod$. Then the canonical arrow $N\otimes_A M\to N\otimes_B M$ has a right adjoint in $A-mod$, in particular it is continuous.
\end{Lm}
\begin{proof} Consider the functor $\cF: \bfitDelta^{op}\to A-mod$, $[n]\mapsto N\otimes_A (B^{\otimes n}_A)\otimes_A M$, here $B^{\otimes n}_A$ is the $n$-th tensor power of $B$ in $A-mod$. Since all the maps $m: B\otimes_A B\to B$, $N\otimes_A B\to B$, $B\otimes_A M\to M$ admit right adjoints in $A-mod$, we may pass to right adjoints in $\cF$ and get $\cF^R: \bfitDelta\to A-mod$. So, $ins_0: N\otimes_A M\to N\otimes_B M$ has a right adjoint in $A-mod$ given by $\ev_0: N\otimes_B M\to N\otimes_A M$.
\end{proof}

\begin{Rem} i) One has $M^c_B\subset M^c_A$. Indeed, for $m, y\in M$ one has canonically $\und{\HOM}_A(m, y)\,\iso\, u^R\und{\HOM}_B(m, y)$ in $A$.

\medskip\noindent
ii) If $B_1, B_2$ are rigid relative to $A$ then $B_1\otimes_A B_2$ is also rigid relative to $A$. 

\medskip\noindent
iii) Let $M$ be left dualizable as a left $B$-module category and $M^{\vee, B}\in B^{rm}-mod$ be the left dual of $M$ in the sense of (\cite{Chen}, A.2.1). Then $M$ is dualizable in $A-mod$, and $M^{\vee, A}\,\iso\, M^{\vee, B}$ by (\cite{Chen}, A.3.8).
\end{Rem}

\begin{Lm} 
\label{Lm_1.2.5_rigidityG}
Assume $B\in CAlg(A-mod)$, and $B$ rigid relative to $A$. Let $M\in B-mod$. Assume $M$ is dualizable in $A-mod$, let $M^{\vee, A}$ be its dual in $A-mod$. Then $M$ is dualizable in $B-mod$ and $\Fun_B(M, B)\,\iso\, M^{\vee, A}$ canonically in $B-mod$. 
\end{Lm}
\begin{proof}
This is analogous to (\cite{G}, ch. I.1, 9.4.4). Note that $\phi_B: B\,\iso\, B^{\vee}$ is an isomorphism of $B$-bimodules under our assumption. Let $E\in B-mod$. Note that $M^{\vee, A}=\Fun_A(M, A)$ is naturally a right $B$-module. 

We claim that there is a canonical isomorphism 
$$
\Fun_B(M, E)\,\iso M^{\vee, A}\otimes_B E.
$$ 
Indeed, $\Fun_B(M, E)\,\iso\, \underset{[n]\in\bfitDelta}{\lim} \Fun_A(B^{\otimes n}_A\otimes_A M, E)$, where we have denoted by $B^{\otimes n}_A$ the $n$-th tensor power of $B$ in $A-mod$. In turn, for $[n]\in\bfitDelta$, 
$$
\Fun_A(B^{\otimes n}_A\otimes_A M, E)\,\iso\, M^{\vee, A}\otimes_A (B^{\vee})^{\otimes n}_A\otimes_A E
$$
Each transition functor in the diagram
$$
\underset{[n]\in\bfitDelta}{\lim} M^{\vee, A}\otimes_A (B^{\vee})^{\otimes n}_A\otimes_A E
$$
admits a left adjoint, so the latter limit identifies with 
$$
\underset{[n]\in\bfitDelta^{op}}{\colim}  M^{\vee, A}\otimes_A B^{\otimes n}_A\otimes_A E\,\iso\, M^{\vee, A}\otimes_B E.
$$
as desired. Our claim follows.
\end{proof}

\begin{Cor} Assume $B\in CAlg(A-mod)$, and $B$ rigid relative to $A$. 

\noindent
i) $B$ is dualizable in $B\otimes_A B^{rm}$-modules, and its dual in $B\otimes_A B^{rm}-mod$ identifies with $B$.\\
ii) For $M\in B-mod$, $N\in B^{rm}-mod$ one has $\Fun_{B\otimes_A B^{rm}}(B, M\otimes_A N)\,\iso\, N\otimes_B M$ canonically.
\end{Cor}
\begin{proof}
i) Combine Proposition~\ref{Pp_1.0.9G} with Lemma~\ref{Lm_1.2.5_rigidityG}.

\smallskip\noindent
ii) One has canonically $B\otimes_{B\otimes_A B^{rm}}(M\otimes_A N)\,\iso\, N\otimes_B M$.
\end{proof}

\printindex


\begin{thebibliography}{99}
\bibitem{GLC2} D. Arinkin, D. Beraldo, J. Campbell, L. Chen, J. Faergeman, D. Gaitsgory, K. Lin, S. Raskin, N. Rozenblyum, Proof of the geometric Langlands conjecture II: Kac-Moody localization and the FLE, arXiv:2405.03648
\bibitem{GLC4} D. Arinkin, D. Beraldo, L. Chen, J. Faergeman, D. Gaitsgory, K. Lin, S. Raskin, N. Rozenblyum, Proof of the geometric Langlands conjecture IV: ambidexterity, arXiv:2409.08670 
\bibitem{AGKRRV} D. Arinkin, D. Gaitsgory, D. Kazhdan, S. Raskin, N. Rozenblyum, Y. Varshavsky, The stack of local systems with restricted variation and geometric Langlands theory with nilpotent singular support, arXiv:2010.01906
\bibitem{BBD} A. Beilinson, J. Bernstein, P. Deligne,  Faisceaux pervers, Asterisque vol. 100, Paris, Soci\'et\'e math\'ematique de France, 1982, DOI : 10.24033/ast.1042
\bibitem{BD_chiral} A. Beilinson, V. Drinfeld, Chiral algebras, AMS Colloquium publications, vol. 51 (2004), Asterisque 100, {\tt https://doi.org/10.1090/coll/051}
\bibitem{GCL3} Justin Campbell, Lin Chen, Joakim Faergeman, Dennis Gaitsgory, Kevin Lin, Sam Raskin, Nick Rozenblyum, Proof of the geometric Langlands conjecture III: compatibility with parabolic induction,  	arXiv:2409.07051
\bibitem{Chen} L. Chen, Nearby cycles on Drinfeld - Gaitsgory - Vinberg interpolation Grassmannian and long intertwining functor, Duke Math. J. 172(3), 447 - 543 (2023)
\bibitem{DL} G. Dhillon, S. Lysenko, Semi-infinite parabolic $\IC$-sheaf, arXiv:2310.0638
\bibitem{DL2} G. Dhillon, S. Lysenko, Semi-infinite parabolic $\IC$-sheaf II: the Ran space version, arXiv:2508.01527
\bibitem{DL3} G. Dhillon, S. Lysenko, Some residues of the geometric Eisenstein series, in preparation
\bibitem{FG} J. Francis, D. Gaitsgory, Chiral Koszul diality, Selecta Mathematica, 
Vol. 18, 27 - 87, (2012)
\bibitem{Ga1} D. Gaitsgory, Sheaves of categories and the notion of 1-affineness, Contemporary Mathematics 643 (2015), 127 - 226, {\tt http://dx.doi.org/10.1090/conm/643/12899}
\bibitem{Gai19SI} D. Gaitsgory, The semi-infinite intersection cohomoogy sheaf, arxiv version 6
\bibitem{Gai19Ran} D. Gaitsgory, The semi-infinite intersection cohomology sheaf-II: the Ran space version, In: Baranovsky, V., Guay, N., Schedler, T. (eds) Representation Theory and Algebraic Geometry. Trends in Mathematics. Birkh\"auser, {\tt https://doi.org/10.1007/978-3-030-82007-7\_6}
\bibitem{Ga} D. Gaitsgory, The Atiyah-Bott formula for the cohomology of the moduli space of bundles on a curve, arXiv:1505.02331 version 2
\bibitem{Gai_de_Jong} D. Gaitsgory, On De Jong's conjecture, Israel J. Math. 157 (2007), 155 - 191,\\  {\tt https://doi.org/10.1007/s11856-006-0006-2}
\bibitem{Ga_Eis_and_Quantum_groups} D. Gaitsgory, Eisenstein series and quantum groups, arXiv:1505.02329
\bibitem{Ga_central} D. Gaitsgory, Central elements in the affine Hecke algebra via nearby cycles, Invent. Math. 144 (2001), p. 253 - 280, {\tt https://doi.org/10.1007/s002220100122}
\bibitem{GLC1} D. Gaitsgory, S. Raskin, Proof of the geometric Langlands conjecture I: construction of the functor, arXiv:2405.03599 
\bibitem{GLC5} D. Gaitsgory, S. Raskin, Proof of the geometric Langlands conjecture V: the multiplicity one theorem, arXiv:2409.09856 
\bibitem{Ga_contractibility_Ran} D. Gaitsgory, Contractibility of the space of rational maps, Inv. Math. 191, p. 91 - 196, (2013) {\tt https://doi.org/10.1007/s00222-012-0392-5}
\bibitem{GLurie} D. Gaitsgory, J. Lurie, Weil's Conjecture for Function Fields: Volume I (AMS-199). Princeton University Press (2019)
\bibitem{GLys} D. Gaitsgory, S. Lysenko, Parameters and duality for the metaplectic geometric Langlands theory, Selecta Math., (2018) Vol. 24, Issue 1, 227-301, {\tt https://doi.org/10.1007/s00029-017-0360-4} Corrected version: of Dec 21, 2022 on arXiv:1608.00284 version 7.
\bibitem{GLys2} D. Gaitsgory, S. Lysenko, Metaplectic Whittaker category and quantum groups : the "small" FLE, arXiv:1903.02279, version April 21, 2020
\bibitem{G} D. Gaitsgory, N. Rozenblyum, A study in derived algebraic geometry, book version available at {\tt https://people.mpim-bonn.mpg.de/gaitsgde/Book}
\bibitem{GRV} D. Gaitsgory, N. Rozenblyum, Y. Varshavsky,
Applications of (higher) categorical trace I: the definition of $\AGCat$,  	arXiv:2512.21726
\bibitem{HR} J. Hall, D. Rydh, Algebraic groups and compact generation of their derived categories of representations, Indiana University Mathematics Journal
Vol. 64, No. 6 (2015), 1903 - 1923, DOI:10.1512/iumj.2015.64.5719
\bibitem{HTT} J. Lurie, Higher topos theory, Annals of Mathematics Studies, Nu. 170, 
Princeton University Press (2009), DOI: 10.1515/9781400830558
\bibitem{HA} J. Lurie, Higher algebra, version Sept. 18, 2017 available at {\tt https://www.math.ias.edu/~lurie}
\bibitem{Ly} S. Lysenko,  Comments to Rozenblyum, Gaitsgory, A Study in Derived Algebraic Geometry, available at {\tt https://lysenko.perso.math.cnrs.fr/index\_sub.html}
\bibitem{Ly3} S. Lysenko, Comments to D. Gaitsgory, S. Lysenko,  Parameters and duality for the metaplectic geometric Langlands theory, available at {\tt https://lysenko.perso.math.cnrs.fr/index\_sub.html}
\bibitem{Ly4} S. Lysenko, Assumptions on the sheaf theory for the small FLE paper, available at\\ {\tt https://lysenko.perso.math.cnrs.fr/index\_sub.html}
\bibitem{Ly8} S. Lysenko, My notes of the embryo GL seminar, available at\\ {\tt https://lysenko.perso.math.cnrs.fr/index\_sub.html}
\bibitem{Ly9} S. Lysenko, Comments to Beilinson, Drinfeld, Chiral algebras, available at\\ {\tt https://lysenko.perso.math.cnrs.fr/index\_sub.html}
\bibitem{Ly2} S. Lysenko, Comments: Sheaves of categories and notion of 1-affineness, available at\\ {\tt https://lysenko.perso.math.cnrs.fr/index\_sub.html}
\bibitem{MV} I. Mirkovic, K. Vilonen, Geometric Langlands duality and
representations of algebraic groups over commutative rings, Ann. of Math., 166 (2007), 95 - 143, DOI: 10.4007/annals.2007.166.95
\bibitem{PT} G. Peschke, W. Tholen, Diagrams, Fibrations, and the Decomposition of Colimits, arXiv:2006.10890
\bibitem{R} S. Raskin, Chiral categories, version 4 Sept. 2019, available at\\ {\tt https://gauss.math.yale.edu/$\sim$sr2532/}
\bibitem{Ras2} S. Raskin, Chiral principal series categories I: finite-dimensional calculations,
Advances in Mathematics, Vol. 388 (2021), 107856, {\tt https://doi.org/10.1016/j.aim.2021.107856}
\bibitem{R2} S. Raskin, Chiral principal series categories II: the factorizable Whitaker category, available at\\ {\tt https://gauss.math.yale.edu/$\sim$sr2532/}
\bibitem{Reich} R. C. Reich, Twisted geometric Satake equivalence via gerbes on the factorizable grassmannian, Repr. Theory, An Electronic Journal of the AMS, Volume 16, 345 - 449 (2012), DOI: 10.1090/s1088-4165-2012-00420-4
\end{thebibliography}
\end{document}